\documentclass[a4paper, reqno, 11pt]{amsart}

\usepackage[english]{babel}
\usepackage{amsmath}
\usepackage{amssymb}
\usepackage{amsthm}
\usepackage{enumerate}
\usepackage{enumitem}
\usepackage{ifthen}
\usepackage{bbm}
\usepackage{color}
\provideboolean{shownotes} 
\setboolean{shownotes}{true}
\usepackage{hyperref}
\usepackage{graphicx}
\usepackage{pgfplots}
\usepackage{tikz,tikz-3dplot} 
\usepackage{mathtools}
\usepackage{comment}
\usepackage{float}
\usepackage{geometry}
\newcommand{\margnote}[1]{
\ifthenelse{\boolean{shownotes}}%
{\marginpar{\raggedright\tiny\texttt{#1}}}%
{}%
}

\newcommand{\hole}[1]{
\ifthenelse{\boolean{shownotes}}%
{\begin{center} \fbox{ \rule {.25cm}{0cm}
\rule[-.1cm]{0cm}{.4cm} \parbox{.85\textwidth}{\begin{center}
\texttt{#1}\end{center}} \rule {.25cm}{0cm}}\end{center}}
{}
}
\usepackage{mathrsfs}
\newtheorem{thm}{Theorem}[section]
\newtheorem{prop}[thm]{Proposition}
\newtheorem{lem}[thm]{Lemma}

\newtheorem{rem}[thm]{Remark}

\newtheorem{defn}[thm]{Definition}

\newcommand{\e}{\varepsilon}		       
\newcommand{\R}{\mathbb{R}}
\newcommand{\T}{\mathbb{T}}
\newcommand{\N}{\mathbb{N}}
\newcommand{\Z}{\mathbb{Z}}
\newcommand{\C}{\mathbb{C}}
\newcommand{\dive}{\mathop{\mathrm {div}}}
\newcommand{\curl}{\mathop{\mathrm {curl}}}

\newcommand{\de}{\mathrm{d}}	
\newcommand{\U}{\mathcal{U}}
\newcommand{\Lin}{\mathcal{L}}
\newcommand{\Id}{\mathrm{Id}}

\newcommand{\Lip}{\mathrm{Lip}}

\newcommand{\Op}{\mathcal{OPS}}

\newcommand{\balpha}{{\boldsymbol{\alpha}}}

\numberwithin{equation}{section}

\begin{document}

\title[Large amplitude traveling waves MHD]{Large amplitude traveling waves for the non-resistive MHD system}

\author[G. Ciampa]{Gennaro Ciampa}
\address[G.\ Ciampa]{Dipartimento di Matematica ``Federigo Enriques", Universit\`a degli Studi di Milano, Via Cesare Saldini 50, 20133 Milano, Italy.}
\email[]{\href{gciampa@}{gennaro.ciampa@unimi.it}}


\author[R. Montalto]{Riccardo Montalto}
\address[R.\ Montalto]{Dipartimento di Matematica ``Federigo Enriques", Universit\`a degli Studi di Milano, Via Cesare Saldini 50, 20133 Milano, Italy.}
\email[]{\href{rmontalto@}{riccardo.montalto@unimi.it}}

\author[S. Terracina]{Shulamit Terracina}
\address[S.\ Terracina]{Dipartimento di Matematica ``Federigo Enriques", Universit\`a degli Studi di Milano, Via Cesare Saldini 50, 20133 Milano, Italy.}
\email[]{\href{sterracina@}{shulamit.terracina@unimi.it}}

\begin{abstract}
We prove the existence of large amplitude bi-periodic traveling waves (stationary in a moving frame) of the two-dimensional non-resistive Magnetohydrodynamics (MHD) system with a traveling wave external force with large velocity speed $\lambda (\omega_1, \omega_2)$ and of amplitude of order $O(\lambda^{1^+})$ where $\lambda \gg 1$ is a large parameter. For most values of $\omega = (\omega_1, \omega_2)$ and for $\lambda \gg 1$ large enough, we construct bi-periodic traveling wave solutions of arbitrarily large amplitude as $\lambda \to + \infty$. More precisely, we show that the velocity field is of order $O(\lambda^{0^+})$, whereas the magnetic field is close to a constant vector as $\lambda \to + \infty$. Due to the presence of small divisors, the proof is based on a nonlinear Nash-Moser scheme adapted to construct nonlinear waves of large amplitude. The main difficulty is that the linearized equation at any approximate solution is an unbounded perturbation of large size of a diagonal operator and hence the problem is not perturbative. The invertibility of the linearized operator is then performed by using tools from micro-local analysis and normal forms together with a sharp analysis of high and low frequency regimes w.r. to the large parameter $\lambda \gg 1$. To the best of our knowledge, this is the first result in which global in time, large amplitude solutions are constructed for the 2D non-resistive MHD system with periodic boundary conditions and also the first existence results of large amplitude quasi-periodic solutions for a nonlinear PDE in higher space dimension.
\end{abstract}

\maketitle


\medskip


\tableofcontents

\section{Introduction}
In this paper we consider the two dimensional non-resistive MHD equations on the two-dimensional torus $\T^2$, $\T := \R /(2 \pi \Z)$
\begin{equation}\label{eq:mhd}
\begin{cases}
\partial_t u+(u\cdot \nabla)u+\nabla p= \Delta u+(b\cdot \nabla)b+  {\bf f}(t, x),\\
\partial_t b+(u\cdot \nabla)b=(b\cdot \nabla)u,\\
\dive u=0,\qquad
\dive b=0,
\end{cases}
\end{equation}
where the unknowns are the velocity field of the fluid $u: \R \times \T^2 \to \R^2$, the magnetic field $b : \R \times \T^2 \to \R^2$ and the pressure $p : \R \times \T^2 \to \R$, whereas the external force acting on the fluid ${\bf f} : \R \times \T^2 \to \R^2$ is given. This system describes the behavior of a viscous, electrically conducting, incompressible fluid, whose resistivity is negligible and it is commonly used to model strongly collisional plasmas in astrophysics. We refer to \cite{Cowling, Landau} for a derivation of the model, however we recall that the equations are given by a combination of the Navies-Stokes equations and the Faraday-Maxwell system via Ohm’s law. The equations \eqref{eq:mhd} were analyzed by Moffat \cite{Moff85} in the context of topological hydrodynamics with the aim to construct magneto-static equilibria (and thus also steady solutions of the Euler equations) with arbitrary prescribed topology. In fact, the second equation in \eqref{eq:mhd} is ``topology preserving" since the integral lines of the magnetic field $b$ are transported by the flow of the velocity $u$ as they evolve in time. Indeed, as long as one considers smooth solutions, the flow of $u$ is a diffeomorphism and then the magnetic lines keep their topology unchanged during the evolution. This is also known in the literature as Alfven's Theorem. The ambitious program of Moffat was to construct magneto-static equilibria in the infinite limit (w.r.t. time) of solutions to a topology preserving diffusive equation. We point out that Moffat's original proposal did not involve external forces in the system \eqref{eq:mhd}, but we refer to \cite{Nunez} for the forced case. This procedure is known as {\em magnetic relaxation} and some partial results have been obtained in recent years, see \cite{BFV, Brenier14, CP, Nunez}. We also point out that Enciso
and Peralta-Salas \cite{Annals} proved the existence of (infinite energy) stationary solutions of the Euler equations with vortex lines of prescribed topology, however it is an open problem whether such solutions arise
as limits of the system \eqref{eq:mhd}. In contrast, in the resistive case the topology of magnetic lines is expected to change leading to {\em magnetic reconnection} phenomena. The latter are of extreme relevance to astrophysicists being a phenomenon observed virtually everywhere in the universe.  We refer to \cite{CCL, MagnPhys, Priest} and references therein for an overview of this topic from a numerical and analytical point of view.\\

The mathematical aspects of the MHD equations have been extensively studied in recent decades, especially with regard to the Cauchy problem. Before going into the details of our result, we provide a (non-exhaustive) review of the literature on the problem, with a particular focus on the two-dimensional case. In the presence of the resistive term $\Delta b$ in the second equation, global existence and uniqueness of finite energy weak solutions (in 2D) has been proved in \cite{DL72}. Furthermore, if the initial datum and the external force are smooth, the well-posedness of \eqref{eq:mhd} (with the resistive term $\Delta b$) in 2D has been shown in \cite{ST}. In contrast, for the three-dimensional case we only have unique {\em local} smooth solutions, see again \cite{DL72, ST}. For the non-resistive case, the local existence of smooth solutions in 2D has been proved in \cite{JN} when the initial datum and the body force belong to the space $H^s$ with $s\geq 3$, while the 3D case was treated in \cite{FMRR, Feff} at a sharp level of Sobolev regularity. Moreover, global weak solutions in 2D for the system without the viscous term $\Delta u$ in the first equations but the resistive term $\Delta b$ in the second equations have been constructed in \cite{K}. This case is less difficult than \eqref{eq:mhd}. Indeed in \eqref{eq:mhd}, the equations for the magnetic field look like the vorticity formulation of the 3D Euler equation (the term $b \cdot \nabla u$ is similar to the streching term). 
Global well-posedness results in the three-dimensional \eqref{eq:mhd} equations are available under some geometrical constraints on the initial datum. In particular, special type of axisymmetric solutions have been constructed for the system \eqref{eq:mhd} in \cite{Lei}. See also \cite{Zineb MHD} for the analogous result for the system with no viscosity $\Delta u$ in the first equations but with resistive term $\Delta b$ in the second equation.

Finally, assuming that the initial datum $b_0$ is a small perturbation of a constant steady state, the existence of global-in-time smooth {\em small} solutions (namely the velocity field is small and the magnetic field is close to a constant steady state) of \eqref{eq:mhd} has been established in \cite{LXZ, RWXZ}. See also \cite{LZ, XZ} for similar results in the three dimensional case. However, the existence of global-in-time smooth solutions for \eqref{eq:mhd} with generic periodic initial data (no perturbation of a constant steady state, big velocity) is still open even in 2D. 

\smallskip

\noindent
The aim of this paper is to construct quasi-periodic solutions (global-in-time) of large amplitude for the system \eqref{eq:mhd} on the bi-dimensional torus $\T^2$. We consider the case of bi-periodic traveling wave solutions. We shall assume that the external force ${\bf f}$ is a smooth bi-periodic traveling wave propagating in the direction $\omega = (\omega_1, \omega_2) \in \R^2$, with large amplitude $\lambda^{1 + \eta}$ ($0 < \eta \ll 1$, $\lambda \gg 1$) and with large velocity speed $\lambda \omega$. More precisely, we take 
\begin{equation}\label{ipotesi media nulla f}
\begin{aligned}
&{\bf f}(t, x) = \lambda^{1 + \eta} f(x + \lambda \omega t) \quad \text{where} \\
&f \in {\mathcal C}^\infty(\T^2, \R^2)\,, \quad \int_{\T^2}f(x)\, d x = 0\,,\\
& \omega = (\omega_1, \omega_2) \in {\mathcal O} \quad \text{where} \quad {\mathcal O} \subset \R^2 \quad \text{is a bounded domain}, \\
& \lambda \gg 1 \quad \text{is a large parameter and } \quad 0 < \eta \ll 1 \quad \text{is a small parameter}. 
\end{aligned}
\end{equation}
Note that $\lambda^{1 + \eta}$ represents the amplitude of the forcing term and $\lambda$ is the size of the frequency (velocity speed). We shall prove that for $\lambda \gg 1$ large enough, $0 < \eta \ll 1$ small enough and for a full measure set of frequencies $\omega \in {\mathcal O}$, the system \eqref{eq:mhd} admits traveling wave solutions $u(x + \lambda \omega t)$, $b(x + \lambda \omega t)$, $p(x + \lambda \omega t)$ of {\it large amplitude} of order $O(\lambda^{0^+})$ as $\lambda \to + \infty$. This is a {\it small divisors} problem, typical in the KAM (Kolmogorov-Arnold-Moser) theory for PDEs, since the linearized equation at the origin has spectrum accumulating to zero for most values of the parameter $\omega$. The major difficulty in this problem is that we look for {\it large quasi-periodic solutions} and hence the problem is not perturbative: indeed, even at a formal level, if one tries to construct by perturbation theory the solution (by expanding in decreasing powers of $\lambda$), one gets that the size of the expected solution is essentially given by the ratio
\begin{equation}\label{ratio introduction}
\text{size of the solution} = \dfrac{\text{ size of the perturbation}}{\text{size of the frequency}} = \dfrac{O(\lambda^{1^+})}{O(\lambda)} = O(\lambda^{0^+}).
\end{equation}
Other difficulties are: the equations \eqref{eq:mhd} are strongly coupled, this makes quite difficult their analysis; the system \eqref{eq:mhd} is a system of PDEs in higher space dimension with derivatives in the nonlinearity, for which it is particularly difficult to prove KAM-type results. Indeed the problem that we study is even non trivial for small amplitude solutions. 

\noindent 
We overcome the afore-mentioned difficulties by implementing a Nash-Moser scheme adapted for constructing solutions of large amplitude. The key difficulty is the inversion of the linearized operator at any approximate solution which is the sum of a diagonal operator plus a large variable coefficients operator of size $O(\lambda^{0^+})$. In order to invert such an operator we develop normal form methods based on micro-local analysis and pseudo-differential calculus which allows to use the {\it stabilizing effect} of the large velocity speed $\lambda \omega$. This requires some sharp analysis in high and low frequency regimes w.r. to the large parameter $\lambda$ combined with tools from micro-local analysis. We refer to the next section for a more precise description of the main ideas of the proof. 

\noindent
We conclude this part of the introduction with the following remarks. 
Our result is the first one in which one shows the existence of {\it large amplitude} quasi-periodic solutions for a non-integrable PDE with a perturbation (forcing term) of large size. Up to now large amplitude quasi-periodic solutions have been constructed only for small perturbations of defocusing NLS and KdV equations, by exploiting the integrable structures of these two equations, see \cite{BKM18}, \cite{BKM}. We also mention that KAM and Normal Form techniques have been developed to study the dynamics of one dimensional linear wave and Klein Gordon equations with a time dependent quasi-periodic potential of size $O(1)$ and large frequencies (in which the ratio \eqref{ratio introduction} is small for $\lambda \gg 1$ large enough), see \cite{FranzoiMaspero}, \cite{Franzoi}.

\subsection{Main result} We are interested in the existence of {\em large amplitude traveling  wave} solutions of \eqref{eq:mhd}. As we have already mentioned, the unknowns are the magnetic field $b:\R\times\T^2 \to \R^2$, the velocity of the fluid $u: \R\times\T^2 \to \R^2$ and the pressure $p:\R\times\T^2 \to \R$, whereas, the data of the problem are the {\bf large parameter} $ \lambda \gg 1$, the small parameter $0 < \eta \ll 1$, $\omega = (\omega_1, \omega_2)\in {\mathcal O} \subset \R^2$ is the two-dimensional frequency vector in a bounded domain ${\mathcal O} \subset \R^2$ and the external force ${\bf f}$ is a bi-periodic traveling wave of the form \eqref{ipotesi media nulla f}. 
We then look for smooth solutions $(u, b, p)$ of the system \eqref{eq:mhd}, having the form
\begin{equation}\label{forma soluzioni 0}
\begin{aligned}
& u(t, x) := U(x + \lambda \omega t), \quad U : \T^2 \to \R^2\,, \quad \int_{\T^2 } U(x)\, d x = 0\,, \\
& b(t, x) := \mathbf b + B(x + \lambda \omega t), \quad \mathbf b \neq 0, \quad B : \T^2 \to \R^2, \quad \int_{\T^2} B(x)\, d x = 0\,, \\
& p(t, x) := P(x + \lambda \omega t), \quad \int_{\T^2} P(x)\, d x = 0\,, \\
& \text{with frequency} \quad \omega = (\omega_1, \omega_2) \in {\mathcal O} \subseteq \R^2\,. 
\end{aligned}
\end{equation}
The main feature of solutions of the form \eqref{forma soluzioni 0} is that {\it they look steady in the reference frame} $y_1 = x_1 + \lambda \omega_1 t$, $y_2 = x_2 + \lambda \omega_2 t$. This implies that  plugging the ansatz \eqref{forma soluzioni 0} in the equations \eqref{eq:mhd}, we obtain a stationary problem of the form 
\begin{equation}\label{eq:mhd 2}
\begin{cases}
\lambda \omega \cdot \nabla U +(U\cdot \nabla) U +\nabla P= \Delta U+( \mathbf b + B)\cdot \nabla B+ \lambda^{1 + \eta} f ,\\
\lambda \omega \cdot \nabla B +(U\cdot \nabla)B =(\mathbf b + B) \cdot \nabla U,\\
\dive U=\dive B =0
\end{cases}
\end{equation}
where $\omega \cdot \nabla = \omega_1 \partial_{x_1} + \omega_2 \partial_{x_2}$. We look for large amplitude solution (growing as $\lambda \to + \infty$) $(U, B, P) \in H^s_0(\T^2, \R^2) \times H^s_0(\T^2, \R^2) \times H^s_0(\T^2, \R)$ for $s \gg 0$ large enough where 
\begin{equation}\label{def sobolev introduzione}
\begin{aligned}
& H^s(\T^2, \R^n)  := \Big\{ u(x) = \sum_{k \in \Z^2} \widehat u(k) e^{i k \cdot x} \in L^2(\T^2, \R^n) : \| u \|_s  := \Big( \sum_{k \in \Z^2} \langle k \rangle^{2 s} |\widehat u(k)|^2 \Big)^{\frac12} < \infty \Big\}\,, \\
& H^s_0(\T^2, \R^n)  := \Big\{ u \in H^s(\T^2, \R^n) : \int_{\T^2} u(x)\, d x = 0 \Big\}\,. \\
\end{aligned}
\end{equation}
where $\langle k \rangle:= (1 + |k|^2)^{\frac12}$. As a notation, we often write $H^s \equiv H^s(\T^2) \equiv H^s(\T^2, \R)$ and $H^s_0 \equiv H^s_0(\T^2) \equiv H^s_0(\T^2, \R)$. We denote by ${\mathcal M}^{(n)}$ the Lebesgue measure on $\R^n$.

\noindent
 We shall make the following assumption on the forcing term $f$ and on the average of the magnetic field ${\bf b}$ that guarantees that we construct non trivial solutions.
\begin{itemize}
\item{\bf Assumption.} Let ${\mathbf b} \in \R^2 \setminus \{ 0 \}$ be the average of the magnetic field and let $f := (f_1, f_2) \in {\mathcal C}^\infty(\T^2, \R^2)$, $\int_{\T^2} f(x)\, d x = 0$. Let $F := {\rm curl}(f) := \partial_{x_1} f_2 - \partial_{x_2} f_1$, $F(x) = \sum_{k \in \Z^2 \setminus \{ 0 \}} \widehat F(k) e^{i x \cdot k}$.
\begin{equation}\label{assumption b f}
\begin{aligned}
& \text{There exists} \quad \overline k \in \Z^2 \setminus \{ 0 \} \quad \text{such that}  \\
& \mathbf b \cdot \overline k \neq 0 \quad \text{and} \quad \widehat F(\overline k) \neq 0\,. 
\end{aligned}
\end{equation}
\end{itemize}

 Our main result is the following.
\begin{thm}\label{thm:main}
Let $f \in {\mathcal C}^\infty(\T^2, \R^2)$ with zero average and $\mathbf b \in \R^2 \setminus \{ 0 \}$ satisfies \eqref{assumption b f}. There exist $\bar S>0$, $\eta \in (0, 1)$ such that for any $S\geq\bar S$, there exist $\lambda_0\equiv\lambda_0(f, \eta, S, \mathbf b)\gg 1$ and $C_1 \equiv C_1(f, \eta, S, \mathbf b)\,,\, C_2 \equiv C_2(f, \eta, S, \mathbf b) \gg 1$ such that for any $\lambda>\lambda_0$ the following holds. There exists a Borel set $\mathcal{O}_\lambda\subset\mathcal{O}$ of asymptotically full Lebesgue measure, i.e. $\lim_{\lambda\to+ \infty}{\mathcal M}^{(2)}(\mathcal{O}\setminus\mathcal{O}_\lambda)=0$, such that for every $\omega\in\mathcal{O}_\lambda$ there exists $\Big(U(\cdot; \omega),B(\cdot; \omega),P(\cdot; \omega) \Big) \in H^S_0(\T^2, \R^2) \times H^S_0(\T^2, \R^2) \times H^S_0(\T^2)$, $U, B, P \neq 0$ that solves the MHD equations \eqref{eq:mhd 2}. Moreover
\begin{equation}\label{stima U B P main}
\begin{aligned}
& \inf_{\omega\in \mathcal{O}_\lambda}\Big( \|U(\cdot;\omega)\|_S +\|B(\cdot;\omega)\|_S \Big)\geq \inf_{\omega\in \mathcal{O}_\lambda} \|U(\cdot;\omega)\|_S \geq  C_1 \lambda^\eta\,, \\
& \text{if} \quad {\rm div}(f) \neq 0 \quad \text{then } \quad  \inf_{\omega\in \mathcal{O}_\lambda} \|P(\cdot;\omega)\|_S  \geq C_1 \lambda^{1 + \eta}\\ 
& \sup_{\omega\in \mathcal{O}_\lambda} \|U(\cdot;\omega)\|_S  \leq C_2\lambda^{3 \eta}\,, \quad \sup_{\omega \in {\mathcal O}_\lambda} \|B(\cdot;\omega)\|_S \leq C_2 \lambda^{- 3 \eta}\,, \quad \sup_{\omega\in \mathcal{O}_\lambda} \|P(\cdot;\omega)\|_S \leq C_2 \lambda^{1 + \eta}\,. 
\end{aligned}
\end{equation}
\end{thm}
We now make some comments on the assumption \eqref{assumption b f}. This assumption is verified for instance if ${\bf b} \in \R^2 \setminus \{ 0 \}$ is an irrational vector, i.e. ${\bf b} \cdot k \neq 0$ for any $k \in \Z^2 \setminus \{ 0 \}$ and if $F = {\rm curl}(f) \neq 0$ is a non trivial forcing term. If ${\bf b}$ is a rational vector, for instance ${\bf b} = (1, 0)$, then the assumption \eqref{assumption b f} holds if (for instance) the forcing term $F$ satisfies $\widehat F(\overline k) \neq 0$ for $\overline k = (k_1, 0)$, $k_1 \in \Z \setminus \{ 0 \}$. 

\medskip

\noindent
Theorem \ref{thm:main} will be obtained by Theorem \ref{thm:main2} below, dealing with the vorticity formulation of the system \eqref{eq:vmhd}. We define the vorticity and the current density as
\begin{equation}\label{def vorticita e densita}
\begin{aligned}
& \Omega:=\partial_{x_1} U_2-\partial_{x_2} U_1, \quad U (x) = (U_1(x), U_2(x))\,,  \\
& J:=\partial_{x_1} B_2-\partial_{x_2} B_1\,, \quad B(x) = (B_1(x), B_2(x))\,. 
\end{aligned}
\end{equation}
Applying the $\curl$ operator to both the equations in \eqref{eq:mhd}, we obtain the following system
\begin{equation}\label{eq:vmhd}
\begin{cases}
\lambda \omega \cdot \nabla \Omega+(U\cdot \nabla)\Omega= \Delta \Omega+(\mathbf b + B )\cdot \nabla J+\lambda^{1 + \eta} F,\\
\lambda \omega \cdot \nabla J+(U \cdot \nabla)J=(\mathbf b + B )\cdot \nabla \Omega+2H(U,B),\\
U =\U\Omega:=\nabla^\perp (-\Delta)^{-1}\Omega,\hspace{0.3cm} B =\U J:=\nabla^\perp (-\Delta)^{-1}J,
\end{cases}
\end{equation}
where we recall that $F=\curl f = \partial_{x_1} f_2 - \partial_{x_2} f_1$, the orthogonal gradient is defined as $\nabla^\perp:=(\partial_{x_2},-\partial_{x_1})$, $(-\Delta)^{-1}$ denotes the inverse of the Laplacian (i.e. the Fourier multiplier with symbol $|\xi|^{-2}$ for $\xi\in\Z^2\setminus\{0\}$), and the bilinear form $H$ is defined by
\begin{equation}
H(U,B):=\partial_{x_1} B\cdot\nabla U_2-\partial_{x_2} B\cdot\nabla U_1.
\end{equation}
We remark that, once \eqref{eq:vmhd} is solved, the pressure is recovered from the first equation in \eqref{eq:mhd 2} via the elliptic equation
\begin{equation}\label{equazione P}
\Delta P=\lambda^{1 + \eta} {\rm div} f+\dive [(B\cdot\nabla)B]-\dive[(U\cdot\nabla)U].
\end{equation}
Since $\int_{\T^2}\Omega(x)\de x$ and $\int_{\T^2}J(x)\de x$ are conserved quantities, we shall restrict to the space of zero average vector fields in $x$. We look for large amplitude waves of order $ O(\lambda^{ 0^+})$, therefore it is convenient to rescale the variables 
\begin{equation}\label{riscalamento}
\begin{aligned}
& \Omega \mapsto \lambda^{ \delta} \Omega, \quad J \mapsto \lambda^{ \delta} J, \\
& \text{where we fix} \quad \delta := 3 \eta.
\end{aligned}
\end{equation}
Under the latter rescaling, the system \eqref{eq:vmhd} reads as 
\begin{equation}\label{mappa nonlineare iniziale}
\begin{aligned}
& {\mathcal F}(\Omega, J) = 0\,, \\
&{\mathcal F}(\Omega, J) := \begin{pmatrix}
\lambda \omega \cdot \nabla \Omega - \Delta \Omega  - \mathbf b \cdot \nabla J + \lambda^{ \delta } \Big[U\cdot \nabla\Omega -  B \cdot \nabla J \Big] - \lambda^{1 - \frac23 \delta} F \\
\lambda \omega \cdot \nabla J - \mathbf b \cdot \nabla \Omega + \lambda^{ \delta} \Big[ (U \cdot \nabla)J -  B \cdot \nabla \Omega - 2H(U,B) \Big]\, 
\end{pmatrix},\\
& U =\U\Omega:=\nabla^\perp (-\Delta)^{-1}\Omega,\hspace{0.3cm} B =\U J:=\nabla^\perp (-\Delta)^{-1}J
\end{aligned}
\end{equation}
and we look for solutions  $\Omega, J \in H^s_0(\T^2)$ for $s \gg 1$ large enough. The following theorem holds. 
\begin{thm}\label{thm:main2}
Let $f \in {\mathcal C}^\infty(\T^2, \R^2)$ with zero average and $\mathbf b \in \R^2 \setminus \{ 0 \}$ satisfies \eqref{assumption b f}. There exist $\bar S>0$, $\delta \in (0, 1)$ such that for any $S\geq\bar S$, there exist $\lambda_0=\lambda_0(f, \delta, S, \mathbf b)\gg 1$ and $C_1 \equiv C_1(f, \delta, S, \mathbf b)\,,\, C_2 \equiv C_2(f, \delta,  S, \mathbf b)>0$ such that for any $\lambda>\lambda_0$ the following holds. There exists  a Borel set $\mathcal{O}_\lambda\subset\mathcal{O}$ of asymptotically full Lebesgue measure, i.e. $\lim_{\lambda\to\infty}{\mathcal M}^{(2)}(\mathcal{O}\setminus\mathcal{O}_\lambda)=0$, such that for any $\omega\in\mathcal{O}_\lambda$, there exists  $(\Omega(\cdot;\omega), J(\cdot;\omega)) \in H^S_0(\T^2) \times H^S_0(\T^2)$ with $\Omega, J \neq 0$ which is a zero of $\mathcal{F}$ defined in \eqref{mappa nonlineare iniziale}, namely $\mathcal{F}\Big(\Omega(\cdot; \omega),J(\cdot; \omega) \Big)=0$ for any $\omega \in {\mathcal O}_\lambda$. Moreover
\begin{equation}\label{stima Omega J main}
\begin{aligned}
& \inf_{\omega\in \mathcal{O}_\lambda} \Big( \|\Omega(\cdot;\omega)\|_S +\|J(\cdot;\omega)\|_S \Big) \geq \inf_{\omega\in \mathcal{O}_\lambda}  \|\Omega(\cdot;\omega)\|_S   \geq C_1 \lambda^{- \frac23 \delta}\,,\\
& \sup_{\omega\in \mathcal{O}_\lambda}  \|\Omega(\cdot;\omega)\|_S \leq C_2\,, \quad \sup_{\omega\in \mathcal{O}_\lambda} \|J(\cdot;\omega)\|_S  \leq C_2 \lambda^{- 2 \delta} \\
\end{aligned}
\end{equation}
\end{thm}
Our approach is based on a set of techniques and tools (normal forms, microlocal analysis, Nash-Moser implicit function theorem, etc.) that are now commonly referred to as KAM (Kolmogorov-Arnold-Moser) techniques for PDEs. 
In recent years, these techniques have been successfully developed to construct nonlinear periodic and quasi-periodic waves of {\it small amplitude} in Fluid Mechanics. This has been the case of the two-dimensional water waves equations, see \cite{Rick Inv, BeM} for time quasi-periodic standing waves, \cite{Luca arma, Luca cpam, FG} for time quasi-periodic traveling wave solutions and \cite{Iooss1}, \cite{Iooss2} for diamond waves in the three-dimensional water waves equations. Moreover, we also mention the recent results on the evolution of vortex-patches for active-scalar equations \cite{BHM, Roulley1, GSIP, HHN, HHR, HassR, HR, Roulley}. By considering quasi-periodic external forces, time quasi-periodic solutions of the forced Euler equations have been constructed in \cite{BM}, in \cite{M} for the Navier-Stokes equations and in \cite{FM} quasi-periodic solutions of 2D Navier-Stokes equations were constructed in the vanishing viscosity limit and it was proved the convergence to quasi-periodic solutions of the forced Euler equation (uniformly and globally in time) as the viscosity goes to zero.

 We finally remark that quasi-periodic solutions to the (unforced) Euler equations were constructed also in \cite{CF, EPL-JDE} with a completely different approach (see also the recent extension to almost periodic solutions \cite{FranzoiMontaltoalmost}). In particular, the solutions are built by suitably gluing localized traveling profiles, so they are not of KAM type since no small divisors are involved. 
 \subsection{Main ideas, novelties and strategy of the proof} We construct large amplitude bi-periodic solutions of the 2D non-resistive MHD equations (we work in the vorticity-current formulation \eqref{mappa nonlineare iniziale}) of size $O(\lambda^\delta)$, $0 < \delta \ll 1$, $\lambda \gg 1$. Due to the presence of {\it small divisors}, the scheme of the proof is based on a Nash-Moser iteration scheme in scale of Sobolev spaces $H^s$. The main difficulty compared with previous works is that, since we produce solutions of large amplitude, the problem is not perturbative. The heart of the proof lies in the analysis of the linearized operator $\Lin$, associated to \eqref{eq:vmhd}, arising at each step of the Nash-Moser iteration. 

\noindent
For the whole introduction, we use the following notation. Given $\kappa, m \in \R$ we write $O(\lambda^\kappa |D|^m)$ to denote an operator of size $\lambda^\kappa$ and of order $m$.

\noindent
The linearized operator ${\mathcal L}$ is a variable coefficients perturbation of large size of a diagonal operator, namely it has the form 
\begin{equation}\label{linearized introduction}
\begin{aligned}
& {\mathcal L}  := {\mathcal D} + {\mathcal P}\,, \\
& {\mathcal D}  := \begin{pmatrix}
\lambda \omega \cdot \nabla - \Delta & 0 \\
0 & \lambda \omega \cdot \nabla 
\end{pmatrix}\,, \quad {\mathcal P} := \begin{pmatrix}
a(x) \cdot \nabla & d(x) \cdot \nabla \\
d(x) \cdot \nabla & a(x) \cdot \nabla
\end{pmatrix} + 
\begin{pmatrix}
\mathcal{R}_1 & \mathcal{R}_2 \\
\mathcal{R}_3 & \mathcal{R}_4
\end{pmatrix}\,, \\
& \text{where} \quad d(x), a(x) = O(\lambda^\delta) \quad \text{and}  \\
& {\mathcal R}_1, {\mathcal R}_2 = O(\lambda^\delta |D|^{- 1}) \ldots, {\mathcal R}_3\,,\, {\mathcal R}_4 = O(\lambda^\delta |D|^0) \,.  
\end{aligned}
\end{equation}
We want to show the invertibility of ${\mathcal L}$ and provide {\it tame estimates} for its inverse ${\mathcal L}^{- 1}$. The {\it formal idea} behind the procedure is the following. By imposing the standard diophantine condition on $\omega$, i.e. 
\begin{equation}\label{diofantea introduzione}
|\omega \cdot k| \geq \frac{\gamma}{|k|^\tau}, \quad \forall k \in \Z^2 \setminus \{ 0 \}, \quad 0 < \gamma \ll 1, \quad \tau \gg 0,
\end{equation}
one has that ${\mathcal D}$ is invertible with {\it loss of derivatives} and $\| {\mathcal D}^{- 1} \|_{{\mathcal B}(H^{s + \tau}, H^s)} = O(\lambda^{- 1})$. Hence formally ${\mathcal D}^{- 1} {\mathcal P} = O(\lambda^{\delta - 1}) \ll 1$ if $\delta \ll 1$ is small enough and $\lambda \gg 1$ is large enough. This heuristic argument suggests that the fact that the velocity speed is of size $O(\lambda)$ is the key point to enter in a perturbative regime. On the other hand, due to the small divisors, the Neumann series argument cannot be implemented directly and hence one needs to implement normal form arguments based on pseudo-differential calculus in order to reduce to a situation in which the heuristic argument described above can be made rigorous. As we mentioned already before, the major difficulty in this procedure is that the perturbation is large for $\lambda \gg 1$.  Now, we shall describe more in details the main ideas and points of such a procedure. 

\noindent
In order to invert the operator $\Lin$ we have to solve the linear system of 2 equations in 2 unknowns (recall \eqref{linearized introduction}) of the type
\begin{equation}\label{sistema lineare intro}
\begin{aligned}
& \begin{cases}
{\mathcal L}^{(1)} h_1 + {\mathcal L}^{(2)} h_2 = g_1,  \\
{\mathcal L}^{(3)} h_1 + {\mathcal L}^{(4)} h_2 = g_2\,, 
\end{cases} \\
& {\mathcal L}^{(1)} := \lambda \omega \cdot \nabla - \Delta + a(x) \cdot \nabla + {\mathcal R}_1\,, \\
& {\mathcal L}^{(2)} := d(x) \cdot \nabla + {\mathcal R}_2 \,, \quad  {\mathcal L}^{(3)} := d(x) \cdot \nabla + {\mathcal R}_3\,, \\
& {\mathcal L}^{(4)} := \lambda \omega \cdot \nabla  + a(x) \cdot \nabla + {\mathcal R}_4\,. 
\end{aligned}
\end{equation}
The leading idea to solve this system is to try to reduce it to the invertibility of a scalar transport type operator with large variable coefficients, that can be inverted by extending the strategy developed in \cite{BM}, \cite{FM} in the framework of small amplitude solutions. This reduction can be done since ${\mathcal L}^{(1)}$ is {\it dissipative}, due to the presence of the Laplacian and this avoid the small divisors problem at least in the first equation. On the other hand, we mention that this inversion is anyway subtle since the perturbation has large variable coefficients and requires a splitting in high and low frequencies with respect to the large parameter $\lambda$.
We can summarize the strategy of the inversion in three main steps that we shall describe separately below:
\begin{itemize}
\item[(1)] decoupling of the equations up to an arbitrarily regularizing remainder;
\item[(2)] inversion of the first equation;
\item[(3)] inversion of the second equation.
\end{itemize}
{\bf Decoupling of the equations up to an arbitrarily regularizing remainder.}

\noindent
First of all, in Section \ref{sezione decoupling ordine uno off diag}, we implement an iterative procedure that decouples the equations up to a remainder of large size but regularizing, namely we use pseudo-differential operators in order to conjugate  $\Lin$ to another linear operator of the form
\begin{equation}\label{cal L1 intro 0}
\Lin_1 = \begin{pmatrix}
{\mathcal L}_1^{(1)} & {\mathcal L}_{- N}^{(2)} \\
{\mathcal L}_{- N}^{(3)} & {\mathcal L}_1^{(4)}
\end{pmatrix},
\end{equation}
with
\begin{equation}\label{cal L1 intro 1}
\begin{aligned}
{\mathcal L}_1^{(1)} & =  \lambda \omega \cdot \nabla - \Delta +  a(x) \cdot \nabla +  {\mathcal R}_0^{(1)},\\
{\mathcal L}_4^{(1)}  & =  \lambda \omega \cdot \nabla  + a(x) \cdot \nabla + {\mathcal R}_0^{(4)},
\end{aligned}
\end{equation}
where ${\mathcal R}_0^{(1)},{\mathcal R}_0^{(4)} = O(\lambda^{\delta M} |D|^0)$ are operators of order $0$ and of size $O(\lambda^{\delta M})$ for some $M \equiv M(N) \gg 1$ large enough, whereas the remainders ${\mathcal L}_{- N}^{(2)},{\mathcal L}_{- N}^{(3)} = O(\lambda^{\delta M} |D|^{- N})$ are smoothing operators of order $-N$ (meaning that they are bounded linear operators from $H^s$ into $H^{s+N}$) and of size $O(\lambda^{\delta M})$. The presence of the Laplacian in the first equation is crucial in this procedure, since it makes possible to solve {\em homological equations} which in turn provides the generator of the transformations that decouple the equations. In order to simplify the exposition we shall describe the first step of such an iterative procedure in which we eliminate from ${\mathcal L}$ in \eqref{linearized introduction}, the highest order off-diagonal term 
\begin{equation}\label{kill off diag introduction}
{\mathcal Q}_1:= \begin{pmatrix}
0 & d(x) \cdot \nabla \\
d(x) \cdot \nabla & 0
\end{pmatrix}\,. 
\end{equation} 
We look for a map 
$$
\Phi := {\rm exp}(\Psi)\,, \quad \Psi := \begin{pmatrix}
0 & {\rm Op}(\psi_1(x, \xi)) \\
{\rm Op}(\psi_2(x, \xi)) & 0
\end{pmatrix},
$$ in such a way to eliminate-normalize the off-diagonal term ${\mathcal Q}_1$. If one computes $\Phi^{- 1} {\mathcal L} \Phi$, by a direct calculation, using pseudo-differential calculus, it turns out that in order to cancel the new off-diagonal term of order one, the symbols $\psi_1(x, \xi)$ and $\psi_2(x, \xi)$ has to satisfy 
\begin{equation}\label{omologica intro decoupling 0}
\begin{aligned}
& \lambda \omega \cdot \nabla \psi_1(x, \xi) + |\xi|^2 \psi_1(x, \xi) + i \,d(x) \cdot \xi = \text{lower order terms}, \\
& \lambda \omega \cdot \nabla \psi_2(x, \xi) - |\xi|^2 \psi_2(x, \xi) + i \,d(x) \cdot \xi = \text{lower order terms}.
\end{aligned}
\end{equation}
The two latter equations are solved in the same way, hence let us discuss how to solve the first one. The main issue is that $d(x) = O(\lambda^\delta) \gg 1$, but one needs estimates on $\psi_1$ which are uniformly bounded w.r. to $\lambda$, in such a way that ${\rm exp}(\Psi)$ has not bad divergences w.r. to the large parameter $\lambda$. A convenient way to proceed is then to introduce a suitable cut-off function $\chi_\lambda(\xi)$ which is supported on $|\xi| \geq \lambda^{6 \delta}$. Then, we split the symbol $i\, d(x) \cdot \xi$ as 
$$
i d(x) \cdot \xi = i \chi_\lambda(\xi) d(x) \cdot \xi  + i (1 - \chi_\lambda (\xi)) d(x) \cdot \xi \,. 
$$
The symbol $ (1 - \chi_\lambda (\xi)) d(x) \cdot \xi $ is smoothing and it is of order $- N$, $N \gg 1$ with estimates  
$$
|(1 - \chi_\lambda (\xi)) d(x) \cdot \xi | \lesssim \lambda^{6 \delta (N + 1)}\langle \xi \rangle^{- N} ,
$$
and hence we shall solve
\begin{equation}\label{omologica intro decoupling}
\lambda \omega \cdot \nabla \psi_1(x, \xi) + |\xi|^2 \psi_1(x, \xi) + i \chi_\lambda(\xi) d(x) \cdot \xi = 0\,. 
\end{equation}
This equation can be solved by expanding $\psi_1$ and $d$ in Fourier series w.r. to $x$ and one obtains that 
$$
\psi_1(x, \xi) = \sum_{k \in \Z^2} \widehat \psi_1(k, \xi) e^{i k \cdot x}, \quad d(x) = \sum_{k \in \Z^2} \widehat d(k) e^{i k \cdot x}, \quad \widehat \psi_1(k, \xi) = - \dfrac{i \chi_\lambda(\xi) \widehat d(k) \cdot \xi}{i \lambda \omega \cdot k + |\xi|^2}\,.
$$
By the latter formula, one has that $\psi_1$ is supported on $|\xi| \geq \lambda^{6 \delta}$ and (recall that $d(x) = O(\lambda^\delta)$),
$$
|\psi_1(x, \xi)| \lesssim \lambda^\delta |\xi|^{- 1} \stackrel{|\xi|^{\frac12} \geq \lambda^{3 \delta}}{\lesssim} \lambda^\delta \lambda^{- 3 \delta} |\xi|^{- \frac12} \lesssim \lambda^{- 2 \delta} |\xi|^{- \frac12}\,. 
$$
We then find $\psi_1, \psi_2$ of order $\frac12$ and of size $O(\lambda^{- 2 \delta}) \ll 1$ for $\lambda \gg 1$ large enough. The outcome of this procedure, by estimating all the terms in the expansion of ${\mathcal L}^{(0)} := \Phi^{- 1} {\mathcal L } \Phi$ is then given by 
$$
\begin{aligned}
& {\mathcal L}^{(0)} := {\mathcal D} + \begin{pmatrix}
a(x) \cdot \nabla & 0 \\
0 & a(x) \cdot \nabla
\end{pmatrix} + \begin{pmatrix}
0 & O(\lambda^{3 \delta} |D|^{\frac12}) \\
O(\lambda^{3 \delta} |D|^{\frac12}) & 0
\end{pmatrix} + O(\lambda^{3 \delta} |D|^{0}) +O(\lambda^{M \delta} |D|^{- N}), \\
&   M := 6 (N + 1)\,. 
\end{aligned}
$$ 
We remark that in ${\mathcal L}^{(0)}$, the size of the remainders is $O(\lambda^{3 \delta})$, which is worse than the ones in ${\mathcal L}$ (which is of order $O(\lambda^\delta)$), cf. \eqref{linearized introduction}. This is due to the presence of the term ${\rm Op}\Big(i d(x) \cdot \xi \big(\psi_1 - \psi_2 \big)\Big)$ which is of order $O(\lambda^{3 \delta} |D|^0)$, see \eqref{cancellazione psi 1 - psi 2}, which follows from the fact that ${\rm Op}(\psi_1 - \psi_2)$ is of order $- 1$, thanks to a cancellation exploited in Lemma \ref{lem:equazione-psi}-$(ii)$ (even if $\psi_1, \psi_2$ are only of order $- \frac12$). 
 In order to normalize the lower order terms, the procedure is quite similar to the one described above and the equations to be solved are similar to \eqref{omologica intro decoupling 0}, \eqref{omologica intro decoupling}. The only difference is that the size of the remainders to be normalized is $O(\lambda^{3 \delta})$. In order to get uniform estimates w.r. to $\lambda$ of the maps along the iteration, one uses that the cut-off function $\chi_\lambda$ is supported on $|\xi| \geq \lambda^{6 \delta}$. We refer to sub-section \ref{subsub:decit} for more details.

\medskip

\noindent
{\bf Inversion of the first equation and reduction to the second component.}

\noindent
After the previous decoupling procedure, we can invert the first equation (namely the operator ${\mathcal L}_1^{(1)}$ in \eqref{cal L1 intro 1}) by implementing a Neumann series argument. This is done in subsection \ref{sezione inversione cal L 1 1}. The operator ${\mathcal L}_1^{(1)}$ is of the form 
\begin{equation}\label{forma cal L 1 (1)}
\begin{aligned}
& {\mathcal L}_1^{(1)} := L_\lambda + {\mathcal R}_1^{(1)}\,, \\
& L_\lambda := \lambda \omega \cdot \nabla - \Delta\,, \quad {\mathcal R}_1^{(1)} = O(\lambda^{\delta M} |D|)\,. 
\end{aligned}
\end{equation}
In order to invert this operator by Neumann series we need that $L_\lambda^{- 1} {\mathcal R}_1^{(1)}$ is of order zero and small in size. Clearly 
$$
L_\lambda = {\rm diag}_{k \neq 0} i \lambda \omega \cdot k + |k|^2,
$$
and hence $|i \lambda \omega \cdot k + |k|^2| \geq |k|^2$ for any $k \in \Z^2 \setminus \{ 0 \}$, meaning that $L_\lambda^{- 1} \sim (- \Delta)^{- 1}$ gains two space derivatives. This fact then balance the fact that ${\mathcal R}_1^{(1)}$ is unbounded of order one. On the other hand, this argument is not enough to use Neumann series, since in this way $L_\lambda^{- 1} {\mathcal R}_1^{(1)} = O(\lambda^{\delta M} |D|^{- 1})$, whose size is still large w.r. to $\lambda \gg 1$. Hence, in order to overcome this problem we require a gain of only one derivative in the estimate of $L_\lambda^{- 1}$ and we gain also in size, in such a way to compensate $O(\lambda^{\delta M})$. This is done in Lemma \ref{lem:inversione lineare cal D lambda}. More precisely, for some constant $\mathtt p > 0$ (depending on the constant $\tau$ appearing in the small divisors \eqref{diofantea introduzione}), One shows that 
$$
L_\lambda^{- 1} = O(\lambda^{- \mathtt p} |D|^{- 1}),
$$
by analyzing separately the regimes $|k| \geq \lambda^{\mathtt p}$ and $|k| \leq \lambda^{\mathtt p}$. For large frequencies one uses that $|k|^2 \geq \lambda^{\mathtt p} |k|$ and for low frequencies one uses the diophantine condition on $\omega$. We remark that the dissipation provided by the laplacian $- \Delta$ is used in a crucial way to analyze the high frequency regime $|k| \geq \lambda^{\mathtt p}$. One then obtains that 
$$
L_\lambda^{- 1} {\mathcal R}_1^{(1)} = O(\lambda^{- {\mathtt p} + \delta M} |D|^0) ,
$$
which is small and bounded by choosing $0 < \delta \ll \frac{{\mathtt p}}{M}$ and $\lambda \gg 1$ large enough. The precise, quantitative Neumann series argument is performed in Lemma \ref{lemma inversione cal L 1 1} and one shows that $({\mathcal L}_1^{(1)})^{- 1} = O(\lambda^{- \mathtt p} |D|^{- 1})$. The inversion of ${\mathcal L}_1^{(1)}$ then reduces the invertibility of ${\mathcal L}_1$ in \eqref{cal L1 intro 0}, \eqref{cal L1 intro 1} to the invertibility of a transport type operator with large variable coefficients of the form 
\begin{equation}\label{operatore trasporto introduzione}
{\mathcal P} := \lambda \omega \cdot \nabla + a(x) \cdot \nabla + {\mathcal R}_0 , \quad {\mathcal R}_0 = O(\lambda^{\delta M} |D|^0),
\end{equation}
see \eqref{sistema cal L1 a}-\eqref{sistema cal L1 c} and Lemma \ref{lemma op cal T seconda equazione} in sub-section \ref{sezione riduzione seconda componente}.

\medskip

\noindent
{\bf Inversion of the transport operator ${\mathcal P}$.}

\noindent
The inversion of the transport operator ${\mathcal P}$ is done in section \ref{sezione riduzione seconda componente}. This is an adaptation of the methods developed in \cite{BM}, \cite{FM} (for small amplitude waves). We implement a normal form procedure transforming ${\mathcal P}$ to ${\mathcal P}_1$ where the  operator ${\mathcal P}_1$ is a smoothing perturbation of size $\lambda^{\delta M}$ of a diagonal operator, namely it has the form
\begin{equation}\label{forma cal P1}
\begin{aligned}
& {\mathcal P}_1  = {\mathcal D}_1  + {\mathcal R}_1\,, \\
& {\mathcal D}_1 := {\rm diag}_{k \neq 0} \mu(k)\,, \quad \mu(k) = i \lambda \omega \cdot k + O(\lambda^{\delta M}), \quad k \in \Z^2 \setminus \{ 0 \}\,, \\
& {\mathcal R}_1 = O(\lambda^{\delta M} |D|^{- N})\,. 
\end{aligned}
\end{equation}
This is proved in Propositions \ref{prop:riduzione1}, \ref{prop coniugio cal P con cal A} and \ref{proposizione regolarizzazione ordini bassi}. Also in this steps, it is essential that the size of the frequency is of order $O(\lambda)$. This allows to estimate uniformly in $\lambda$ the size of the normal form transformations that we need in order to transform ${\mathcal P}$ in ${\mathcal P}_1$. Then in sub-section \ref{sezione inversione trasporto}, we invert ${\mathcal P}_1$ (and then ${\mathcal P}$) by a Neumann series argument. By imposing diophantine conditions on the eigenvalues of ${\mathcal D}_1$ 
$$
|\mu(k)| \geq \frac{\lambda \gamma}{|k|^\tau}, \quad k \in \Z^2 \setminus \{ 0 \},
$$
and by choosing $N \sim \tau$, one gets that 
$$
{\mathcal D}_1^{- 1} {\mathcal R}_1 = O(\lambda^{\delta M - 1} |D|^0),
$$
which is small for $\delta \ll \frac{1}{M}$ and $\lambda \gg 1$ large enough. This last step then conclude the whole invertibility procedure of the linearized operator ${\mathcal L}$ which is summarized in Section \ref{inversione linearized totale}. We point out that by an appropriate choice of the parameters which enters in the procedure, ${\mathcal L}^{- 1}$ has size $\| {\mathcal L}^{- 1} \|_{{\mathcal B}(H^{s + \bar \sigma}, H^s)} = O(\lambda^{-  \delta M})$ (for some large constant $\bar \sigma \gg 1$). 

\medskip

\noindent
{\bf The nonlinear Nash-Moser scheme.} We conclude this explanation of the scheme of the proof with some comments on the nonlinear Nash-Moser scheme implemented on the map ${\mathcal F}$. The first problem is already in the first approximation of the solution. Indeed $(\Omega, J) = (0, 0)$ is not a good approximation, since ${\mathcal F}(0, 0) = O(\lambda^{1 + \eta - \delta}) = O(\lambda^{1^-})$ is not uniformly bounded w.r. to $\lambda$. Hence in section \ref{sezione soluzione approssimata}, we construct a non trivial approximate solution $(\Omega_{app}, 0)$, with $\Omega_{app} \neq 0$ satisfying the property that ${\mathcal F}(\Omega_{app}, 0)$ satisfies estimates which are uniform w.r. to the large parameter $\lambda$. This is enough in order to implement a rapidly convergent Nash-Moser scheme in Section \ref{sezione:NASH}. A crucial point is actually the gain of size $O(\lambda^{-  \delta M})$ in the estimate of ${\mathcal L}^{- 1}$, described above. This gain is actually used in order to obtain that the difference between two consecutive approximate solutions  converges to zero super exponentially and it is uniformly bounded w.r. to $\lambda \gg 1$. We point out that the average of the magnetic field $\mathbf b \neq 0$  in \eqref{forma soluzioni 0} and the forcing term $F$ have to satisfy the non-degeneracy condition \eqref{assumption b f} in order to guarantee that the solutions $(\Omega, J)$ that we produce are different from zero, see Proposition \ref{costruzione soluzioni approssimate} and Section \ref{sezione dim dei teoremi principali}. Indeed, to show that $\Omega \neq 0$, one provides a lower bound on $\Omega_{app}$ (in Proposition \ref{costruzione soluzioni approssimate}) and then by a perturbative argument one also shows that $\Omega$ satisfies the same lower bound since the difference $\Omega - \Omega_{app}$ is small. The proof that $J \neq 0$ is more delicate and it is done by contraddiction. The key point is to shows that ${\bf b} \cdot \nabla \Omega$ is not identically zero. This is true (by the assumption \eqref{assumption b f}) for the approximate solution $\Omega_{app}$ and it is proved also for $\Omega$ by a perturbative argument. 
%
%
%

\medskip

\noindent
\subsection*{Remarks and future perspectives} We conclude this introduction with some final remarks and future perspectives. By using the techniques developed in this paper (with some techinal adaptations), one can also prove the same kind of result for the inviscid, resistive MHD system, namely there is not the viscous term $\Delta u$ in the first equations of \eqref{eq:mhd}  whereas, there is the resistive term $\Delta b$ in the second equations of \eqref{eq:mhd}. As we described above, the presence of the dissipation in one of the two equations is used in a crucial way in the analysis of the linearized equation and in particular it allows to (approximately) decouple the system in a parabolic equation and a trasnport equation. The decoupling procedure for the non resistive inviscid case (no $\Delta u$ in the first equation, no $\Delta b$ in the second equation) seems to be quite difficult and requires new ideas. A natural extension of our result would be to analyze the general case of quasi periodic solutions with an arbitrary number of frequencies $\omega = (\omega_1, \ldots, \omega_\nu)$. In this case the solutions one look for are not stationary in a moving frame, hence space and time are really independent unlike the case of diamond waves where the time derivative becomes a space derivative. A major issue in the extension of our strategy seems to be in the reduction procedure to the scalar transport equation, in which one has to invert the dissipative operator $\partial_t - \Delta$ on the space of time quasi-periodic functions $u(\omega t, x)$, $u : \T^\nu \times \T^2 \to \R$. In this inversion then one loses the ``structure of dynamical system" and this seems to be a non-trivial technical problem in implementing KAM-Normal forms methods for reducing the linearized equation to constant coefficients. Other future perspectives are the analysis of the three dimensional case and the analysis of the inviscid case.

\subsection*{Outline of the paper} The paper is organized as follows. In Section \ref{sezione functional setting} we recall the functional setting and some general lemmata which will then be used consistently throughout the paper. In Section \ref{sezione linearizzato}, we compute the linearized operator $\mathcal{L}$ of the vorticity-current system \eqref{eq:vmhd} and we state some properties of it that will be used in the sequel. We then implement the decoupling procedure of $\Lin$ (up to smoothing remainders) in Section \ref{sezione decoupling ordine uno off diag}. Section \ref{inversione prima equazione calore} will be devoted to the inversion of the first equation of the decoupled linear operator, whereas in Section \ref{sezione op trasporto} we invert the transport operator ${\mathcal P}$ and in Section \ref{inversione linearized totale} we conclude the invertibility of the linearized operator ${\mathcal L}$. In Section \ref{sezione soluzione approssimata}, we construct an approximate solution that will be the starting point of the Nash-Moser iteration. The convergence of the Nash-Moser scheme and the measure estimates on the set of non-resonant frequencies will be shown in Section \ref{sezione:NASH} and \ref{sezione stime di misura} respectively. Finally, the proof of the main theorems is provided in Section \ref{sezione dim dei teoremi principali}.

\subsection*{Acknowledgements}
G. Ciampa, R. Montalto and S. Terracina are supported by the ERC STARTING GRANT 2021 ``Hamiltonian Dynamics, Normal Forms and Water Waves" (HamDyWWa), Project Number: 101039762. Views and opinions expressed are however those of the authors only and do not necessarily reflect those of the European Union or the European Research Council. Neither the European Union nor the granting authority can be held responsible for them.

\noindent
The work of the author Riccardo Montalto is also supported by PRIN 2022 ``Turbulent effects vs Stability in Equations from Oceanography" (TESEO), project number: 2022HSSYPN. 

Riccardo Montalto and Shulamit Terracina are also supported by INDAM-GNFM and Gennaro Ciampa is supported by INDAM-GNAMPA. 
\medskip

\noindent
The authors warmly thank Luca Franzoi, Renato Luc\`a and Michela Procesi for many useful discussions and comments. 

\section{Functions spaces, norms and linear operators}\label{sezione functional setting}
In this section we fix the notation and we state some techincal tools concerning function spaces and pseudo-differential operators.
We consider
\begin{equation}\label{lower bound indice s0}
s_0>  5 \quad \text{that can be arbitrarily large but fixed for the whole paper}
\end{equation}

\medskip

\noindent
{\bf Notations.} In the whole paper, the notation $ A \lesssim_{s, m, \alpha} B $ means
that $A \leq C(s, m, \alpha) B$ for some constant $C(s, m,  \alpha) > 0$ depending on 
the Sobolev index $ s $, the constants $ \alpha, m $. If $s = s_0$ we simply write $\lesssim\,,\, \lesssim_{m, \alpha}$ instead of $\lesssim_{s_0}, \lesssim_{s_0, m, \alpha}$. We always omit to write the dependence on  $\tau$, which is the constant appearing in the non-resonance conditions (see for instance \eqref{diofantea introduzione} in the introduction). 

\noindent
We denote by $\N$ the set of the positive integer numbers $\N = \{1,2 \ldots \}$ and $\N_0 := \N \cup \{ 0 \}$. Given $f : \Omega \to \C$, $\Omega \subseteq \R^d$ and a multi-index $\beta = (\beta_1, \ldots, \beta_d) \in \N_0^d$, we write $\partial_x^\beta f = \partial_{x_1}^{\beta_1} \ldots \partial_{x_d}^{\beta_d} f(x_1, \ldots, x_d)$. We also define the lenght of a multi-index as $|\beta| = \beta_1 + \ldots + \beta_d$. 

\noindent
For any $s\ge 0$ we define the scale of Sobolev spaces 
\begin{equation}\label{def sobolev}
\begin{aligned}
& H^s(\T^2, \mathbb K^n)  := \Big\{ u(x) = \sum_{k \in \Z^2} \widehat u(k) e^{i k \cdot x} \in L^2(\T^2, \mathbb K^n) : \| u \|_s  := \Big( \sum_{k \in \Z^2} \langle k \rangle^{2 s} |\widehat u(k)|^2 \Big)^{\frac12} < \infty \Big\}\,, \\
& H^s_0(\T^2, \mathbb K^n)  := \Big\{ u \in H^s(\T^2, \mathbb K^n) : \int_{\T^2} u(x)\, d x = 0 \Big\}\,. \\
\end{aligned}
\end{equation}
where $\mathbb K= \C, \R$.
As a notation, we often write $H^s \equiv H^s(\T^2) \equiv H^s(\T^2, \mathbb K)$ and $H^s_0 \equiv H^s_0(\T^2) \equiv H^s_0(\T^2, \mathbb K)$. 

 For $s>s_0$ one has $H^s(\T^{2})\subset {\mathcal C}^0(\T^{2})$ where ${\mathcal C}^0(\T^2)$ is the set of continuous functions on $\T^2$, and $H^s(\T^{2})$ is an algebra and they have the interpolation structure
\begin{equation}
\|uv\|_s\leq C(s)\|u\|_s\|v\|_{s_0}+C(s_0)\|u\|_{s_0}\|v\|_s,
\end{equation}
where $C(s),C(s_0)$ are positive constants independent from $u,v$. 

We also denote by $L^2_0(\T^2, \mathbb K^n)$, the space of $L^2$ functions from $\T^2$ to $\mathbb K^n$ with zero average. Given two Banach spaces $X_1, X_2$, we denote by ${\mathcal B}(X_1, X_2)$ the space of bounded, linear operators from $X_1$ to $X_2$ equipped by the standard operator norm $\| \cdot \|_{{\mathcal B}(X_1, X_2)}$. If $X_1 = X_2$ we write ${\mathcal B}(X_1) \equiv {\mathcal B}(X_1, X_2)$. 

\begin{defn}[Weighted norms]
Let ${\bf O} \subset \R^2$ be a bounded subset of $\R^2$ and let $(X, \| \cdot \|_X)$ be a Banach space. Given $\gamma \in (0, 1)$ and a Lipschitz function $u : {\bf O} \to X$, we define 
$$
\begin{aligned}
& \| u \|_X^{\rm sup} := \sup_{\omega \in {\bf O}} \| u(\omega) \|_X\,, \quad \| u \|_X^{\Lip} := \sup_{\begin{subarray}{c}
\omega_1, \omega_2 \in {\bf O} \\
\omega_1 \neq \omega_2
\end{subarray}} \dfrac{\| u(\omega_1) - u(\omega_2) \|_X}{|\omega_1 - \omega_2|}, \\
& \| u \|_X^{\Lip(\gamma)} := \| u \|_X^{\rm sup} + \gamma \| u \|_X^{\Lip}\,. 
\end{aligned}
$$
If $X = H^s(\T^2)$, we write $\| \cdot \|_s^{\rm sup}, \| \cdot \|_s^{\Lip}$ instead of $\| \cdot \|_{H^s}^{\rm sup}, \| \cdot \|_{H^s}^{\Lip}$ and 
\begin{equation}\label{Sbolev lip gamma}
\| u \|_s^{\Lip(\gamma)} := \| u \|_s^{\rm sup} + \gamma \| u \|_{s - 1}^{\Lip}\,. 
\end{equation} 
If $X = \R, \C$, we write $| \cdot |^{\rm sup}, |\cdot|^{\Lip}, |\cdot|^{\Lip(\gamma)}$ instead of $\| \cdot \|_X^{\rm sup}, \| \cdot \|_X^{\Lip}\,,\, \| \cdot \|_X^{\Lip(\gamma)}$. 
\end{defn}
For any $N\in \N$ we define the smoothing operators (Fourier truncations)
\begin{equation}\label{def:smoothing}
\Pi_N u (x):=\sum_{|k|\leq N} \widehat u(k)e^{i k\cdot x},\,\,\,\,\Pi_N^\perp={\rm Id}-\Pi_N.
\end{equation}
\begin{lem}[Smoothing estimates]\label{lem:smoothing}
The operators $\Pi_N,\Pi^\perp_N$ satisfy the smoothing estimates
\begin{equation}
\|\Pi_N u\|_s^{\Lip(\gamma)}\leq N^a \|u\|_{s-a}^{\Lip(\gamma)},\,\,\,0\leq a\leq s,
\end{equation}
\begin{equation}
\|\Pi_N^\perp u\|_s^{\Lip(\gamma)}\leq N^{-a}\|u\|_{s+a}^{\Lip(\gamma)},\,\,\,\forall a\geq 0.
\end{equation}
\end{lem}

\begin{lem}\label{sobolev tame}
For all $s\geq s_0$, $u, v \in H^s(\T^2)$
\begin{equation}\label{interpolazione bassa alta}
\|uv\|_s^{\Lip(\gamma)}\leq C(s)\|u\|_s^{\Lip(\gamma)}\|v\|_{s_0}^{\Lip(\gamma)}+C(s_0)\|u\|_{s_0}^{\Lip(\gamma)}\|v\|_s^{\Lip(\gamma)}.
\end{equation}
Moreover for $s_1 \leq s \leq s_2$ one has the following interpolation inequality 
\begin{equation}\label{interpolazione s1 s2 s}
\| u \|_s^{\Lip(\gamma)} \lesssim (\| u \|_{s_1}^{\Lip(\gamma)})^\lambda ( \| u \|_{s_2}^{\Lip(\gamma)})^{1 - \lambda}, \quad \lambda = \frac{s_2 - s}{s_2 - s_1}, \quad 1 - \lambda = \frac{s - s_1}{s_2 - s_1}.
\end{equation}
As a consequence, one has that for any $a_0,b_0\geq 0$, $p,q>0$ 
\end{lem}

\subsection{Diophantine equations}
Let $ {\mathbf O} \subset \R^2$ be a bounded domain of $\R^2$. Given $\gamma \in (0, 1)$, $\tau > 0$, we define the set of Diophantine vectors ${\rm DC}(\gamma,\tau)$, where 
\begin{equation}\label{def:DC}
{\rm DC}(\gamma,\tau):=\left\{\omega\in {\mathbf  O} :|\omega\cdot k |\geq\frac{\gamma}{|k|^\tau}, \quad \forall k \in\Z^2\setminus\{0\}\right\},
\end{equation}
the equation
\begin{equation}\label{eq:tr}
\omega\cdot\nabla v=u,
\end{equation}
with $u$ satisfying $\int_{\T^2}u(x)\de x=0$, has the periodic solution
$$
v(x):=(\omega\cdot\nabla)^{-1}u(x)=\sum_{k \in\Z^2\setminus\{0\}}\frac{\widehat{u}(k)}{i\omega\cdot k}e^{i x \cdot k}.
$$
The following estimates hold (see \cite[Lemma 2.2]{BKM18}).
\begin{lem}\label{lem:trasporto}
Let $s\geq s_0$ and $\omega\in {\rm DC}(\gamma,\tau)$. Then, for any $u\in H^{s+\tau}(\T^2)$ with $\hat{u}(0)=0$, the linear equation \eqref{eq:tr} has a unique solution $v := (\omega \cdot \nabla)^{- 1} u\in H^s_0(\T^2)$ with $\hat{v}(0)=0$ such that
\begin{equation}
\|(\omega\cdot\nabla)^{-1}u\|_s\leq C\gamma^{-1}\|u\|_{s+\tau},
\end{equation}
If $u=u_\omega\in H^{s+2\tau+1}_0(\T^2)$ is Lipschitz continuous in $\omega\in{\bf O}\subseteq\R^2$, then the solution $v=v_\omega\in H^s_0(\T^2)$ is Lipschitz continuous in ${\rm DC}(\gamma,\tau)$ and satisfies the estimate
\begin{equation}
\|(\omega\cdot\nabla)^{-1}u\|_s^{\mathrm{Lip}(\gamma)}\leq C\gamma^{-1}\|u\|_{s+2\tau+1}^{\mathrm{Lip}(\gamma)}.
\end{equation}
\end{lem}


\subsection{Linear operators}
\begin{defn}
Let $\mathcal{A}$ be a linear operator acting on $L^2(\T^2)$. We say that an operator $\mathcal{A}$ is real if it maps real valued functions into real valued functions. 
\end{defn}
\begin{lem}
Given  $\mathcal{A}_{1}, \mathcal{A}_{2}$ real linear operators, then $\mathcal{A}_{1} \circ \mathcal{A}_{2}$ and $\mathcal{A}_{1}^{-1}$ are real operators.
\end{lem}

We represent a real linear operator $\mathcal{R}$ acting on $(f,g)\in L^2(\T^2,\R^2)$ by a matrix
\begin{equation}\label{matrice.rap}
\mathcal{R}{f \choose g}=\begin{pmatrix}
\mathcal{A}_1 & \mathcal{A}_2 \\
\mathcal{A}_3 & \mathcal{A}_4
\end{pmatrix}{f\choose g},
\end{equation}
where $\mathcal{A}_1,\mathcal{A}_2,\mathcal{A}_3,\mathcal{A}_4$ are real operators acting on the scalar components $f,g\in L^2(\T^2)$. The action of an operator $\mathcal{A}\in {\mathcal B}(L^2(\T^2))$ on a function $u\in L^2(\T^2,\R^2)$ can be defined as
\begin{equation}\label{def:op-matrice}
\mathcal{A}u(x)=\sum_{k,k'\in\Z^2}\hat{\mathcal{A}}_k^{k'}\hat{u}(k')e^{i x\cdot k}.
\end{equation}
So, we can identify an operator $\mathcal{A}$ with the matrix $(\hat{\mathcal{A}}_k^{k'})_{k,k'\in\Z^2}$.

\begin{defn}\label{def:diagonal}
Let $\mathcal{A}$ be an operator as in \eqref{def:op-matrice}. We define $\mathcal{D}_{\mathcal{A}}$ as the operator
\begin{equation}
\mathcal{D}_{\mathcal{A}}:=\mathrm{diag}_{k\in\Z^2}\hat{\mathcal{A}}_k^k,\qquad (\mathcal{D}_{\mathcal{A}})_k^{k'}=\begin{cases}
\hat{\mathcal{A}}_k^{k'} & k=k',\\
0 &\mbox{otherwise}.
\end{cases}
\end{equation}
In particular, we say that the operator $\mathcal{A}$ is diagonal if $\mathcal{D}_{\mathcal{A}}=\mathcal{A}$.
\end{defn}

\subsection{Pseudo-differential operators and norms}
\begin{defn}[Pseudo-differential operators and symbols]\label{def:pseudo-diff}
Let $m \in \R$, $s \geq 0$ and $\alpha \in \N_{0}$. We say that a function $a : \T^2 \times \R^2 \to \C$ belongs to ${\mathcal S}^m_{s, \alpha}$ if there exists a constant $C(s, \alpha) > 0$ such that for any $\beta \in \N^2$, $|\beta| \leq \alpha$ satisfies the inequality
\begin{equation}\label{def:symb}
\| \partial^\beta_\xi a(\cdot ,\xi) \|_s\leq C(s, \alpha)\langle\xi\rangle^{m-|\beta|}, \quad \forall \xi \in \R^2.
\end{equation}
We say that a linear operator ${\mathcal A}$ belongs to the class $\Op_{s, \alpha}^m$ if there is $a \in {\mathcal S}^m_{s, \alpha}$ such that ${\mathcal A} = {\rm Op}(a)$, namely 
$$
{\mathcal A} u(x) := \sum_{k \in \Z^2} a(x, k) \widehat u(k) e^{i x \cdot k}, \quad u(x) = \sum_{k \in \Z^2} \widehat u(k) e^{i x \cdot k}\,. 
$$
We set 
$$
{\mathcal S}^m := \bigcap_{\begin{subarray}{c}
s \geq 0 \\
\alpha \in \N
\end{subarray}} {\mathcal S}^m_{s, \alpha}\,, \quad \Op^m := \bigcap_{\begin{subarray}{c}
s \geq 0 \\
\alpha \in \N
\end{subarray}} \Op^m_{s, \alpha},
$$ 
which are the classes of classical symbols and classical pseudo-differential operators of order $m$. We also denote by ${\mathcal S}^{- \infty} := \cap_{m \in \R} {\mathcal S}^m$ the class of smoothing symbols and $\Op^{- \infty} := \cap_{m \in \R} \Op^{m}$ the class of smoothing pseudo-differential operators. 

\noindent
For a matrix of pseudo-differential operators
\begin{equation}\label{matricidisimboli}
\mathcal{A}=
\begin{pmatrix}
\mathcal{A}_1 & \mathcal{A}_2\\
\mathcal{A}_3 & \mathcal{A}_4
\end{pmatrix},
\end{equation} 
we say that $\mathcal{A}\in \Op^m_{s, \alpha}$ if ${\mathcal A}_i \in \Op^m_{s, \alpha}$ and $\mathcal{A}\in \Op^m$ if ${\mathcal A}_i \in \Op^m$,  $i = 1,2,3,4$. 
\end{defn}
We will say that $\mathcal{A}$ is a block diagonal matrix operator if $\mathcal{A}_2=\mathcal{A}_3=0$. On the other hand, we will say that that $\mathcal{A}$ is a block off-diagonal matrix operator if $\mathcal{A}_1=\mathcal{A}_4=0$. If instead an operator is diagonal in the sense of the Definition \ref{def:diagonal}, it will be specified.

\medskip 

The following characterization for real pseudo-differential operators holds.
\begin{lem}[Lemma 2.10 \cite{BeM}]
Let $\mathcal A = {{\rm Op}(a(x,\xi))}$ be a pseudo-differential operator. Then $\mathcal A $ is real if and only if the symbol $a(x, \xi) = \overline{a(x, - \xi)}$ for any $(x, \xi) \in \T^2 \times \R^2$.
\end{lem}

We now introduce a norm which controls the regularity in $x$, and the decay in $\xi$, of the symbol $a(x, \xi)\in\mathcal S^{m}_{s, \alpha}$, together with its derivatives $\partial_{\xi}^{\beta}a(x,\xi)\in \mathcal S^{m-|\beta|}_{s, \alpha}$, $|\beta|\le \alpha$ in the Sobolev norm $\|\cdot \|_{s}$.

\begin{defn}[Weighted norms]\label{def:family-pseudo-diff}
Let $m\in\R$, $s\geq 0$, $\alpha \in\N_{0}$. If $\mathcal{A} = {\rm Op}(a) \in \Op^m_{s, \alpha}$, we define the norm
\begin{equation}
|\mathcal{A}|_{m,s,\alpha}:=\sup_{|\beta|\leq \alpha}\sup_{\xi\in\R^2}\|\partial^\beta_\xi a(\cdot,\xi) \|_s \langle\xi\rangle^{-m+|\beta|}<\infty;
\end{equation}

If $\mathcal{A}=\mathcal{A}(\omega)$ depends in a Lipschitz way on a parameter $\omega \in {\bf O} \subset \R^2$, we define
\begin{equation}
|\mathcal{A}|_{m,s, \alpha}^{\Lip(\gamma)}:=\sup_{|\beta|\leq \alpha}\sup_{\xi\in\R^2}\|\partial^\beta_\xi a(\cdot,\xi) \|_s^{\Lip(\gamma)} \langle\xi\rangle^{-m+|\beta|}.
\end{equation}

The norm of a $2\times 2$ matrix of pseudo-differential operators
in  $\Op^m_{s, \alpha}$ of the form \eqref{matricidisimboli}
is 
\begin{equation}\label{simbolimatrici}
	| \mathcal{A}|_{m,s, \alpha}^{\Lip(\gamma)}:=\max\{ |\mathcal{A}_{i}|_{m, s, \alpha}^{\Lip(\gamma)}  , i = 1,2,3,4\} \, . 
\end{equation}

\end{defn}

The pseudo-differential norm $|\cdot|_{m,s,\alpha}^{\Lip(\gamma)}$ satisfies the following elementary properties: for any $s\leq s'$, $\alpha\leq \alpha'$, and $m\leq m'$,
\begin{equation}
|\cdot|_{m,s,\alpha}^{\Lip(\gamma)}\leq |\cdot|_{m,s',\alpha}^{\Lip(\gamma)},\,\,\, |\cdot|_{m,s,\alpha}^{\Lip(\gamma)}\leq |\cdot|_{m,s,\alpha'}^{\Lip(\gamma)},\,\,\,|\cdot|_{m',s,\alpha}^{\Lip(\gamma)}\leq |\cdot|_{m,s,\alpha}^{\Lip(\gamma)}.
\end{equation}
For a Fourier multiplier ${\rm Op}(g(\xi))$ of order $m$, one has
\begin{equation}\label{fourier multiplier norm}
|{\rm Op}(g)|_{m,s,\alpha}^{\Lip(\gamma)}=|{\rm Op}(g)|_{m,0,\alpha}^{\Lip(\gamma)}\leq C(m,\alpha),\,\,\forall s\geq 0,
\end{equation}
and for a function $a(x;\omega)$,
\begin{equation}\label{multiplication norms}
|{\rm Op}(a)|_{0,s,\alpha}^{\Lip(\gamma)}=|{\rm Op}(a)|_{0,s,0}^{\Lip(\gamma)}\lesssim \|a\|_s^{\Lip(\gamma)}.
\end{equation}


\textbf{Composition.} If $\mathcal{A}=\mathrm{Op}(a)\in \Op^m$ and $\mathcal{B}=\mathrm{Op}(b)\in \Op^{m'}$, then the composition $\mathcal{A}\mathcal{B}:=\mathcal{A}\circ\mathcal{B}$ is a pseudo-differential operator with symbol $\sigma_{{\mathcal A}{\mathcal B}}\in {\mathcal S}^{m+m'}$
\begin{equation}\label{def:sigma-comp}
\sigma_{{\mathcal A}{\mathcal B}}(x,\xi)
:=\sum_{k\in\Z^2}a(x,\xi+k)\widehat{b}(k,\xi)e^{ix\cdot k}=\sum_{k,k'\in\Z^2}\hat{a}(k'-k,\xi+k)\widehat{b}(k,\xi)e^{ix\cdot k'},
\end{equation}
where the symbol $\hat{\cdot}$ denotes the Fourier coefficient with respect to $x$. We will also use the notation $a\# b$ to express the symbol of the composition. Moreover, for any $N\geq 0$, $\sigma_{AB}$ admits the asymptotic expansion
\begin{equation}\label{espansione.composizione.simboli}
\sigma_{{\mathcal A}{\mathcal B}}(x,\xi)
:=\sum_{|\beta| \leq N - 1}\frac{1}{i^{|\beta|}\beta!} \partial_\xi^\beta a(x,\xi)\partial_x^\beta b(x,\xi)+r_N(x,\xi),
\end{equation}
with remainder $r_N \in {\mathcal S}^{m+m'-N}$. The remainder $r_N$ has the explicit formula
\begin{equation}\label{remainder cal RN composizione}
r_N(x,\xi):=\frac{1}{i^{N} (N - 1)!}\int_0^1 (1-\tau)^{N - 1} \sum_{|\beta| = N}\sum_{k\in\Z^2} (\partial^{\beta}_\xi a)(x,\xi+ \tau k)\widehat{\partial^{\beta}_x b}(k,\xi)e^{ix\cdot k}\de\tau.
\end{equation}
We recall the following tame estimate for the composition of two pseudo-differential operators.
\begin{lem}[Composition of pseudo-differential operators]\label{lem:composition-pseudodiff}
Let $s\geq s_0$, $m,m'\in\R$, $\alpha\in \N_{0}$. 
\begin{itemize}
\item[(i)] Let $\mathcal{A}=\mathrm{Op}(a)\in \Op^m_{s, \alpha}$, $\mathcal{B}=\mathrm{Op}(b)\in \Op^{m'}_{s, \alpha}$. Then the product $\sigma_{ab}(x, \xi) := a(x, \xi) b(x, \xi)$ satisfies for any $s \geq s_0$, $\alpha \in \N_0$ the estimate
\begin{equation}\label{prodotto.simboli}
|{\rm Op}(\sigma_{ab})|_{m + m', s, \alpha}^{\Lip(\gamma)} \lesssim_{s, \alpha} |{\mathcal A} |_{m, s, \alpha}^{\Lip(\gamma)} |{\mathcal B} |_{m', s_0, \alpha}^{\Lip(\gamma)} + | {\mathcal A} |_{m, s_0, \alpha}^{\Lip(\gamma)} | {\mathcal B} |_{m', s, \alpha}^{\Lip(\gamma)}\,. 
\end{equation}
\item[(ii)] Let $s \geq s_0, \alpha \in \N_{0}$, $\mathcal{A}=\mathrm{Op}(a)\in \Op^m_{s, \alpha}$, $\mathcal{B}=\mathrm{Op}(b)\in \Op^{m'}_{s  + |m| + \alpha, \alpha}$. Then, the composition operator $\mathcal{AB}$ belongs to $\Op^{m+m'}_{s, \alpha}$ and it satisfies the estimate
\begin{equation}\label{est:composition}
|\mathcal{AB}|^{\Lip(\gamma)}_{m+m',s,\alpha}\lesssim_{s,m,\alpha}|\mathcal{A}|_{m,s,\alpha}^{\Lip(\gamma)}|\mathcal{B}|_{m',s_0+|m|+\alpha,\alpha}^{\Lip(\gamma)}+|\mathcal{A}|_{m,s_0,\alpha}^{\Lip(\gamma)}|\mathcal{B}|_{m',s +|m|+\alpha,\alpha}^{\Lip(\gamma)}.
\end{equation}
Moreover if $\mathcal{A}=\mathrm{Op}(a)\in \Op^m_{s, N}$, $\mathcal{B}=\mathrm{Op}(b)\in \Op^{m'}_{s + 2 N + |m| , 0}$ then  the remainder ${\mathcal R}_N = {\rm Op}(r_N) \in \Op^{m + m' - N}_{s, 0}$ satisfies the estimates 
\begin{equation}\label{stima cal RN comp memoirs}
|{\mathcal R}_N|_{m + m' - N, s, 0}^{\Lip(\gamma)} \lesssim_{s, N} |\mathcal{A}|_{m,s,N }^{\Lip(\gamma)}|\mathcal{B}|_{m',s_0+2N + |m|,0}^{\Lip(\gamma)}+|\mathcal{A}|_{m,s_0, N }^{\Lip(\gamma)}|\mathcal{B}|_{m',s+|m|,0}^{\Lip(\gamma)}
\end{equation}
\item[(iii)] The estimates \eqref{prodotto.simboli}-\eqref{stima cal RN comp memoirs} hold verbatim for matrices of pseudo-diffrential operators $\mathcal{A}\in\Op^m_{s, \alpha}$ and $\mathcal{B}\in\Op^{m'}_{s, \alpha}$.
\item[(iv)] Let $s \geq s_0, \alpha \in \N_{0}$, $\mathcal{A}=\mathrm{Op}(a)\in \Op^m_{s + |m| + |m'| + \alpha + 1, \alpha + 1}$, $\mathcal{B}=\mathrm{Op}(b)\in \Op^{m'}_{s + |m| + |m'| + \alpha + 1, \alpha + 1}$. Then, the commutator $[{\mathcal A}, {\mathcal B}] :=  {\mathcal A} {\mathcal B} - {\mathcal B} {\mathcal A} \in \Op^{m + m' - 1}_{s, \alpha}$ and it satisfies the estimate
\begin{equation}\label{est:commutator}
\begin{aligned}
|[{\mathcal A}, {\mathcal B}]|^{\Lip(\gamma)}_{m+m' - 1,s,\alpha} & \lesssim_{s,m,m', \alpha}|\mathcal{A}|_{m,s + |m| + |m'|+\alpha + 1,\alpha + 1}^{\Lip(\gamma)}|\mathcal{B}|_{m',s_0+|m| + |m'|+\alpha + 1,\alpha + 1}^{\Lip(\gamma)} \\
& \quad +|\mathcal{A}|_{m,s_0 + |m|+ |m'| +\alpha + 1,\alpha + 1}^{\Lip(\gamma)}|\mathcal{B}|_{m',s+|m|+ |m'| + \alpha + 1,\alpha + 1}^{\Lip(\gamma)}.
\end{aligned}
\end{equation}\end{itemize}
\end{lem}
\begin{proof}
See Lemma 2.13 in \cite{BeM}.
\end{proof}
\begin{lem}[Exponential map]\label{lem:exponential}
Let $s \geq s_0$, $\alpha \in \N_{0}$, $\mathcal{A}={\rm Op}(a(x,\xi;\omega))\in \Op^0_{s + \alpha, \alpha}$, then $\Phi = {\rm exp}({\mathcal A})$ satisfies the estimate

\begin{equation}
|\Phi-\Id|^{\Lip(\gamma)}_{0,s,\alpha}\leq |\mathcal{A}|^{\Lip(\gamma)}_{0,s,\alpha} \exp(C(s,\alpha)|\mathcal{A}|^{\Lip(\gamma)}_{0,s_0 + \alpha,\alpha}).
\end{equation}
The same statement holds even for matrix valued symbol ${\mathcal A} = \begin{pmatrix}
{\mathcal A}_1 & {\mathcal A}_2 \\
{\mathcal A}_3 & {\mathcal A}_4
\end{pmatrix} \in \Op^0_{s + \alpha, \alpha}$. 
\end{lem}
\begin{proof}
See Lemma 2.13 in \cite{BKM}.
\end{proof}

We also define, for $0\leq n\leq N-1$,
\begin{equation}\label{cancellittiEspliciti}
\begin{aligned}
a\#_{n} b&:=\sum_{|\beta|=n}\frac{1}{\beta! i^{|\beta|}} (\partial_{\xi}^\beta a)(\partial_x^\beta b)\in \mathcal{S}^{m+m'-n}\,, 
\\ 
a\#_{< N} b&:=\sum_{n=0}^{N-1} a\#_n b\in \mathcal{S}^{m+m'}\,, \qquad a\#_{\geq N} b:=r_N:=r_{N}({\mathcal A},{\mathcal B})\in \mathcal{S}^{m+m'-N}\,.
\end{aligned}
\end{equation}

\begin{rem}\label{composizione.simboli.reali}
We remark that if ${\rm Op}(a), {\rm Op}(b)$ are real operators, then ${\rm Op}(a\#b), {\rm Op}(a\#_{n})$ and ${\rm Op}(a\#_{< N} b)$ are so.
\end{rem}

\begin{lem}\label{lemma composizione 2}
Let $m, m' \in \R$, $s \geq s_0$, $\alpha\in\N_{0}, N \in \N$, ${\mathcal A} = {\rm Op}(a) \in \Op^m_{s, \alpha + N}$, ${\mathcal B} = {\rm Op}(b) \in \Op^{m'}_{s + 2 N + |m|, \alpha}$. Then 
the composition operator ${\mathcal A} {\mathcal B}$ admits the expansion 
$$
{\mathcal A} {\mathcal B} = {\rm Op}( ab) + {\rm Op}(\sigma_N) + {\mathcal R}_N 
$$
where 
$$
\begin{aligned}
|{\rm Op}(\sigma_N)|_{m + m' - 1, s, \alpha}^{\Lip(\gamma)} & \lesssim_{s, \alpha, m, m'}  | {\mathcal A} |_{m, s , \alpha + N }^{\Lip(\gamma)} | {\mathcal B}|_{m', s_0 + N, \alpha }^{\Lip(\gamma)}  +  | {\mathcal A}|_{m, s_0  , \alpha + N}^{\Lip(\gamma)} |{\mathcal B}|_{m', s + N, \alpha }^{\Lip(\gamma)}
\end{aligned}
$$
and the remainder ${\mathcal R}_N$ satisfies for any $s \geq s_0$, the estimate 
$$
|{\mathcal R}_N|_{m + m' - N , s, 0}^{\Lip(\gamma)} \lesssim_{s, N} |{\mathcal A}|_{m,s, N}^{\Lip(\gamma)} |{\mathcal B}|_{m', s_0 + 2N + |m|,0}^{\Lip(\gamma)} +  |{\mathcal A}|_{m,s_0, N}^{\Lip(\gamma)} |{\mathcal B}|_{m', s + 2 N  + |m|, 0}^{\Lip(\gamma)} \,. 
$$
The same statement holds true also for matrices of pseudo-differential operators of the form \eqref{matricidisimboli} ${\mathcal A} = {\rm Op}(a) \in \Op^m_{s, \alpha + N}$ and ${\mathcal B} = {\rm Op}(b) \in \Op^{m'}_{s + 2 N + |m|, \alpha}$.
\end{lem}
\begin{proof}
By \eqref{def:sigma-comp}, \eqref{remainder cal RN composizione}, we define 
$$
\sigma_N (x, \xi) := \sum_{1 \leq |\beta| \leq N - 1} \frac{1}{i^{|\beta|}\beta!} \partial_\xi^\beta a(x,\xi)\partial_x^\beta b(x,\xi) \,, \quad {\mathcal R}_N :={\rm Op}(r_N(x, \xi))\,. 
$$ 
The claimed bound on ${\mathcal R}_N$ follows directly by \eqref{stima cal RN comp memoirs}. Moreover by Lemma \ref{lem:composition-pseudodiff}-$(i)$, one obtains that for any $1 \leq |\beta| \leq N - 1$, one gets
$$
\begin{aligned}
|{\rm Op}(\partial_\xi^\beta a \partial_x^\beta b)|_{m + m' - 1, s, \alpha}^{\Lip(\gamma)} & \leq |{\rm Op}(\partial_\xi^\beta a \partial_x^\beta b)|_{m + m' - |\beta|, s, \alpha}^{\Lip(\gamma)}  \\
& \lesssim_{s, \alpha}  |{\rm Op}(\partial_\xi^\beta a) |_{m - |\beta|, s, \alpha}^{\Lip(\gamma)} | {\rm Op}(\partial_x^\beta b) |_{m' , s_0, \alpha}^{\Lip(\gamma)}  \\
& \quad + |{\rm Op}(\partial_\xi^\beta a) |_{m - |\beta|, s_0, \alpha}^{\Lip(\gamma)} | {\rm Op}(\partial_x^\beta b) |_{m', s, \alpha}^{\Lip(\gamma)} \\
& \lesssim_{s, \alpha}  |{\rm Op}(a) |_{m , s, \alpha + N}^{\Lip(\gamma)} | {\rm Op}( b) |_{m' , s_0 + N, \alpha}^{\Lip(\gamma)}  \\
& \quad + |{\rm Op}( a) |_{m , s_0, \alpha + N}^{\Lip(\gamma)} | {\rm Op}( b) |_{m', s + N, \alpha}^{\Lip(\gamma)}\,.
\end{aligned}
$$
The latter bound implies then the estimate on ${\rm Op}(\sigma_N)$ and the proof of the lemma is then concluded. 
\end{proof}
Given two linear operators ${\mathcal A}$ and ${\mathcal B}$, we define for any $\ell \in \N$, the linear operator ${\rm Ad}^\ell({\mathcal A}) {\mathcal B}$ as follows 
$$
\begin{aligned}
& {\rm Ad}({\mathcal A}){\mathcal B} := [{\mathcal B}, {\mathcal A}]\,, \\
&{\rm Ad}^{\ell + 1}({\mathcal A}){\mathcal B}  := [{\rm Ad}^{\ell }({\mathcal A}){\mathcal B}\,, {\mathcal A}], \quad \ell \geq 1\,.
\end{aligned}
$$

By \eqref{espansione.composizione.simboli}, \eqref{remainder cal RN composizione}, 
\eqref{cancellittiEspliciti}
one deduces the expansion for any $N\geq2$
\begin{equation}\label{espstar}
a\star b : =\sigma_{{\mathcal A}{\mathcal B}}-\sigma_{{\mathcal B}{\mathcal A}} = - i \{a,b\}+
\sum_{2\le |\beta|\le N-1}
(a\#_{|\beta|}b-b\#_{|\beta|}a) + \mathtt{r}_N
\end{equation}
where 
\begin{equation}\label{espstar2}
\{a, b\}:= \nabla_\xi a \cdot \nabla_x b -  \nabla_x a \cdot \nabla_\xi b\,,
\qquad
 \mathtt{r}_N=r_{N}({\mathcal A},{\mathcal B})- r_{N}({\mathcal B},{\mathcal A})\,.
\end{equation}
By construction,
\begin{equation}\label{adjdefpseudo}
{\rm Op}(a\star b)={\rm Ad}({\mathcal B}){\mathcal A}\,.
\end{equation}

By iterating Lemma \ref{lemma composizione 2}, one obtains the following lemma.  
\begin{lem}\label{Ad matriciali stima raffinata}
Let $m > 0, m' \in \R$, $\ell, N\in\N, \alpha \in \N_{0}$, $\ell \leq N$, $s \geq s_0$. Then there exists a constant $\sigma \equiv \sigma (N, m, m') > 0$ large enough such that if  ${\mathcal A} = \begin{pmatrix}
{\mathcal A}_1 & {\mathcal A}_2 \\
{\mathcal A}_3 & {\mathcal A}_4
\end{pmatrix} \in \Op^{- m}_{s + \sigma, \alpha + \sigma}$, ${\mathcal B} = \begin{pmatrix}
{\mathcal B}_1 & {\mathcal B}_2 \\
{\mathcal B}_3 & {\mathcal B}_4
\end{pmatrix} \in \Op^{m'}_{s + \sigma, \alpha + \sigma}$, one has that 
$$
{\rm Ad}^\ell({\mathcal A}) {\mathcal B} = {\mathcal C}_{\ell, N} + {\mathcal R}_{\ell, N},
$$
where 
$$
\begin{aligned}
|{\mathcal C}_{\ell, N}|_{ m' - \ell m, s, \alpha}^{\Lip(\gamma)} & \lesssim_{s, \alpha, N}  | {\mathcal A} |_{m, s + \sigma , \alpha + \sigma}^{\Lip(\gamma)} \Big( | {\mathcal A} |_{m, s_0 + \sigma , \alpha + \sigma}^{\Lip(\gamma)} \Big)^{\ell - 1} |{\mathcal B}|_{m', s_0 + \sigma, \alpha + \sigma}^{\Lip(\gamma)} \\
& \qquad  +  \Big( | {\mathcal A} |_{m, s_0 + \sigma , \alpha + \sigma}^{\Lip(\gamma)} \Big)^\ell |{\mathcal B}|_{m', s + \sigma, \alpha+ \sigma}^{\Lip(\gamma)},
\end{aligned}
$$
and 
$$
\begin{aligned}
|{\mathcal R}_{\ell, N}|_{ m' - m N, s, 0}^{\Lip(\gamma)} & \lesssim_{s, N} |{\mathcal A}|_{m,s + \sigma, \sigma}^{\Lip(\gamma)}\Big(  |{\mathcal A}|_{m,s_0 + \sigma, \sigma}^{\Lip(\gamma)} \Big)^{\ell - 1}  |{\mathcal B}|_{m', s_0 + \sigma, \sigma}^{\Lip(\gamma)}   \\
& \qquad +  \Big( |{\mathcal A}|_{m,s_0 + \sigma, \sigma}^{\Lip(\gamma)} \Big)^\ell  |{\mathcal B}|_{m', s + \sigma, \sigma}^{\Lip(\gamma)}  \,. 
\end{aligned}
$$
 In particular ${\rm Ad}^N({\mathcal A}) {\mathcal B}$ satisfies the estimate 
\begin{equation}\label{stima Ad N cal A cal B}
\begin{aligned}
|{\rm Ad}^N({\mathcal A}) {\mathcal B}|_{ m' - m N, s, 0}^{\Lip(\gamma)} & \lesssim_{m, m', s, N} |{\mathcal A}|_{m,s + \sigma, \sigma}^{\Lip(\gamma)}\Big(  |{\mathcal A}|_{m,s_0 + \sigma, \sigma}^{\Lip(\gamma)} \Big)^{N - 1}  |{\mathcal B}|_{m', s_0 + \sigma, \sigma}^{\Lip(\gamma)}   \\
& \qquad +  \Big( |{\mathcal A}|_{m,s_0 + \sigma, \sigma}^{\Lip(\gamma)} \Big)^N  |{\mathcal B}|_{m', s + \sigma, \sigma}^{\Lip(\gamma)}\,. 
\end{aligned}
\end{equation}
\end{lem}


\begin{lem}[Conjugation by an exponential map]\label{coniugio senza alpha exponential map}
Let $m > 0, m' \in \R$, $\ell, N\in\N, \alpha \in \N_{0}$, $s \geq s_0$. Then there exists a constant $\sigma \equiv \sigma (N, m, m') > 0$ large enough such that if  ${\mathcal A} = \begin{pmatrix}
{\mathcal A}_1 & {\mathcal A}_2 \\
{\mathcal A}_3 & {\mathcal A}_4
\end{pmatrix} \in \Op^{- m}_{s + \sigma, \alpha + \sigma}$, ${\mathcal B} = \begin{pmatrix}
{\mathcal B}_1 & {\mathcal B}_2 \\
{\mathcal B}_3 & {\mathcal B}_4
\end{pmatrix} \in \Op^{m'}_{s + \sigma, \alpha + \sigma}$. If $|{\mathcal A}|_{- m, s_0 + \sigma, \alpha + \sigma}^{\Lip(\gamma)} \leq C$ for some constant $C > 0$, then one has the following expansion
$$
e^{- {\mathcal A}} {\mathcal B} e^{\mathcal A} = {\mathcal B} + {\mathcal C}_N + {\mathcal R}_N,
$$
where for any $s \geq s_0\,,\, \alpha \in \N_{0}$, 
$$
|{\mathcal C}_N|^{\Lip(\gamma)}_{m' - m, s, \alpha} \lesssim_{s, \alpha, N} |{\mathcal A}|_{- m, s + \sigma, \alpha + \sigma}^{\Lip(\gamma)} |{\mathcal B}|_{m', s_0 + \sigma, \alpha + \sigma}^{\Lip(\gamma)} +  |{\mathcal A}|_{- m, s_0 + \sigma, \alpha + \sigma}^{\Lip(\gamma)} |{\mathcal B}|_{m', s + \sigma, \alpha + \sigma}^{\Lip(\gamma)},
$$
and 
$$
|{\mathcal R}_N|^{\Lip(\gamma)}_{m' - N m, s, 0} \lesssim_{s, \alpha, N} |{\mathcal A}|_{- m, s + \sigma,  \sigma}^{\Lip(\gamma)} |{\mathcal B}|_{m', s_0 + \sigma,  \sigma}^{\Lip(\gamma)} +  |{\mathcal A}|_{- m, s_0 + \sigma,  \sigma}^{\Lip(\gamma)} |{\mathcal B}|_{m', s + \sigma,  \sigma}^{\Lip(\gamma)}.
$$
\end{lem}
\begin{proof}
By the Lie expansion, one has that   
\begin{equation}\label{espansione di Lie}
\begin{aligned}
e^{-\mathcal{A}}\mathcal{B} e^\mathcal{A} & =\mathcal{B}+ \sum_{\ell=1}^{N - 1} \frac{{\rm Ad}^\ell({\mathcal A}) {\mathcal B} }{\ell!}+{\mathcal Q}_N\,,  \\
 {\mathcal Q}_N&  := \frac{1}{(N - 1)!}\int_0^1(1-\tau)^{N - 1}e^{-\tau {\mathcal A}} \circ  {\rm Ad}^N ({\mathcal A}) {\mathcal B} \circ e^{\tau {\mathcal A}}\, d \tau.
\end{aligned}
\end{equation}
Then by applying Lemma \ref{Ad matriciali stima raffinata} in order to expand any term $\frac{{\rm Ad}^\ell({\mathcal A}) {\mathcal B} }{\ell!}$, we then obtain there exists $\sigma \equiv \sigma(N, m, m') \gg 0$ large enough such that if $|{\mathcal A}|^{\Lip(\gamma)}_{m, s_0 + \sigma , \alpha + \sigma} \lesssim 1$, then 
$$
\begin{aligned}
\frac{{\rm Ad}^\ell({\mathcal A}) {\mathcal B} }{\ell!} & = {\mathcal C}_{\ell, N} + {\mathcal R}_{ \ell, N}, \\
|{\mathcal C}_{\ell, N}|_{ m' - \ell m, s, \alpha}^{\Lip(\gamma)} & \lesssim_{s, \alpha, N}  | {\mathcal A} |_{m, s + \sigma , \alpha + \sigma}^{\Lip(\gamma)}  |{\mathcal B}|_{m', s_0 + \sigma, \alpha + \sigma}^{\Lip(\gamma)}  +   | {\mathcal A} |_{m, s_0 + \sigma , \alpha + \sigma}^{\Lip(\gamma)} |{\mathcal B}|_{m', s + \sigma, \alpha+ \sigma}^{\Lip(\gamma)}\,, \\
|{\mathcal R}_{\ell, N}|_{m' - m N, s,0}^{\Lip(\gamma)} & \lesssim_{s, \alpha, N} |{\mathcal A}|_{- m, s + \sigma,  \sigma}^{\Lip(\gamma)} |{\mathcal B}|_{m', s_0 + \sigma,  \sigma}^{\Lip(\gamma)} +  |{\mathcal A}|_{- m, s_0 + \sigma,  \sigma}^{\Lip(\gamma)} |{\mathcal B}|_{m', s + \sigma,  \sigma}^{\Lip(\gamma)}\,. 
\end{aligned}
$$
Moreover by the estimate \eqref{lem:exponential} (applied for $\alpha = 0$), the estimate \eqref{stima Ad N cal A cal B} and by the composition estimate \eqref{est:composition} (still applied for $\alpha = 0$), one obtains the bound for ${\mathcal Q}_N$ in \eqref{espansione di Lie}
$$
|{\mathcal Q}_{N}|_{m' - m N, s,0}^{\Lip(\gamma)}  \lesssim_{s, \alpha, N} |{\mathcal A}|_{- m, s + \sigma,  \sigma}^{\Lip(\gamma)} |{\mathcal B}|_{m', s_0 + \sigma,  \sigma}^{\Lip(\gamma)} +  |{\mathcal A}|_{- m, s_0 + \sigma,  \sigma}^{\Lip(\gamma)} |{\mathcal B}|_{m', s + \sigma,  \sigma}^{\Lip(\gamma)}\,.
$$
The claimed statement then follows by defining 
$$
{\mathcal C}_N := \sum_{\ell = 1}^{N - 1} {\mathcal C}_{\ell, N}, \quad {\mathcal R}_N := {\mathcal Q}_N + \sum_{\ell = 1}^{N - 1} {\mathcal R}_{\ell, N}\,. 
$$
\end{proof}
Given an operator ${\mathcal A} = {\rm Op}(a(x, \xi)) \in \Op^m_{s, \alpha}$, we define the averaged symbol $\langle a \rangle_x(\xi)$ as 
\begin{equation}\label{definizione simbolo mediato}
\langle a \rangle_x(\xi) := \frac{1}{(2 \pi)^d} \int_{\T^d} a(x, \xi)\,d x\,. 
\end{equation}
The following elementary lemma holds. 
\begin{lem}\label{stima simbolo mediato}
Let ${\mathcal A} = {\rm Op}(a(x, \xi)) \in \Op^m_{s, \alpha}$, $s \geq s_0$, $\alpha \in \N_{0}$. Then 
$$
|{\rm Op}(\langle a \rangle_x)|_{m, s, \alpha}^{\Lip(\gamma)} \lesssim |{\mathcal A}|_{m, s_0, \alpha}^{\Lip(\gamma)}\,. 
$$
Moreover, if $\mathcal A$ is real, then one has that $\langle a\rangle_{x}(\xi)=\overline{\langle a\rangle_{x}(-\xi)}$.
\end{lem}

Following the argument used in Lemma $2.21$ of \cite{BeM} one can prove the following result.

\begin{lem}[\textbf{Action of a pseudo-differential operator}]\label{lemma azione tame pseudo-diff}
Let $N > 0$, ${\mathcal A} = {\rm Op}(a) \in {\mathcal S}^{- N}_{s, 0}$, $s \geq s_0$. Then 
$$
\| {\mathcal A} h \|_{s + N}^{\Lip(\gamma)} \lesssim_{s, m} |{\mathcal A}|_{- N, s_0, 0}^{\Lip(\gamma)} \| h \|_s^{\Lip(\gamma)} +  |{\mathcal A}|_{ - N, s, 0}^{\Lip(\gamma)} \| h \|_{s_0}^{\Lip(\gamma)}\,. 
$$
\end{lem}
We now collect some properties of composition operators. We consider a diffeomorphism of the ${2}$-dimensional torus defined by
\begin{equation}
y=x+\balpha(x)\,\,\,\iff \,\,\,x=y+\check{\balpha}(y),
\end{equation}
where $\balpha$ is a small real valued smooth vector function, and the induced operators
\begin{equation}\label{def:B}
(\mathcal{A}u)(x):=u(x+\balpha(x)),\,\,\,\,(\mathcal{A}^{-1}u)(y):=u(y+\check{\balpha}(y)).
\end{equation}
\begin{lem}\label{lem:changevar}
Let $s \geq s_0$, $\balpha (\cdot; \omega) \in H^s$, $\omega \in {\bf O} \subset \R^2$, $\|\balpha\|_{s_0}^{\Lip(\gamma)}\leq \delta(s_0)$ small enough. Then, the composition operator $\mathcal{A}$ satisfies the following tame estimates
\begin{equation}
\|\mathcal{A} u \|_s^{\Lip(\gamma)} \lesssim_{s}\|u\|_s^{\Lip(\gamma)} +\| \balpha\|_s^{\Lip(\gamma)} \|u\|_{s_0}^{\Lip(\gamma)},
\end{equation} 
and the function $\check{\alpha}$ defined by the inverse diffeomorphism satisfies \begin{equation}\label{diffeo.e.inverso}
\|\check{\balpha}\|_{s}^{\Lip(\gamma)}\lesssim_{s}\|\balpha\|_{s}^{\Lip(\gamma)},
\end{equation}
and as a consequence 
\begin{equation}
\|\mathcal{A}^{-1} u \|_s^{\Lip(\gamma)} \lesssim_{s}\|u\|_s^{\Lip(\gamma)} +\| \balpha\|_s^{\Lip(\gamma)} \|u\|_{s_0}^{\Lip(\gamma)}.
\end{equation}
\end{lem}
\begin{proof}
It follows by the same argument of Lemma 2.30 in \cite{BeM}.
\end{proof}
Since the linearized operator has invariance properties on the space of zero average function we introduce the projections $\Pi_0$ and $\Pi_0^\bot$ as follows 
\begin{equation}\label{def media media nulla}
\Pi_0 h := \frac{1}{(2 \pi)^2} \int_{\T^2} h(x)\, d x, \quad \Pi_0^\bot := {\rm Id} - \Pi_0\,. 
\end{equation}
Note that given an even cut-off function $\eta_0$ such that 
\begin{equation}\label{def cut off eta 0}
\begin{aligned}
& \eta_0 \in {\mathcal C}^\infty(\R^2, \R), \quad \eta_0 \quad\text{is even}\,, \\
& \eta_0(\xi ) = 1, \quad \forall |\xi| \leq \frac12, \quad \eta_0(\xi) = 0 \quad \forall |\xi| \geq \frac23,
\end{aligned}
\end{equation}
one has that 
\begin{equation}\label{Pi 0 Pi 0 bot op}
\begin{aligned}
& \Pi_0 = {\rm Op}(\eta_0(\xi)) \in \Op^{- \infty}\,, \\
& |\Pi_0|_{- m, s, \alpha} \lesssim_{m, s, \alpha} 1, \quad \forall m, s\geq 0, \quad \alpha \in \N_{0} \,, \\
&  |\Pi_0^\bot|_{0, s, \alpha} \lesssim_{s, \alpha} 1, \quad \forall s \geq 0, \quad \alpha \in \N_{0}\,. 
\end{aligned}
\end{equation}
Finally, one has that $\Pi_{0}$ is real since $\eta_{0}$ is even in $\xi$ and therefore $|\Pi_0^\bot$ is so.
We state the following lemma, see Lemma 4.2 in \cite{FM}
\begin{lem}\label{invertibilita cal A bot}
Let ${\mathcal A}$ be the map in \eqref{def:B} and let us define ${\mathcal A}_\bot := \Pi_0^\bot {\mathcal A} \Pi_0^\bot$. Then under the same assumptions of Lemma \ref{lem:changevar}, one has that ${\mathcal A}_\bot : H^s_0 \to H^s_0$ is invertible and ${\mathcal A}_\bot^{- 1} = \Pi_0^\bot {\mathcal A}^{- 1} \Pi_0^\bot : H^s_0 \to H^s_0$ and if $u(\cdot; \omega) \in H^s_0$, then
$$
\| {\mathcal A}_\bot^{\pm 1}u  \|_s^{\Lip(\gamma)} \lesssim_s  \|u\|_s^{\Lip(\gamma)} +\| \alpha\|_s^{\Lip(\gamma)} \|u\|_{s_0}^{\Lip(\gamma)}\,. 
$$
\end{lem}

We also prove the following lemma that we shall use in the sequel.
\begin{lem}\label{pseudo media nulla}
Let $N > 0$, $m \in \R$, $s \geq s_0$. Then there is $\sigma  \gg 0$ large enough such that if ${\mathcal A} = {\rm Op}(a) \in \Op^m_{s + \sigma, 0}$ then the operator ${\mathcal A}_\bot := \Pi_0^\bot {\mathcal A} \Pi_0^\bot$ (recall \eqref{def media media nulla}-\eqref{Pi 0 Pi 0 bot op}) satisfies ${\mathcal A}_\bot = {\mathcal A} + {\mathcal R}_\bot$ where ${\mathcal R}_\bot $ satisfies the estimate $|{\mathcal R}_\bot|_{- N, s, 0}^{\Lip(\gamma)} \lesssim_{N, s} |{\mathcal A}|_{m, s + \sigma, 0}^{\Lip(\gamma)}$.
Finally, if ${\mathcal A}$ is real, then the operators ${\mathcal A}_\bot$ and ${\mathcal R}_\bot$ are real.
\end{lem}
\begin{proof}
One has 
$$
{\mathcal A}_\bot = {\mathcal A} + {\mathcal R}_\bot, \quad {\mathcal R}_\bot := - \Pi_0 {\mathcal A} \Pi_0^\bot - {\mathcal A} \Pi_0,
$$
hence the claimed bound on ${\mathcal R}_\bot$ follows by  \eqref{Pi 0 Pi 0 bot op} and the composition estimate \eqref{est:composition}. 
The operator ${\mathcal A}_\bot$ is real by composition with $\Pi_{0}^{\bot}$ and ${\mathcal R}_\bot$ by difference.
\end{proof}

We now mention some elementary properties of the Laplacian operator $-\Delta$ and of its inverse $(-\Delta)^{-1}$ acting on functions with zero average in $x$:
\begin{equation}
-\Delta u(x)=\sum_{k \neq 0}|k|^2 \widehat{u}(k)e^{i x\cdot k},\qquad (-\Delta)^{-1} u(x)=\sum_{k\neq 0}\frac{1}{|k|^2} \widehat{u}(k)e^{i x\cdot k}.
\end{equation}
By the properties of the cut off function $\eta_0$ we can identify $(- \Delta)^{- 1}$ with  ${\rm Op}\Big(\dfrac{1 - \eta_0(\xi)}{|\xi|^2} \Big)$ since the action of these two operators on functions with zero average is the same. 
By recalling Definition \ref{def:family-pseudo-diff} one easily checks that
\begin{equation}
|-\Delta|_{2,s,\alpha}\lesssim_\alpha 1,\qquad |(-\Delta)^{-1}|_{-2,s,\alpha}\lesssim_\alpha 1.
\end{equation}
In what follows, we denote by $\mathcal{U}:H^s(\T^2)\to H^{s+1}(\T^2)$ the Biot-Savart operator, which is defined on zero average functions as follows
\begin{equation}\label{def bio savart operator}
\mathcal{U}f(x):=\nabla^{\perp}(-\Delta)^{-1}f(x)= \begin{pmatrix}
\partial_{x_2} (- \Delta)^{- 1} f \\
-\partial_{x_1} (- \Delta)^{- 1} f
\end{pmatrix}\,. 
\end{equation}
One asily verifies that  $\mathcal{U}\in\Op^{-1}$ and satisfies
\begin{equation}\label{stima biot savart1}
|\mathcal{U}|_{-1,s,\alpha}\lesssim_\alpha 1, \quad \forall s \geq 0, \quad \alpha \in \N_{0}\,. 
\end{equation}
Finally, given $f\in {\mathcal C}^\infty(\T^2)$ we define the operator $\mathcal{R}(f)$ acting on $h\in H^s(\T^2)$ as follows
\begin{equation}\label{op one smoothing lin bio savart}
\mathcal{R}(f)h=[\nabla^\perp(-\Delta)^{-1} h\cdot\nabla]f = {\mathcal U} h \cdot \nabla f,
\end{equation}
which is a pseudo-differential operator of order $-1$ satisfying
\begin{equation}\label{stima biot savart2}
\begin{aligned}
& |\mathcal{R}(f)|_{-1,s,\alpha}\lesssim_{\alpha} \|f\|_{s+1}, \quad \forall s \geq 0, \quad \alpha \in \N_{0}\,. \\
\end{aligned}
\end{equation}

\subsection{A quantitative Egorov Theorem}
We also prove a quantitative version of the Egorov theorem. The proof follows word by word the arguments in 
\cite{BFPT} (See also Theorem $3.4$ of \cite{FGP19} for the main idea). All the following sums over $k_{1}, k_{2}, k_{3}\in \N_{0}$ such that  $k_{1}+k_{2}+k_{3}=s$ are actually sums over $k_{1}, k_{2}, k_{3}\in\N_{0}$ such that  $k_{1}+k_{2}+k_{3}=s, k_{1}+k_{2}\ge 1 $.

\begin{thm}{\bf (Egorov Theorem).}\label{quantitativeegorov}
Fix $ m\in\R$, $M > \max \{ 0 , -m\}$. There exists  $\sigma:=\sigma(m, M) \gg 0$ large enough such that for any $S\ge s_{0} + \sigma$ and for any $\alpha\ge0$ there exists $\varepsilon= \varepsilon(S, m, M) \ll 1$ small enough such that, if (recall the definition of ${\mathcal A}$ in \eqref{def:B}) 
\begin{equation}\label{buf}
\|\balpha\|^{\Lip(\gamma)}_{s_0+\sigma} \leq \varepsilon\,, 
\end{equation}
then given a symbol
 $w( x, \xi)\in {\mathcal S}^m_{S, \alpha+\sigma}$, Lipschitz in the variable $\omega\in \mathbf{O}$,
then the following hold.
\begin{equation}\label{simbotransportato}
{\mathcal A}^{- 1} {\rm Op}(w(x, \xi)){\mathcal A}= {\rm Op}\Big(q( x, \xi) \Big) + {\mathcal R}
\end{equation}
where, for any $s\in [s_{0}, S-\sigma]$ and for any $\alpha\ge 0$,  $q\in {\mathcal S}_{s, \alpha}^m$ and satisfies the estimates 
\begin{equation}\label{parlare}
|{\rm Op}(q) |_{m,s,\alpha}^{\Lip(\gamma)}
\lesssim_{m, M, s, \alpha} |{\rm Op}(w)|_{m,s,\alpha+\sigma}^{\Lip(\gamma)}+\sum_{k_{1}+k_{2}+k_{3}=s}
		| {\rm Op}(w) |_{m,k_1,\alpha+k_2+\sigma}^{\Lip(\gamma)} \|\balpha\|^{\Lip(\gamma)}_{k_3+\sigma}\,,
\end{equation}
\begin{equation}\label{troppo}
\begin{aligned}
|\Delta_{12} {\rm Op}(q) |_{m, s,\alpha} \lesssim_{m, M, s, \alpha} &| {\rm Op}(w)|_{m, s+1,\alpha+\sigma} \|\Delta_{12} \balpha\|_{s+1} + |\Delta_{12} {\rm Op}(w)|_{m, s, \alpha+\sigma}
\\
 &+ \sum_{k_{1}+k_{2}+k_{3}=s+1} |{\rm Op}(w)|_{m,k_1,\alpha+k_2+\sigma} \|\balpha\|_{k_3+\sigma_1} \|\Delta_{12}\balpha \|_{s_0+1} 
 \\
 &+  \sum_{k_{1}+k_{2}+k_{3}=s} |\Delta_{12}{\rm Op}(w)|_{m,k_1,\alpha+k_2+\sigma} \|\balpha\|_{k_3+\sigma}\,. 
 \end{aligned}
\end{equation}
Furthermore the remainder ${\mathcal R} \in {\mathcal B}(H^s, H^{s + M})$, for any $s\in [s_{0}, S-\sigma]$, satisfies the estimate
\begin{equation}\label{francia1}
 \| {\mathcal R} h \|_{s + M}^{\Lip(\gamma)} \lesssim_{m,s, M} {\mathfrak M}(s) \| h \|_{s_0}^{\Lip(\gamma)} + {\mathfrak M}(s_0) \| h \|_s^{\Lip(\gamma)}  
\end{equation}
where
\[
 {\mathfrak M}(s) := 
 \sum_{k_1 + k_2 + k_3 = s} | {\rm Op}(w) |_{m,k_1,k_2+\sigma}^{\Lip(\gamma)} \|\balpha\|_{k_3+\sigma}^{\Lip(\gamma)}
\]
Finally, if ${\rm Op}(w(x, \xi))$ is a real operator, then ${\rm Op}(q(x, \xi))$ and $\mathcal R$ are a real operators.
\end{thm}

In the proof we use  the following lemma proved 
in the Appendix of \cite[Lemma A.7]{FGP19}. Given a square matrix $A$, we denote by $A^T$ the transpose matrix and if $A$ is invertible we denote by $A^{- T}$ the inverse of the transpose matrix. 
\begin{lem}\label{Lemmino} 
There exists $\sigma\gg 0$ such that for any $S\ge s_{0} +\sigma$, if
$\balpha \in C^S (\T^{2}, \R^{2})$ is a real function satisfying 
$\| \balpha \|^{\Lip(\gamma)} _{s_0+\sigma}< 1$
then, for any symbol $w\in \mathcal{S}_{S, \alpha}^m$, 
\begin{equation}
		A w:=w\Big( x+\balpha(x), ({\rm Id} + \nabla_x \balpha(x))^{-T}\xi \Big)
	\end{equation}
	is a symbol in $  \mathcal{S}_{S, \alpha}^m$ satisfying, for any $s\in [s_0, S-\sigma] $ and for any $\alpha\ge0$,    
	\begin{equation}\label{stima}
		| {\rm Op}(A w) |^{\Lip(\gamma)} _{m, s, \alpha}\le | {\rm Op}(w) |^{\Lip(\gamma)} _{m, s, \alpha}
		+C\sum_{
		k_{1}+k_{2}+k_{3}=s
		}
		| {\rm Op}(w) |^{\Lip(\gamma)} _{m, k_1, \alpha+k_2} \| \balpha \|^{\Lip(\gamma)} _{k_3+\sigma} \, ,
	\end{equation}
	for some $C=C(s, \alpha)>0$.
	 For $s=s_0$  we have the rougher estimate 
$	| {\rm Op}(A w) |^{\Lip(\gamma)} _{m, s_0, \alpha}\lesssim | {\rm Op}(w) |^{\Lip(\gamma)}_{m, s_0, \alpha+s_0} $. 
\end{lem}

\begin{proof}[Proof of Thm. \ref{quantitativeegorov}]
Let us consider for $\tau\in[0,1]$ the composition operators
\begin{equation}\label{ignobel}
\mathcal{A}^{\tau}h(x):=h(x+\tau\balpha(x))\,, 
\quad 
(\mathcal{A}^{\tau})^{-1}h( y):=
h(y+\breve{\balpha}(\tau; y))\, , 
\end{equation}

For any $\tau\in[0,1]$, the conjugated operator
	\[
	P^{\tau}:=\mathcal{A}^{\tau} \circ {\rm Op}(w)\circ  (\mathcal{A}^{\tau})^{-1}
	\]
	solves  the Heisenberg equation
	\begin{equation}\label{ars}
	\partial_{\tau} P^{\tau}=[{X}^{\tau}, P^{\tau}]\,,  \quad P^{0}={\rm Op}(w) \, , 
	\end{equation}
	where\footnote{By using the definition of $\mathcal{A}^{\tau}$ in \eqref{ignobel}
	and \eqref{def:Calphagenerator}-\eqref{def:Calpha2}
	one can easily check that
	\begin{equation}\label{flussodiffeo}
      \partial_{\tau}\mathcal{A}^{\tau}=X^{\tau}\mathcal{A}^{\tau}\,,\quad 
     \mathcal{A}^{0}={\rm Id}\,.
	\end{equation}
	}
	\begin{equation}\label{def:Calphagenerator}
		{X}^{\tau}:=b(\tau;x) \cdot \nabla_{x}={\rm Op}(\chi)\,,\qquad 
		\chi :=\chi(\tau;x,\xi) :=\chi(\tau;x,\xi):=i b(\tau;x)\cdot\xi\,,
	\end{equation}
	and 
	\begin{equation}\label{def:Calpha2}
b(\tau;x)= \left(\mathbb{I} + \tau \nabla_{x} \balpha  \right)^{-1} \balpha
\,.
\end{equation}
Notice that
\begin{equation}\label{normachichi}
| {\rm Op}(\chi) |_{1,s, \alpha}^{\Lip(\gamma)}\lesssim_{s}\|b\|_{s}^{\Lip(\gamma)}
\lesssim_{s}\|\balpha\|_{s+1}^{\Lip(\gamma)}\,,\qquad \forall \,p\geq0\,.
\end{equation}
	For simplicity  we omit the dependence of the symbols on the variables
$\omega\in{\bf O} \subset \R^2$.
	We look for an approximate solution of \eqref{ars} 
	of the form\footnote{We suppose  that $ m + M \geq 2 $ 
	is an integer, otherwise we  can replace $ m + M $  with 
	$ [m + M] + 1 $.} 
	\begin{equation}\label{balconata0}
		Q^{\tau}:={\rm Op}(q(\tau;x,\xi))\,,\quad 
		q=q(\tau;x,\xi)=\sum_{k=0}^{m+M-1} q_{m-k}(\tau;x, \xi)\,,
	\end{equation}
	where $q_{m-k}$ are  symbols in $S^{m-k}$ to be determined iteratively	so that
\begin{equation}\label{approssimo}
	\partial_{\tau} Q^{\tau}=[{X}^{\tau}, Q^{\tau}] + \mathcal{M}^\tau\,, \quad Q^{0}={\rm Op}(w) \, , 
\end{equation}
	where $\mathcal{M}^\tau={\rm Op}(\mathtt{r}_{-M}(\tau;x,\xi))$ with symbol $\mathtt{r}_{-M}\in \mathcal{S}^{-M}$.
	Passing to the symbols we obtain 
	(recall \eqref{def:sigma-comp}, \eqref{balconata0} and \eqref{espstar})
	\begin{equation}\label{probapproxsimboloq}
\left\{\begin{aligned}
			&\partial_{\tau}q(\tau;x,\xi)=\chi(\tau;x,\xi)\star q(\tau;x,\xi) + \mathtt{r}_{-M}(\tau;x,\xi)
			\\
			&q(0;x,\xi)=w(x,\xi)\,.
		\end{aligned}\right.
	\end{equation}
	where the unknowns are now $q(\tau;x,\xi),\mathtt{r}_{-M}(\tau;x,\xi)$.

	We  expand $\chi(\tau;x,\xi)\star q(\tau;x,\xi)$ into a sum of symbols with decreasing orders.
	Using that $\chi$ is linear in $\xi\in\R^{2}$ (see \eqref{def:Calphagenerator}), 
	\eqref{espansione.composizione.simboli}, \eqref{remainder cal RN composizione},
	and by the expansion \eqref{balconata0} (together with the ansatz that $q_{m-k}\in S^{m-k}$)
	we note that 
	\[
	\begin{aligned}
	\chi\star q&=
	\chi\#q-q\#\chi
	=\chi q+\frac{1}{i}\nabla_{\xi}\chi\cdot	\nabla_{x}q-  
	q\#\chi
	\\&
	\stackrel{\eqref{balconata0}}{=}
	\chi q+
	\sum_{k=0}^{m+M-1} \frac{1}{i}\nabla_{\xi}\chi\cdot	\nabla_{x}q_{m-k}-\Big(\sum_{k=0}^{m+M-1} q_{m-k}(\tau;x, \xi)\Big)\#\chi
	\\&
	\stackrel{\eqref{cancellittiEspliciti}}{=}	\sum_{k=0}^{m+M-1} \frac{1}{i}\{\chi,q_{m-k}\}
	-\sum_{k=0}^{m+M-1}
	\sum_{n=2}^{m-k+M}q_{m-k}\#_{n}\chi-\sum_{k=0}^{m+M-1}q_{m-k}\#_{\geq m-k+M+1}\chi\,.
	\end{aligned}
	\]
	
	By rearranging the sums we can write
\begin{equation}\label{civilwar}
\begin{aligned}
	\chi\star q&=\underbrace{-i\{\chi,q_{m}\}}_{{\rm ord}\; m}
	-\overbrace{\sum_{k=1}^{m+M-1}\underbrace{\big(-i\{\chi,q_{m-k}\}+r_{m-k}\big)}_{{\rm ord}\; m-k} }^{{\rm orders\, from}\;-M+1 \;{\rm to}\;  m-1}
	-\underbrace{\mathtt{r}_{-M}}_{{\rm ord}\;  -M}
\end{aligned}
\end{equation}
where we defined, denoting $\mathtt{w}=\mathtt{w}(k,h):=k-h+1$,
\begin{equation}\label{sperobene}
\begin{aligned}
r_{m-k}&:=-\sum_{h=0}^{k-1}q_{m-h}\#_{\mathtt{w}}\chi
\stackrel{\eqref{cancellittiEspliciti}}{\in}  \mathcal{S}^{(m-h)+1-(k-h+1)}_{s, \alpha}\equiv \mathcal{S}^{m-k}_{s, \alpha}\,,
\end{aligned}
\end{equation}
\begin{equation}\label{sperobeneResto}
\begin{aligned}
\mathtt{r}_{-M}&:=\sum_{k=0}^{m+M-1}q_{m-k}\#_{\geq M+m-k+1}\chi
\stackrel{\eqref{cancellittiEspliciti}}{\in}  \mathcal{S}^{m-k+1-(m-k+1+M)}_{s, \alpha}\equiv  \mathcal{S}^{-M}_{s, \alpha}\,.
\end{aligned}
\end{equation}
We have reduced the problem to  finding symbols $q_{m-k}\in \mathcal{S}^{m-k}_{s, \alpha}$, $0\leq k\leq m+M-1$, which solve for $k=0$
(recall the form of $\chi$ in \eqref{def:Calphagenerator})
	\begin{equation}\label{ordm}
		\left\{\begin{aligned}
			&\partial_{\tau}q_{m}(\tau;x,\xi)= \{b(\tau;x)\xi, q_{m}(\tau;x,\xi)\}
			\\
			&q_{m}(0;x,\xi)=w(x,\xi)\,,
		\end{aligned}\right.
	\end{equation}
while for $1\le k\le m+M-1$ 
	\begin{equation}\label{ordmmenok}
	\left\{\begin{aligned}
		&\partial_{\tau}q_{m-k}(\tau;x,\xi)= \{b(\tau;x)\xi, q_{m-k}(\tau;x,\xi)\}+r_{m-k}(\tau;x,\xi)
		\\
		&q_{m-k}(0;x,\xi)=0\,.
	\end{aligned}\right.
\end{equation}
We note that the symbols $r_{m-k}$ in \eqref{sperobene} with $1\leq k\leq m+M-1$ depend only
on $q_{m-h}$ with $0\leq h<k$. This means that equations \eqref{ordmmenok} can be solved iteratively.

\vspace{0.5em}
\noindent
{\bf Order $m$.} To solve \eqref{ordm}, 	we consider the solutions of the Hamiltonian system
	\begin{equation}\label{charsys}
		\left\{\begin{aligned}
			&\frac{d}{ds}x(s)=-b(s;x(s))\\
			&\frac{d}{ds}\xi(s)=( \nabla_{x} b(s;x(s)))^{T}\xi(s)
		\end{aligned}\right.\qquad (x(0),\xi(0))=(x_0,\xi_0)\in \T^{2}\times \R^{2}\,.
	\end{equation}
	One can note that if $q_m$ is a solution of \eqref{ordm}, then 
	it is constant when evaluated along the flow of \eqref{charsys}. In other words
	setting
	$g(\tau):=q_{m}(\tau;x(\tau),\xi(\tau))$ one has that 
	$	\frac{d}{d\tau}g(\tau)=0 $ implies that $ g(\tau)=g(0) $, for any $ \tau\in[0,1] $. 
	Let us denote by $\gamma^{\tau_0,\tau}(x,\xi)$ the solution 
	of the characteristic system \eqref{charsys}
	with initial condition $\gamma^{\tau_0,\tau_0}=(x,\xi)$\footnote{
		In other words $(x(\tau),\xi(\tau))=\gamma^{0,\tau}(x_0,\xi_0)$
		and the inverse flow is given by $\gamma^{\tau,0}(x,\xi)=(x_0,\xi_0)$.
	}.
	By a standard ODE argument the flow is regular w.r.t. $x$ and $\xi$.
	Then the equation \eqref{ordm} has the solution
	\begin{equation}\label{ars3}
		q_m(\tau;x, \xi)=w(\gamma^{\tau, 0}(x, \xi))
	\end{equation}
	where $\gamma^{\tau, 0}(x, \xi)$ has the explicit form (recall \eqref{def:Calpha2})
	\begin{equation}\label{ars40}
		\gamma^{\tau, 0}(x, \xi)=\big(f(\tau;x), g(\tau;x)\xi\big)\,, \qquad 
		f(\tau;x):=x+\tau\balpha(x)\,, \quad g(\tau;x):= \left( \mathbb{I} + \tau \nabla_{x} \balpha(x) \right)^{-T}\,.
	\end{equation}
	
Indeed
\begin{align*}
\dfrac{\partial}{\partial \tau} (x+\tau\balpha(x(\tau))) &= \dfrac{\partial}{\partial \tau} x(\tau) + \balpha(x(\tau)) + \tau\nabla_x \balpha(x(\tau)) \dfrac{\partial}{\partial \tau} x(\tau)\\
&= \balpha(x(\tau))  - (\mathbb{I} + \tau\nabla_x \balpha(x(\tau)) ) b(x(\tau))\\
&= \balpha(x(\tau))  - (\mathbb{I} + \tau\nabla_x \balpha(x(\tau)) ) (\mathbb{I} + \tau\nabla_x \balpha(x(\tau)) )^{-1} \balpha(x(\tau))=0
\end{align*}

For the second one, let us start looking at $\nabla_x b$
\begin{align*}
\partial_{x_k}[b(\tau, x)] &= \partial_{x_k} [(\mathbb{I} + \tau\nabla_x \balpha(x) )^{-1} \balpha(x)]\\
&= -\tau (\mathbb{I} + \tau\nabla_x \balpha(x) )^{-1} ( \partial_{x_k}\nabla_x \balpha) (\mathbb{I} + \tau\nabla_x \balpha(x) )^{-1} \balpha + (\mathbb{I} + \tau\nabla_x \balpha(x) )^{-1}\partial_{x_k} \balpha\\
&= -\tau (\mathbb{I} + \tau\nabla_x \balpha(x) )^{-1} (\partial_{x_k}\nabla_x \balpha)[ b ]+ (\mathbb{I} + \tau\nabla_x \balpha(x) )^{-1}\partial_{x_k} \balpha
\end{align*}
And therefore 
\begin{equation}\label{shop}
(\mathbb{I} + \tau\nabla_x \balpha(x) )\partial_{x_k} b(\tau, x) = \partial_{x_k} \balpha - \tau (\partial_{x_k} \nabla_x\balpha) [b]
\end{equation}
\begin{align}
\dot{\xi}_k(\tau) &\overset{\eqref{charsys}}{=} \sum_{j=1}^{2} \partial_{x_k} b_j(x(\tau)) \xi_j(\tau)
\overset{\eqref{ars40}}{=} \sum_{j=1}^{2} \partial_{x_k} b_j(x(\tau))[(\mathbb{I}+\tau\nabla_x\balpha)^T\xi]_j\nonumber\\
&=  \partial_{x_k} b(x(\tau))\cdot(\mathbb{I}+\tau(\nabla_x\balpha)^T)\xi=(\mathbb{I}+\tau(\nabla_x\balpha)^T)\partial_{x_k} b(x(\tau))\cdot\xi\nonumber\\
&\overset{\eqref{shop}}{=} \partial_{x_k} \balpha(x(\tau))\cdot\xi - \tau(\partial_{x_k}\nabla_x\balpha[b])\cdot\xi=  \partial_{x_k} \balpha(x(\tau))\cdot\xi - \tau \sum_{j, p=1}^{2} ((\partial_{x_k}\partial_{x_p} \balpha_j) b_p )\xi_j\nonumber\\
&=  \partial_{x_k} \balpha(x(\tau))\cdot\xi - \tau \sum_{j=1}^{2} \Big( \sum_{p=1}^{2}  \partial_{x_k}(\nabla_x\balpha)_j^p b_p\Big) \xi_j = \partial_{x_k} \balpha(x(\tau))\cdot\xi - \tau\partial_{x_k}\big( (\nabla_x\balpha b)\cdot \xi\big)\label{uguaglia}
\end{align}
and therefore
\begin{equation}
\dot{\xi}(\tau)= (\nabla_x \balpha)^T[\xi] -\tau \nabla_x \big( (\nabla_x\balpha b)\cdot \xi\big)
\end{equation}

By the other hand, if $\xi(\tau)=(\mathbb{I}+\tau\nabla_x \balpha(x(\tau)))^T[\xi]$ then 
\begin{equation}
\xi_k(\tau)=\xi_k +\tau\sum_{j=1}^{2}\partial_{x_k}\balpha_j(x(\tau)) \xi_j
\end{equation}
and therefore
\begin{equation}
\partial_\tau \xi_k(\tau)=\sum_{j=1}^{2} \partial_{x_k}\balpha_j(x(\tau)) \xi_j + \tau \sum_{j, p=1}^{2} ((\partial_{x_k}\partial_{x_p} \balpha_j)\partial_\tau x_p(\tau) )\xi_j
\end{equation}
that, if $\xi(\tau)$ solves \eqref{charsys}, it is equal to \eqref{uguaglia}.

	\vspace{0.5em}
	\noindent
	Hence by Lemma \ref{Lemmino} the symbol $q_{m}(\tau;x,\xi)$ in \eqref{ars3} 
	satisfies  \eqref{parlare}.
		
	\vspace{0.5em}
	\noindent
	{\bf Order $m-k$ with $1\leq k\leq m+M-1$.} 
	We now assume inductively that 
	we have already found the appropriate solutions $q_{m-h}\in  \mathcal{S}^{m-h}$ 
	of \eqref{ordmmenok}
	with $0\leq h<k$, for some $k\ge 1$. We assume moreover that 
	\begin{equation}\label{zeppelinIndut}
		| {\rm Op}(q_{m-h}) | ^{{\rm Lip} (\gamma)}_{m-h, s, \alpha}\lesssim_{m, s, \alpha, M}  
		\sum_{k_{1}+k_{2}+k_{3}=s}  
		| {\rm Op}(w) |^{{\rm Lip} (\gamma)}_{m, k_1, k_2+\sigma_{h}+\alpha} 
		\| \balpha \|^{{\rm Lip} (\gamma)}_{k_3+\sigma_{h}}\,,\quad 1\leq h<k\,,
	\end{equation}
	for some non decreasing sequence of parameters $\sigma_{h}$ 
	depending only on $|m|$ and $M$. 
	Notice that in \eqref{zeppelinIndut} the Sobolev  norm of  $\balpha$ 
	does not contain  the parameter $\alpha$. This is due to the fact that $\chi$ is linear in $\xi$.
	We now construct the solution $q_{m-k}$ of \eqref{ordmmenok} 
	satisfying estimate \eqref{zeppelinIndut} with $h=k$.
	Using Lemma \ref{lemma composizione 2} we deduce 
(recall also that $\chi$ is linear in $\xi$ and $\mathtt{w}:=k-h+1$)
\begin{equation}\label{stimaResti}
\begin{aligned}
| {\rm Op}(r_{m-k}) | _{m-k,s, \alpha}^{{\rm Lip}(\gamma)}
&\stackrel{\eqref{sperobene}}{\lesssim}
\sum_{h=0}^{k-1} \sum_{|\beta|=\mathtt{w}}
\frac{1}{h! } |(\nabla_{\xi}^\beta q_{m-h})(\nabla_x^\beta \chi)|_{m-k,s, \alpha}^{{\rm Lip}(\gamma)}
\\&\stackrel{\eqref{normachichi}, \eqref{prodotto.simboli}}{\lesssim_{m,s, \alpha}}
| {\rm Op}(q_{m}) | _{m,s, \alpha+k+1}^{{\rm Lip}(\gamma)}
\|\balpha\|_{s_0+k+2}^{{\rm Lip}(\gamma)}
+
| {\rm Op}(q_{m}) | _{m,s_0, \alpha+k+1}^{{\rm Lip}(\gamma)}
\|\balpha\|_{s+k+2}^{{\rm Lip}(\gamma)}
\\&+
\sum_{h=1}^{k-1}
| {\rm Op}(q_{m-h}) | _{m-h,s, \alpha+\mathtt{w}}^{{\rm Lip}(\gamma)}
\|\balpha\|_{s_0+1+\mathtt{w}}^{{\rm Lip}(\gamma)}
+
| {\rm Op}(q_{m-h}) | _{m-h,s_0, \alpha+\mathtt{w}}^{{\rm Lip}(\gamma)}
\|\balpha\|_{s+1+\mathtt{w}}^{{\rm Lip}(\gamma)}\, .
\end{aligned}
\end{equation} 
Let us consider first the second  summand, which is the most complicated: by  the inductive assumption \eqref{zeppelinIndut} with $1 \leq h\leq k-1$, 
we deduce 
\[
\begin{aligned}
\sum_{h=1}^{k-1}
&| {\rm Op}(q_{m-h}) | _{m-h,s, \alpha+\mathtt{w}}^{{\rm Lip}(\gamma)}
\|\balpha\|_{s_0+1+\mathtt{w}}^{{\rm Lip}(\gamma)}
+
| {\rm Op}(q_{m-h}) | _{m-h,s_0, \alpha+\mathtt{w}}^{{\rm Lip}(\gamma)}
\|\balpha\|_{s+1+\mathtt{w}}^{{\rm Lip}(\gamma)}
\\&
\stackrel{\eqref{zeppelinIndut}}{\lesssim_{m,s, \alpha,M}}
\sum_{h=0}^{k-1}
\sum_{k_{1}+k_{2}+k_{3}=s} 
		| {\rm Op}(w) |^{{\rm Lip} (\gamma)}_{m, k_1, k_2+\sigma_{h}+p+\mathtt{w}} 
		\| \balpha \|^{{\rm Lip} (\gamma)}_{k_3+\sigma_{h}+\mathtt{w}}
\|\balpha\|_{s_0+1+\mathtt{w}}^{{\rm Lip}(\gamma)}
\\&\qquad\quad+
\sum_{h=0}^{k-1}
\sum_{k_{1}+k_{2}+k_{3}=s_0} 
		| {\rm Op}(w) |^{{\rm Lip} (\gamma)}_{m, k_1, k_2+\sigma_{h}+\alpha+\mathtt{w}} 
		\| \balpha \|^{{\rm Lip} (\gamma)}_{k_3+\sigma_{h}+\mathtt{w}}
\|\balpha\|_{s+1+\mathtt{w}}^{{\rm Lip}(\gamma)}
\\&\lesssim_{m, s, \alpha,M} \sum_{h=0}^{k-1}
\sum_{k_{1}+k_{2}+k_{3}=s} 
| {\rm Op}(w) |^{{\rm Lip} (\gamma)}_{m, k_1, k_2+\sigma_{h}+\alpha+\mathtt{w}} 
\| \balpha \|^{{\rm Lip} (\gamma)}_{k_3+\sigma_{h}+\mathtt{w}}
\\ &
\qquad\quad+
\sum_{h=0}^{k-1} \sum_{1\le k_1+k_2 \le s_0}
| {\rm Op}(w) |^{{\rm Lip} (\gamma)}_{m, k_1, k_2+\sigma_{h}+\alpha+\mathtt{w}} 
\|\balpha\|_{s-s_0+1+\mathtt{w} +s_0 }^{{\rm Lip}(\gamma)} 
\\&	
\lesssim_{m, s, \alpha, M} 
\sum_{s}^{*} 
| {\rm Op}(w) |^{{\rm Lip} (\gamma)}_{m, k_1, k_2+\widehat{\sigma}_{k-1}+\alpha} 
\| \balpha \|^{{\rm Lip} (\gamma)}_{k_3+\widehat{\sigma}_{k-1}} \, . 
\end{aligned}
\]
In the fourth and fifth line we have used  the smallness assumption
\eqref{buf} 
(assuming $\sigma\geq {\sigma}_{k-1}+1+ k 
$), and in the last line we have set $\hat\sigma_{k-1}:= \sigma_{k-1} +k +2 +s_0$.
In dealing with the first summand (i.e. the third line in \eqref{sperobene}) we proceed in the same way, only we substitute \eqref{parlare} instead of \eqref{zeppelinIndut}.
We conclude that 
\begin{equation}
	\label{verdone1}
| {\rm Op}(r_{m-k}) | _{m-k,s, \alpha}^{{\rm Lip}(\gamma)}
{\lesssim}_{m, s, \alpha, M} 
\sum_{k_{1}+k_{2}+k_{3}=s} 
| {\rm Op}(w) |^{{\rm Lip} (\gamma)}_{m, k_1, k_2+\widehat{\sigma}_{k-1}+\alpha} 
\| \balpha \|^{{\rm Lip} (\gamma)}_{k_3+\widehat{\sigma}_{k-1}}\,.
\end{equation}
	We now reason similarly to  $q_m$, and solve \eqref{ordmmenok} by variation of constants.
	We define $f_{m-k}(\tau):=q_{m-k}(\tau;x(\tau),\xi(\tau))$
	where $x(\tau),\xi(\tau)$ are the solution of the Hamiltonian system \eqref{charsys}.
	One has that, if $q_{m-k}$ solves \eqref{ordmmenok}, then
	\[
	\frac{d}{d\tau}f_{m-k}(\tau)=r_{m-k}(\tau;x(\tau),\xi(\tau))\quad \Rightarrow
	\quad
	f_{m-k}(\tau)=\int_{0}^{\tau} r_{m-k}(\sigma ; x(\sigma),\xi(\sigma))d\sigma\,,
	\]
	where we used that $f(0)=q_{m-k}(0,x(0),\xi(0))=0$.
	Therefore the solution of \eqref{ordmmenok} is
	\begin{equation*}\label{gnomo}
		q_{m-k}(\tau;x, \xi)=
		\int_0^{\tau} r_{m-k}(\gamma^{0, \sigma} \gamma^{\tau, 0}(x, \xi)) \,d\sigma \, .
	\end{equation*}
	We observe also that 
	\begin{equation*}
		\gamma^{0, \sigma} \gamma^{\tau, 0}(x, \xi)=(\tilde{f}, \tilde{g}\,\xi)\,,\qquad \sigma,\tau\in[0,1]\,,
	\end{equation*}
	with
	\begin{equation*}
		\tilde{f}(\sigma, \tau,x):=x+\tau \balpha(x)+\breve{\balpha}(\sigma, x+\tau \balpha(x))\,, 
		\qquad 
		\tilde{g}(\sigma, \tau,x):= \left( \nabla_{x} f (\sigma, \tau, x) \right)^{-T} \, . 
	\end{equation*}
	Thus if $\tilde{A} r:=r(\tilde{f}, \tilde{g}\xi)$ we have (recall that $\tau\in [0, 1]$)
	\begin{equation*}
		\begin{aligned}
			| {\rm Op}(q_{m-k}) |^{{\rm Lip}(\gamma)}_{m-k, s, \alpha}
			&\lesssim_{s, \alpha} 
			|{\rm Op} (\tilde{A} r_{m-k} )|^{{\rm Lip}(\gamma)}_{m-k, s, \alpha}\,,
			\\
			| {\rm Op}(q_{m-k}) | ^{{\rm Lip}(\gamma)}_{m-k, s_0, \alpha}
			&\lesssim_{\alpha} 
			| {\rm Op}(\tilde{A} r_{m-k} )|^{{\rm Lip}(\gamma)}_{m-k, s_0, \alpha}
			\lesssim
			| {\rm Op}(r_{m-k}) | ^{{\rm Lip}(\gamma)}_{m-k, s_0, \alpha+s_0}\,,
		\end{aligned}
	\end{equation*}
	and by Lemma \ref{Lemmino} with  
	$\balpha\rightsquigarrow  \tau \balpha(x)+\breve{\balpha}(\sigma, x+\tau \balpha(x))$\footnote{Notice that 
	the function $t(\tau,\sigma,x):=\tau \balpha(x)+\breve{\balpha}(\sigma, x+\tau \balpha(x))$
	satisfies, using that $y+\breve{\balpha} $ is the inverse diffeomorphism of $x+\tau\balpha$  
	(recall for instance  \eqref{diffeo.e.inverso}) ,
	the estimate	$ \| t(\tau,\sigma)\|_{s}^{{\rm Lip} (\gamma)}\lesssim_s  \|  {\balpha}\|_{s+s_0}^{{\rm Lip} (\gamma)} $
	uniformly in $\tau,\sigma\in[0,1]$.
	}
	\begin{equation}\label{interpol}
		| {\rm Op}(q_{m-k}) | ^{{\rm Lip}(\gamma)}_{m-k, s, \alpha}\lesssim_{s, \alpha} 
		| {\rm Op}(r_{m-k}) | ^{{\rm Lip}(\gamma)}_{m-k, s, \alpha}
		+\sum_{k_{1}+k_{2}+k_{3}=s} 
		| {\rm Op}(r_{m-k}) | ^{{\rm Lip}(\gamma)}_{m-k, k_1, \alpha+k_2} \| \balpha \|^{{\rm Lip}(\gamma)}_{k_3+\widetilde{\sigma}}\,,
	\end{equation}
	where $\widetilde{\sigma}$ depends only on $s_{0}$.
	It remains to prove that actually the symbol
	$q_{m-k}$ satisfies the bound \eqref{zeppelinIndut} with $h=k$ and for some 
	new $\sigma_{k}\geq \sigma_{k-1}$ depending only on $|m|$ and $M$.
By substituting \eqref{verdone1} in the estimate \eqref{interpol}
we get 
\begin{equation}\label{verdone2}
\begin{aligned}
| {\rm Op}(q_{m-k}) | ^{{\rm Lip}(\gamma)}_{m-k, s, \alpha}&\lesssim_{m,s, \alpha,M} 
\sum_{k_{1}+k_{2}+k_{3}=s} 
		| {\rm Op}(w) |^{{\rm Lip} (\gamma)}_{m, k_1, k_2+\widehat{\sigma}_{k-1}+\alpha} 
		\| \balpha \|^{{\rm Lip} (\gamma)}_{k_3+\widehat{\sigma}_{k-1}}
		\\&+\sum_{k_{1}+k_{2}+k_{3}=s}
		\Big(
	\sum_{k'_{1}+k'_{2}+k'_{3}=k_1}
		| {\rm Op}(w) |^{{\rm Lip} (\gamma)}_{m, k_1', k_2'+\widehat{\sigma}_{k-1}+\alpha+k_2} 
		\| \balpha \|^{{\rm Lip} (\gamma)}_{k_3'+\widehat{\sigma}_{k-1}}
			\Big) \| \balpha \|^{{\rm Lip}(\gamma)}_{k_3+\widetilde{\sigma}}\,.
\end{aligned}
\end{equation}
By Lemma \ref{sobolev tame}
we deduce that (if $k_3'\neq0$, otherwise we do nothing)
\[
\begin{aligned}
\| \balpha \|^{{\rm Lip} (\gamma)}_{k'_3+\widehat{\sigma}_{k-1}}\| \balpha \|^{{\rm Lip}(\gamma)}_{k_3+\widetilde{\sigma}}
&\lesssim_{s}
\left(\| \balpha \|^{{\rm Lip} (\gamma)}_{\widehat{\sigma}_{k-1}} \right)^{\frac{k_{3}}{k_{3}+k_{3}'}}\left(\| \balpha \|^{{\rm Lip} (\gamma)}_{k_3'+k_3+\widehat{\sigma}_{k-1}} \right)^{\frac{k_{3}'}{k_{3}+k_{3}'}} \left(\| \balpha \|^{{\rm Lip} (\gamma)}_{\widetilde{\sigma}} \right)^{\frac{k_{3}'}{k_{3}+k_{3}'}}\left(\| \balpha \|^{{\rm Lip} (\gamma)}_{k_3'+k_3+\widetilde{\sigma}}  \right)^{\frac{k_{3}}{k_{3}+k_{3}'}}
\\
&\lesssim_{s} 
\| \balpha \|^{{\rm Lip} (\gamma)}_{\widehat{\sigma}_{k-1}+\widetilde{\sigma}} \| \balpha \|^{{\rm Lip} (\gamma)}_{k_3'+k_3+\widehat{\sigma}_{k-1}+\widetilde{\sigma}} 
\stackrel{\eqref{buf}}{\lesssim}
\| \balpha \|^{{\rm Lip}(\gamma)}_{k'_3+k_3+\widehat{\sigma}_{k-1}+\widetilde{\sigma}} \, . 
\end{aligned}
\]
Therefore the \eqref{verdone2} becomes (by renaming the indexes in the sum)
\[
\begin{aligned}
| {\rm Op}(q_{m-k}) | ^{{\rm Lip}(\gamma)}_{m-k, s, \alpha}&\lesssim_{m,s, \alpha,M} 
\sum_{k_{1}+k_{2}+k_{3}=s} 
		| {\rm Op}(w) |^{{\rm Lip} (\gamma)}_{m, k_1, k_2+\widehat{\sigma}_{k-1}+\alpha} 
		\| \balpha \|^{{\rm Lip} (\gamma)}_{k_3+\widehat{\sigma}_{k-1}}
		\\&+\sum_{k_{1}+k_{2}+k_{3}=s}
	\sum_{k_{1}+k_{2}+k_{3}=k_1} 
		| {\rm Op}(w) |^{{\rm Lip} (\gamma)}_{m, k_1', k_2'+\widehat{\sigma}_{k-1}+\alpha+k_2} 
		\| \balpha \|^{{\rm Lip} (\gamma)}_{k_3'+k_3+\widetilde{\sigma}+\widehat{\sigma}_{k-1}}
		\\&\lesssim_{m,s, \alpha,M} 
			\sum_{k_{1}+k_{2}+k_{3}=s} | {\rm Op}(w) |^{{\rm Lip} (\gamma)}_{m, k_1, k_2+{\sigma}_{k}+\alpha} 
		\| \balpha \|^{{\rm Lip} (\gamma)}_{k_3+{\sigma}_{k}}
\end{aligned}
\]
for some $\sigma_{k}\geq \widehat{\sigma}_{k-1}+\widetilde{\sigma}$ depending only on $|m|$ and $M$.
This is the estimate \eqref{zeppelinIndut} with $h=k$.
To prove \eqref{troppo} we reason as a above, using that $\Delta_{12}{\rm Op}(q)={\rm Op}(\Delta_{12}q)$ and Leibnitz rule.

Let us now estimate the remainder term.
Similarly to \eqref{stimaResti}, 
we have
\begin{equation}\label{stimaRestiRHO}
\begin{aligned}
|{\rm Op}(\mathtt{r}_{-M})|_{-M,s, 0}^{{\rm Lip}(\gamma)}
&\stackrel{\eqref{sperobeneResto}}{\lesssim}
\sum_{k=0}^{m+M-1} |{\rm Op}(r_{m-k+1+M}(q_{m-k},\chi)) |_{-M,s, 0}^{{\rm Lip}(\gamma)}
\\&
\stackrel{\eqref{stima cal RN comp memoirs}}{\lesssim_{m,s,M}}
\sum_{k=0}^{m+M-1}
| {\rm Op}(q_{m-k}) | ^{{\rm Lip} (\gamma)}_{m-k, s, |m|+1+M} | \chi |^{{\rm Lip} (\gamma)}_{1, s_0+3|m|+3+3M, 0} 
\\&
\qquad\quad 
+ | {\rm Op}(q_{m-k}) | ^{{\rm Lip} (\gamma)}_{m-k, s_0, |m|+1+M} | \chi |^{{\rm Lip} (\gamma)}_{1, s+3|m|+3+3M, 0}
\\&
\stackrel{\eqref{normachichi}}{\lesssim_{m,s,M}}
\sum_{k=0}^{m+M-1}
| {\rm Op}(q_{m-k}) | ^{{\rm Lip} (\gamma)}_{m-k, s, |m|+1+M} \|\balpha\|^{{\rm Lip} (\gamma)}_{s_0+3|m|+4+3M} 
\\&
\qquad\quad 
+ | {\rm Op}(q_{m-k}) | ^{{\rm Lip} (\gamma)}_{m-k, s_0, |m|+1+M} \|\balpha\|^{{\rm Lip} (\gamma)}_{s+3|m|+4+3M}\,.
\end{aligned}
\end{equation} Substituting the bound \eqref{zeppelinIndut} we get
\begin{equation}\label{stimaerrinomenorho}
\begin{aligned}
	|{\rm Op}(\mathtt{r}_{-M} )|_{-M,s, 0}^{{\rm Lip}(\gamma)}
	&{\lesssim_{m,s,M}}
	\sum_{k=0}^{m+M-1}
	\big(	\sum_{k_{1}+k_{2}+k_{3}=s} | {\rm Op}(w) |^{{\rm Lip} (\gamma)}_{m, k_1, k_2+{\sigma}_{k}+|m|+1+M} 
	\| \balpha \|^{{\rm Lip} (\gamma)}_{k_3+{\sigma}_{k}}\big) \|\balpha\|^{{\rm Lip} (\gamma)}_{s_0+3|m|+4+3M} 
	\\&+
	\sum_{k=0}^{m+M-1}
\big(	\sum_{k_{1}+k_{2}+k_{3}=s_0} | {\rm Op}(w) |^{{\rm Lip} (\gamma)}_{m, k_1, k_2+{\sigma}_{k}+|m|+1+M} 
\| \balpha \|^{{\rm Lip} (\gamma)}_{k_3+{\sigma}_{k}}\big) \|\balpha\|^{{\rm Lip} (\gamma)}_{s+3|m|+4+3M} 
\\ & 
\lesssim_{m,s,M}\sum_{k_{1}+k_{2}+k_{3}=s} | {\rm Op}(w) |^{{\rm Lip} (\gamma)}_{m, k_1, k_2+\widehat{\sigma}} 
\| \balpha \|^{{\rm Lip} (\gamma)}_{k_3+\widehat{\sigma}}\,,
\end{aligned}
\end{equation}
for some $\widehat{\sigma}\geq \sigma_{m+M-1}+3|m|+4+3M$ 
(recall that $\sigma_{k}$ is a non-decreasing sequence in $k$)
and provided that \eqref{buf} holds for some $\sigma\geq \widehat{\sigma}$.
In the last inequality we also used the interpolation estimate in Lemma \ref{sobolev tame}
following the reasoning used for \eqref{verdone2}.
	\noindent
	We now conclude the construction of the solution of \eqref{ars}.
We set (recall \eqref{balconata0})
	\begin{equation*}
		P^{\tau}=Q^{\tau}+R^{\tau}\,, \qquad Q^\tau={\rm Op}(q)\in \mathcal{OPS}^m\,.
	\end{equation*}
	By construction we have that (recall \eqref{balconata0},  \eqref{probapproxsimboloq}) 
	\[
	\partial_{\tau} Q^{\tau}=[{X}^{\tau}, Q^{\tau}] + \mathcal{M}^\tau\,, \quad  \mathcal{M}^\tau:={\rm Op}(\mathtt{r}_{-M})  \, , 
	\]
so
	\begin{equation}\label{equazioneRRtauesplicita}
		\begin{cases}
			\partial_{\tau} R^{\tau}=[{X}^{\tau}, R^{\tau}]-\mathcal{M}^{\tau}\\
			R^{0}=0\, . 
		\end{cases}
	\end{equation}
	We set $V^{\tau}:=(\mathcal{A}^{\tau})^{-1}\circ R^{\tau}\circ \mathcal{A}^{\tau}$
	and we note that 
	$V^{0}=(\mathcal{A}^{0})^{-1}\circ R^{0}\circ\mathcal{A}^{0}
	={\rm Id}\circ R^{0}\circ {\rm Id}=R^{0}=0$.
	Moreover, by using \eqref{equazioneRRtauesplicita}, \eqref{flussodiffeo}, we have
\[
\begin{aligned}
\partial_{\tau}V^{\tau}&=
\partial_{\tau}((\mathcal{A}^{\tau})^{-1})\circ R^{\tau}\circ \mathcal{A}^{\tau}+
(\mathcal{A}^{\tau})^{-1}\circ(\partial_{\tau}R^{\tau})\circ\mathcal{A}^{\tau}
+(\mathcal{A}^{\tau})^{-1}\circ R^{\tau}\circ(\partial_{\tau}\mathcal{A}^{\tau})
\\&
\stackrel{\eqref{equazioneRRtauesplicita}, \eqref{flussodiffeo}}{=}
-(\mathcal{A}^{\tau})^{-1}\circ X^{\tau}\circ R^{\tau}\circ  \mathcal{A}^{\tau}
\\&+
(\mathcal{A}^{\tau})^{-1}\circ\big([{X}^{\tau}, R^{\tau}]-\mathcal{M}^{\tau}\big)\circ\mathcal{A}^{\tau}
+(\mathcal{A}^{\tau})^{-1}\circ R^{\tau}\circ X^{\tau}\circ \mathcal{A}^{\tau}
\\&
=
-(\mathcal{A}^{\tau})^{-1}\circ \mathcal{M}^{\tau}\circ\mathcal{A}^{\tau}\,.
\end{aligned}
\]
Therefore $V^{\tau}$
 solves the problem 
	\begin{equation*}
		\left\{\begin{aligned}
			&\partial_{\tau} V^{\tau} = 
			-(\mathcal{A}^{\tau})^{-1}\circ\mathcal{M}^{\tau}\circ  \mathcal{A}^{\tau}
			\\&
			V^{0}=0\,,
		\end{aligned}\right.
	\end{equation*}
	whose solution is 
	\begin{equation*}
		V^{\tau}=-
		\int_0^{\tau} 
		(\mathcal{A}^t)^{-1} \circ
		\mathcal{M}^{t}\circ \mathcal{A}^{t}\,dt \,.
	\end{equation*}
	We then deduce that 
	\begin{equation*} 
		R^{\tau}=-
		\int_0^{\tau} \mathcal{A}^{\tau}\circ (\mathcal{A}^t)^{-1} \circ
		\mathcal{M}^t \circ \mathcal{A}^{t} \circ(\mathcal{A}^{\tau})^{-1}\,dt\,.
	\end{equation*}
	By using \eqref{interpolazione bassa alta}, Lemmata \ref{lem:changevar} and \ref{lemma azione tame pseudo-diff} and Lemma \ref{sobolev tame} we obtain the bound \eqref{francia1}.
	
Finally, if ${\rm Op}(w(x, \xi))$ is a real operator then by induction one can prove that ${\rm Op}(q(x, \xi))$ is real. Indeed, ${\rm Op}(q_m(\tau;x, \xi))={\rm Op}(w(\gamma^{\tau, 0}(x, \xi)))$ so it is real and ${\rm Op}(q_{m-k})$ is real by \eqref{gnomo}, \eqref{sperobene} and Lemma \ref{composizione.simboli.reali}.
Therefore, by difference also $\mathcal R$ is a real operator.	
\end{proof}

\section{The linearized operator}\label{sezione linearizzato}
We will show the existence of bi-periodic traveling waves for the MHD system, namely we will find zeroes of the nonlinear map ${\mathcal F}$ defined in \eqref{mappa nonlineare iniziale}, i.e. 
\begin{equation}\label{def:F}
\begin{aligned}
& \mathcal{F}(\Omega,J)  :={\lambda \omega\cdot \nabla \Omega- \Delta \Omega - \mathbf b \cdot \nabla J+\lambda^\delta[(U\cdot \nabla)\Omega-(B\cdot \nabla)J]-\lambda^{1 -\frac23 \delta} F \choose \lambda\omega\cdot \nabla J - \mathbf b \cdot \nabla \Omega + \lambda^\delta[(U\cdot \nabla)J-(B\cdot \nabla)\Omega-2\mathcal{H}(\Omega,J)]}, \\
& U := {\mathcal U} \Omega, \quad B :=  {\mathcal U} J
\end{aligned}
\end{equation}
where we recall that ${\mathcal U}$ is the Biot-Savart operator defined in \eqref{def bio savart operator} and  $\mathcal{H}$ is defined as follows
\begin{equation}
\begin{aligned}
\mathcal{H}(\Omega,J) & =H(\mathcal{U}\Omega,\mathcal{U}J) \\
& = - \partial_{x_1} {\mathcal U} J \cdot \nabla \partial_{x_1} (- \Delta)^{- 1} \Omega - \partial_{x_2} {\mathcal U} J \cdot \nabla \partial_{x_2} (- \Delta)^{- 1} \Omega\,. 
\end{aligned}
\end{equation}
The main point is the inversion of the linearized operator at any $(\Omega, J)$ where $\Omega, J \in C^\infty(\T^2)$, with zero average. We shall make the following ansatz that will be used systematically in all our proofs. 
\begin{equation}\label{ansatz}
\begin{aligned}
& {\mathcal I} := (\Omega, J) \quad \text{satisfies} \quad \| {\mathcal I} \|_{s_0 + \sigma}^{\Lip(\gamma)} \leq C(s_0, \sigma) ,\\
&   \text{for some large constants} \quad \sigma ,\, C(\sigma, s_0) \gg 0\,. 
\end{aligned}
\end{equation}
The linearized operator at $(\Omega, J)$ is given by $\Lin:=D \mathcal{F}(\Omega,J)$
\begin{equation}\label{Lin3}
\Lin(\hat{\Omega},\hat{J}):={(\lambda \omega\cdot \nabla -\Delta)\hat{\Omega}+ a(x)\cdot\nabla \hat{\Omega} +  d(x)\cdot\nabla \hat{J}+\mathcal{R}_1(\hat{\Omega})+\mathcal{R}_2(\hat{J})
\choose 
\lambda \omega\cdot \nabla  \hat{J}+ a(x)\cdot\nabla \hat{J} + d(x)\cdot\nabla \Omega+\mathcal{R}_3(\hat{\Omega})+\mathcal{R}_4(\hat{J})
},
\end{equation}
where the operators $\mathcal{R}_i$ are defined as follows (recall the definition \eqref{op one smoothing lin bio savart})
\begin{equation}\label{def:R1}
\begin{aligned}
\mathcal{R}_1 [\widehat \Omega]&= \lambda^\delta {\mathcal R}(\Omega) [\widehat \Omega] = \lambda^\delta  (\U\hat{\Omega}\cdot\nabla)\Omega,\\
\mathcal{R}_2 [\widehat J]&= - \lambda^\delta {\mathcal R}(J)[\widehat J] = - \lambda^\delta (\U \hat{J}\cdot\nabla)J ,\\
\mathcal{R}_3[\widehat \Omega]&= \lambda^\delta (\U\hat{\Omega}\cdot\nabla)J-2 \lambda^\delta \mathcal{H}(\hat{\Omega},J) \\ 
& = \lambda^\delta {\mathcal R}(J)[\widehat \Omega] + 2 \lambda^\delta  \partial_{x_1} {\mathcal U} J \cdot \nabla \partial_{x_1} (- \Delta)^{- 1} \widehat\Omega \\
& \quad  + 2 \lambda^\delta  \partial_{x_2} {\mathcal U} J \cdot \nabla \partial_{x_2} (- \Delta)^{- 1} \widehat\Omega  , \\
\mathcal{R}_4 [\widehat J]&=\lambda^\delta (\U \hat{J}\cdot\nabla)\Omega-2 \lambda^\delta \mathcal{H}(\Omega,\hat{J}) \\
& = \lambda^\delta {\mathcal R}(\Omega)[\widehat J] + 2 \lambda^\delta  \partial_{x_1} {\mathcal U} \widehat J \cdot \nabla \partial_{x_1} (- \Delta)^{- 1} \Omega  \\
& \quad + 2 \lambda^\delta  \partial_{x_2} {\mathcal U} \widehat J \cdot \nabla \partial_{x_2} (- \Delta)^{- 1} \Omega   ,
\end{aligned}
\end{equation}
and the vector fields $a(x), d(x)$ are given by
\begin{equation}\label{definizione a}
a(x):= \lambda^\delta (\U\Omega) (x),
\end{equation}
\begin{equation}\label{definizione m}
d(x):=- \mathbf b - \lambda^\delta (\U J) (x).
\end{equation}
By using that ${\mathcal U} \Omega, {\mathcal U} J$ are zero divergence vector fields and have zero average, a direct calculation shows that $a \cdot \nabla, d \cdot \nabla, {\mathcal R}_i$, $i = 1, \ldots, 4$ leave invariant the space of zero average functions, namely 
\begin{equation}\label{invarianze medie nulle}
\begin{aligned}
& [a \cdot \nabla, \Pi_0^\bot] = [d \cdot \nabla, \Pi_0^\bot] = 0, \quad [{\mathcal R}_i, \Pi_0^\bot ] = 0, \quad i = 1, \ldots, 4 \\
&  \text{and hence} \quad a \cdot \nabla = \Pi_0^\bot a \cdot \nabla \Pi_0^\bot, \quad d \cdot \nabla = \Pi_0^\bot d \cdot \nabla \Pi_0^\bot\,, \\
& {\mathcal R}_i = \Pi_0^\bot {\mathcal R}_i \Pi_0^\bot, \quad i = 1, \ldots, 4\,. \\
& \text{The latter properties imply that} \quad {\mathcal L} = \Pi_0^\bot {\mathcal L} \Pi_0^\bot\,. 
\end{aligned}
\end{equation}
Moreover, since $\mathcal{U}$ is real then all the operators in ${\mathcal L}$ in \eqref{Lin3} are real.
With these notations, we can rewrite the operator $\Lin$ in the following matrix formulation
\begin{equation}\label{block representation cal L}
\Lin
=
\begin{pmatrix}
\lambda \omega\cdot \nabla -\Delta+  a(x)\cdot\nabla &  d(x)\cdot\nabla\\
 d(x)\cdot\nabla &\lambda \omega\cdot \nabla +  a(x)\cdot\nabla
\end{pmatrix}
+
\begin{pmatrix}
\mathcal{R}_1 & \mathcal{R}_2 \\
\mathcal{R}_3 & \mathcal{R}_4
\end{pmatrix}.
\end{equation}
The following lemma holds.
\begin{lem}\label{stime linearized originario}
Assume \eqref{ansatz} with $\sigma = 1$. Then for any $s \geq s_0$, $\alpha \in \N_{0}$ one has 
\begin{equation}\label{stima in a m cal R i}
\begin{aligned}
& \| a \|_s^{\Lip(\gamma)} \lesssim_s \lambda^\delta \| {\mathcal I} \|_{s }^{\Lip(\gamma)} \,, \quad  \, \| d \|_s^{\Lip(\gamma)}  \lesssim_s \lambda^\delta (1 +  \| {\mathcal I} \|_{s}^{\Lip(\gamma)})\,, \\
& {\mathcal R}_1, {\mathcal R}_2 \in \Op^{- 1}\,, \quad {\mathcal R}_3, {\mathcal R}_4 \in \Op^0\,, \\
& |{\mathcal R}_1|_{- 1, s, \alpha}^{\Lip(\gamma)} \,,\, |{\mathcal R}_2|_{- 1, s, \alpha}^{\Lip(\gamma)}\,,\, |{\mathcal R}_3|_{0, s, \alpha}^{\Lip(\gamma)}\,,\, |{\mathcal R}_4|_{0, s, \alpha}^{\Lip(\gamma)} \lesssim_{s, \alpha} \lambda^\delta \| {\mathcal I} \|_{s + 1}^{\Lip(\gamma)}\,.
\end{aligned}
\end{equation}
Let $s_1 \geq s_0$, $\alpha \in \N_{0}$ and assume that ${\mathcal I}_1, {\mathcal I}_2$ satisfy \eqref{ansatz} with $s_1 + 1$ instead of $s_0 + 1$. Then 
\begin{equation}\label{stime iniziali Delta 12 a m cal R i}
\begin{aligned}
& \| \Delta_{12} a \|_{s_1}\,,\, \| \Delta_{12} d \|_{s_1} \lesssim_{s_1} \lambda^\delta \| {\mathcal I}_1 - {\mathcal I}_2 \|_{s_1 }\,, \\
& |\Delta_{12}{\mathcal R}_1|_{- 1, s_1, \alpha} \,,\, |\Delta_{12} {\mathcal R}_2|_{- 1, s_1, \alpha}\,,\, | \Delta_{12} {\mathcal R}_3|_{0, s_1, \alpha}\,,\, |\Delta_{12} {\mathcal R}_4|_{0, s_1, \alpha} \lesssim_{s_1, \alpha} \lambda^\delta \| {\mathcal I}_1 - {\mathcal I}_2 \|_{s_1 + 1}\,.
\end{aligned}
\end{equation}
\end{lem}
\begin{proof}
The lemma follows by the definitions \eqref{def:R1}-\eqref{definizione m} and by applying Lemma \ref{sobolev tame}, and the estimates \eqref{fourier multiplier norm}, \eqref{multiplication norms}, \eqref{stima biot savart1}, \eqref{stima biot savart2}. 
\end{proof}
We now state the following lemma which summarize some elementary properties of the map ${\mathcal F}$ that we shall use in the sequel. 
\begin{lem}\label{prop elementari mappa cal F}
Let ${\mathcal F}$ be the nonlinear map defined in \eqref{mappa nonlineare iniziale}. We have the following.
\begin{itemize}
\item[$(i)$] The map ${\mathcal F}$ is invariant on the space of zero average functions. More precisely for any $s \geq s_0$, ${\mathcal F} : H^{s + 2}_0 \times H^{s + 2}_0 \to H^s_0 \times H^s_0$. 

\noindent
\item[$(ii)$] Let $s \geq s_0$, ${\mathcal I}(\omega) \in {\mathcal C}^\infty(\T^2, \R^2)$ (with zero average) with $\| {\mathcal I} \|_{s_0 + 1}^{\Lip(\gamma)} \leq C_0$. Then the linearized operator ${\mathcal L}({\mathcal I}) := D {\mathcal F}({\mathcal I})$ satisfies the following properties. One has that $[{\mathcal L}({\mathcal I}), \Pi_N] =  [\Pi_N^\bot , {\mathcal L}({\mathcal I})]$ and for any $a > 0$, 
$$
\| [{\mathcal L}({\mathcal I}), \Pi_N^\bot][\widehat{\mathcal I}] \|_{s_0}^{\Lip(\gamma)} \lesssim_s \lambda^\delta N^{1 - a} \Big(  \| \widehat{\mathcal I} \|_{s_0 + a }^{\Lip(\gamma)} + \| {\mathcal I} \|_{s_0 +a}^{\Lip(\gamma)} \| \widehat{\mathcal I} \|_{s_0 + 1}^{\Lip(\gamma)}\Big),
$$
and 
$$
\| [{\mathcal L}({\mathcal I}), \Pi_N][\widehat{\mathcal I}] \|_{s_0 + a}^{\Lip(\gamma)} \lesssim_s \lambda^\delta N \Big(  \| \widehat{\mathcal I} \|_{s_0  + a}^{\Lip(\gamma)} + \| {\mathcal I} \|_{s_0 +a+ 1}^{\Lip(\gamma)} \| \widehat{\mathcal I} \|_{s_0 + 1}^{\Lip(\gamma)}\Big)\,.
$$

\item[$(iii)$] Let $s \geq s_0$, ${\mathcal I}(\omega)\,,\, \widehat{\mathcal I}(\omega) \in H^{s + 1}_0 \times H^{s + 1}_0$ and let us define 
$$
{\mathcal Q}(\widehat{\mathcal I}, \widehat{\mathcal I}) := {\mathcal F}({\mathcal I} + \widehat{\mathcal I}) - {\mathcal F}({\mathcal I} ) - {\mathcal L}({\mathcal I})[\widehat{\mathcal I}]\,.
$$
Then, one has 
$$
\| {\mathcal Q}(\widehat{\mathcal I}, \widehat{\mathcal I}) \|_{s}^{\Lip(\gamma)} \lesssim_s  \lambda^{ \delta} \| \widehat{\mathcal I} \|_{s + 1}^{\Lip(\gamma)} \| \widehat{\mathcal I} \|_{s_0 + 1}^{\Lip(\gamma)} \,. 
$$
\end{itemize}
\end{lem}
\begin{proof}
The fact that if ${\mathcal I}$ has zero average, then ${\mathcal F}({\mathcal I})$ has zero average follows by explicit intgration by parts, using that $\Omega, J, U, B$ have zero average and $U = {\mathcal U} \Omega$, $B = {\mathcal U} J$ are zero divergence vector fields. All the claimed bounds follow by explicit calculations, by applying interpolation estimates \eqref{interpolazione bassa alta}, the smoothing estimates of Lemma \ref{lem:smoothing}, The estimates of Lemma \ref{stime linearized originario} and using the trivial fact that $\Pi_N$ and $\Pi_N^\bot$ commute with the diagonal operator $\begin{pmatrix}
\lambda \omega \cdot \nabla - \Delta & 0 \\
0 & \lambda \omega \cdot \nabla
\end{pmatrix}$ .
\end{proof}


\section{Decoupling of linearized operator up to a smoothing remainder}\label{sezione decoupling ordine uno off diag}
The aim of this section is to decouple the equations for the velocity and the magnetic field. We will implement an iterative procedure whose goal is to remove the operators in the off-diagonal entries of $\Lin$ up to an arbitrary regularizing term of order $-N$, for some $N\in \N$ fixed. The trickiest part comes from the decoupling of the order 1 which we present in Subsection \ref{sub:dec1}. Then, we illustrate the iterative procedure in Subsection \ref{subsub:decit}. We use the notation $\mathcal{F}$ to denote block diagonal matrix operators, while we denote with $\mathcal{Q}$ block off-diagonal matrix operators. Lastly, we denote with $\mathcal{R}$ the remainders term.
\subsection{Decoupling of the first order part}\label{sub:dec1}
We look for a transformation which removes the transport term $ d(x)\cdot\nabla$ from the off-diagonal entries of $\Lin$ up to a remainder of order $\frac12$. The main technical obstruction comes from the fact that $d(x) = O(\lambda^\delta)$, $\lambda \gg 1$ large enough. 

\noindent
Then, we consider a transformation $\Psi_\bot $ of the form
 \begin{equation}\label{def:Psi}
\Psi_\bot := \Pi_0^\bot \Psi \Pi_0^\bot, \quad \Psi :=\begin{pmatrix}
0 & \Psi_1\\
\Psi_2 & 0
\end{pmatrix},
\end{equation}
where $\Psi_1 = {\rm Op}(\psi_1(x,\xi)),\Psi_2 = {\rm Op}(\psi_2(x,\xi))$ where $\psi_1,\psi_2 \in {\mathcal S}^{- \frac12}$ have to be determined in order to cancel the highest order off diagonal term $ \begin{pmatrix}
0 & d(x) \cdot \nabla \\
d(x) \cdot \nabla & 0
\end{pmatrix}$. We split ${\mathcal L}$ as 
\begin{equation}\label{mathcal L expansion 0}
\begin{aligned}
{\mathcal L} & = {\mathcal D} + {\mathcal F}_1+{\mathcal Q}_1 + {\mathcal R}_0\,, \\
{\mathcal D} & := \begin{pmatrix}
\lambda \omega \cdot \nabla - \Delta & 0 \\
0 & \lambda \omega \cdot \nabla
\end{pmatrix}\,, \\
{\mathcal F}_1 & :=  \begin{pmatrix}
a \cdot \nabla & 0 \\
0 & a \cdot \nabla
\end{pmatrix}\,, \\
{\mathcal Q}_1 & :=  \begin{pmatrix}
0 &  d\cdot \nabla \\
d\cdot \nabla & 0
\end{pmatrix}\,, \\
{\mathcal R}_0 & :=  \begin{pmatrix}
\mathcal{R}_1 & \mathcal{R}_2 \\
\mathcal{R}_3 & \mathcal{R}_4
\end{pmatrix}.
\end{aligned}
\end{equation}
Note that $\mathcal{D}$ is a diagonal operator in the sense of Definition \ref{def:diagonal}. We compute ${\mathcal L}^{(0)} := e^{- \Psi_\bot} {\mathcal L} e^{\Psi_\bot}$: we apply the Lie expansion and we get that
\begin{equation}\label{def:R11}
\begin{aligned}
e^{- \Psi_\bot} {\mathcal D} e^{\Psi_\bot}  &= {\mathcal D} + [{\mathcal D}, \Psi_\bot] + \frac12  [[{\mathcal D}, \Psi_\bot]\,,\, \Psi_\bot ] +  {\mathcal R}_1^{(0)}\,,\\
e^{- \Psi_\bot} ({\mathcal F}_1+{\mathcal Q}_1) e^{\Psi_\bot} & = {\mathcal F}_1+{\mathcal Q}_1 + [{\mathcal F}_1, \Psi_\bot] + [{\mathcal Q}_1, \Psi_\bot] +  {\mathcal R}_2^{(1)}\,, \\
{\mathcal R}_1^{(0)} &:= \frac12 \int_0^1 (1 - \tau)^2 e^{- \tau \Psi_\bot} [[[{\mathcal D}, \Psi_\bot], \Psi_\bot]\,,\, \Psi_\bot] e^{\tau \Psi_\bot}\, d\tau \,,\\
{\mathcal R}_2^{(0)} & := \frac12 \int_0^1 (1 - \tau)^2 e^{- \tau \Psi_\bot}[ [\mathcal{F}_1+{\mathcal Q}_1, \Psi_\bot]\,, \Psi_\bot] e^{\tau \Psi_\bot}\, d \tau \,.
\end{aligned}
\end{equation}
Thus we have that
\begin{equation}\label{prima espansione cal L (1)}
\begin{aligned}
{\mathcal L}^{(0)}& ={\mathcal D} +{\mathcal F}_1+ [{\mathcal D}, \Psi_\bot] +{\mathcal Q}_1  + \frac12 [[{\mathcal D}, \Psi_\bot]\,,\, \Psi_\bot ] + [{\mathcal F}_1, \Psi_\bot] + [{\mathcal Q}_1, \Psi_\bot]  \\
& \quad + {\mathcal R}_1^{(0)}+  {\mathcal R}_2^{(0)}+e^{- \Psi_\bot}{\mathcal R}_0e^{\Psi_\bot}.
\end{aligned}
\end{equation}
We shall construct $\Psi$ in such a way that 
$$
[\mathcal{D},\Psi_\bot]+{\mathcal Q}_1 = \text{an operator of order}\quad \frac12.
$$
Moreover, by exploiting the explicit formulas for $\psi_1$ and $\psi_2$, we will show that there is a regularization effect for which $ \frac12  [[{\mathcal D}, \Psi_\perp]\,,\, \Psi_\perp ] \,,\, [{\mathcal Q}_1, \Psi_\perp]$ are operators of order zero.

\subsubsection{Construction of $\Psi$} 
We define an even cut-off function $\chi \in C^\infty(\R^2)$ such that 
$$
\begin{aligned}
& 0 \leq \chi \leq 1, \quad \chi(\xi) = 1, \quad \forall |\xi| \geq 1, \\
&  \chi(\xi) = 0, \quad \forall |\xi| \leq \frac12\, ,
\end{aligned}
$$
and define $\chi_\lambda$ as 
\begin{equation}\label{cut off per eq omologica}
\chi_\lambda(\xi) := \chi\Big( \frac{\xi}{\lambda^{6\delta}} \Big)\,. 
\end{equation}


We need the following lemmas.

\begin{lem}\label{lemma smoothing merda semi norme}
Let $m \in \R$, $f \in {\mathcal S}^m$ and let us consider the symbol $f_\lambda(x, \xi) := \big( 1 - \chi_\lambda(\xi) \big) f(x, \xi)$. Then one has that for any $N \in \N$, $f_\lambda \in {\mathcal S}^{- N}$ and for any $s \geq 0$, $\alpha \in \N_{0}$, 
$$
|{\rm Op}( f_\lambda) |_{- N, s, \alpha}^{\Lip(\gamma)} \lesssim_{s, N, \alpha} \lambda^{6 \delta(m + N)} |{\rm Op}( f) |_{m, s, \alpha}^{\Lip(\gamma)}\,. 
$$
\end{lem}
\begin{proof}
By the properties of the cut-off function $\chi_\lambda$ one has that for any $ K \in \N$, $\beta \in \N^2$
\begin{equation}\label{prop elementare 1 - chi lambda}
|\langle \xi \rangle^K \partial_\xi^\beta (1-\chi_\lambda(\xi))| \lesssim_{K, \beta} \lambda^{6 \delta(K - |\beta|)}\,.
\end{equation}
Hence for any $\beta \in \N^2$, $s \geq s_0$, for any $\xi \in \R^2$, one has that 
$$
\begin{aligned}
\langle \xi \rangle^{N+|\beta|} \| \partial_\xi^\beta f_\lambda(\cdot, \xi) \|_s & \lesssim_{s, \beta}  \sum_{\beta_1 + \beta_2 = \beta}  \langle \xi \rangle^{N+|\beta|} |\partial_\xi^{\beta_1} ( 1- \chi_\lambda)(\xi)| \| \partial_\xi^{\beta_2} f(\cdot, \xi) \|_s \\
& \lesssim_{s, \beta} \sum_{\beta_1 + \beta_2 = \beta} \langle \xi \rangle^{N+|\beta|}|\partial_\xi^{\beta_1} ( 1 - \chi_\lambda)(\xi)| \langle \xi \rangle^{m - |\beta_2|} |{\rm Op}( f) |_{m, s, |\beta_2|} \\
& \lesssim_{s, \beta}  |{\rm Op}( f) |_{m, s, |\beta|}\sum_{\beta_1 + \beta_2 = \beta} \langle \xi \rangle^{N +|\beta_1| + |\beta_2|+ m-|\beta_2|}|\partial_\xi^{\beta_1}(1- \chi_\lambda(\xi))|  \\
& \lesssim_{s, \beta} \lambda^{6 \delta(m + N)} |{\rm Op}( f) |_{m, s, |\beta|},
\end{aligned}
$$
where in the last line we used the property \eqref{prop elementare 1 - chi lambda} and that for $\beta_1\neq 0$ and the function $\partial_\xi^{\beta_1}(1- \chi_\lambda(\xi))\neq 0$ as long as $\frac{\lambda^{6\delta}}{2}\leq |\xi|\leq \lambda^{6\delta}$. The Lipschitz estimate can be proved similarly. 
\end{proof}
In the next lemma we establish some properties of the solutions of the homological equations that one has to solve in the reduction procedure of this section and the one of section \ref{subsub:decit}. 
\begin{lem}\label{lem:equazione-psi}
$(i)$. Let $\omega \in {\rm DC}(\gamma, \tau)$ and $a_1, a_2 \in {\mathcal S}^m$. Then there exist $\psi_1, \psi_2 \in {\mathcal S}^{m - \frac32}$ such that 
\begin{equation}
\begin{cases}
& \big(\lambda \omega \cdot \nabla + |\xi|^2 \big) \psi_1(x, \xi) +  \chi_\lambda(\xi) a_1(x, \xi)   = 0\,,\\
& \big(\lambda \omega \cdot \nabla - |\xi|^2 \big) \psi_2(x, \xi) + \chi_\lambda(\xi) a_2(x, \xi) = 0\,.
\end{cases}\label{eq:psi12}
\end{equation}
Moreover, for any $s \geq s_0$, $\alpha \in \N_{0}$, $i = 1, 2$, one gets that $| {\rm Op}(\psi_i) |_{m - \frac32, s, \alpha}^{\Lip(\gamma)} \lesssim_{s, \alpha} \lambda^{- 3\delta} |{\rm Op}(a_i) |_{m, s + \tau + 1, \alpha}^{\Lip(\gamma)}$. 
Finally, if ${\rm Op}(a_1(x, \xi)), {\rm Op}(a_2(x, \xi))$ are real operators then ${\rm Op}(\psi_1(x, \xi)), {\rm Op}(\psi_2(x, \xi))$ are real operators.

\noindent
$(ii)$ Let $\omega \in {\rm DC}(\gamma, \tau)$ and let us assume that $a_1 (x, \xi) = a_2(x, \xi) = i  \, d(x) \cdot \xi$ where the function $d$ is defined in \eqref{definizione m}. Then, the two corresponding solutions $\psi_1,  \psi_2$ of \eqref{eq:psi12} satisfy the following property: $\psi_1 - \psi_2 \in {\mathcal S}^{- 1}$ and 
\begin{equation}\label{bound psi 1 - psi 2}
| {\rm Op}(\psi_1 - \psi_2) |_{- 1, s, \alpha}^{\Lip(\gamma)} \lesssim_{s, \alpha} \lambda^\delta \gamma^{-1} \big( 1 +  \| {\mathcal I} \|_{s + 2 \tau + 3}^{\Lip(\gamma)} \big), \quad \forall s \geq s_0, \quad \alpha \in \N_{0}\,. 
\end{equation}
\end{lem}
\begin{proof}
{\sc Proof of $(i)$.} We prove the lemma for $\psi_1$. The corresponding properties of $\psi_2$ can be proved in a similar way. To shorten notations we write $\psi \equiv \psi_1$ and $a \equiv a_1$. We expand $a(x, \xi)$ and $\psi(x, \xi)$ in Fourier series 
$$
a(x, \xi) = \sum_{k \in \Z^2} \widehat a(k, \xi ) e^{i x \cdot k}, \quad \psi(x, \xi) = \sum_{k \in \Z^2} \widehat \psi (k, \xi) e^{i x \cdot k}\,. 
$$
Then we need to determine $\widehat \psi (k, \xi)$, $k \in \Z^2$, $\xi \in \R^2$ in such a way that 
$$
\Big( i \lambda \omega \cdot k + |\xi|^2 \Big) \widehat \psi (k, \xi) + \chi_\lambda(\xi) \widehat a(k, \xi)  = 0,
$$
and therefore we define 
\begin{equation}\label{def psi 1 widehat}
\widehat \psi(k, \xi) = - \dfrac{\chi_\lambda(\xi)\,\widehat a(k, \xi)}{i \lambda \omega \cdot k + |\xi|^2}, \quad (k, \xi) \in \Z^2 \times \R^2\,. 
\end{equation}
Note that by the definition of the cut-off function $\chi_\lambda$, one has that
\begin{equation}\label{prop supporto chi lambda}
\begin{aligned}
& \chi_\lambda(\xi) \neq 0 \Longrightarrow |\xi| \geq \frac12 \lambda^{6\delta}\,,  \\
& \partial_\xi^\beta \chi_\lambda(\xi) \neq 0 \Longrightarrow |\xi| \sim \lambda^{6\delta}, \quad \beta \in \N^2 ,
\end{aligned}
\end{equation}
 and 
\begin{equation}\label{stima chi lambda 1}
|\partial_\xi^\beta \chi_\lambda(\xi)| \lesssim \lambda^{- 6\delta |\beta|} \lesssim_\beta
 \langle \xi \rangle^{-  |\beta|}, \quad \forall \beta \in \N^2, \quad \forall \xi \in \R^2\,. 
\end{equation}
 A direct computation shows that (use that $\chi$ and all its derivatives vanish near the origin) for any $k \in \Z^2, \lambda^{6\delta} \lesssim |\xi|$, one has that 
\begin{equation}\label{stima chi lambda 2}
\Big| \partial_\xi^\beta \Big(\dfrac{1}{i \lambda \omega \cdot k + |\xi|^2} \Big) \Big| \lesssim_\beta \langle \xi \rangle^{- 2 - |\beta|} \lesssim_\beta \lambda^{- 3 \delta} \langle \xi \rangle^{- \frac32 - |\beta|} , \quad \forall \beta \in \N^2\,. 
\end{equation}
Hence, the estimates \eqref{stima chi lambda 1}, \eqref{stima chi lambda 2} imply that  
$$
\Big| \partial_\xi^\beta \Big( \dfrac{\chi_\lambda(\xi)\,\widehat a(k, \xi)}{i \lambda \omega \cdot k + |\xi|^2} \Big) \Big| \lesssim_\beta \lambda^{- 3\delta} \sum_{\beta_1 + \beta_2 = \beta} \langle \xi \rangle^{-\frac32 - |\beta_1|} |\partial_\xi^{\beta_2} \widehat a(k, \xi)|,
$$
which implies that 
\begin{equation}\label{stima psi}
\begin{aligned}
\| \partial_\xi^\beta \psi(\cdot, \xi) \|_s & \lesssim_{\beta} \lambda^{- 3\delta} \sum_{\beta_1 + \beta_2 = \beta}  \langle \xi \rangle^{-\frac32 - |\beta_1|} \| \partial_\xi^{\beta_2} a(\cdot, \xi) \|_s  \nonumber\\
& \lesssim_\beta \lambda^{- 3\delta} | {\rm Op}(a) |_{m, s, |\beta|} \sum_{\beta_1 + \beta_2 = \beta} \langle \xi \rangle^{- \frac32 - |\beta_1|} \langle \xi \rangle^{m - |\beta_2|} \nonumber\\
& \lesssim_\beta  \lambda^{- 3\delta} | {\rm Op}(a) |_{m, s, |\beta|} \langle \xi \rangle^{m - \frac32 - |\beta|}\,. 
\end{aligned}
\end{equation}

The latter chain of inequalities then implies the bound
\begin{equation}\label{bound psi a sup}
|{\rm Op}(\psi)|_{m - \frac32, s, \alpha} \lesssim_\alpha \lambda^{- 3 \delta} | {\rm Op}(a) |_{m, s, \alpha}\,, \quad \forall s \geq s_0, \quad \alpha \in \N_{0}\,. 
\end{equation}
Now we compute the Lipschitz norm w.r.t. the variable $\omega$. Let $\omega_1, \omega_2 \in {\rm DC}(\gamma, \tau)$, we have that
\begin{equation}\label{formula psi omega 1 psi omega 2}
\begin{aligned}
\psi(x,\xi;\omega_1)-\psi(x,\xi;\omega_2)&=\sum_{k\in\Z^2}\chi_\lambda(\xi)\,\left[\dfrac{\widehat a(k, \xi;\omega_2)}{i \lambda \omega_2 \cdot k + |\xi|^2}-\dfrac{\widehat a(k, \xi;\omega_1)}{i \lambda \omega_1 \cdot k + |\xi|^2}\right]e^{ix\cdot k}\\
&= \Phi_1(x, \xi) + \Phi_2(x, \xi)\,, \\
\Phi_1(x, \xi) & := \sum_{k\in\Z^2}\dfrac{\chi_\lambda(\xi)}{i \lambda \omega_2 \cdot k + |\xi|^2}\, \Big( \widehat a(k, \xi;\omega_2)-\widehat a(k, \xi;\omega_1) \Big)e^{ix\cdot k},\\
\Phi_2(x, \xi)& := \sum_{k\in\Z^2} \Gamma(k, \xi) \,\widehat a(k, \xi;\omega_1) e^{ix\cdot k}\,, \\
\Gamma(k, \xi) & := \chi_\lambda(\xi) \left[\dfrac{i\lambda(\omega_1-\omega_2)\cdot k}{(i \lambda \omega_2 \cdot k + |\xi|^2)(i \lambda \omega_1 \cdot k + |\xi|^2)}\right]\,. 
\end{aligned}
\end{equation}
Arguing as in \eqref{stima psi}, one gets that for any $s \geq s_0$, $\alpha \in \N_{0}$, 
\begin{equation}\label{stima Phi 1}
\gamma |{\rm Op}(\Phi_1)|_{m - \frac32, s, \alpha} \lesssim_{s, \alpha} \lambda^{- 3 \delta} \gamma |{\rm Op}(a)|^{\rm Lip}_{m - \frac32, s, \alpha} |\omega_1 - \omega_2| \lesssim_{s, \beta} \lambda^{- 3 \delta}  |{\rm Op}(a)|^{\Lip(\gamma)}_{m - \frac32, s, \alpha} |\omega_1 - \omega_2| \,.
\end{equation}
We then estimate $\Phi_2$. Arguing by induction, one can show that since $\omega_1 \in {\rm DC}(\gamma, \tau)$, one gets that for $\lambda^{6 \delta} \lesssim |\xi|$
\begin{equation}\label{stima chi lambda 2 b}
\Big| \partial_\xi^\beta \Big(\dfrac{1}{i \lambda \omega_1 \cdot k + |\xi|^2} \Big) \Big| \lesssim_\beta \frac{|k|^\tau }{\lambda \gamma |\xi|^{|\beta|}}  \quad \forall \beta \in \N^2\,. 
\end{equation}
Therefore, the estimates \eqref{stima chi lambda 1}, \eqref{stima chi lambda 2}, \eqref{stima chi lambda 2 b} imply that for any $\xi \in \R^2$, $\beta \in \N^2$, $\omega_1, \omega_2 \in {\rm DC}(\gamma, \tau)$, 
\begin{equation}\label{stima Gamma k xi}
\begin{aligned}
|\partial_\xi^\beta \Gamma(k, \xi)| & \lesssim_\beta \lambda |k| |\omega_1 - \omega_2| \sum_{\beta_1 +\beta_2 +\beta_3 = \beta} |\partial_\xi^{\beta_1} \chi_\lambda(\xi)| \Big| \partial_\xi^{\beta_2} \Big(\dfrac{1}{i \lambda \omega_1 \cdot k + |\xi|^2} \Big) \Big| \Big| \partial_\xi^{\beta_3} \Big(\dfrac{1}{i \lambda \omega_2 \cdot k + |\xi|^2} \Big)\Big| \\
& \lesssim_\beta \lambda |k| |\omega_1 - \omega_2| \sum_{\beta_1 + \beta_2 + \beta_3 = \beta} \langle \xi \rangle^{- |\beta_1|} \lambda^{- 3 \delta} \langle \xi \rangle^{- \frac32 - |\beta_2|}  |k|^\tau \lambda^{- 1} \gamma^{- 1} \langle \xi \rangle^{- |\beta_3|}  \\
& \lesssim_\alpha |k|^{\tau + 1} \gamma^{- 1} \lambda^{- 3 \delta}  \langle \xi \rangle^{- \frac32 - |\beta|} |\omega_1 - \omega_2|\,. 
\end{aligned}
\end{equation}
Hence, for any $s \geq s_0, \beta \in \N^2$, one obtains that 
\begin{equation}\label{stima Phi 2 alpha s}
\begin{aligned}
\| \partial_\xi^\beta \Phi_2(\cdot, \xi)\|_s^2 & \lesssim_{\beta} \sum_{\beta_1 + \beta_2 = \beta}   \sum_{k \in \Z^2} \langle k \rangle^{2 s} |\partial_\xi^{\beta_1} \Gamma(k, \xi)|^2 |\partial_\xi^{\beta_2} \widehat a(k, \xi; \omega_1)|^2 \\
& \stackrel{\eqref{stima Gamma k xi}}{\lesssim_\beta} \Big(  \gamma^{- 1} \lambda^{- 3 \delta}  |\omega_1 - \omega_2| \Big)^2 \sum_{\beta_1 +\beta_2 = \beta}  \big( \langle \xi \rangle^{- \frac32 - |\beta_1|} \big)^2  \sum_{k \in \Z^2} \langle k \rangle^{2 (s + \tau + 1)}  |\partial_\xi^{\beta_2} \widehat a(k, \xi; \omega_1)|^2 \\
& \lesssim_\beta \Big(  \gamma^{- 1} \lambda^{- 3 \delta}  |\omega_1 - \omega_2| \Big)^2 \sum_{\beta_1 + \beta_2 = \beta}  \big( \langle \xi \rangle^{- \frac32 - |\beta_1|} \big)^2 \| \partial_\xi^{\beta_2} a(\cdot, \xi; \omega_1) \|_{s + \tau + 1}^2  \\
& \lesssim_\beta \Big(  \gamma^{- 1} \lambda^{- 3 \delta}  |\omega_1 - \omega_2| \Big)^2 \sum_{\beta_1 + \beta_2 = \beta}  \big( \langle \xi \rangle^{- \frac32 - |\beta_1|} \langle \xi \rangle^{m - |\beta_2|} \big)^2  |{\rm Op}(a)|_{m, s + \tau + 1, |\beta_{2}|}^2 \\
& \lesssim_\beta \Big(  \gamma^{- 1} \lambda^{- 3 \delta} \langle \xi \rangle^{m - \frac32 - |\beta|}  |{\rm Op}(a)|_{m, s + \tau + 1, |\beta|}  |\omega_1 - \omega_2|  \Big)^2\,.
\end{aligned}
\end{equation}
Hence \eqref{stima psi}, \eqref{formula psi omega 1 psi omega 2}, \eqref{stima Phi 1}, \eqref{stima Phi 2 alpha s} imply the claimed bound
$$
|{\rm Op}(\psi)|^{\Lip(\gamma)}_{m-3/2,s,\alpha}\lesssim_{s,\alpha}\lambda^{-3\delta} |{\rm Op}(a) |_{m,s+\tau + 1,\alpha}^{\Lip(\gamma)}, \quad \forall s \geq s_0, \quad \alpha \in \N_{0}\,. 
$$
By \eqref{def psi 1 widehat} and by the fact that $\chi_{\lambda}(\xi)$ is even, one can see that ${\rm Op}(\psi_{1})$ is a real operator. For $\psi_{2}$ one can reason as above.

\medskip

\noindent
{\sc Proof of $(ii)$.} By the proof of the item $(i)$, one gets that for any $k \in \Z^2, \xi \in \R^2$
$$
\begin{aligned}
& \psi_1(x, \xi) = \sum_{k \in \Z^2} \widehat\psi_1(k, \xi) e^{i x \cdot \xi}\,, \quad \psi_2(x, \xi) = \sum_{k \in \Z^2} \widehat\psi_2(k, \xi) e^{i x \cdot \xi}\,, \\
& \widehat \psi_1(k, \xi) = -  \dfrac{\chi_\lambda(\xi)\, i  \, \widehat d(k) \cdot \xi}{i \lambda \omega \cdot k + |\xi|^2} \,, \quad \widehat\psi_2(k, \xi) = -  \dfrac{\chi_\lambda(\xi)\, i  \, \widehat d(k) \cdot \xi}{i \lambda \omega \cdot k - |\xi|^2},
\end{aligned}
$$
and hence 
$$
\begin{aligned}
p(x, \xi) & := \psi_1(x, \xi) - \psi_2(x, \xi) = \sum_{k \in \Z^2} \widehat p(k, \xi) e^{i k \cdot \xi}\,, \\
\widehat p(k, \xi) & := \widehat\psi_1(k, \xi) - \widehat\psi_2(k, \xi) = - \dfrac{  2 i  \widehat d(k) \cdot \xi \chi_\lambda(\xi) |\xi|^2}{\lambda^2 (\omega \cdot k)^2 + |\xi|^4}\,.
\end{aligned}
$$
A direct calculation shows that (recall \eqref{stima chi lambda 1}) for any $\beta \in \N^2$, one gets 
$$
|\partial_\xi^\beta \widehat p(k, \xi)| \lesssim_\beta   \langle \xi \rangle^{- 1 - |\beta|} |\widehat d(k)|,
$$
implying that for any $s \geq s_0$, $\alpha \in \N_{0}$,
\begin{equation}\label{stima sup psi 1 - psi 2}
|{\rm Op}(p)|_{- 1, s, \alpha} \lesssim_{s, \alpha}  \| d \|_s \stackrel{\eqref{stima in a m cal R i}}{\lesssim_{s, \alpha}} \lambda^\delta(1 +  \| {\mathcal I} \|_{s + 1}) \, . 
\end{equation}
and for any given $\omega_1, \omega_2 \in {\rm DC}(\gamma, \tau)$, one has that 
\begin{align*}
\widehat p(k, \xi; \omega_1)- \widehat p(k, \xi; \omega_2)&=-2 \frac{i  \widehat d(k;\omega_1) \cdot \xi \chi_\lambda(\xi) |\xi|^2}{\lambda^2 (\omega_1 \cdot k)^2 + |\xi|^4}+2 \frac{i  \widehat d(k;\omega_2) \cdot \xi \chi_\lambda(\xi) |\xi|^2}{\lambda^2 (\omega_2 \cdot k)^2 + |\xi|^4}\\
&= \Gamma_1(k, \xi) + \Gamma_2(k, \xi)\,, \\
\Gamma_1(k, \xi) & := -2i  \frac{[\widehat d(k;\omega_1)-\widehat d(k;\omega_2)] \cdot \xi \chi_\lambda(\xi) |\xi|^2}{\lambda^2 (\omega_1 \cdot k)^2 + |\xi|^4}\\
\Gamma_2(k, \xi)& := 2i \widehat d(k;\omega_2) \cdot \xi \chi_\lambda(\xi) |\xi|^2\frac{\lambda^2\big[(\omega_1-\omega_2)\cdot k\big]\big[ (\omega_1+\omega_2)\cdot k\big]}{(\lambda^2 (\omega_1 \cdot k)^2 + |\xi|^4)(\lambda^2 (\omega_2 \cdot k)^2 + |\xi|^4)}.
\end{align*}
Let $\omega_1, \omega_2 \in {\rm DC}(\gamma, \tau)$, $\lambda^{6 \delta} \lesssim |\xi|$. Then, one can easily show by induction that 
\begin{equation}\label{stima denominatori quadri omega 1 omega 2}
\begin{aligned}
& \Big| \partial_\xi^\beta \Big( \frac{1}{\lambda^2 (\omega_1 \cdot k)^2 + |\xi|^4} \Big) \Big| \lesssim_\beta \langle \xi \rangle^{- 4 - |\beta|}, \quad \forall \beta \in \N^2\,,  \\
& \Big| \partial_\xi^\beta \Big( \frac{1}{\lambda^2 (\omega_2 \cdot k)^2 + |\xi|^4} \Big) \Big| \lesssim_\beta \lambda^{- 2} \gamma^{- 2} |k|^{2 \tau} \langle \xi \rangle^{- |\beta|}, \quad \forall \beta \in \N^2\,. 
\end{aligned}
\end{equation}
In the first estimate above, one always estimate the denominators as $\lambda^2 (\omega_1 \cdot k)^2 + |\xi|^4 \geq |\xi|^4$ whereas in the second one, $\lambda^2 (\omega_2 \cdot k)^2 + |\xi|^4 \geq \lambda^2 (\omega_2 \cdot k)^2 \geq \lambda^2 \gamma^2 |k|^{- 2 \tau}$ for $k \in \Z^2 \setminus \{ 0 \}$, by using that $\omega_2$ is diophantine.

By using the estimate in \eqref{stima denominatori quadri omega 1 omega 2} and the estimate \eqref{stima chi lambda 1} on $\chi_\lambda$, one proves that 
\begin{equation}\label{stima Gamma 1}
|\partial_\xi^\beta \Gamma_1(k, \xi)| \lesssim_\beta   \langle \xi \rangle^{- 1 - |\beta|} |\widehat d(k;\omega_1)-\widehat d(k;\omega_2)|, \quad \forall \beta \in \N^2,
\end{equation}
and 
\begin{equation}\label{stima Gamma 2}
|\partial_\xi^\beta \Gamma_2(k, \xi)| \lesssim_\beta |k|^{2 (\tau + 1)}  \gamma^{- 2} \langle \xi \rangle^{- 1 - |\beta|}  |\widehat d(k; \omega_2)| |\omega_1 - \omega_2|, \quad \forall \beta \in \N^2\,. 
\end{equation}
Hence, by the estimates \eqref{stima Gamma 1}, \eqref{stima Gamma 2} one deduces that for any $s \geq s_0$, $\alpha \in \N_{0}$
\begin{equation}\label{stima Lip differenza psi 1 psi 2}
\begin{aligned}
\gamma \Big|{\rm Op}\Big(p(\cdot; \omega_1) - p(\cdot; \omega_2) \Big) \Big|_{- 1, s, \alpha} & \lesssim_\alpha \gamma \| d(\cdot; \omega_1) - d(\cdot; \omega_2) \|_s + \gamma^{- 1}\| d \|_{s + 2 \tau + 2}^{\rm sup} |\omega_1 - \omega_2| \\
& {\lesssim_{s, \alpha}}  \gamma^{- 1} \| d \|_{s + 2 \tau + 2}^{\Lip(\gamma)} |\omega_1 - \omega_2| \\
& \stackrel{\eqref{stima in a m cal R i}}{\lesssim_{s, \alpha}} \lambda^\delta \gamma^{- 1} \big(1 + \| {\mathcal I} \|_{s + 2 \tau + 3}^{\Lip(\gamma)} \big) |\omega_1 - \omega_2|\,. 
\end{aligned}
\end{equation}
The claimed statement then follows by the bounds \eqref{stima sup psi 1 - psi 2}, \eqref{stima Lip differenza psi 1 psi 2}. 
\end{proof}

We are now in position to define the operator $\Psi$.
\begin{lem}\label{lem:defPsi}
Assume that \eqref{ansatz} holds for some $\sigma \gg 0$ large enough and  let $\omega \in {\rm DC}(\gamma, \tau)$. Then there exist two real operators $\Psi_1 = {\rm Op}(\psi_1(x,\xi)),\Psi_2 = {\rm Op}(\psi_2(x,\xi))$ where $\psi_1,\psi_2\in \mathcal S^{-1/2}$ such that the operator $\Psi_\bot$ defined as
$$
\Psi_\bot = \Pi_0^\bot \Psi \Pi_0^\bot, \quad \Psi=\begin{pmatrix}
0 & {\rm Op}(\psi_1(x,\xi))\\
{\rm Op}(\psi_2(x,\xi)) & 0
\end{pmatrix},
$$
satisfies the equation
\begin{equation}
[\mathcal{D},\Psi_\bot]+\mathcal{Q}_1= \Pi_0^\bot \Big( {\mathcal Q}_\Psi + {\mathcal R}_\Psi \Big) \Pi_0^\bot,
\end{equation}
with $\Psi\in \Op^{-1/2}$ such that 
\begin{equation}\label{est:Psi}
\begin{aligned}
& |\Psi|_{-1/2,s,\alpha}^{\Lip(\gamma)}\lesssim_{s}1 + \|\mathcal{I}\|_{s+\sigma}^{\Lip(\gamma)}, \quad \forall s \geq s_0, \quad \alpha \in \N_{0}\,, \\
& \Psi_\bot = \Psi + {\mathcal R}_\bot^\Psi\,, \\
& |{\mathcal R}_\bot^\Psi|_{- N, s, 0}^{\Lip(\gamma)} \lesssim_{s, N}1 +  \|\mathcal{I}\|_{s+\sigma}^{\Lip(\gamma)}, \quad \forall s \geq s_0, \quad N \in \N,
\end{aligned}
\end{equation}
and the matrix operators ${\mathcal Q}_\Psi, {\mathcal R}_\Psi$ are off-diagonal, real, and satisfy the estimates
\begin{equation}\label{est:Q_Psi}
\begin{aligned}
|\mathcal{Q}_\Psi|_{1/2,s,\alpha}^{\Lip(\gamma)}&\lesssim_s1 + \|{\mathcal I}\|_{s+\sigma}^{\Lip(\gamma)}, \quad \forall s \geq s_0, \quad \alpha \in \N_{0}\,,\\
|\mathcal{R}_\Psi |_{-N,s,0}^{\Lip(\gamma)}&\lesssim_{N,s}\lambda^{6 \delta(N + 1)} ( 1 + \|{\mathcal I}\|_{s + \sigma}^{\Lip(\gamma)}), \quad \forall s \geq s_0\,. 
\end{aligned}
\end{equation}
Moreover let $s_1 \geq s_0$, $\alpha \in \N_{0}$  and let us assume that ${\mathcal I}_1, {\mathcal I}_2$ satisfy \eqref{ansatz} with $s_1 + \sigma$ instead of $s_0 + \sigma$. Then we also have that
\begin{equation}\label{Delta 12 Psi R Psi Q Psi}
\begin{aligned}
&  |\Delta_{12} \Psi|_{-1/2,s_1,\alpha}\,,\, |\Delta_{12} \mathcal{Q}_\Psi|_{1/2,s_1,\alpha}\lesssim \| {\mathcal I}_1 - {\mathcal I}_2 \|_{s_1 + \sigma}, \\
\end{aligned}
\end{equation}
\end{lem}
\begin{proof}
First of all, we define $\psi_1$ and $\psi_2$ to be, respectively, the solutions of the equation \eqref{eq:psi12} with $a_1$ and $a_2$ given by
\begin{equation}\label{equazione per psi}
a_1(x,\xi)=a_2(x,\xi):=i\, d(x)\cdot\xi.
\end{equation}
Thus, since $a_i\in {\mathcal S}^{1}$ and are their associated pseudo-differential operators are real, we can apply Lemma \ref{lem:equazione-psi} 
and Lemma \ref{stime linearized originario} obtaining the existence of $\psi_1,\psi_2\in \mathcal S^{-1/2}$ such that $\Psi_{1}, \Psi_{2}$ are real 
and satisfy for some $\sigma \gg 0$ large enough
\begin{equation}\label{stima-psi}
|\Psi_{i}|_{-1/2,s,\alpha}^{\Lip(\gamma)}\lesssim_{s, \alpha} \lambda^{- 3 \delta} \| d \|_{s + \tau + 1}^{\Lip(\gamma)} \lesssim_{s, \alpha} 1 + \| {\mathcal I} \|_{s + \sigma}^{\Lip(\gamma)}, \quad \forall s \geq s_0, \quad \alpha \in \N_{0}.
\end{equation}
In particular, it follows that
\begin{equation}\label{est:Psi0}
|\Psi|_{-1/2,s,\alpha}^{\Lip(\gamma)}\lesssim_{s, \alpha} 1 + \| {\mathcal I} \|_{s + \sigma}^{\Lip(\gamma)}, \quad \forall s \geq s_0, \quad \alpha \in \N_{0}\,. 
\end{equation}
The expansion of $\Psi_\bot$ and the estimate of ${\mathcal R}_\bot^\Psi$ then follows by applying Lemma \ref{pseudo media nulla}.

\noindent 
Now, we consider $[\mathcal{D},\Psi_\bot] + {\mathcal Q}_1$. Using that $[{\mathcal D}, \Pi_0^\bot] = 0$ and that ${\mathcal Q}_1 = \Pi_0^\bot {\mathcal Q}_1 \Pi_0^\bot$ (recall \eqref{invarianze medie nulle}), one has that 
\begin{equation}\label{prop eq omologica Pi 0 bot}
[\mathcal{D},\Psi_\bot] + {\mathcal Q}_1 = \Pi_0^\bot \Big( [{\mathcal D}, \Psi] + {\mathcal Q}_1 \Big) \Pi_0^\bot\,.
\end{equation}
and by splitting $i\,  d(x) \cdot \xi =  i  \chi_\lambda(\xi) d(x) \cdot \xi+ i  (1 - \chi_\lambda(\xi)) d(x) \cdot \xi$, $i = 1, 2$ (recall \eqref{cut off per eq omologica}), and by using that  by a direct calculation 
\begin{equation}\label{omologica psi  1 psi  2 laplace}
\begin{aligned}
& - \Delta \circ {\rm Op}\Big(\psi_{ 1}(x, \xi) \Big) = {\rm Op}\Big( \psi_{1}(x, \xi) |\xi|^2 -  2 i \xi \cdot \nabla_x \psi_{1}(x, \xi) - \Delta_x \psi_{1}(x, \xi) \Big)\,, \\
&  {\rm Op}(\psi_{ 2}) \circ ( - \Delta) = {\rm Op}\Big( \psi_{2}(x, \xi) |\xi|^2 \Big)\,,
\end{aligned}
\end{equation}
one obtains that 
\begin{equation}\label{def:Qpsi}
\begin{aligned}
& [{\mathcal D}, \Psi] + {\mathcal Q}_1 \\
&   = \begin{pmatrix}
0 & [\lambda \omega \cdot \nabla\,,\, {\rm Op}\big( \psi_{ 1} \big)] - \Delta \circ {\rm Op}\big( \psi_{ 1} \big)    \\
[\lambda \omega \cdot \nabla\,,\, {\rm Op}\big( \psi_{2} \big)] + {\rm Op}\big( \psi_{ 2} \big)  \circ \Delta   & 0 
\end{pmatrix}  \\
& \quad + \begin{pmatrix}
0 &  {\rm Op}\big(   i  \chi_\lambda(\xi) d(x) \cdot \xi \big) \\
 {\rm Op}\big(  i  \chi_\lambda(\xi) d(x) \cdot \xi  \big) & 0
\end{pmatrix} \\
& \quad + \begin{pmatrix}
0 &  {\rm Op}\big(   i  \big( 1 - \chi_\lambda(\xi) \big) d(x) \cdot \xi \big) \\
 {\rm Op}\big(  i  \big( 1 - \chi_\lambda(\xi) \big) d(x) \cdot \xi  \big) & 0
\end{pmatrix} \\ 
& \stackrel{\eqref{omologica psi  1 psi  2 laplace}}{=} \begin{pmatrix}
0 & {\rm Op}\big(\big( \lambda \omega \cdot \nabla + |\xi|^2 \big) \psi_{1}(x, \xi) + i  \chi_\lambda(\xi) d(x) \cdot \xi\big)\\
{\rm Op}\big(\big( \lambda \omega \cdot \nabla - |\xi|^2 \big) \psi_{ 2}(x, \xi) + i  \chi_\lambda(\xi) d(x) \cdot \xi\big) & 0
\end{pmatrix} \\
& \qquad + {\mathcal Q}_{\Psi} + {\mathcal R}_{\Psi}\,, \\
& {\mathcal Q}_{\Psi}  :=  \begin{pmatrix}
0 & {\rm Op}\Big(- 2 i \xi \cdot \nabla_x \psi_{ 1}(x, \xi) - \Delta_x \psi_{1}(x, \xi) \Big) \\
0 & 0
\end{pmatrix}\,, \\
& {\mathcal R}_{\Psi} := \begin{pmatrix}
0 & {\rm Op}\Big(  i  \big( 1 - \chi_\lambda(\xi) \big) d(x) \cdot \xi \Big) \\
 {\rm Op}\Big(  i  \big( 1 - \chi_\lambda(\xi) \big) d(x) \cdot \xi \Big) & 0
\end{pmatrix}\,. 
\end{aligned}
\end{equation}

By the explicit form in \eqref{def:Qpsi}, by the fact that $\chi_{\lambda}(\xi)$ is even and by the reality of the symbol $\psi_{1}$ and the function $d(x)$, we have that ${\mathcal Q}_{\Psi}$ and ${\mathcal R}_{\Psi}$ are real matrix operators. 

By applying the estimate \eqref{est:Psi0}, together with Lemma \ref{lemma smoothing merda semi norme} and the estimates \eqref{stima in a m cal R i} on $d(x)$, one obtains that 

\begin{align}\label{def:qpsi}
\begin{aligned}
& |{\mathcal Q}_\Psi|_{1/2,s,\alpha}\lesssim_{s, \alpha} 1+  \| {\mathcal I} \|_{s + \sigma}, \quad \forall s \geq s_0, \quad \alpha \in \N_{0}\,, \\
& |{\mathcal R}_\Psi|_{-N,s,0}\lesssim_{N, s}\lambda^{6 \delta(N + 1)}( 1 + \|{\mathcal I}\|_{s + \sigma}), \quad \forall s \geq s_0,
\end{aligned}
\end{align}
(for some constant $\sigma \gg 0$ large enough). Moreover, we apply Lemma \ref{lem:equazione-psi} to get that
\begin{align*}
&\big(\lambda \omega \cdot \nabla + |\xi|^2 \big) \psi_1(x, \xi) + i  \chi_\lambda(\xi) d(x) \cdot \xi = 0 ,\\
& \big(\lambda \omega \cdot \nabla - |\xi|^2 \big) \psi_2(x, \xi) + i  \chi_\lambda(\xi) d(x) \cdot \xi =  0,
\end{align*}
and hence we have obtained that
$$
[\mathcal{D},\Psi]+\mathcal{Q}_1=\mathcal{Q}_\Psi+{\mathcal R}_\Psi
$$
and this concludes the proof. The estimates \eqref{Delta 12 Psi R Psi Q Psi} can be proved by similar arguments, by using also the estimates \eqref{stime iniziali Delta 12 a m cal R i}. 
\end{proof}
We now analyze the conjugation of ${\mathcal L}$ by means of the map $\Phi := {\rm exp}(\Psi_\bot)$ (recall \eqref{def:Psi}). This is the content of the following lemma. 
\begin{lem}[Conjugation of $\Lin$]\label{coniugazione L}
Let $N \in \N$, $\gamma \in (0, 1)$, $\tau > 0$, $\lambda^{- \delta} \gamma^{- 1} \leq 1$, $\omega\in {\rm DC}(\gamma,\tau)$. Then there exists $\sigma \equiv \sigma_N \gg 0$ such that if \eqref{ansatz} is fulfilled the following holds. There exists an invertible real map $\Phi$ satisfying
\begin{equation}\label{stima Phi Phi inv}
\Phi^{\pm 1} : H^s_0 \to H^s_0 , \quad |\Phi^{\pm 1}|_{0, s, 0}^{\Lip(\gamma)} \lesssim_s 1 + \| {\mathcal I} \|_{s + \sigma}^{\Lip(\gamma)}, \quad \forall s \geq s_0,
\end{equation}
such that the operator $\Lin^{(0)}:=\Phi^{- 1}\Lin \Phi$ in \eqref{prima espansione cal L (1)} admits the expansion
\begin{equation}\label{espressione cal L0}
{\mathcal L}^{(0)}  = {\mathcal D} + {\mathcal F}_1  + \Pi_0^\bot \Big( \mathcal{Q}_{1/2}^{(0)}  +  {\mathcal R}_0^{(0)} +\mathcal{R}_{-N}^{(0)} \Big) \Pi_0^\bot ,
\end{equation}
where ${\mathcal F}_1 = \begin{pmatrix}
a \cdot \nabla & 0 \\
0 & a \cdot \nabla
\end{pmatrix}$ is defined in \eqref{mathcal L expansion 0}, $\mathcal{Q}_{1/2}^{(0)}\in\Op^{1/2}$ is a block off-diagonal matrix-operator, ${\mathcal R}_0^{(0)} \in \Op^0$ and $\mathcal{R}_{-N}^{(0)}$ is an operator of order $- N$ satisfying 
\begin{equation}\label{bounds Q 12 R0 E - N L 0}
\begin{aligned}
|\mathcal{Q}_{1/2}^{(0)}|_{1/2,s,\alpha}^{\Lip(\gamma)}\,,\,|{\mathcal R}_0^{(0)}|_{0, s, \alpha}^{\Lip(\gamma)}&\lesssim_{s,\alpha}  \lambda^{3\delta} (1 + \| {\mathcal I} \|_{s + \sigma}^{\Lip(\gamma)}), \quad \forall s \geq s_0, \quad \alpha \in \N_{0}\,,\\
|\mathcal{R}_{-N}^{(0)}|_{-N,s,0}^{\Lip(\gamma)}&\lesssim_{s, N} \lambda^{6 \delta(N + 1)} ( 1 + \| {\mathcal I} \|_{s + \sigma}^{\Lip(\gamma)}), \quad \forall s \geq s_0\,.
\end{aligned}
\end{equation}
Moreover let $s_1 \geq s_0$, $\alpha \in \N_{0}$ and let ${\mathcal I}_1, {\mathcal I}_2$ satisfy \eqref{ansatz} with $s_1 + \sigma$ instead of $s_0 + \sigma$. Then 
\begin{equation}\label{stime Delta 12 primo step decoupling}
\begin{aligned}
& |\Delta_{12} \Phi^{\pm 1}|_{0, s_1, 0} \lesssim \| {\mathcal I}_1 - {\mathcal I}_2 \|_{s_1 + \sigma}\,, \quad  \\
& |\Delta_{12} \mathcal{Q}_{1/2}^{(0)}|_{1/2,s_1,\alpha}\,,\,|\Delta_{12} {\mathcal R}_0^{(0)}|_{0, s_1, \alpha} \lesssim_{s_1,\alpha}  \lambda^{3\delta} \| {\mathcal I}_1 - {\mathcal I}_2 \|_{s_1 + \sigma}\,. 
\end{aligned}
\end{equation}
Finally the operators $\mathcal L^{(0)}, \mathcal F_{1}, \mathcal{Q}_{1/2}^{(0)}, {\mathcal R_{0}^{(0)}}$ and $\mathcal R_{-N}^{(0)} $ are real matrix operators.
\end{lem}
\begin{proof}
 To simplify the notations we write $\| \cdot \|_s$ instead of $\| \cdot \|_s^{\Lip(\gamma)}$ and $|\cdot|_{m, s, \alpha}$ instead of $|\cdot|_{m, s, \alpha}^{\Lip(\gamma)}$. First of all, by the estimates \eqref{est:Psi} and by Lemma \ref{lem:exponential}, one immediately gets the estimates \eqref{stima Phi Phi inv} on $\Phi := {\rm exp}(\Psi_\bot)$.
 Moreover, since $\Psi_{\perp}$ is real, then also $\Phi$ is so.
  Then, we analyze all the terms appearing in the expansions \eqref{def:R11}, \eqref{prima espansione cal L (1)} of ${\mathcal L}^{(0)} = \Phi^{- 1} {\mathcal L} \Phi$. 

\noindent
{\sc Expansion of the term $\frac12 [[{\mathcal D}, \Psi_\bot], \Psi_\bot] + [{\mathcal F}_1, \Psi_\bot] + [{\mathcal Q}_1, \Psi_\bot] $.} By applying Lemma \ref{lem:defPsi}, one obtains that 
$$
\begin{aligned}
& \frac12 [[{\mathcal D}, \Psi_\bot], \Psi_\bot] + [{\mathcal Q}_1, \Psi_\bot] + [{\mathcal F}_1, \Psi_\bot] =  [{\mathcal F}_1, \Psi_\bot] +  \frac12 [ \mathcal{Q}_1\,,\, \Psi_\bot ]  \\
& \quad + \frac12 [ \Pi_0^\bot \mathcal{Q}_\Psi \Pi_0^\bot , \Psi_\bot] + \frac12 [\Pi_0^\bot \mathcal{R}_\Psi \Pi_0^\bot \,,\, \Psi_\bot] \,.
\end{aligned}
$$
By the estimates \eqref{est:Psi}, \eqref{est:Q_Psi} and by applying the composition Lemmata \ref{lem:composition-pseudodiff} (to estimate $ [ \Pi_0^\bot \mathcal{R}_\Psi \Pi_0^\bot\,,\, \Psi_\bot]$), \ref{lemma composizione 2} (to expand $[ \Pi_0^\bot \mathcal{Q}_\Psi \Pi_0^\bot , \Psi_\bot]$) together with the ansatz \eqref{ansatz}, one easily obtains that 
\begin{equation}\label{iena 0}
\begin{aligned}
& \frac12 [ \Pi_0^\bot \mathcal{Q}_\Psi \Pi_0^\bot, \Psi_\bot]  = \Pi_0^\bot \big( {\mathcal F}_1^{(1)} + {\mathcal R}_{- N}^{(1)} \big) \Pi_0^\bot \quad \text{where} \quad {\mathcal F}_1^{(1)} \in \Op^0 \quad \text{has zero off-diagonal entries} \\
& |{\mathcal F}_1^{(1)} |_{0, s, \alpha} \lesssim_{s, \alpha}1 +  \| {\mathcal I} \|_{s + \sigma}, \quad \forall s \geq s_0, \quad \alpha \in \N_{0}\,, \\
& |{\mathcal R}_{- N}^{(1)}|_{- N, s, 0} \lesssim_{s, N} 1 + \| {\mathcal I} \|_{s + \sigma}, \quad \forall s \geq s_0,   \\
&   |[\Pi_0^\bot \mathcal{R}_\Psi \Pi_0^\bot \,,\, \Psi]|_{- N, s, 0} \lesssim_{s, N} \lambda^{6 \delta(N + 1)}( 1 +  \| {\mathcal I} \|_{s + \sigma}), \quad \forall s \geq s_0\,.
\end{aligned}
\end{equation}
We analyze in more details the terms $ [ \mathcal{Q}_1\,,\, \Psi_\bot ]\,,\,  [ \mathcal{F}_1\,,\, \Psi_\bot ]$. By the properties \eqref{invarianze medie nulle}, $[\Pi_0^\bot, {\mathcal Q}_1] = 0$, $[\Pi_0^\bot, {\mathcal F}_1] = 0$ and therefore
\begin{equation}\label{Q1 Psi bot medie fuori}
\begin{aligned} 
&[{\mathcal F}_1, \Psi_\bot] = \Pi_0^\bot [{\mathcal F}_1, \Psi] \Pi_0^\bot\,, \qquad  [{\mathcal Q}_1, \Psi_\bot] = \Pi_0^\bot [{\mathcal Q}_1, \Psi] \Pi_0^\bot\,.
\end{aligned}
\end{equation}
Moreover, by an explicit calculation, one has that (recalling that  $\Psi_1 = {\rm Op}(\psi_1(x,\xi)),\Psi_2 = {\rm Op}(\psi_2(x,\xi))$)
$$
\begin{aligned}
&  [ \mathcal{Q}_1\,,\, \Psi ] =   \begin{pmatrix}
0 &  d \cdot \nabla \\
d \cdot \nabla & 0
\end{pmatrix} \begin{pmatrix}
0 & \Psi_1 \\
\Psi_2 & 0
\end{pmatrix}  -  \begin{pmatrix}
0 & \Psi_1 \\
\Psi_2 & 0
\end{pmatrix} \begin{pmatrix}
0 &  d \cdot \nabla \\
d \cdot \nabla & 0
\end{pmatrix} \\
& =   \begin{pmatrix}
(d \cdot \nabla) \circ  \Psi_2 - \Psi_1 \circ  (d \cdot \nabla)     & 0 \\
0 & (d \cdot \nabla) \circ  \Psi_1  -  \Psi_2 \circ  (d \cdot \nabla)    
\end{pmatrix}\,.
\end{aligned}
$$
Then by applying Lemma \ref{lemma composizione 2} and by using the estimates on $m(x)$ and $\Psi_1$ of Lemmata \ref{stime linearized originario}, \ref{lem:defPsi}, one obtains the expansion
\begin{equation}\label{iena 1}
\begin{aligned}
&  [ \mathcal{Q}_1\,,\, \Psi ]  \\
& =   \begin{pmatrix}
 {\rm Op}\Big(  i d(x) \cdot \xi \big( \psi_2(x, \xi) - \psi_1(x, \xi) \big)\Big) & 0 \\
0 &  {\rm Op}\Big(  i d(x) \cdot \xi \big( \psi_1(x, \xi) - \psi_2(x, \xi) \big)\Big)
\end{pmatrix}    \\
& \quad + {\mathcal F}_1^{(2)} + {\mathcal R}_{- N}^{(2)}\,, \\
& {\mathcal F}_1^{(2)} \in \Op^{- \frac12}\,, \quad  |{\mathcal F}_1^{(2)}|_{- \frac12, s, \alpha} \lesssim_{s, \alpha} 1 +  \| {\mathcal I} \|_{s + \sigma}^{\Lip(\gamma)}, \quad \forall s \geq s_0, \quad \alpha \in \N_{0}\,, \\
& |{\mathcal R}_{- N}^{(2)}|_{- N, s, 0}  \lesssim_{s, N} 1 +  \| {\mathcal I} \|_{s + \sigma}, \quad \forall s \geq s_0\,. 
\end{aligned}
\end{equation}
Furthermore, by using Lemmata \ref{lem:composition-pseudodiff}-$(i)$, \ref{stime linearized originario}, \ref{lem:equazione-psi}-$(ii)$ one obtains that 
\begin{equation}\label{cancellazione psi 1 - psi 2}
\begin{aligned}
&   i \,d \cdot \xi \big( \psi_1 - \psi_2 \big) \in {\mathcal S}^{0}, \\
&  \Big|  {\rm Op} \Big( i\,d \cdot \xi \big( \psi_1 - \psi_2 \big) \Big)  \Big|_{0, s, \alpha} \lesssim_{s, \alpha} \lambda^{2 \delta} \gamma^{- 1}(1 +  \| {\mathcal I} \|_{s + \sigma}) \\
& \qquad \stackrel{\lambda^{- \delta} \gamma^{- 1} \leq 1}{\lesssim_{s, \alpha}}  \lambda^{3 \delta} (1 +  \| {\mathcal I} \|_{s + \sigma}) \quad \forall s \geq s_0, \quad \alpha \in \N_{0}\,.
\end{aligned}
\end{equation}
Moreover, by applying Lemma \ref{lemma composizione 2} and by using the estimates on $a(x)$ and $\Psi_1, \Psi_2$ of Lemmata \ref{stime linearized originario}, \ref{lem:defPsi}, one obtains the expansion
\begin{equation}\label{iena 2}
\begin{aligned}
& [{\mathcal F}_1, \Psi]  = \begin{pmatrix}
0 & [ a \cdot \nabla , \Psi_1] \\
[ a \cdot \nabla, \Psi_2] & 0
\end{pmatrix} = {\mathcal Q}_1^{(3)} + {\mathcal R}_{- N}^{(3)}\,, \\
& {\mathcal Q}_1^{(3)} \in \Op^{- \frac12}, \quad |{\mathcal Q}_1^{(3)}|_{- \frac12, s, \alpha} \lesssim_{s, \alpha} 1+ \| {\mathcal I} \|_{s + \sigma}, \quad \forall s \geq s_0, \quad \alpha \in \N_{0}\,, \\
& |{\mathcal R}_{- N}^{(3)}|_{- N, s, 0} \lesssim_{s, N}1 +  \| {\mathcal I} \|_{s + \sigma}, \quad \forall s \geq s_0\,. 
\end{aligned}
\end{equation}

Hence, by summarizing \eqref{iena 0}, \eqref{Q1 Psi bot medie fuori}, \eqref{iena 1}, \eqref{iena 2}, \eqref{cancellazione psi 1 - psi 2}, one gets the final expansion 
\begin{equation}\label{espansione finale cal D Psi Psi}
\begin{aligned}
& \frac12 [[{\mathcal D}, \Psi_\bot], \Psi_\bot] + [{\mathcal F}_1, \Psi_\bot] + [{\mathcal Q}_1, \Psi_\bot] = \Pi_0^\bot \big( {\mathcal S}_1^{(4)} + {\mathcal R}_{- N}^{(4)} \big) \Pi_0^\bot \,,  \\
& {\mathcal S}_1^{(4)} \in \Op^0\,, \quad |{\mathcal S}_1^{(4)}|_{0, s, \alpha} \lesssim_{s, \alpha} \lambda^{{3}\delta}(1 +  \| {\mathcal I} \|_{s + \sigma}), \quad \forall s \geq s_0, \quad \alpha \in \N_{0}\,, \\
& |{\mathcal R}_{- N}^{(4)}|_{- N, s, 0} \lesssim_{s, N} \lambda^{6 \delta(N + 1)}(1 +  \| {\mathcal I} \|_{s + \sigma}), \quad \forall s \geq s_0\,. 
\end{aligned}
\end{equation}

\medskip

\noindent
{\sc Analysis of the term ${\mathcal R}_1^{(0)}+  {\mathcal R}_2^{(0)}+e^{- \Psi_\bot }{\mathcal R}_0e^{\Psi_\bot}$ in \eqref{def:R11}, \eqref{prima espansione cal L (1)}.} By using Lemmata \ref{lemma composizione 2}, \ref{coniugio senza alpha exponential map}, \ref{stime linearized originario}, \ref{lem:defPsi}, the expansions \eqref{iena 1}, \eqref{iena 2}, \eqref{espansione finale cal D Psi Psi}, one obtains that 
\begin{equation}\label{espansione finale cal D Psi Psi 2}
\begin{aligned}
& {\mathcal R}_1^{(1)}+  {\mathcal R}_2^{(1)}+e^{- \Psi_\bot}{\mathcal R}_0e^{\Psi_\bot} = \Pi_0^\bot \big( {\mathcal S}_1^{(5)} + {\mathcal R}_{- N}^{(5)} \big) \Pi_0^\bot\,,  \\
&{{\mathcal S}_1^{(5)}}\in \Op^0\,, \quad |{\mathcal S}_1^{(5)}|_{0, s, \alpha} \lesssim_{s, \alpha} \lambda^{3\delta}  ( 1 +  \| {\mathcal I} \|_{s + \sigma}), \quad \forall s \geq s_0, \quad \alpha \in \N_{0}\,, \\
& |{\mathcal R}_N^{(5)}|_{- N, s, 0} \lesssim_{s, N} \lambda^{6 \delta(N + 1)} ( 1 + \| {\mathcal I} \|_{s + \sigma}), \quad \forall s \geq s_0\,. 
\end{aligned}
\end{equation}
The claimed expansion \eqref{espressione cal L0} and the claimed bounds \eqref{bounds Q 12 R0 E - N L 0} then follows by Lemma \ref{lem:defPsi}, the expansions  \eqref{espansione finale cal D Psi Psi}, \eqref{espansione finale cal D Psi Psi 2} and by defining 
$$
{\mathcal Q}_{\frac12}^{(0)} := {\mathcal Q}_\Psi\,, \quad {\mathcal R}_0^{(0)} := {\mathcal S}_1^{(4)} + {\mathcal S}_1^{(5)} \,, \quad {\mathcal R}_{- N}^{(0)} := {\mathcal R}_\Psi + {\mathcal R}_{- N}^{(4)} + {\mathcal R}_{- N}^{(5)}\,. 
$$
The bounds \eqref{stime Delta 12 primo step decoupling} follows by similar arguments, using also the estimates \eqref{stime iniziali Delta 12 a m cal R i}, \eqref{Delta 12 Psi R Psi Q Psi}. 
Finally we discuss the algebraic properties of the operators. $\mathcal L^{(0)}$ is real by composition, $ \mathcal F_{1}$ by definition, $\mathcal{Q}_{1/2}^{(0)}$ by Lemma \ref{lem:defPsi}.
By Remark \ref{composizione.simboli.reali} we have that ${\mathcal S}_1^{(4)}, {\mathcal S}_1^{(5)}$ are real and therefore also ${\mathcal R_{0}^{(0)}}$ and finally by difference $ {\mathcal R}_\Psi, {\mathcal R}_{- N}^{(4)}, {\mathcal R}_{- N}^{(5)}$ are real and hence $\mathcal R_{-N}^{(0)} $.
\end{proof}

\bigskip

\subsection{The iterative step}\label{subsub:decit}
We now describe the iterative argument that will eventually allow us to decouple the equations up to an arbitrary smoothing operator of order $-N$.

\begin{lem}\label{sub:decit}
Let $N \in \N$, $\gamma \in (0, 1)$, $\tau > 0$ and assume that $\lambda^{- \delta} \gamma^{- 1} \leq 1$. Then there exists $\sigma \equiv \sigma_N \gg 0$ large enough such that if \eqref{ansatz} holds, then the following statements holds for all $n \in \{ 0, \ldots, 2 N + 1\}$. There exists a linear operator 
\begin{equation}\label{formula cal L (n) decoupling}
\Lin^{(n)}=\mathcal{D}+\mathcal{F}_1 + \Pi_0^\bot \Big( {\mathcal F}_0^{(n)} +\mathcal{Q}_{ \frac{1 - n}{2} }^{(n)}+\mathcal{R}_{-N}^{(n)} \Big) \Pi_0^\bot,
\end{equation}
defined for all $\omega\in \mathrm{DC}(\gamma,\tau)$, where $\mathcal{F}_0^{(n)}\in\Op^0$ is a block diagonal matrix-operator, $\mathcal{Q}_{ \frac{1 - n}{2}}^{(n)}\in\Op^{ \frac{1 - n}{2}}$ is a block off-diagonal matrix-operator, and $\mathcal{R}_{-N}^{(n)}$ is an operator of order $- N$ satisfying the estimates 
\begin{equation}\label{stime decoupling cal L n}
\begin{aligned}
|\mathcal{F}_0^{(n)}|_{0,s,\alpha}^{\Lip(\gamma)}\,,\, |\mathcal{Q}_{\frac{1 - n}{2}}^{(n)}|_{\frac{1 - n}{2},s,\alpha}^{\Lip(\gamma)}&\lesssim_{s,\alpha} \lambda^{3 \delta } \big( 1 + \| {\mathcal I} \|_{s + \sigma}^{\Lip(\gamma)} \big), \quad \forall s \geq s_0, \quad \alpha \in \N_{0}\,,\\
|\mathcal{R}_{-N}^{(n)}|^{\Lip(\gamma)}_{-N,s,0}&\lesssim_{s, N} \lambda^{6 \delta(N + 1)} \| {\mathcal I} \|_{s + \sigma}^{\Lip(\gamma)}, \quad \forall s \geq s_0.
\end{aligned}
\end{equation}
For $n \in \{1, 2, \ldots, 2 N + 1 \}$ there exists a real, linear and invertible map $\Phi_{n - 1}$ satisfying 
\begin{equation}\label{stima decoupling iterative Phin}
\Phi_{n - 1}^{\pm 1} : H^s_0 \to H^s_0, \quad |\Phi_{n - 1}^{\pm 1}|_{0, s, 0}^{\Lip(\gamma)} \lesssim_s 1 + \| {\mathcal I} \|_{s + \sigma}^{\Lip(\gamma)}, \quad \forall s \geq s_0,
\end{equation} 
such that the operator $\Lin^{(n)}$ satisfies the conjugation
$$
\Lin^{(n)}=\Phi_{n - 1}^{- 1} \Lin^{(n-1)} \Phi_{n - 1}\,. 
$$
Let $s_1 \geq s_0$, $\alpha \in \N_{0}$ and let ${\mathcal I}_1, {\mathcal I}_2$ satisfy \eqref{ansatz} with $s_1 + \sigma$ instead of $s_0 + \sigma$. Then 
\begin{equation}\label{Delta 12 pezzi di cal L (0)}
\begin{aligned}
& |\Delta_{12} \Phi_{n - 1}^{\pm 1}|_{0, s_1, 0} \lesssim_{s_1} \| {\mathcal I}_1 - {\mathcal I}_2 \|_{s_1 + \sigma}\,, \\
& |\Delta_{12} \mathcal{F}_0^{(n)}|_{0,s_1,\alpha}\,,\,  |\Delta_{12}\mathcal{Q}_{\frac{1 - n}{2}}^{(n)}|_{\frac{1 - n}{2},s_1,\alpha} \lesssim_{s_1,\alpha} \lambda^{3 \delta }\| {\mathcal I}_1 - {\mathcal I}_2 \|_{s_1 + \sigma}\,.  \\
\end{aligned}
\end{equation}
Finally, the opeators $\Lin^{(n)}, {\mathcal F}_0^{(n)}, \mathcal{Q}_{ \frac{1 - n}{2} }^{(n)}$ and $\mathcal{R}_{-N}^{(n)} $ are real matrix operators.
\end{lem}
\begin{proof}
To simplify notations we write $\| \cdot \|_s$ for $\| \cdot \|_s^{\Lip(\gamma)}$ and $|\cdot|_{m, s, \alpha}$ for $|\cdot|_{m, s, \alpha}^{\Lip(\gamma)}$. 

\medskip

\noindent
PROOF OF THE STATEMENT FOR $n=0$. The claimed statement follows from Lemma \ref{coniugazione L}. \\
\\
PROOF OF THE INDUCTION STEP.
Let $n\geq 0$ and we assume that we have an operator of the form
\begin{equation}
{\mathcal L}^{(n)}  = {\mathcal D} +\mathcal{F}_1+ \Pi_0^\bot \Big( {\mathcal F}_0^{(n)} +  {\mathcal Q}_{\frac{1 - n}{2}}^{(n)}+{\mathcal R}_{-N}^{(n)} \Big) \Pi_0^\bot\,,
\end{equation}
where ${\mathcal F}_0^{(n)}, {\mathcal Q}_{\frac{1 - n}{2}}^{(n)}, {\mathcal R}_{-N}^{(n)}$ satisfy the properties \eqref{stime decoupling cal L n}. 
Our goal is to normalize the off-diagonal part 
\begin{equation}\label{bla bla cal Qn off}
{\mathcal Q}_{\frac{1 - n}{2}}^{(n)} = \begin{pmatrix}
0 & {\rm Op}\Big( q_{n, 1}(x, \xi) \Big) \\
{\rm Op}\Big(q_{n, 2}(x, \xi) \Big) & 0
\end{pmatrix} \in \Op^{\frac{1 - n}{2}}.
\end{equation}
We look for 
\begin{equation}\label{bla bla Psin off}
\Psi_{n, \bot} := \Pi_0^\bot \Psi_n \Pi_0^\bot\,, \qquad \Psi_n = \begin{pmatrix}
0 & {\rm Op}\Big( \psi_{n, 1}(x, \xi) \Big) \\
{\rm Op}\Big( \psi_{n, 2}(x, \xi) \Big) & 0
\end{pmatrix} \in \Op^{- \frac{n}{2} - 1},
\end{equation}
where the symbols $\psi_{n, 1}$ and $\psi_{n, 2}$ have to be determined and we consider $\Phi_n = {\rm exp}(\Psi_{n, \bot})$. 
We analyze the conjugation $ {\mathcal L}^{(n + 1)} := \Phi_n^{- 1} {\mathcal L}^{(n)} \Phi_n$. By Lie expansion one computes 
\begin{equation}\label{prima espansione cal L n Phi n}
\begin{aligned}
{\mathcal L}^{(n + 1)} & = {\mathcal D} + {\mathcal F}_1 + [{\mathcal D}\,,\, \Psi_{n, \bot}] + \Pi_0^\bot \Big(  {\mathcal Q}_{\frac{1 - n}{2}}^{(n)} + {\mathcal F}_0^{(n)} \Big) \Pi_0^\bot + {\mathcal S}^{(n)}_1 + {\mathcal S}^{(n)}_2 + {\mathcal S}_3^{(n)}\,, \\
{\mathcal S}_1^{(n)} & := \frac12 \int_0^1 (1 - \tau) e^{- \tau \Psi_{n, \bot}} [[{\mathcal D}, \Psi_{n, \bot}], \Psi_{n, \bot}] e^{\tau \Psi_{n, \bot}}\, d \tau  \\
{\mathcal S}_2^{(n)} & :=   \int_0^1 e^{- \tau \Psi_{n, \bot}} \Big[{\mathcal F}_1 +\Pi_0^\bot  {\mathcal F}_0^{(n)} \Pi_0^\bot  + \Pi_0^\bot {\mathcal Q}_{\frac{1 - n}{2}}^{(n)} \Pi_0^\bot \,, \Psi_{n, \bot} \Big] e^{\tau \Psi_{n, \bot}}\, d \tau  \\
{\mathcal S}_3^{(n)} & := \Phi_n^{- 1} \Pi_0^\bot {\mathcal R}_{- N}^{(n)} \Pi_0^\bot \Phi_n\,.  
\end{aligned}
\end{equation}
\noindent
{\sc Construction of $\Psi_n$ and expansion of $[{\mathcal D}, \Psi_{n, \bot}] +\Pi_0^\bot  {\mathcal Q}_{\frac{1 - n}{2}}^{(n)} \Pi_0^\bot$.} Since $[{\mathcal D}, \Pi_0^\bot ] = 0$, one has that 
$$
[{\mathcal D}, \Psi_{n, \bot}] +\Pi_0^\bot  {\mathcal Q}_{\frac{1 - n}{2}}^{(n)} \Pi_0^\bot = \Pi_0^\bot \Big( [{\mathcal D}, \Psi] + {\mathcal Q}_{\frac{1 - n}{2}}^{(n)} \Big) \Pi_0^\bot\,. 
$$By Lemma \ref{lem:equazione-psi} and using the induction estimates \eqref{stime decoupling cal L n} on ${\mathcal Q}_{\frac{1 - n}{2}}^{(n)}$, We choose $\psi_{n, i}$, $i = 1, 2$ so that 
\begin{equation}\label{solution omologica dec step n}
\begin{aligned}
& \big( \lambda \omega \cdot \nabla + |\xi|^2 \big) \psi_{n, 1}(x, \xi) + \chi_\lambda(\xi) q_{n, 1}(x, \xi) = 0\,, \\
& \big( \lambda \omega \cdot \nabla - |\xi|^2 \big) \psi_{n, 2}(x, \xi) + \chi_\lambda(\xi) q_{n, 2}(x, \xi) = 0\,, \\
\end{aligned}
\end{equation}
with $\psi_{n, 1}, \psi_{n, 2} \in S^{- \frac{n}{2} - 1}$ satisfy
\begin{equation}
 | {\rm Op}(\psi_{n, i}) |_{- \frac{n}{2} - 1, s, \alpha} \lesssim_\alpha \lambda^{- 3 \delta} | {\rm Op}(q_{n, i}) |_{\frac{1 - n}{2}, s + \tau + 1, \alpha} \lesssim_{s, \alpha}   1 + \| {\mathcal I} \|_{s + \sigma} , \quad \forall s \geq s_0, \quad \alpha \in \N_{0}\,. 
\end{equation}
Moreover the latter properties, together with Lemma \ref{pseudo media nulla} imply that 
\begin{equation}\label{Psi n bot Psi n}
\begin{aligned}
& \Psi_{n, \bot} = \Psi_n +  {\mathcal R}_\bot^{\Psi_n}\,, \\
& |{\mathcal R}_\bot^{\Psi_n}|_{- N, s, 0} \lesssim_{s, N} 1 + \| {\mathcal I} \|_{s + \sigma}, \quad \forall s \geq s_0\,. 
\end{aligned}
\end{equation}
Furthermore, by splitting $q_{n, i} = \chi_\lambda q_{n , i} + (1 - \chi_\lambda) q_{n, i}$, $i = 1, 2$ (recall \eqref{cut off per eq omologica}), and by using that  by a direct calculation 
\begin{equation}\label{omologica psi n 1 psi n 2}
\begin{aligned}
 - \Delta \circ {\rm Op}\Big(\psi_{n, 1}(x, \xi) \Big) &= {\rm Op}\Big( \psi_{n, 1}(x, \xi) |\xi|^2 - 2 i \xi \cdot \nabla_x \psi_{n, 1}(x, \xi) - \Delta_x \psi_{n, 1}(x, \xi) \Big)\,, \\
  {\rm Op}(\psi_{n , 2}) \circ ( - \Delta) &= {\rm Op}\Big( \psi_{n, 2}(x, \xi) |\xi|^2 \Big)\,,
\end{aligned}
\end{equation}
one obtains that 
\begin{equation}\label{espansione cal D Psin Q 12 - n}
\begin{aligned}
& [{\mathcal D}, \Psi_n] + {\mathcal Q}_{\frac{1 - n}{2}}^{(n)} \\
&   = \begin{pmatrix}
0 & [\lambda \omega \cdot \nabla\,,\, {\rm Op}\big( \psi_{n, 1} \big)] - \Delta \circ {\rm Op}\big( \psi_{n, 1} \big)  \\
[\lambda \omega \cdot \nabla\,,\, {\rm Op}\big( \psi_{n, 2} \big)] + {\rm Op}\big( \psi_{n, 2} \big)  \circ \Delta    & 0 
\end{pmatrix}  \\
& \quad + \begin{pmatrix}
0 &  {\rm Op}\big(  \chi_\lambda(\xi) q_{n, 1}(x, \xi) \big) \\
 {\rm Op}\big( \chi_\lambda(\xi)q_{n, 2}(x, \xi) \big) & 0
\end{pmatrix} \\
& \quad + \begin{pmatrix}
0 &  {\rm Op}\big(  \big( 1 - \chi_\lambda(\xi) \big) q_{n, 1}(x, \xi) \big) \\
 {\rm Op}\big( \big( 1 - \chi_\lambda(\xi) \big)q_{n, 2}(x, \xi) \big) & 0
\end{pmatrix} \\ 
& = \begin{pmatrix}
0 & \big( \lambda \omega \cdot \nabla + |\xi|^2 \big) \psi_{n, 1}(x, \xi) + \chi_\lambda(\xi) q_{n, 1}(x, \xi) \\
\big( \lambda \omega \cdot \nabla - |\xi|^2 \big) \psi_{n, 2}(x, \xi) + \chi_\lambda(\xi) q_{n, 2}(x, \xi) & 0
\end{pmatrix} \\
& \qquad + {\mathcal Q}_{\Psi_n} + {\mathcal R}_{\Psi_n}\,, \\
& \stackrel{\eqref{omologica psi n 1 psi n 2}}{=} {\mathcal Q}_{\Psi_n} + {\mathcal R}_{\Psi_n}\,, 
\end{aligned}
\end{equation}
\begin{equation}\label{espansione cal D Psin Q 12 - n bis}
\begin{aligned}
& {\mathcal Q}_{\Psi_n}  :=  \begin{pmatrix}
0 & {\rm Op}\Big(2 i \xi \cdot \nabla_x \psi_{n, 1}(x, \xi) - \Delta_x \psi_{n, 1}(x, \xi) \Big) \\
0 & 0
\end{pmatrix}\,, \\
& {\mathcal R}_{\Psi_n} := \begin{pmatrix}
0 & {\rm Op}( \big( 1 - \chi_\lambda(\xi) \big)q_{n, 1}(x, \xi) ) \\
 {\rm Op}\big( \big( 1 - \chi_\lambda(\xi) \big)q_{n, 2}(x, \xi) \big) & 0
\end{pmatrix}\,. 
\end{aligned}
\end{equation}
By using Lemma \ref{lemma smoothing merda semi norme}, by the estimates \eqref{solution omologica dec step n} and by the induction estimates \eqref{stime decoupling cal L n}, one easily gets that

\begin{equation}\label{stime cal Q R Psi n}
\begin{aligned}
& {\mathcal Q}_{\Psi_n} \in \Op^{- \frac{n}{2}}\,, \quad | {\mathcal Q}_{\Psi_n}|_{- \frac{n}{2}, s, \alpha} \lesssim_{s, \alpha}1 +  \| {\mathcal I} \|_{s + \sigma}, \quad \forall s \geq s_0, \quad \alpha \in \N_{0}\,, \\
& |{\mathcal R}_{\Psi_n}|_{- N, s, 0} \lesssim_{s, N} \lambda^{6 \delta (N + \frac{1 - n}{2})} \lambda^{3 \delta}  (1 + \| {\mathcal I} \|_{s + \sigma}) \\
&\qquad  \qquad \quad  \lesssim_{s, N} \lambda^{6 \delta (N + \frac{1}{2})} \lambda^{3 \delta} (1 + \| {\mathcal I} \|_{s + \sigma}) \\
&\qquad  \qquad \quad \lesssim_{s, N} \lambda^{6 \delta (N + 1)} (1 + \| {\mathcal I} \|_{s + \sigma}) , \quad \forall s \geq s_0.
\end{aligned}
\end{equation}
The estimates \eqref{solution omologica dec step n}, \eqref{Pi 0 Pi 0 bot op}, together with Lemma \ref{lem:exponential} and the ansatz \eqref{ansatz} allows to deduce that 
\begin{equation}\label{stima exp tau Psi n}
\sup_{\tau \in [- 1, 1]} |e^{\tau \Psi_{n, \bot}}|_{0, s, 0} \lesssim_s 1 + \| {\mathcal I} \|_{s + \sigma}, \quad \forall s \geq s_0\,,
\end{equation}
form which one deduces the claimed bound \eqref{stima decoupling iterative Phin} for $\Phi_n^{\pm 1} = e^{\pm \Psi_{n, \bot}}$. 

Since by induction one has that ${\mathcal Q}_{\frac{1 - n}{2}}^{(n)}$, then by using Lemma \ref{lem:equazione-psi} one deduce that $\Phi_n$ is real and therefore by composition also ${\mathcal L}^{(n + 1)}$ is so.
\medskip

\noindent
{\sc Analysis of ${\mathcal S}_1^{(n)}$.} By recalling \eqref{prima espansione cal L n Phi n} and using \eqref{espansione cal D Psin Q 12 - n}, \eqref{espansione cal D Psin Q 12 - n bis}, \eqref{stime cal Q R Psi n} one writes 
$$
\begin{aligned}
{\mathcal S}_1^{(n)}& = {\mathcal S}_a^{(n)} + {\mathcal S}_b^{(n)}\,, \\
{\mathcal S}_a^{(n)} & := \frac12 \int_0^1 (1 - \tau) e^{- \tau \Psi_{n, \bot}} [\Pi_0^\bot {\mathcal Q}_{\Psi_n} \Pi_0^\bot\,, \Psi_{n, \bot}] e^{\tau \Psi_{n, \bot}}\, d \tau\,, \\
 {\mathcal S}_b^{(n)} & :=  \frac12 \int_0^1 (1 - \tau) e^{- \tau \Psi_{n, \bot}} [\Pi_0^\bot {\mathcal R}_{\Psi_n} \Pi_0^\bot\,,\, \Psi_{n, \bot}] e^{\tau \Psi_{n, \bot}}\, d \tau \,.
\end{aligned}
$$
Clearly ${\mathcal S}_1^{(n)}$ is invariant on the space of zero average functions, namely ${\mathcal S}_1^{(n)} = \Pi_0^\bot {\mathcal S}_1^{(n)} \Pi_0^\bot$, ${\mathcal S}_a^{(n)} = \Pi_0^\bot {\mathcal S}_a^{(n)} \Pi_0^\bot$ and ${\mathcal S}_b^{(n)} = \Pi_0^\bot {\mathcal S}_b^{(n)} \Pi_0^\bot$. By the estimates \eqref{solution omologica dec step n}, \eqref{stime cal Q R Psi n}, \eqref{stima exp tau Psi n}, \eqref{Pi 0 Pi 0 bot op} by the composition estimate \eqref{est:composition} (using also the ansatz \eqref{ansatz}), one gets that the operator ${\mathcal S}_b^{(n)}$ is of order $- N$ and it satisfies the estimate 
\begin{equation}\label{stima cal Sb n}
| {\mathcal S}_b^{(n)}|_{- N, s, 0} \lesssim_{s, N} \lambda^{6 \delta(N + 1)}(1 +  \| {\mathcal I} \|_{s + \sigma}), \quad \forall s \geq s_0\,.
\end{equation}
Moreover, by \eqref{solution omologica dec step n}, \eqref{stime cal Q R Psi n}, Lemma \ref{lemma composizione 2} (to expand $[ \Pi_0^\bot {\mathcal Q}_{\Psi_n} \Pi_0^\bot \,, \Psi_{n, \bot}]$),  Lemma \ref{coniugio senza alpha exponential map} (to expand $e^{- \tau \Psi_{n, \bot}}[ \Pi_0^\bot {\mathcal Q}_{\Psi_n} \Pi_0^\bot \,, \Psi_{n, \bot}] e^{\tau \Psi_{n, \bot}}$, $\tau \in [0, 1]$) ,using again \eqref{ansatz}), one obtains the following expansion for the operator ${\mathcal S}_a^{(n)}$: 
\begin{equation}\label{espansione cal S a (n)}
\begin{aligned}
& {\mathcal S}_a^{(n)}  = \Pi_0^\bot \Big( {\mathcal S}_c^{(n)} + {\mathcal S}_d^{(n)} \Big) \Pi_0^\bot\,, \\
& {\mathcal S}_c^{(n)} \in \Op^{- n - 1}\,, \quad |{\mathcal S}_c^{(n)}|_{- n- 1, s, \alpha} \lesssim_{s, \alpha}1 +  \|{\mathcal I} \|_{s + \sigma}, \quad \forall s \geq s_0, \quad \forall \alpha \in \N_{0}\,, \\
& |{\mathcal S}_d^{(n)}|_{- N, s, 0} \lesssim_s 1+  \| {\mathcal I}\|_{s + \sigma}, \quad \forall s \geq s_0\,. 
\end{aligned}
\end{equation}
This concludes the analysis of the remainder ${\mathcal S}_1^{(n)}$. 

\medskip

\noindent
{\sc Analysis of ${\mathcal S}_2^{(n)}$.} Recall the formula \eqref{prima espansione cal L n Phi n} for ${\mathcal S}_2^{(n)}$. By recalling the expression of ${\mathcal F}_1$ in \eqref{mathcal L expansion 0}, the estimates of $a(x)$ in Lemma \ref{stime linearized originario} and by the induction estimates \eqref{stime decoupling cal L n} on $\mathcal{F}_0^{(n)}, {\mathcal Q}_{\frac{1 - n}{2}}^{(n)}$, one obtains that 
\begin{equation}\label{F1 F0 n Q n}
\begin{aligned}
& {\mathcal F}_1 + {\mathcal F}_0^{(n)} + {\mathcal Q}_{\frac{1- n}{2}}^{(n)} \in \Op^1 \quad \text{and} \\
& \Big|{\mathcal F}_1 + {\mathcal F}_0^{(n)} + {\mathcal Q}_{\frac{1- n}{2}}^{(n)} \Big|_{1, s, \alpha} \lesssim_{s, \alpha} \lambda^{3 \delta}  \big( 1 +  \| {\mathcal I} \|_{s + \sigma} \big), \quad \forall s \geq s_0, \quad \alpha \in \N_{0}\,. 
\end{aligned}
\end{equation}
Hence by \eqref{F1 F0 n Q n}, \eqref{solution omologica dec step n}, Lemma \ref{lemma composizione 2} (to expand $[\Pi_0^\bot \big( {\mathcal F}_1 + {\mathcal F}_0^{(n)} + {\mathcal Q}_{\frac{1- n}{2}}^{(n)} \big) \Pi_0^\bot, \Psi_{n, \bot}]$), \ref{coniugio senza alpha exponential map} (to expand $e^{- \tau \Psi_{n, \bot}}[\Pi_0^\bot \big( {\mathcal F}_1 + {\mathcal F}_0^{(n)} + {\mathcal Q}_{\frac{1- n}{2}}^{(n)} \big) \Pi_0^\bot, \Psi_{n, \bot}] e^{\tau \Psi_{n, \bot}}$), using again \eqref{ansatz}, one obtains the following expansion for the operator ${\mathcal S}_2^{(n)}$: 
\begin{equation}\label{espansione cal S 2 (n)}
\begin{aligned}
& {\mathcal S}_2^{(n)}  = \Pi_0^\bot \Big( {\mathcal S}_e^{(n)} + {\mathcal S}_f^{(n)} \Big) \Pi_0^\bot\,, \\
& {\mathcal S}_e^{(n)} \in \Op^{- \frac{n}{2} }\,, \quad |{\mathcal S}_e^{(n)}|_{- \frac{n}{2}, s, \alpha} \lesssim_{s, \alpha} \lambda^{3 \delta} \big( 1 +  \|{\mathcal I} \|_{s + \sigma} \big), \quad \forall s \geq s_0, \quad \forall \alpha \in \N_{0}\,, \\
& |{\mathcal S}_f^{(n)}|_{- N, s, 0} \lesssim_s \lambda^{3 \delta} \big( 1 +  \| {\mathcal I}\|_{s + \sigma} \big) \\
& \qquad \qquad \lesssim_s \lambda^{6 \delta(N + 1)} \big( 1 +  \| {\mathcal I}\|_{s + \sigma} \big)  , \quad \forall s \geq s_0\,. 
\end{aligned}
\end{equation}
This concludes the analysis of the remainder ${\mathcal S}_2^{(n)}$. 

\medskip

\noindent
{\sc Analysis of ${\mathcal S}_3^{(n)}$.} By \eqref{prima espansione cal L n Phi n}, by the estimates \eqref{stima exp tau Psi n}, the induction estimate \eqref{stime decoupling cal L n} on ${\mathcal R}_{- N}^{(n)}$, and by the composition estimate \eqref{est:composition} (together with the ansatz \eqref{ansatz}), one deduces the ${\mathcal S}_3^{(n)}$ is an operator of order $- N$ and it satisfies the estimate
\begin{equation}\label{stima cal S3 (n)}
|{\mathcal S}_3^{(n)}|_{- N, s, 0} \lesssim_s \lambda^{6 \delta(N + 1)}\big( 1 +  \| {\mathcal I} \|_{s + \sigma} \big), \quad \forall s \geq s_0\,. 
\end{equation}

\medskip

\noindent
Finally, by collecting \eqref{prima espansione cal L n Phi n}, \eqref{espansione cal D Psin Q 12 - n}, \eqref{espansione cal D Psin Q 12 - n bis}, \eqref{stime cal Q R Psi n}, \eqref{stima cal Sb n}, \eqref{espansione cal S a (n)}, \eqref{espansione cal S 2 (n)}, \eqref{stima cal S3 (n)} one gets that 
\begin{equation}\label{prop quasi finali cal L n + 1}
\begin{aligned}
{\mathcal L}^{(n + 1)} & = \Phi_n^{- 1} {\mathcal L}^{(n)} \Phi_n =  \mathcal{D}+\mathcal{F}_1 + \Pi_0^\bot \Big( {\mathcal F}_0^{(n)} + {\mathcal R}_{- \frac{n}{2}}^{(n + 1)} +\mathcal{R}_{-N}^{(n + 1)} \Big) \Pi_0^\bot\,, \\
{\mathcal R}_{- \frac{n}{2}}^{(n + 1)} & := {\mathcal Q}_{\Psi_n} + {\mathcal S}_c^{(n)} + {\mathcal S}_e^{(n)}    \,, \\
\mathcal{R}_{-N}^{(n + 1)} & := {\mathcal R}_{\Psi_n} + {\mathcal S}_b^{(n)} + {\mathcal S}_d^{(n)} + {\mathcal S}_f^{(n)} + {\mathcal S}_3^{(n)}\,, \\
|{\mathcal R}_{- \frac{n}{2}}^{(n + 1)}|_{- \frac{n}{2}, s, \alpha} & \lesssim_{s, \alpha} \lambda^{3 \delta} \big( 1 +  \| {\mathcal I} \|_{s + \sigma} \big), \quad \forall s \geq s_0, \quad \alpha \in \N_{0}\,, \\
|\mathcal{R}_{-N}^{(n + 1)}|_{- N, s, 0} & \lesssim_{s, N} \lambda^{6 \delta(N + 1)}\big( 1 +  \| {\mathcal I} \|_{s + \sigma} \big), \quad \forall s \geq s_0\,. 
\end{aligned}
\end{equation}
for some $\sigma \equiv \sigma_N \gg 0$ large enough. Then the claimed properties \eqref{stime decoupling cal L n} for ${\mathcal L}^{(n + 1)}$ then follows by \eqref{prop quasi finali cal L n + 1} by splitting ${\mathcal R}_{- \frac{n}{2}}^{(n + 1)} = {\mathcal F}_{- \frac{n}{2}}^{(n + 1)} + {\mathcal Q}_{- \frac{n}{2}}^{(n + 1)}$ where ${\mathcal F}_{- \frac{n}{2}}^{(n + 1)}$ is the diagonal part of ${\mathcal R}_{- \frac{n}{2}}^{(n + 1)}$ and ${\mathcal Q}_{- \frac{n}{2}}^{(n + 1)}$ is the off-diagonal part of ${\mathcal R}_{- \frac{n}{2}}^{(n + 1)}$ and by defining ${\mathcal F}_0^{(n + 1)} := {\mathcal F}_0^{(n)} + {\mathcal F}_{- \frac{n}{2}}^{(n + 1)}$. The claimed expansion \eqref{formula cal L (n) decoupling} and the claimed estimates \eqref{stime decoupling cal L n} at the step $n + 1$ has then been proved. The estimates \eqref{Delta 12 pezzi di cal L (0)} at the step $n + 1$ follows by similar arguments. 

To prove that ${\mathcal F}_0^{(n+1)}, \mathcal{Q}_{ -\frac{n}{2} }^{(n+1)} $  are real one can reason as done in Lemma \ref{coniugazione L} by using the inductive hypothesis on ${\mathcal F}_0^{(n)}, \mathcal{Q}_{ \frac{1 - n}{2} }^{(n)}$ and finally  ${\mathcal R}_{- \frac{n}{2}}^{(n + 1)}$ is real by difference.
\end{proof}

\begin{prop} \label{proposizione L1}
Let $N \in \N$, $M := 6(N + 1)$, $\gamma \in (0, 1)$, $\tau > 0$, $\lambda^{- \delta} \gamma^{- 1} \leq 1$, $\omega \in {\rm DC}(\gamma, \tau)$. Then there exists $\sigma \equiv \sigma_N \gg  0$ such that if \eqref{ansatz} holds, then there exists a real, invertible map ${\bf \Phi}_N$ satisfying 
\begin{equation}\label{stima Phi}
{\bf \Phi}_N^{\pm 1} : H^s_0 \to H^s_0\,, \quad |{\bf \Phi}_N^{\pm 1} |_{0, s, 0}^{\Lip(\gamma)}\ \lesssim_{s, N} 1 + \| {\mathcal I}\|_{s + \sigma}^{\Lip(\gamma)}, \quad \forall s \geq s_0
\end{equation}
and such that the linear operator $\Lin$ in \eqref{block representation cal L} transforms as follows: 
\begin{equation}\label{forma cal L1 dopo decoupling}
\begin{aligned}
\Lin_1 & := {\bf \Phi}_N^{-1} \Lin {\bf \Phi}_N = \begin{pmatrix}
{\mathcal L}_1^{(1)} & {\mathcal L}_{- N}^{(2)} \\
{\mathcal L}_{- N}^{(3)} & {\mathcal L}_1^{(4)}
\end{pmatrix} \,, \\
{\mathcal L}_1^{(1)}& :=  \lambda \omega \cdot \nabla - \Delta +  a(x) \cdot \nabla +\Pi_0^\bot  {\mathcal S}_0^{(1)} \Pi_0^\bot + {\mathcal S}_{- N}^{(1)} \,, \\
{\mathcal L}_4^{(1)}  & :=  \lambda \omega \cdot \nabla  +  a(x) \cdot \nabla +  \Pi_0^\bot {\mathcal S}_0^{(4)} \Pi_0^\bot  + {\mathcal S}_{- N}^{(4)},
\end{aligned}
\end{equation}
where ${\mathcal S}_0^{(1)}, {\mathcal S}_0^{(4)} \in\Op^0$, whereas the remainders $ {\mathcal L}_{- N}^{(2)}\,,\,  {\mathcal L}_{- N}^{(3)}, {\mathcal S}_{- N}^{(1)}, {\mathcal S}_{- N}^{(4)}$ are operators of order $- N$ satisfying 
\begin{equation}\label{prop cal Q (1) 1234}
\begin{aligned}
& |{\mathcal S}_0^{(1)}|_{0 ,s,\alpha}^{\Lip(\gamma)}\,,\,|{\mathcal S}_0^{(4)}|_{0 ,s,\alpha}^{\Lip(\gamma)}  \lesssim_{s, \alpha}\lambda^{3\delta}\big( 1 + \|\mathcal{I}\|^{\Lip(\gamma)}_{s+\sigma} \big), \quad \forall s \geq s_0, \quad \alpha \in \N_{0}\,, \\
& {\mathcal L}_{- N}^{(2)}, {\mathcal L}_{- N}^{(3)}, {\mathcal S}_{- N}^{(1)}, {\mathcal S}_{- N}^{(1)} \in {\mathcal B}(H^s_0\,,\, H^{s+ N}_0), \quad \forall s \geq s_0 \quad \text{and} \\
& \|{\mathcal L}_{- N}^{(2)} h \|_{s + N}^{\Lip(\gamma)}\,,\,\|{\mathcal L}_{- N}^{(3)}h \|_{s + N}^{\Lip(\gamma)} \,,\, \|{\mathcal S}_{- N}^{(1)} h\|_{s + N}^{\Lip(\gamma)}\,,\,\|{\mathcal S}_{- N}^{(4)} h \|_{s + N}^{\Lip(\gamma)}  \\
& \lesssim_s\lambda^{M \delta} \Big( \| h \|_s^{\Lip(\gamma)} +  \|\mathcal{I}\|^{\Lip(\gamma)}_{s+\sigma} \| h \|_{s_0}^{\Lip(\gamma)} \Big), \quad \forall s \geq s_0\,. 
\end{aligned}
\end{equation}

\noindent
Let $s_1 \geq s_0$, $\alpha \in \N_{0}$ and let ${\mathcal I}_1, {\mathcal I}_2$ satisfy \eqref{ansatz} with $s_1 + \sigma$ instead of $s_0 + \sigma$. Then 
\begin{equation}\label{Delta 12 dopo decoupling}
 | \Delta_{12} {\mathcal S}_0^{(1)}|_{0 ,s_1 ,\alpha}\,,\,|\Delta_{12} {\mathcal S}_0^{(4)}|_{0 ,s_1,\alpha}  \lesssim_{s_1, \alpha}\lambda^{3\delta}\|{\mathcal I}_1 - {\mathcal I}_2\|_{s_1+\sigma}
\end{equation}
Finally, the operators $\mathcal L_{1}, {\mathcal S}_0^{(1)}, {\mathcal S}_0^{(4)}, {\mathcal L}_{- N}^{(2)}\,,\,  {\mathcal L}_{- N}^{(3)}, {\mathcal S}_{- N}^{(1)}, {\mathcal S}_{- N}^{(4)}$ are real.
\end{prop}
\begin{proof}
We apply Lemmata \ref{coniugazione L}, \ref{sub:decit}. We set 
$$
{\bf \Phi}_N :=  \Phi \circ \Phi_0 \circ \ldots \circ \Phi_{2 N}\,.
$$
The estimate \eqref{stima Phi} then follows by \eqref{stima Phi Phi inv}, \eqref{stima decoupling iterative Phin} and by applying the composition estimate \eqref{est:composition} (using also \eqref{ansatz}). Moreover, the operator $\Lin_1 \equiv {\Lin}^{(2 N + 1)} = {\bf \Phi}_N^{- 1} \Lin {\bf \Phi}_N$ is of the form \eqref{forma cal L1 dopo decoupling} by setting 
$$
\begin{aligned}
&{\mathcal F}_0^{(2 N + 1)} \equiv \begin{pmatrix}
{\mathcal S}^{(1)}_{0} & 0  \\
0 & {\mathcal S}^{(4)}_{0} 
\end{pmatrix}\,, \qquad \Pi_0^\bot \Big(  {\mathcal Q}_{- N}^{(2 N + 1)} + {\mathcal R}_{- N}^{(2 N + 1)} \Big) \Pi_0^\bot \equiv \begin{pmatrix}
{\mathcal S}^{(1)}_{- N} & {\mathcal L}^{(2)}_{- N}  \\
{\mathcal L}^{(3)}_{- N}  & {\mathcal S}^{(4)}_{- N} 
\end{pmatrix}\,. 
\end{aligned}
$$ 
Then by the estimate \eqref{stime decoupling cal L n} and using also Lemma \ref{lemma azione tame pseudo-diff} one deduces the estimates \eqref{prop cal Q (1) 1234}. Finally, the estimates \eqref{Delta 12 dopo decoupling} are a direct consequence of the estimates \eqref{Delta 12 pezzi di cal L (0)} for $\Delta_{12} {\mathcal F}_0^{(2 N + 1)}$. 
Since the maps $ \Phi, \Phi_0, \ldots, \Phi_{2 N}$  are real by Lemmata \ref{sub:decit} and \ref{coniugazione L}, then the map ${\bf \Phi}_N $ is real by composition, together with ${\Lin}^{(2 N + 1)}$. To prove the realty of ${\mathcal S}_0^{(1)}, {\mathcal S}_0^{(4)}, {\mathcal L}_{- N}^{(2)}\,,\,  {\mathcal L}_{- N}^{(3)}, {\mathcal S}_{- N}^{(1)}, {\mathcal S}_{- N}^{(4)}$  one reasons as done in the proof of Lemma \ref{coniugazione L}.
\end{proof}

\section{Inversion of the first equation}\label{inversione prima equazione calore}
\subsection{Inversion of the operator ${\mathcal L}_1^{(1)}$}\label{sezione inversione cal L 1 1}
In order to invert the operator ${\mathcal L}_1$, we first invert the operator 
\begin{equation}\label{def cal L 1 1}
{\mathcal L}_1^{(1)} := \lambda \omega \cdot \nabla - \Delta + {\mathcal R}_1^{(1)}, \quad {\mathcal R}_1^{(1)} := a(x) \cdot \nabla + \Pi_0^\bot {\mathcal S}_0^{(1)} \Pi_0^\bot  + {\mathcal S}_{- N}^{(1)}\,, 
\end{equation}
given in \eqref{forma cal L1 dopo decoupling}.  Note that by Lemma \ref{stime linearized originario}, by Lemma \ref{lemma azione tame pseudo-diff} and by the estimates \eqref{prop cal Q (1) 1234}, 
and by the ansatz \eqref{ansatz}, one obtains that 
\begin{equation}\label{prop cal R 1 (1)}
\begin{aligned}
& {\mathcal R}_1^{(1)} \in {\mathcal B}(H^{s + 1}_0, H^s_0), \quad \forall s \geq s_0 \quad \text{and}   \\
& \| {\mathcal R}_1^{(1)} h \|_s^{\Lip(\gamma)} \lesssim_s \lambda^{\delta M} \Big( \| h \|_{s + 1}^{\Lip(\gamma)} + \| {\mathcal I} \|_{s + \sigma}^{\Lip(\gamma)} \| h \|_{s_0 + 1}^{\Lip(\gamma)} \Big), \quad \forall s \geq s_0\,. 
\end{aligned}
\end{equation}
Then, we write  $\mathcal{L}_1^{(1)}$ as
\begin{equation}\label{lem: inversione laplaciano}
\mathcal{L}_1^{(1)}:= L_\lambda + {\mathcal R}_1^{(1)}, \quad L_\lambda := \lambda \omega \cdot \nabla - \Delta.
\end{equation}
We have the following.
\begin{lem}\label{lem:inversione lineare cal D lambda}
Let $\lambda>1$, $\gamma \in (0, 1)$, $0 < \mathtt p <  \frac{1}{\tau + 2}$ and $\omega\in {\rm DC}(\gamma,\tau)$. Then the operator $L_\lambda$ is invertible and its inverse is a continuous operator $L_\lambda^{-1}:H^s_0(\T^2)\to H^{s+1}_0(\T^2)$ satisfying the bound 
\begin{equation}\label{stima L lambda inv guadagno}
\|L_\lambda^{-1}\|_{\mathcal{B}(H_0^s,H_0^{s+1})}\lesssim \gamma^{- 1} \lambda^{- \mathtt p}.
\end{equation}
Moreover, one also has {\it the estimate with loss of derivatives}
\begin{equation}\label{stima L lambda inv perdita}
\|L_\lambda^{-1}\|_{\mathcal{B}(H_0^{s + \tau},H_0^{s})}\lesssim \gamma^{- 1} \lambda^{- 1}, \quad \forall s \geq 0\,. 
\end{equation}
\end{lem}
\begin{proof}
We denote by $\Lambda(k) \equiv \Lambda(k; \omega)$, $\omega \in {\rm DC}(\gamma, \tau)$, the eigenvalues of the operator $L_\lambda \equiv L_\lambda(\omega)$, which are given by the formula
\begin{equation}
\Lambda(k):=i\lambda\omega\cdot k+|k|^2, \quad k \in \Z^2 \setminus \{ 0 \}\,. 
\end{equation}
We look for a lower bound on the eigenvalues $\Lambda(k)$. Let ${\mathtt p} > 0$. In the high frequency regime, we have that
\begin{equation}\label{Omega k high}
|\Lambda (k)|\geq |k|^2\geq \lambda^{\mathtt p}\,|k| \quad \text{for} \quad  |k| \geq \lambda^{\mathtt p}. 
\end{equation}
On the other hand, for low frequencies $|k| \leq \lambda^{{\mathtt p}}$ we use the Diophantine condition, obtaining that 
\begin{equation}\label{Omega k low}
|\Lambda(k)|\geq \lambda |\omega\cdot k| \geq \lambda \frac{\gamma}{|k|^\tau} \geq \gamma\lambda^{1-\tau {\mathtt p}}  \quad \text{for} \quad  |k| \leq \lambda^{\mathtt p}\,. 
\end{equation}
Thus, for any $h\in H^s_0(\T^2)$, one has that 
\begin{equation}\label{D lambda inverse sup}
\begin{aligned}
\|L_\lambda^{-1}h\|_{s+1}^2&=\sum_{k \neq 0} \langle k \rangle^{2(s+1)}\frac{|\hat{h}(k)|^2}{|\Lambda (k)|^2}\\
&=\sum_{0 < |k|\leq \lambda^{\mathtt p}} \langle k \rangle^{2(s+1)}\frac{|\hat{h}(k)|^2}{|\Lambda (k)|^2} +\sum_{|k|\geq \lambda^{\mathtt p}} \langle k \rangle^{2(s+1)}\frac{|\hat{h}(k)|^2}{|\Lambda (k)|^2}\\
&\lesssim \gamma^{-2}\lambda^{2(\tau {\mathtt p} -1)}\sum_{|k|\leq \lambda^{\mathtt p}} \langle k \rangle^{2(s+1)}|\hat{h}(k)|^2\\
&+\lambda^{-2{\mathtt p}}\sum_{|k|\geq \lambda^{\mathtt p}} \langle k \rangle^{2 (s + 1)}\frac{|\hat{h}(k)|^2}{|k|^2}\\
&\lesssim \left(\gamma^{-2} \lambda^{2 (\tau {\mathtt p} + {\mathtt p} - 1)}+\lambda^{-2 {\mathtt p}}\right)\|h\|_s^2 \lesssim \lambda^{- 2 {\mathtt p}} \gamma^{- 2} \| h \|_s^2,
\end{aligned}
\end{equation}
since $\gamma \in (0, 1)$, $0<{\mathtt p} <\frac{1}{\tau + 2}$, which implies that $\lambda^{\tau {\mathtt p} + {\mathtt p} - 1} \lesssim \lambda^{- {\mathtt p}}$. The latter chain of inequalities then implies \eqref{stima L lambda inv guadagno}. In order to prove \eqref{stima L lambda inv perdita}, one simply uses that $\omega \in {\rm DC}(\gamma, \tau)$, namely 
$$
|\Lambda(k)| \geq \lambda |\omega \cdot k| \geq \frac{\lambda \gamma}{|k|^\tau}, \quad \forall k \in \Z^2 \setminus \{ 0 \}
$$
which implies that $\| L_\lambda^{- 1} h \|_s \lesssim \lambda^{- 1}\gamma^{- 1} \| h \|_{s + \tau}$ for any $s \geq 0$, $h \in H^{s + \tau}_0$. 

\end{proof}
We now fix the constants $\mathtt p$ and $\delta$ as 
\begin{equation}\label{prima def a delta}
\begin{aligned}
& {\mathtt{p}}:=  \delta (M + 1)\,, \quad 0 < \delta < \frac{1}{ (M + 1)(\tau + 2)}\,, \\
& \text{where we recall that} \qquad M = 6(N + 1)\,. 
\end{aligned}
\end{equation} 

\begin{lem}\label{lemma inversione cal L 1 1}
Let $\gamma \in (0, 1)$, $\tau > 0$, $N \in \N$ and $\omega\in {\rm DC}(\gamma,\tau)$. Then there exists $\sigma \equiv \sigma_N \gg 0$ large enough such that if \eqref{ansatz} holds, then  for any $S > s_0 + \sigma$ there exists $\varepsilon(S) \ll 1$ small enough such that if $\lambda^{- \delta} \gamma^{- 1} \leq \varepsilon(S)$, then the following holds
\begin{itemize}
\item{\bf Invertibility of ${\mathcal L}_1^{(1)}$ with gain of regularity.} The operator $\Lin_1^{(1)}$ is invertible with inverse $\left(\Lin_1^{(1)}\right)^{-1}:H^s_0(\T^2)\to H_0^{s+1}(\T^2)$, for any $s_0  \leq s \leq S - \sigma$ and it satisfies the estimate 
\begin{equation}\label{stima L 1 1 gain}
\Big\| \left(\Lin_1^{(1)}\right)^{-1} h \Big\|_{s + 1}\lesssim_s \lambda^{- \delta M}\Big( \| h \|_s + \| {\mathcal I} \|_{s + \sigma} \| h \|_{s_0} \Big)\,, \quad \forall s_0  \leq s \leq S - \sigma\,. 
\end{equation}
\item{\bf Lipschitz estimate of $({\mathcal L}_1^{(1)})^{- 1}$ with loss of derivatives.} For any $s_0 \leq s \leq S - \sigma$, for $h(\cdot; \omega) \in H^{s + 2 \tau + 1}_0$, $\omega \in {\rm DC}(\gamma, \tau)$, the operator $({\mathcal L}_1^{(1)})^{- 1}$ satisfies the estimate 
\begin{equation}\label{stima L 1 1 loss}
\Big\| \left(\Lin_1^{(1)}\right)^{-1} h \Big\|_{s}^{\Lip(\gamma)} \lesssim_s  \lambda^{- \delta M} \Big( \| h \|_{s + 2 \tau + 1}^{\Lip(\gamma)} + \| {\mathcal I} \|_{s + \sigma}^{\Lip(\gamma)} \| h \|_{s_0+ 2 \tau + 1 }^{\Lip(\gamma)} \Big)\,.
\end{equation}
\end{itemize}
\end{lem}
\begin{proof}
{\sc Proof of \eqref{stima L 1 1 gain}.} First of all, by recalling \eqref{lem: inversione laplaciano}, we rewrite the operator $\Lin_1^{(1)}$ as
\begin{equation}\label{cal L 11 P lambda D lambda}
\Lin_1^{(1)}= L_\lambda \Big( {\rm Id} + {\mathcal P}_\lambda  \Big), \quad {\mathcal P}_\lambda  := L_\lambda^{- 1} {\mathcal R}_1^{(1)}.
\end{equation}
In order to invert ${\mathcal L}_1^{(1)}$, we need to invert the operator $ {\rm Id} + {\mathcal P}_\lambda$. We will make this by Neumann series. Indeed, by the definition of $\mathtt p$ in \eqref{prima def a delta}, by the estimate \eqref{prop cal R 1 (1)} and by the estimate \eqref{stima L lambda inv guadagno} in Lemma \ref{lem:inversione lineare cal D lambda}, one has that for $\sigma \equiv \sigma_N \gg 0$ large enough, $\| {\mathcal I} \|_{s_0 + \sigma} \lesssim 1$ (see \eqref{ansatz}), for any $S \geq s_0 + \sigma$, for any $s_0 + 1  \leq s \leq S - \sigma$, one has that 
\begin{equation}\label{stima D lambda inv cal R 1 1}
\begin{aligned}
{\mathcal P}_\lambda & : H^s_0 \to H^s_0 \,, \\
\| {\mathcal P}_\lambda h \|_s & \lesssim  \lambda^{- \delta (M + 1)} \gamma^{- 1}  \| {\mathcal R}_1^{(1)} h \|_{s - 1} \\
& \lesssim_s  \lambda^{- \delta(M + 1)} \gamma^{- 1}  \lambda^{ \delta M} \Big( \| h \|_s + \| {\mathcal I} \|_{s + \sigma} \| h \|_{s_0} \Big), \\
& \lesssim_s \lambda^{- \delta} \gamma^{- 1} \Big( \| h \|_s + \| {\mathcal I} \|_{s + \sigma} \| h \|_{s_0} \Big)\,. 
\end{aligned}
\end{equation}
 The estimate \eqref{stima D lambda inv cal R 1 1} together with the ansatz \eqref{ansatz} gives the estimate
\begin{equation}\label{stima D lambda inv cal R 1 1 s0}
\begin{aligned}
\| {\mathcal P}_\lambda h \|_{s_0} & \lesssim \lambda^{- \delta} \gamma^{- 1}    \| h \|_{s_0}\,.
\end{aligned}
\end{equation}
By iterating the estimates \eqref{stima D lambda inv cal R 1 1}, \eqref{stima D lambda inv cal R 1 1 s0}, one gets that for any integer $n \in \N$ 
\begin{equation}\label{stima D lambda inv cal R 1 1 potenze}
\begin{aligned}
\| {\mathcal P}_\lambda^n h \|_s & \leq  \Big(C(s) \lambda^{- \delta }\gamma^{- 1}  \Big)^n  \Big( \| h \|_s + \| {\mathcal I} \|_{s + \sigma} \| h \|_{s_0} \Big), \quad \forall s_0 + 1 \leq s \leq S - \sigma, \\
\| {\mathcal P}_\lambda^n h \|_{s_0} & \leq  \Big(C(s_0) \lambda^{- \delta }\gamma^{- 1}  \Big)^n   \| h \|_{s_0},
\end{aligned}
\end{equation}
for some constants $C(s) \gg C(s_0) \gg 0$ large enough. Hence by expanding $({\rm Id} + {\mathcal P}_\lambda)^{- 1} = \sum_{n \geq 0} (- 1)^n {\mathcal P}_\lambda^n$ by Neumann series and by using the estimates \eqref{stima D lambda inv cal R 1 1 potenze} one obtains that for any $s_0 \leq s \leq S - \sigma$, and by taking $\lambda \gg 1$ large enough in such a way that $C(S)  \lambda^{-\delta}  \gamma^{- 1}\ll 1$ 
\begin{equation}\label{stima inv Id + P lambda}
\| ({\rm Id} + {\mathcal P}_\lambda)^{- 1} h \|_s \lesssim_s  \| h \|_s + \| {\mathcal I} \|_{s + \sigma } \| h \|_{s_0}\,, \quad \forall s_0 + 1 \leq s \leq S - \sigma\,.  
\end{equation}
By the estimate of Lemma \ref{lem:inversione lineare cal D lambda}, using that $\mathtt p = \delta(M + 1)$ and $\lambda^{- \delta} \gamma^{- 1} \ll 1$, one has that $\| L_\lambda^{- 1} \|_{{\mathcal B}(H^s_0, H^{s + 1}_0)} \lesssim \lambda^{- \delta M}$ and therefore the claimed bound \eqref{stima L 1 1 gain} on the operator $({\mathcal L}_1^{(1)})^{- 1}$ follows by \eqref{cal L 11 P lambda D lambda}, \eqref{stima inv Id + P lambda}. 

\noindent
{\sc Proof of \eqref{stima L 1 1 loss}.} Let $s_0 \leq s \leq S - \sigma$, $h(\cdot; \omega) \in H^{s + 2 \tau + 1}_0$, $\omega \in {\rm DC}(\gamma, \tau)$. By the estimate \eqref{stima L lambda inv perdita} in Lemma \ref{lem:inversione lineare cal D lambda} and by the estimate \eqref{stima inv Id + P lambda}, one gets that 
\begin{equation}\label{stima L 1 1 loss sup}
\Big\| \left(\Lin_1^{(1)}\right)^{-1} h \Big\|_{s}^{\rm sup} \lesssim_s  \lambda^{- 1} \gamma^{- 1}\Big( \| h \|_{s + \tau}^{\rm sup} + \| {\mathcal I} \|_{s + \sigma}^{\Lip(\gamma)} \| h \|_{s_0+\tau }^{\rm sup} \Big)\,.
\end{equation}
Now let $\omega_1, \omega_2 \in {\rm DC}(\gamma, \tau)$. Since by \eqref{prima def a delta}, $0 < \delta < \frac{1}{M + 1}$, for $\lambda^{- \delta} \gamma^{- 1} \ll 1$, one has that $\lambda^{\delta M} \gamma^{- 1} \lesssim \lambda$. Hence, by \eqref{prop cal R 1 (1)}, \eqref{lem: inversione laplaciano}, one obtains for $s \geq s_0$, $h \in H^{s + 1}_0$
\begin{equation}\label{L 1 1 omega 1 omega 2}
\begin{aligned}
\Big\| \Big( \Lin_1^{(1)}(\omega_2) - \Lin_1^{(1)}(\omega_1) \Big) h \Big\|_s & \lesssim_s \lambda \| h \|_{s + 1} |\omega_1 - \omega_2|  \\
& \qquad +   \lambda^{\delta M} \gamma^{- 1} \Big( \| h \|_{s + 1} + \| {\mathcal I} \|_{s + \sigma}^{\Lip(\gamma)} \| h \|_{s_0 + 1} \Big) |\omega_1 - \omega_2| \\
& \lesssim_s \lambda \Big( \| h \|_{s + 1} + \| {\mathcal I} \|_{s + \sigma}^{\Lip(\gamma)} \| h \|_{s_0 + 1} \Big) |\omega_1 - \omega_2|\,. 
\end{aligned}
\end{equation}
Then, we write
$$
\begin{aligned}
& (\Lin_1^{(1)}(\omega_1))^{-1} h(\cdot; \omega_1) - (\Lin_1^{(1)}(\omega_2))^{-1}h(\cdot; \omega_2)  = {\mathcal G}_1  h(\cdot; \omega_1) + (\Lin_1^{(1)}(\omega_2))^{-1} [h(\cdot; \omega_1) - h(\cdot; \omega_2)]\,, \\
& {\mathcal G}_1  := (\Lin_1^{(1)}(\omega_1))^{-1} - (\Lin_1^{(1)}(\omega_2))^{-1} = (\Lin_1^{(1)}(\omega_1))^{-1}  \Big( \Lin_1^{(1)}(\omega_2) - \Lin_1^{(1)}(\omega_1) \Big) (\Lin_1^{(1)}(\omega_2))^{-1} \,. 
\end{aligned}
$$
By the estimates \eqref{stima L 1 1 loss sup}, \eqref{L 1 1 omega 1 omega 2}, one then obtains that for any $s_0 \leq s \leq S - \sigma$ (use also the ansatz \eqref{ansatz}), 
\begin{equation}\label{cal L 1 1 inv lip}
\begin{aligned}
& \gamma \Big\|  (\Lin_1^{(1)}(\omega_1))^{-1} h(\cdot; \omega_1) - (\Lin_1^{(1)}(\omega_2))^{-1}h(\cdot; \omega_2)  \Big\|_s \\
&  \lesssim_s \lambda^{- 1} \gamma^{- 1} \Big( \| h \|_{s + 2 \tau + 1}^{\rm sup} + \| {\mathcal I}\|_{s + \sigma}^{\Lip(\gamma)} \| h \|_{s _0+ 2 \tau + 1}^{\rm sup} \Big) |\omega_1 - \omega_2| \\
& + \lambda^{- 1} \gamma^{- 1} \Big( \gamma \| h \|_{s + \tau}^{\rm lip} + \| {\mathcal I}\|_{s + \sigma}^{\Lip(\gamma)} \gamma \| h \|_{s _0+ \tau}^{\rm lip} \Big) |\omega_1 - \omega_2| \\
& \lesssim_s \lambda^{- 1} \gamma^{- 1} \Big( \| h \|_{s + 2 \tau + 1}^{\Lip(\gamma)} + \| {\mathcal I}\|_{s + \sigma}^{\Lip(\gamma)} \| h \|_{s _0+ 2 \tau + 1}^{\Lip(\gamma)} \Big) |\omega_1 - \omega_2|\,. 
\end{aligned}
\end{equation}
The claimed bound \eqref{stima L 1 1 loss} then follows by \eqref{stima L 1 1 loss sup}, \eqref{cal L 1 1 inv lip}, using again that $\lambda^{- 1} \gamma^{- 1} \lesssim \lambda^{- \delta M}$. 
\end{proof}
\subsection{Reduction to a scalar transport-type operator}\label{sezione riduzione seconda componente}
In order to invert the linear operator ${\mathcal L}_1$ in \eqref{forma cal L1 dopo decoupling}, we need to solve the system 
\begin{equation}\label{sistema cal L1 a}
\begin{cases}
{\mathcal L}_1^{(1)} h_1 + {\mathcal L}_{- N}^{(2)} h_2 = g_1 , \\
{\mathcal L}_{- N}^{(3)} h_1 + {\mathcal L}_1^{(4)} h_2 = g_2\,. 
\end{cases}
\end{equation}
By Lemma \ref{lemma inversione cal L 1 1}, we have that 
\begin{equation}\label{sistema cal L1 b}
h_1 = ({\mathcal L}_1^{(1)})^{- 1} g_1 -  ({\mathcal L}_1^{(1)})^{- 1} {\mathcal L}_{- N}^{(2)} h_2,
\end{equation}
and we can replace the latter expression of $h_1$ in the second equation of \eqref{sistema cal L1 a}, by obtaining the equation
\begin{equation}\label{sistema cal L1 c}
\begin{aligned}
{\mathcal P}[h_2] & = g_2 - {\mathcal L}_{- N}^{(3)} ({\mathcal L}_1^{(1)})^{- 1} [g_1]\,, \\
{\mathcal P} & := {\mathcal L}_1^{(4)} - {\mathcal L}_{- N}^{(3)} ({\mathcal L}_1^{(1)})^{- 1} {\mathcal L}_{- N}^{(2)}\,. 
\end{aligned}
\end{equation}
Hence, in order to conclude the invertibility of ${\mathcal L}_1$, it is enough to solve the equation \eqref{sistema cal L1 c}, namely to invert the linear operator ${\mathcal P}$.  
The following lemma holds. 
\begin{lem}\label{lemma op cal T seconda equazione}
Let $\gamma \in (0, 1)$, $\tau > 0$, $N \geq 2 \tau + 1$ and $\omega\in {\rm DC}(\gamma,\tau)$.
There exists $\sigma \equiv \sigma_N \gg 0$ large enough such that if \eqref{ansatz} holds then the linear operator ${\mathcal P}$ has the form 
\begin{equation}\label{transport op cal P totale}
{\mathcal P} = \lambda \omega \cdot \nabla +  a(x) \cdot \nabla + \Pi_0^\bot {\mathcal Q} \Pi_0^\bot + {\mathcal R}
\end{equation}
where ${\mathcal Q} \in \Op^0$ and ${\mathcal R}$ is a smoothing operator of order $- N$ satisfying the following properties. The operator ${\mathcal Q}$ satisfies the estimate
\begin{equation}\label{stima cal Q op trasp}
\begin{aligned}
& |{\mathcal Q}|_{0, s, \alpha}^{\Lip(\gamma)} \lesssim_{s, \alpha} \lambda^{\delta M}(1 +  \| {\mathcal I} \|_{s + \sigma}^{\Lip(\gamma)}), \quad \forall s \geq s_0, \quad  \alpha \in \N_{0}\,, \\
\end{aligned}
\end{equation}
and for any $S > s_0 + \sigma$ there exists $\e (S) \ll 1$ small enough such that if $ \lambda^{-   \delta } \gamma^{- 1} \leq \e(S)$, then 
\begin{equation}\label{stima cal R0 op trasp}
\begin{aligned}
& {\mathcal R} \in {\mathcal B}(H^s_0, H^{s + N}_0)\,, \quad \forall s_0 \leq s \leq S - \sigma\,, \\
& \| {\mathcal R} h \|_{s + N}^{\Lip(\gamma)} \lesssim_{s, N} \lambda^{\delta M}\Big(  \| h \|_s^{\Lip(\gamma)} + \| {\mathcal I} \|_{s + \sigma}^{\Lip(\gamma)} \| h \|_{s_0}^{\Lip(\gamma)} \Big)\,, \quad \forall s_0 \leq s \leq S - \sigma\,. 
\end{aligned}
\end{equation}
Let $s_1 \geq s_0$, $\alpha \in \N_{0}$ and assume that ${\mathcal I}_1, {\mathcal I}_2$ satisfies \eqref{ansatz} with $s_1 + \sigma$ instead of $s_0 + \sigma$. Then 
\begin{equation}\label{Delta 12 cal Q inizio trasporto}
|\Delta_{12} {\mathcal Q}|_{0, s_1, \alpha} \lesssim_{s_1, \alpha} \lambda^{\delta M} \| {\mathcal I}_1 - {\mathcal I}_2 \|_{s_1 + \sigma}\,. 
\end{equation}
Finally, the operators $\mathcal{P}, \mathcal{Q}$ and $\mathcal{R}$ are real.
\end{lem}
\begin{proof}
To simplify the notations we write $\| \cdot \|_s$ instead of $\| \cdot \|_s^{\Lip(\gamma)}$. By recalling \eqref{forma cal L1 dopo decoupling}, \eqref{prop cal Q (1) 1234}, \eqref{Delta 12 dopo decoupling}, \eqref{sistema cal L1 c}, the estimates \eqref{stima cal Q op trasp}, \eqref{Delta 12 cal Q inizio trasporto} on ${\mathcal Q} \equiv {\mathcal S}_0^{(4)}$ immediately follows, hence we estimate the operator ${\mathcal R} := - {\mathcal L}_{- N}^{(3)} ({\mathcal L}_1^{(1)})^{- 1} {\mathcal L}_{- N}^{(2)} + {\mathcal S}_{- N}^{(4)}$. Let $S > s_0 + \sigma$, then if $\lambda^{- \delta } \gamma^{- 1}  \leq \varepsilon (S) \ll 1$ is small enough, we can apply the estimate \eqref{stima L 1 1 loss} in Lemma \ref{lemma inversione cal L 1 1} and by using the estimates \eqref{prop cal Q (1) 1234}, Lemma \ref{lemma azione tame pseudo-diff} and the ansatz \eqref{ansatz}, one gets 
$$
\begin{aligned}
& \| {\mathcal L}_{- N}^{(3)} ({\mathcal L}_1^{(1)})^{- 1} {\mathcal L}_{- N}^{(2)} h \|_{s + N}  \leq \| {\mathcal L}_{- N}^{(3)} ({\mathcal L}_1^{(1)})^{- 1} {\mathcal L}_{- N}^{(2)} h \|_{s + 2 N - 2 \tau  - 1} \\
& \lesssim_s \lambda^{\delta M} \Big(  \| ({\mathcal L}_1^{(1)})^{- 1} {\mathcal L}_{- N}^{(2)} h \|_{s + N -  2 \tau - 1} + \| {\mathcal I} \|_{s + \sigma} \| ({\mathcal L}_1^{(1)})^{- 1} {\mathcal L}_{- N}^{(2)} h \|_{s_0 + N - 2 \tau - 1} \Big) \\
& \lesssim_s \lambda^{\delta M} \lambda^{- \delta M}\Big( \| {\mathcal L}_{- N}^{(2)} h \|_{s + N } + \| {\mathcal I} \|_{s + \sigma} \|  {\mathcal L}_{- N}^{(2)} h \|_{s_0 + N } \Big) \\
& \lesssim_s \lambda^{\delta M} \lambda^{- \delta M} \lambda^{\delta M}  \Big( \|  h \|_{s} + \| {\mathcal I} \|_{s + \sigma} \|  h \|_{s_0 } \Big) \\
& {\lesssim_s} \lambda^{\delta M} \Big(  \|  h \|_{s} + \| {\mathcal I} \|_{s + \sigma} \|  h \|_{s_0 } \Big) 
\end{aligned}
$$
which imply the claimed bound for ${\mathcal R}$. 
By Proposition \ref{proposizione L1} and by composition we have that ${\mathcal Q}$ and ${\mathcal R}$ are real and therefore ${\mathcal P}$ is so, since it is a sum of real operators.
The proof of the lemma is then concluded. 
\end{proof}
\section{Normal form reduction and inversion of the transport operator ${\mathcal P}$}\label{sezione op trasporto}

In this section we discuss the invertibility of the linear operator ${\mathcal P}$ in \eqref{transport op cal P totale}. The first step to be implemented is to reduce the highest order part to constant coefficients, namely the transport operator 
\begin{equation}\label{ordine massimo trasporto}
\begin{aligned}
& {\mathcal T}_\lambda := \lambda \omega \cdot \nabla +  a(x) \cdot \nabla = \lambda {\mathcal T}\,, \quad  {\mathcal T} := \omega \cdot \nabla + a_\lambda(x) \cdot \nabla\,, \\
& a_\lambda(x) := \lambda^{ - 1} a(x)  \,. 
\end{aligned}
\end{equation}
By \eqref{definizione a} and Lemma \ref{stime linearized originario}, one clearly has that $a_\lambda$ satisfies 
\begin{equation}\label{stime a lambda pre trasporto}
\begin{aligned}
& \| a_\lambda \|_s^{\Lip(\gamma)} \lesssim_s \lambda^{\delta - 1} \| {\mathcal I} \|_{s }^{\Lip(\gamma)}, \quad \forall s \geq s_0\,, \\
& \text{and if} \quad s_1 \geq s_0, \quad \| a_\lambda({\mathcal I}_1) - a_\lambda({\mathcal I}_2) \|_{s_1} \lesssim_{s_1} \lambda^{\delta - 1} \| {\mathcal I}_1 - {\mathcal I}_2 \|_{s_1} \,. 
\end{aligned}
\end{equation}
We also recall that the function $a(x)$ has zero average and zero divergence
\begin{equation}\label{prop di a per rid trasporto}
\int_{\T^2} a(x)\, d x = 0, \quad {\rm div}(a) = 0\,. 
\end{equation}
%
\subsection{Reduction of the transport operator of order one}
We recall the following result from \cite[Proposition 4.1]{BM} (see also \cite{FGMP}). 
\begin{prop}\label{prop:riduzione1}
Let ${\mathcal T}$ as in \eqref{ordine massimo trasporto} and assume \eqref{prop di a per rid trasporto}. Then there exists $\sigma \equiv \sigma(\tau) \gg 0$ large enough such that if \eqref{ansatz} holds, for $S > s_0 + \sigma$ there exists $ \varepsilon \equiv \varepsilon(S) \ll 1$ small enough such that if 
\begin{equation}\label{condizione gamma}
\lambda^{\delta - 1}\gamma^{-1}\leq \varepsilon,
\end{equation}
then the following holds. There exists an invertible diffeomorphism $x\in\T^2\mapsto x+\balpha(x;\omega)\in\T^2$ with inverse $y\mapsto y+\check{\balpha}(y;\omega)$, defined for all $\omega\in {\rm DC}(\gamma,\tau)$
, satisfying, for any $s_0\leq s \leq S-\sigma$,
\begin{equation}\label{stime alpha alpha cappuccio}
\|\balpha\|_s^{\Lip(\gamma)},\|\check{\balpha}\|_s^{\Lip(\gamma)}\lesssim_s \lambda^{\delta - 1}\gamma^{-1} \| {\mathcal I} \|_{s + \sigma}^{\Lip(\gamma)},
\end{equation}
such that, by defining
\begin{equation}
\mathcal{A}h(x):=h(x+\balpha(x)),\,\, \mbox{with }\,\, \mathcal{A}^{-1}h(y)=h(y+\check{\balpha}(y)),
\end{equation}
one gets the conjugation
$$
\mathcal{A}^{-1}\left(\omega\cdot\nabla+ a_\lambda(x)\cdot\nabla\right)\mathcal{A}=\omega\cdot\nabla.
$$
Moreover, let $s_1 \geq s_0$ and let us assume that ${\mathcal I}_1, {\mathcal I}_2$ satisfy \eqref{ansatz} with $s_1 + \sigma$ instead of $s_0 + \sigma$. Then 
\begin{equation}\label{Delta 12 alpha alpha breve}
\begin{aligned}
& \| \Delta_{12} \balpha \|_{s_1}\,,\, \| \Delta_{12} \check{\balpha} \|_{s_1} \lesssim_{s_1} \lambda^{\delta - 1}\gamma^{-1} \| {\mathcal I}_1 - {\mathcal I}_2 \|_{s_1 + \sigma}\,, \\
& \| (\Delta_{12} {\mathcal A}) h \|_{s_1}\,,\, \| (\Delta_{12} {\mathcal A}^{- 1}) h \|_{s_1}  \lesssim_{s_1} \lambda^{\delta - 1}\gamma^{-1} \| {\mathcal I}_1 - {\mathcal I}_2 \|_{s_1 + \sigma} \| h \|_{s_1 + 1}\,. 
\end{aligned}
\end{equation}
\end{prop}
In the next lemma we provide the full conjugation of the operator ${\mathcal P}$ by means of the map ${\mathcal A}_\bot := \Pi_0^\bot {\mathcal A} \Pi_0^\bot$. 
\begin{prop}\label{prop coniugio cal P con cal A}
Let $\gamma \in (0, 1)$, $\tau > 0$, $N \geq 2 \tau + 1$, $\omega \in {\rm DC}(\gamma, \tau)$. There exists $\sigma \equiv \sigma_N \gg 0$ large enough such that for any $S > s_0 + \sigma$ there exists $\varepsilon(S) \ll 1$ such that if $\lambda^{- \delta } \gamma^{- 1} \leq \varepsilon(S)$ and \eqref{ansatz} holds then for any $\omega \in {\rm DC}(\gamma, \tau)$, the linear operator ${\mathcal P}$ in \eqref{transport op cal P totale} transforms under the action of the map ${\mathcal A}_\bot := \Pi_0^\bot {\mathcal A} \Pi_0^\bot$ as follows. 
\begin{equation}\label{coniugio cal P cal A}
{\mathcal P}_0 := {\mathcal A}_\bot^{- 1} {\mathcal P} {\mathcal A}_\bot = \lambda \omega \cdot \nabla + \Pi_0^\bot {\mathcal Q}_0 \Pi_0^\bot + {\mathcal R}_0 ,
\end{equation}
where ${\mathcal Q}_0 \in \Op^0_{S - \sigma, \alpha}$ for any $\alpha \in \N_{0}$ and ${\mathcal R}_0$ is a smoothing operator of order $- N$ satisfying 
\begin{equation}\label{stima cal Q 0 cal R 0 dopo diffeo}
\begin{aligned}
& |{\mathcal Q}_0|_{0, s, \alpha}^{\Lip(\gamma)} \lesssim_{s, \alpha} \lambda^{\delta M} ( 1 + \| {\mathcal I} \|_{s + \sigma}^{\Lip(\gamma)}), \quad \forall s_0 \leq s \leq S - \sigma, \quad \forall \alpha \in \N_{0}\,,  \\
& {\mathcal R}_0 \in {\mathcal B}(H^s_0, H^{s + N}_0)\,, \quad \forall s_0 \leq s \leq S - \sigma\,,  \\
& \| {\mathcal R}_0 h  \|_{s + N}^{\Lip(\gamma)} \lesssim_{s} \lambda^{\delta M} \big( \| h \|_s^{\Lip(\gamma)} + \| {\mathcal I} \|_{s + \sigma}^{\Lip(\gamma)} \| h \|_{s_0}^{\Lip(\gamma)} \big)\,, \quad  \forall s_0 \leq s \leq S - \sigma \,. 
\end{aligned}
\end{equation}
Let $s_1 \geq s_0$, $\alpha \in \N_{0}$ and assume that ${\mathcal I}_1, {\mathcal I}_2$ satisfy \eqref{ansatz} with $s_1 + \sigma$ instead of $s_0 + \sigma$. Then 
\begin{equation}\label{Delta 12 cal Q0}
|\Delta_{12} {\mathcal Q}_0|_{0, s_1, \alpha} \lesssim_{s, \alpha} \lambda^{\delta M} \| {\mathcal I}_1 - {\mathcal I}_2 \|_{s_1 + \sigma}\,. 
\end{equation}
Finally, the operators $\mathcal{P}_{0}, \mathcal{Q}_{0}$ and $\mathcal{R}_{0}$ are real.
\end{prop}
\begin{proof}
To simplify notations we write $\| \cdot \|_s$ for $\| \cdot \|_s^{\Lip(\gamma)}$ and $|\cdot|_{m, s, \alpha}$ for $|\cdot|_{m, s, \alpha}^{\Lip(\gamma)}$. By \eqref{prima def a delta}, $0 < \delta < \frac12$ hence for $\lambda \gg 1$ large enough, 
$\lambda^{\delta - 1} \gamma^{- 1} \leq \lambda^{- \delta} \gamma^{- 1} \ll 1$ and hence the smallness condition of Proposition \ref{prop:riduzione1} is verified. By \eqref{invarianze medie nulle} $[a(x) \cdot \nabla\,,\, \Pi_0^\bot] = 0$, one has that 
$$
\Pi_0^\bot \big(  \lambda \omega \cdot \nabla +  a(x) \cdot \nabla\big) \Pi_0^\bot = \lambda \omega \cdot \nabla +  a(x) \cdot \nabla
$$
and by applying Lemma \ref{invertibilita cal A bot} (${\mathcal A}_\bot^{- 1} = \Pi_0^\bot {\mathcal A}^{- 1} \Pi_0^\bot$) and Proposition \ref{prop:riduzione1}, one gets that 
\begin{equation}\label{cal R Q (1)}
\begin{aligned}
& {\mathcal A}_\bot^{- 1} \big(  \lambda \omega \cdot \nabla +  a(x) \cdot \nabla\big) {\mathcal A}_\bot = \Pi_0^\bot {\mathcal A}^{- 1} \big(  \lambda \omega \cdot \nabla +  a(x) \cdot \nabla\big) {\mathcal A} \Pi_0^\bot = \lambda \omega \cdot \nabla\,, \\
& {\mathcal A}_\bot^{- 1} {\mathcal Q} {\mathcal A}_\bot = \Pi_0^\bot {\mathcal A}^{- 1} \Pi_0^\bot {\mathcal Q} \Pi_0^\bot {\mathcal A} \Pi_0^\bot = \Pi_0^\bot {\mathcal A}^{- 1} {\mathcal Q} {\mathcal A} \Pi_0^\bot   + {\mathcal R}_{{\mathcal Q}}^{(1)}\,, \\
& {\mathcal R}_{{\mathcal Q}}^{(1)} := - \Pi_0^\bot {\mathcal A}^{- 1} \Pi_0 {\mathcal Q} \Pi_0^\bot {\mathcal A} \Pi_0^\bot - \Pi_0^\bot {\mathcal A}^{- 1}  {\mathcal Q} \Pi_0 {\mathcal A} \Pi_0^\bot \,,
\end{aligned}
\end{equation}
therefore  
$$
\begin{aligned}
{\mathcal A}_\bot^{- 1} {\mathcal P} {\mathcal A}_\bot  & =\lambda\omega\cdot\nabla + \Pi_0^\bot {\mathcal A}^{- 1} {\mathcal Q} {\mathcal A} \Pi_0^\bot + {\mathcal A}_\bot^{- 1} {\mathcal R} {\mathcal A}_\bot + {\mathcal R}_{{\mathcal Q}}^{(1)}\,. 
\end{aligned}
$$
We now analyze the conjugation ${\mathcal A}^{- 1} {\mathcal Q} {\mathcal A}$ of $\mathcal{Q}=\mathrm{Op}(q(x,\xi))$. By Theorem \ref{quantitativeegorov} we have that
\begin{equation}
{\mathcal A}^{- 1} {\mathcal Q} {\mathcal A}=\mathrm{Op}(q_0(x,\xi))+ {\mathcal R}_{\mathcal Q}^{(2)},
\end{equation}
and there exists $\sigma \equiv \sigma_N \gg \mu \gg 0$ such that if \eqref{ansatz} holds then for any $S > s_0 + \sigma$, for any $s_0 \leq s \leq S - \sigma$, for any $\alpha \in \N_{0}$, one has 
 \begin{equation}\label{q0 cal R cal R tilde dopo egorov}
\begin{aligned}
|\mathrm{Op}(q_0)|_{0,s,\alpha}&\lesssim_{s,\alpha}|\mathcal{Q}|_{0,s,\alpha+\mu}+\sum_{k_{1}+k_{2}+k_{3}=s}|\mathcal{Q}|_{0,k_1,\alpha+k_2+\mu}\|\balpha\|_{k_3+\mu}\,, \\
\| {\mathcal R}_{\mathcal Q}^{(2)} h\|_{s+N}& \lesssim_{s,N} {\mathfrak M}(s) \| h \|_{s_0} + {\mathfrak M}(s_0) \| h \|_s\,, \\
{\mathfrak M}(s) := &  \sum_{k_{1}+k_{2}+k_{3}=s} |{\mathcal Q}|_{0 ,k_1,k_2+\mu} \|\balpha\|_{k_3+\mu}
\end{aligned}
\end{equation}
and furthermore, by applying the estimates \eqref{stima cal Q op trasp}, \eqref{stime alpha alpha cappuccio} on ${\mathcal Q}$ and $\balpha$, one obtains that  for $k_{1}+k_{2}+k_{3}=s$
$$
\begin{aligned}
& |\mathcal{Q}|_{0,s,\alpha+\mu}\,,\, |{\mathcal Q} |_{0 ,s+\mu,\mu } \lesssim_{s, \alpha} \lambda^{\delta M}( 1 +  \| {\mathcal I} \|_{s + \sigma})\,, \\
& |\mathcal{Q}|_{0,k_1,\alpha+k_2+\mu}\,,\,  |{\mathcal Q}|_{0 ,k_1,k_2+\mu} \lesssim_{s, \alpha} \lambda^{\delta M} ( 1 + \| {\mathcal I} \|_{s_1 + \sigma})\,, \\
& \|\balpha\|_{s_3+\mu} \lesssim_s \lambda^{\delta - 1}\gamma^{-1} \| {\mathcal I} \|_{s_3 + \sigma},
\end{aligned}
$$
for some $\sigma \equiv \sigma_N \gg \mu \gg 0$ large enough. Hence by the latter estimates and by \eqref{q0 cal R cal R tilde dopo egorov}, using that $\lambda^{\delta - 1} \gamma^{- 1} \leq \lambda^{- \delta} \gamma^{- 1} \ll 1$, one immediately obtains that 
\begin{equation}\label{q0 cal R cal R tilde dopo egorov 2}
\begin{aligned}
|\mathrm{Op}(q_0)|_{0,s,\alpha}\,,\, {\mathfrak M}(s) &\lesssim_{s,\alpha}\lambda^{\delta M} \Big( \| {\mathcal I} \|_{s + \sigma}+\sum_{k_{1}+k_{2}+k_{3}=s} \| {\mathcal I} \|_{k_1+k_{2} + \sigma}\| {\mathcal I} \|_{k_3 + \sigma} \Big) \,. 
\end{aligned}
\end{equation}
Finally by the interpolation estimate \eqref{interpolazione s1 s2 s} one gets that for $k_{1}+k_{2}+k_{3}=s$
$$
\begin{aligned}
\| {\mathcal I} \|_{k_1 +k_{2}+ \sigma} \| {\mathcal I} \|_{k_3 + \sigma} & \leq \| {\mathcal I} \|_{\sigma}^{\frac{s - k_1-k_{2}}{s}} \| {\mathcal I} \|_{s + \sigma}^{\frac{k_1+k_{2}}{s}} \| {\mathcal I} \|_{\sigma}^{\frac{s - k_3}{s}} \| {\mathcal I} \|_{s + \sigma}^{\frac{k_3}{s}} \stackrel{k_{1}+k_{2}+k_{3}=s}{\leq} \| {\mathcal I} \|_\sigma \| {\mathcal I} \|_{s + \sigma} \stackrel{\eqref{ansatz}}{\lesssim} \| {\mathcal I} \|_{s + \sigma},
\end{aligned}
$$
which implies that by \eqref{q0 cal R cal R tilde dopo egorov}, \eqref{q0 cal R cal R tilde dopo egorov 2}, ${\mathcal Q}_0 = {\rm Op}(q_0)$ and ${\mathcal R}_{\mathcal Q}^{(2)}$ satisfy the claimed estimate \eqref{stima cal Q 0 cal R 0 dopo diffeo}.  By the estimates \eqref{stima cal R0 op trasp}, Lemma \ref{lem:changevar}, the estimates \eqref{stime alpha alpha cappuccio} on $\balpha, \breve \balpha$, the property \eqref{Pi 0 Pi 0 bot op} on $\Pi_0$ and Lemma \ref{lemma azione tame pseudo-diff} (using also \eqref{ansatz} and $\lambda^{\delta - 1} \gamma^{- 1} \leq 1$), one obtains that also the remainders ${\mathcal R}_{\mathcal Q}^{(1)}$ and ${\mathcal A}_\bot^{- 1} {\mathcal R} {\mathcal A}_\bot$ satisfy the claimed bound \eqref{stima cal Q 0 cal R 0 dopo diffeo}. The claimed estimate \eqref{stima cal Q 0 cal R 0 dopo diffeo} on ${\mathcal R}_0 := {\mathcal R}_{\mathcal Q}^{(1)} + {\mathcal R}_{\mathcal Q}^{(2)} + {\mathcal A}_\bot^{- 1} {\mathcal R} {\mathcal A}_\bot$ then follows. The claimed bound \eqref{Delta 12 cal Q0} follows by similar arguments, using also \eqref{Delta 12 cal Q inizio trasporto}, \eqref{Delta 12 alpha alpha breve}. 

The operators $\mathcal{Q}_{0}={\rm Op}(q_{0})$ and $\mathcal R$ are real by Theorem \ref{quantitativeegorov} and therefore $\mathcal P$ is so.
\end{proof}

\subsection{Reduction of the lower order terms}
In this section we shall prove the following proposition in which one reduces the operator ${\mathcal P}_0$ to a diagonal one plus a smoothing remainder. 
\begin{prop} \label{proposizione regolarizzazione ordini bassi}
Let $\gamma \in (0, 1)$, $\tau > 0$, $N \geq 2 \tau + 1$, $\omega \in {\rm DC}(\gamma, \tau)$. Then there exists $\sigma \equiv \sigma_N \gg  0$ such that for any $S > s_0 + \sigma$ there exists $\varepsilon = \varepsilon(S) \ll 1$ such that if \eqref{ansatz} and $\lambda^{- \delta } \gamma^{- 1} \leq \varepsilon(S)$ are fullfilled then the following holds. There exists a
real, invertible map ${\mathcal V}$ satisfying 
\begin{equation}\label{stima Phi M reg ordini bassi}
{\mathcal V}^{\pm 1} : H^s_0 \to H^s_0\,, \quad |{\mathcal V}^{\pm 1} |_{0, s, 0}^{\Lip(\gamma)}\ \lesssim_{s} 1 + \| {\mathcal I}\|_{s + \sigma}^{\Lip(\gamma)}, \quad \forall s_0 \leq s \leq S - \sigma,
\end{equation}
and such that the linear operator ${\mathcal P}_0$ in \eqref{coniugio cal P cal A} transforms as follows: 
\begin{equation}\label{trasformazione cal P0 cal V}
{\mathcal P}_1 := {\mathcal V}^{- 1} {\mathcal P}_0 {\mathcal V} = \lambda \omega \cdot \nabla + {\mathcal Z} + {\mathcal R}_1,
\end{equation}
where ${\mathcal Z}(\omega) := {\rm diag}_{k \in \Z^2 \setminus \{ 0 \}} z(k;  \omega)$ is a diagonal operator whose eigenvalues satisfy the estimate  
\begin{equation}\label{stima autovalori cal P 1 dopo descent}
\sup_{k \in \Z^2 \setminus \{ 0 \}} | z(k; \cdot)|^{\Lip(\gamma)} \lesssim \lambda^{\delta M},
\end{equation}
and the remainder satisfies 
\begin{equation}\label{stima cal Z cal R (3)}
\begin{aligned}
& {\mathcal R}_1 \in {\mathcal B}(H^s_0, H^{s + N}_0)\,, \quad \forall s_0 \leq s \leq S - \sigma\,, \\
& \| {\mathcal R}_1 h \|_{s + N}^{\Lip(\gamma)} \lesssim_s \lambda^{\delta M} \Big( \| h \|_s^{\Lip(\gamma)} + \| {\mathcal I} \|_{s + \sigma}^{\Lip(\gamma)} \| h \|_{s_0}^{\Lip(\gamma)} \Big)\,, \quad \forall s_0 \leq s \leq S - \sigma\,. 
\end{aligned}
\end{equation}
Moreover, let ${\mathcal I}_1, {\mathcal I}_2$ satisfy \eqref{ansatz}. Then 
\begin{equation}\label{stime delta 12 cal L (3)}
\begin{aligned}
& \sup_{k \in \Z^2 \setminus \{ 0 \}}|\Delta_{12}  z(k) | \lesssim \lambda^{\delta M} \| {\mathcal I}_1 - {\mathcal I}_2 \|_{s_0 + \sigma}\,.  \\
\end{aligned}
\end{equation}
Finally, the operators ${\mathcal P}_1, {\mathcal Z}$ and ${\mathcal R}_1$ are real.
\end{prop}

Proposition \ref{proposizione regolarizzazione ordini bassi} 
follows by the iterative normal form Lemma \ref{lemma iterativo ordini bassi}. Before proving this, we prove a lemma in which we give study the homological equations required in this normal form procedure. 
\begin{lem}\label{lemma omologica trasporto large}
Let $\gamma \in (0, 1)$, $\tau > 0$, $m \in \R$, $\alpha \in \N_{0}$, $\omega \in {\rm DC}(\gamma, \tau)$ and $a \equiv a(\cdot, \cdot; \omega) \in {\mathcal S}^m_{s + 2 \tau + 1, \alpha}$. Then there exists a unique symbol $f \in {\mathcal S}^m_{s, \alpha}$ with zero average in $x$, that solves the equation 
\begin{equation}\label{equazione diofantea simboli}
\lambda \omega \cdot \nabla f(x, \xi) + a(x, \xi) = \langle a \rangle_x(\xi)\,.
\end{equation}
Moreover
\begin{equation}\label{stima equazione diofantea simboli}
|{\rm Op}(f)|_{m, s, \alpha}^{\Lip(\gamma)} \lesssim \lambda^{- 1} \gamma^{- 1} |{\rm Op}(a)|_{m, s + 2 \tau + 1, \alpha}^{\Lip(\gamma)}\,. 
\end{equation}
Finally, if ${\rm Op}(a)$ is real, then ${\rm Op}(f)$ is real.
\end{lem} 
\begin{proof}
By expanding in Fourier series, one gets that the only solution with zero average of the equation \eqref{equazione diofantea simboli} is given by 
$$
f(x, \xi) = - \sum_{k \in \Z^2 \setminus \{ 0 \}} \dfrac{\widehat a(k, \xi)}{ i \lambda \omega \cdot k} e^{i x \cdot k} \quad \text{defined for} \quad \omega \in {\rm DC}(\gamma, \tau)\,. 
$$
By applying Lemma \ref{lem:trasporto}, one gets that for any $\beta \in \N^2$, $|\beta| \leq \alpha$, 
$$
\| \partial_\xi^\beta f(\cdot, \xi) \|_s^{\Lip(\gamma)} \lesssim \lambda^{- 1} \gamma^{- 1} \| \partial_\xi^\beta a(\cdot, \xi) \|_{s + 2 \tau + 1}^{\Lip(\gamma)}\,.
$$
The claimed bound \eqref{stima equazione diofantea simboli} then immediately follows by recalling Definition \ref{def:family-pseudo-diff}. By the hypothesis on $a$ we have that ${\rm Op}(f)$ is a real operator by definition.
\end{proof}

\begin{lem}\label{lemma iterativo ordini bassi}
Let $\alpha \in \N_{0}$, $\gamma \in (0, 1)$, $\tau > 0$, $N \geq 2 \tau + 1$, $\omega \in {\rm DC}(\gamma, \tau)$. Then there exists $\sigma_i \equiv \sigma_i(\alpha, \tau) \gg 0$, $i = 1,2, \ldots, N$ large enough, $\sigma_1 < \sigma_2 <\ldots <\sigma_N$ such that for any $S > s_0 + \sigma_N$, there exists $\varepsilon(S) \ll 1$ small enough such that if \eqref{ansatz} holds and $\lambda^{- \delta } \gamma^{- 1} \leq \varepsilon(S) \ll 1$, then the following holds.  For any $n = 0, \ldots, N$, there exists a linear operator ${\mathcal P}_0^{(n)}$ of the form 
\begin{equation}\label{def cal L (2) n}
{\mathcal P}_0^{(n)} := \lambda \omega \cdot \nabla + \Pi_0^\bot {\mathcal Z}^{(n)} \Pi_0^\bot  + \Pi_0^\bot {\mathcal Q}_0^{(n)} \Pi_0^\bot + {\mathcal R}_0^{(n)},
\end{equation}
where ${\mathcal Z}^{(n)} = {\rm Op}(z^{(n)}(\xi))\in \Op^0_{S - \sigma_n, \alpha}$, ${\mathcal Q}_0^{(n)} = {\rm Op}(q_0^{(n)}(x, \xi)) \in  \Op^{- n}_{S - \sigma_n, \alpha}$ and ${\mathcal R}_0^{(n)}$ satisfy for any $s_0 \leq s \leq S - \sigma_n$, the estimates
\begin{equation}\label{stima induttiva cal Zn cal Rn (2)}
\begin{aligned}
& |{\mathcal Z}^{(n)}|_{0, s, \alpha}^{\Lip(\gamma)} \lesssim_{s, n} \lambda^{\delta M}( 1 +  \| {\mathcal I} \|_{s_0 + \sigma_n}^{\Lip(\gamma)}), \\
& |{\mathcal Q}_0^{(n)}|_{- n, s, \alpha}^{\Lip(\gamma)} \lesssim_{s, n, \alpha} \lambda^{\delta M} ( 1 + \| {\mathcal I} \|_{s + \sigma_n}^{\Lip(\gamma)}),  \\
& {\mathcal R}_0^{(n)} \in {\mathcal B}(H^s_0, H^{s + N}_0)\,, \\
& \| {\mathcal R}_0^{(n)} h \|_{s + N}^{\Lip(\gamma)} \lesssim_{s, N} \lambda^{\delta M} \Big( \| h \|_s^{\Lip(\gamma)} + \| {\mathcal I} \|_{s + \sigma_n}^{\Lip(\gamma)} \| h \|_{s_0}^{\Lip(\gamma)} \Big)\,. 
\end{aligned}
\end{equation} 
The operators ${\mathcal P}_0^{(n)}, {\mathcal Z}^{(n)}, {\mathcal Q}_0^{(n)}, {\mathcal R}_0^{(n)}$ are real. 
%
For any $n = 1, \ldots, N$ there exists a real, invertible map ${\mathcal T}_n$ defined for any $\omega \in {\rm DC}(\gamma, \tau)$, satisfying 
\begin{equation}\label{stima cal Tn reg ordini bassi}
{\mathcal T}_n^{\pm 1} : H^s_0 \to H^s_0\,, \quad |{\mathcal T}_n^{\pm 1}|_{0 , s, \alpha}^{\Lip(\gamma)} \lesssim_{s, n, \alpha} 1 + \| {\mathcal I} \|_{s + \sigma_n}^{\Lip(\gamma)}, \quad \forall s_0 \leq s \leq S - \sigma_n 
\end{equation}
and 
\begin{equation}\label{coniugazione lemma iterativo ordini bassi}
{\mathcal P}^{(n)}_0 = {\mathcal T}_{n - 1}^{- 1}{\mathcal P}_0^{(n - 1)} {\mathcal T}_{n - 1} 
\quad  \forall \omega \in {\rm DC}(\gamma, \tau)\,. 
\end{equation}
Let $s_1 \geq s_0$, $\alpha \in \N_{0}$ and 
assume that ${\mathcal I}_1, {\mathcal I}_2$ satisfy \eqref{ansatz} with $s_1 + \sigma_n$ instead of $s_0 +\sigma$. 
Then, for any $\omega  \in {\rm DC}(\gamma, \tau)$,
\begin{equation}\label{Delta 12 Tn Zn Rn (2)}
\begin{aligned}
& |\Delta_{12} {\mathcal T}_{n - 1}^{\pm 1} |_{- n, s_1, \alpha} \lesssim_{s_1, n, \alpha}  \| {\mathcal I}_1 - {\mathcal I}_2 \|_{s_1 + \sigma_n}\,, \\
& |\Delta_{12} {\mathcal Z}^{(n)}|_{0, s_1, \alpha}\,,\,  |\Delta_{12} {\mathcal Q}_0^{(n)}|_{- n, s_1, \alpha} \lesssim_{s_1, n, \alpha} \lambda^{\delta M} \| {\mathcal I}_1 - {\mathcal I}_2 \|_{s_1 + \sigma_n}\,. 
\end{aligned}
\end{equation}
\end{lem}
\begin{proof}
The proof is made by implementing an induction normal form procedure. We describe the induction step. Assume that the claimed statement holds at the step $n$. We look for a transformation
$$
{\mathcal T}_n := {\rm exp}({\mathcal M}_{n, \bot}), \quad  {\mathcal M}_{n, \bot} := \Pi_0^\bot {\mathcal M}_n \Pi_0^\bot\,, \quad {\mathcal M}_n = {\rm Op}(m_n(x, \xi))  \in \Op^{- n}
$$
where the symbol $m_n(x, \xi)$ has to be determined in order to normalize ${\mathcal Q}_0^{(n)} = {\rm Op}\big( q_0^{(n)}(x, \xi) \big) \in \Op^{- n}$. By conjugating the operator ${\mathcal P}_0^{(n)}$ by means of ${\mathcal T}_n$, one obtains that 
\begin{equation}\label{T n + - 1 P 0 n a}
\begin{aligned}
{\mathcal P}_0^{(n + 1)} :& = {\mathcal T}_n^{- 1}{\mathcal P}_0^{(n)} {\mathcal T}_n \\
& = \lambda \omega \cdot \nabla + \Pi_0^\bot {\mathcal Z}^{(n)} \Pi_0^\bot  + [\lambda \omega \cdot \nabla \,,\, \Pi_0^\bot {\mathcal M}_n \Pi_0^\bot ] + \Pi_0^\bot {\mathcal Q}_0^{(n)} \Pi_0^\bot   \\
& \quad + {\mathcal Q}_0^{(n + 1)}  + {\mathcal R}_0^{(n + 1)},
\end{aligned}
\end{equation}
where 
\begin{equation}\label{def cal Q0 n + 1}
\begin{aligned}
{\mathcal Q}_0^{(n + 1)} & := \int_0^1 (1 - \tau) e^{- \tau {\mathcal M}_{n, \bot}} [[\lambda \omega \cdot \nabla\,,\, {\mathcal M}_{n, \bot}]\,, {\mathcal M}_{n, \bot}] e^{\tau {\mathcal M}_{n, \bot}}\, d \tau \\
& \quad + \int_0^1 e^{- \tau {\mathcal M}_{n, \bot}} [\Pi_0^\bot {\mathcal Z}^{(n)} \Pi_0^\bot + \Pi_0^\bot {\mathcal Q}_0^{(n)} \Pi_0^\bot \,,\, {\mathcal M}_{n, \bot}] e^{\tau {\mathcal M}_{n, \bot}}\, d \tau\,, \\
{\mathcal R}_0^{(n + 1)} & := {\mathcal T}_n^{- 1} {\mathcal R}_0^{(n)} {\mathcal T}_n\,.  
\end{aligned}
\end{equation}
Note that 
\begin{equation}\label{prop cal Q0 n + 1}
{\mathcal Q}_0^{(n + 1)} = \Pi_0^\bot {\mathcal Q}_0^{(n + 1)} \Pi_0^\bot \quad \text{and} \quad [\lambda \omega \cdot \nabla\,,\, \Pi_0^\bot ] = 0,
\end{equation}
therefore, the term of order $- n$ is then given by
$$
\begin{aligned}
 [\lambda \omega \cdot \nabla \,,\, \Pi_0^\bot {\mathcal M}_n \Pi_0^\bot ] + \Pi_0^\bot {\mathcal Q}_0^{(n)} \Pi_0^\bot & = \Pi_0^\bot \Big( [\lambda \omega \cdot \nabla \,,\, {\mathcal M}_n] + {\mathcal Q}_0^{(n)} \Big) \Pi_0^\bot \\
 &  = \Pi_0^\bot {\rm Op} \Big( \lambda \omega \cdot \nabla m_n(x, \xi) + q_0^{(n)}(x, \xi) \Big) \Pi_0^\bot\,. 
\end{aligned}
$$
We choose the symbol $m_n(x, \xi)$ in such a way that 
\begin{equation}\label{equazione mn q 0 n media q0 n}
\lambda \omega \cdot \nabla m_n(x, \xi) + q_0^{(n)}(x, \xi)  = \langle q_0^{(n)} \rangle_x(\xi), \quad \langle q_0^{(n)} \rangle_x(\xi) := \frac{1}{(2 \pi)^2} \int_{\T^2} q_0^{(n)}(x, \xi)\, d x\,. 
\end{equation}
This can be done by applying Lemma \ref{lemma omologica trasporto large} and using that, since $1 - \delta M > \delta$ (see \eqref{prima def a delta}) and $\lambda^{\delta M - 1} \gamma^{- 1} \leq \lambda^{- \delta} \gamma^{- 1} \leq 1$, one gets 
$$
\begin{aligned}
|{\mathcal M}_n|_{- n, s, \alpha}^{\Lip(\gamma)} & \lesssim \lambda^{- 1} \gamma^{- 1} |{\mathcal Q}_0^{(n)}|_{- n , s + 2 \tau + 1, \alpha}^{\Lip(\gamma)} \stackrel{\eqref{stima induttiva cal Zn cal Rn (2)}}{\lesssim_{s, n, \alpha}} \lambda^{\delta M - 1} \gamma^{- 1} ( 1 +  \| {\mathcal I} \|_{s + \sigma_n + 2 \tau + 1}^{\Lip(\gamma)}) \\
& \lesssim_{s, n, \alpha} 1 +  \| {\mathcal I} \|_{s + \sigma_n + 2 \tau + 1}^{\Lip(\gamma)}, \quad \forall s_0 \leq s \leq S - \sigma_n - 2 \tau - 1\,. 
\end{aligned}
$$
Furthermore, by using also the composition estimate \eqref{est:composition} and the estimates \eqref{Pi 0 Pi 0 bot op} on $\Pi_0^\bot$, the latter estimate together with Lemma \ref{lem:exponential} and the ansatz \eqref{ansatz} imply that 
\begin{equation}\label{stima e tau Mn}
\begin{aligned}
& |{\mathcal M}_{n, \bot}|_{- n, s, \alpha}^{\Lip(\gamma)} \lesssim_{s, n, \alpha} 1 + \| {\mathcal I} \|_{s + \sigma_{n} + \mu(\alpha)}^{\Lip(\gamma)}\,, \quad \forall s_0 \leq s \leq S - \sigma_n - \mu(\alpha)\,, \\
& \sup_{\tau \in [- 1, 1]} | e^{\tau {\mathcal M}_{n, \bot}}|_{0, s, \alpha}^{\Lip(\gamma)} \lesssim_{s, \alpha} 1 + \| {\mathcal I} \|_{s + \sigma_{n} + \mu(\alpha)}^{\Lip(\gamma)}\,, \quad \forall s_0 \leq s \leq S - \sigma_n - \mu(\alpha),
\end{aligned}
\end{equation}
for some constant $\mu(\alpha) \gg 0$ large enough. By inductive hypothesis on $q_{0}^{(n)}$ and by using Lemma \ref{lemma omologica trasporto large} we have that  the maps ${\mathcal M}_n, {\mathcal M}_{n, \bot}$ and ${\mathcal T}_n$ are real. Then by composition and by using the inductive hypotesis on ${\mathcal P}_0^{(n)} $ one also has that ${\mathcal P}_0^{(n + 1)} $ is real.
By \eqref{T n + - 1 P 0 n a}, \eqref{equazione mn q 0 n media q0 n}, one gets that ${\mathcal P}_0^{(n + 1)}$ takes the form 
\begin{equation}\label{seconda forma cal P0 n + 1}
\begin{aligned}
{\mathcal P}_0^{(n + 1)} & = \lambda \omega \cdot \nabla + \Pi_0^\bot {\mathcal Z}^{(n + 1)} \Pi_0^\bot   +\Pi_0^\bot  {\mathcal Q}_0^{(n + 1)} \Pi_0^\bot   + {\mathcal R}_0^{(n + 1)}\,, \\
{\mathcal Z}^{(n + 1)} & := {\mathcal Z}^{(n)} + {\rm Op}\big(\langle q_0^{(n)} \rangle_x(\xi) \big)\,, \\
{\mathcal Q}_0^{(n + 1)} & =  \int_0^1 (1 - \tau) e^{- \tau {\mathcal M}_{n, \bot}} [ \Pi_0^\bot {\mathcal Z}^{(n)} \Pi_0^\bot - \Pi_0^\bot {\mathcal Q}_0^{(n)} \Pi_0^\bot, {\mathcal M}_{n, \bot}] e^{\tau {\mathcal M}_n}\, d \tau \\
& \quad + \int_0^1 e^{- \tau {\mathcal M}_{n, \bot}} [\Pi_0^\bot{\mathcal Z}^{(n)}\Pi_0^\bot + \Pi_0^\bot {\mathcal Q}_0^{(n)} \Pi_0^\bot \,,\, {\mathcal M}_{n, \bot}] e^{\tau {\mathcal M}_{n, \bot}}\, d \tau\,,  \\
{\mathcal R}_0^{(n + 1)} & = {\mathcal T}_n^{- 1} {\mathcal R}_0^{(n)} {\mathcal T}_n.
\end{aligned}
\end{equation}
The estimate of ${\mathcal Z}^{(n + 1)}$ follows by the induction estimates \eqref{stima induttiva cal Zn cal Rn (2)} on ${\mathcal Z}^{(n)}, {\mathcal Q}_0^{(n)}$ and by Lemma \ref{stima simbolo mediato}. The estimate of ${\mathcal Q}_0^{(n + 1)}$ can be done by using Lemma \ref{lem:composition-pseudodiff}, the estimate \eqref{stima e tau Mn}, the induction estimates \eqref{stima induttiva cal Zn cal Rn (2)}, the estimates \eqref{Pi 0 Pi 0 bot op}, together with the ansatz \eqref{ansatz}. The estimate of ${\mathcal R}_0^{(n + 1)}$ follows by Lemma \ref{lemma azione tame pseudo-diff}, together with the estimates \eqref{stima e tau Mn} and the induction estimate on ${\mathcal R}_0^{(n)}$. 

\noindent
The estimate \eqref{Delta 12 Tn Zn Rn (2)} at the step $n + 1$ can be proved by similar arguments. 
$\mathcal Z^{(n+1)}$ is real by inductive hypothesis and Lemma \ref{stima simbolo mediato}, ${\mathcal R}_0^{(n)}$ is real by composition and by difference ${\mathcal Q}_0^{(n)}$ is so.
\end{proof}
{\bf Proof of Proposition \ref{proposizione regolarizzazione ordini bassi}.}
We set 
$$
{\mathcal V} \equiv {\mathcal V}_N :=  {\mathcal T}_0 \circ {\mathcal T}_1 \circ \ldots \circ {\mathcal T}_{N - 1}\,.
$$
The estimate \eqref{stima Phi M reg ordini bassi} then follows by \eqref{stima cal Tn reg ordini bassi}, by Lemma \ref{lem:composition-pseudodiff} and by using the ansatz \eqref{ansatz}. By Lemma \ref{lemma iterativo ordini bassi} and by composition one has that ${\mathcal V}$ is real.
 Moreover the operator ${\mathcal P}_1 \equiv {\mathcal P}_0^{(N)} = {\mathcal V}^{- 1} {\mathcal P}_0 {\mathcal V}$ is of the form \eqref{trasformazione cal P0 cal V} with ${\mathcal Z} := \Pi_0^\bot {\mathcal Z}^{(N)} \Pi_0^\bot$ and ${\mathcal R}_1 := \Pi_0^\bot {\mathcal Q}_0^{(N)} \Pi_0^\bot +  {\mathcal R}_0^{(N)}$ and by composition it is a real operator. Then ${\mathcal Z}$ and ${\mathcal R}_1$ satisfies the desired properties \eqref{stima autovalori cal P 1 dopo descent}, \eqref{stima cal Z cal R (3)}, \eqref{stime delta 12 cal L (3)} by applying \eqref{stima induttiva cal Zn cal Rn (2)}, \eqref{Delta 12 Tn Zn Rn (2)} with $n = N$.

\subsection{Inversion of the operator ${\mathcal P}$}\label{sezione inversione trasporto} In this section we prove the invertibility of the operator ${\mathcal P}$ in \eqref{transport op cal P totale}.  
We define the non-resonant set 
\begin{equation}\label{prime melnikov inversione}
{\mathcal G}_\lambda(\gamma, \tau) \equiv {\mathcal G}_\lambda(\gamma, \tau; {\mathcal I}) := \Big\{ \omega \in {\rm DC}(\gamma, \tau) : |i\, \lambda  \omega \cdot k+ z(k; \omega)| \geq \frac{\lambda \gamma}{|k|^\tau}, \quad \forall k \in \Z^2 \setminus \{ 0 \}  \Big\}.
\end{equation}
For $\tau > 0$, we fix the constants $N, M, \delta$ as
\begin{equation}\label{costanti finale pre NM}
\begin{aligned}
& N := 2 \tau + 2\,, \quad M = 6(N + 1) = 6(2 \tau + 3)\,, \\
& 0 < \delta < \dfrac{1}{(M + 1)(\tau + 2)}
\end{aligned}
\end{equation}
where we recall \eqref{prima def a delta}. We prove the following proposition. 
\begin{prop}\label{invertibilita cal P1}
Let $\gamma \in (0, 1)$, $\tau > 0$ and $M, \delta$ as in \eqref{costanti finale pre NM}. Then there exists $\sigma \equiv \sigma(\tau) \gg 0$ large enough such that for any $S > s_0 + \sigma$, there exists $\varepsilon(S) \ll 1$ small enough such that if \eqref{ansatz} holds and $\lambda^{- \delta } \gamma^{- 1} \leq \varepsilon(S) \ll 1$, then the following holds. For any $\omega \in {\mathcal G}_\lambda(\gamma, \tau)$, the operator ${\mathcal P}$ is invertible and its inverse ${\mathcal P}^{- 1}$ is a real operator which satisfies for any $s_0 \leq s \leq S - \sigma$, the estimate
$$
\| {\mathcal P}^{- 1} h \|_s^{\Lip(\gamma)} \lesssim_s \lambda^{- 1} \gamma^{- 1} \Big( \| h \|_{s + 2 \tau + 1}^{\Lip(\gamma)} + \| {\mathcal I} \|_{s + \sigma}^{\Lip(\gamma)} \| h \|_{s_0 + 2 \tau + 1}^{\Lip(\gamma)} \Big)\,. 
$$
\end{prop}
\begin{proof}
First of all ${\mathcal P}^{-1}$ is real since the operator ${\mathcal P}$ is real.
Then, since by Propositions \ref{prop coniugio cal P con cal A}, \ref{proposizione regolarizzazione ordini bassi}, ${\mathcal P} = {\mathcal A}_\bot {\mathcal V} {\mathcal P}_1 {\mathcal V}^{- 1} {\mathcal A}_\bot^{- 1}$, we need to invert the operator ${\mathcal P}_1$ constructed in Proposition \ref{proposizione regolarizzazione ordini bassi}. We write 
$$
{\mathcal P}_1 = \mathcal{D} + {\mathcal R}_1, \quad \mathcal{D} := \lambda \omega \cdot \nabla + {\mathcal Z}\,. 
$$
The first thing that we discuss is the invertibility of ${\mathcal D}$.\\

\noindent
{\sc Invertibility of ${\mathcal D}$.} We consider the operator ${\mathcal D}(\omega) = {\rm diag}_{k \neq 0} \mu(k;\omega)$ where $\mu(k; \omega) = i\, \lambda \omega \cdot k + z(k; \omega)$, $k \in \Z^2 \setminus \{ 0 \}$. 
If $\omega \in {\mathcal G}_\lambda(\gamma, \tau)$ (see \eqref{prime melnikov inversione}), then $|\mu(k; \omega)| \geq \frac{\lambda \gamma}{|k|^\tau}$, $j \in \Z^2 \setminus \{ 0 \}$, implying that 
$$
{\mathcal D}(\omega)^{- 1} = {\rm diag}_{k \neq 0} \frac{1}{\mu(k; \omega)},
$$
satisfies the bound 
\begin{equation}\label{cal D inv senza lip}
\| {\mathcal D}^{- 1} h \|_s \lesssim \lambda^{- 1} \gamma^{- 1} \| h \|_{s + \tau}, \quad \forall s \geq 0\,, \quad h \in H^{s + \tau}_0(\T^2)\,. 
\end{equation}
Moreover, if $\omega_1, \omega_2 \in {\mathcal G}_\lambda(\gamma, \tau)$, one has that 
$$
\begin{aligned}
\Big| \frac{1}{\mu(k; \omega_1)} - \frac{1}{\mu(k; \omega_2)} \Big|  & = \frac{|\mu(k; \omega_1) - \mu(k; \omega_2)|}{|\mu(k; \omega_1)| |\mu(k; \omega_2)|} \\
& \leq \lambda^{- 1} \gamma^{- 2} |k|^{  2 \tau + 1} |\omega_1 - \omega_2| + \lambda^{- 2} \gamma^{- 2} |k|^{2 \tau} |z(k; \omega_1) - z(k; \omega_2)| \\
& \stackrel{\eqref{stima autovalori cal P 1 dopo descent}}{\lesssim} \lambda^{- 1} \gamma^{- 2} |k|^{  2 \tau + 1} |\omega_1 - \omega_2| + \lambda^{M \delta - 2} \gamma^{- 3} |k|^{2 \tau} |\omega_1 - \omega_2| \\
& \stackrel{\lambda^{M \delta - 1} \gamma^{- 1} \ll 1}{\lesssim} \lambda^{- 1} \gamma^{- 2} |k|^{2 \tau + 1} |\omega_1 - \omega_2|\,.
\end{aligned}
$$
The latter estimate then implies that 
\begin{equation}\label{lip D omega inv h}
\| ({\mathcal D}(\omega_1)^{- 1} - {\mathcal D}(\omega_2)^{- 1}) h \|_s \lesssim \gamma^{- 2} \lambda^{- 1} \| h \|_{s + 2 \tau + 1} |\omega_1 - \omega_2|\,, \quad \forall s \geq 0, \quad \forall h \in H^{s + 2 \tau + 1}_0(\T^2)\,. 
\end{equation}
The estimates \eqref{cal D inv senza lip}, \eqref{lip D omega inv h} imply that 
\begin{equation}\label{stima cal D inv lip gamma}
\| {\mathcal D}^{- 1}   \|_{{\mathcal B}(H^{s + 2 \tau + 1}_0, H^s_0)}^{\Lip(\gamma)} \lesssim \lambda^{- 1} \gamma^{- 1}\,, \quad \forall s \geq 0\,.  
\end{equation}
{\sc Invertibility of ${\mathcal P}_1$.} For $\omega \in {\mathcal G}_\lambda(\gamma, \tau)$, we write 
$$
{\mathcal P}_1 = {\mathcal D} \big( {\rm Id} + {\mathcal F}  \big), \quad {\mathcal F} := {\mathcal D}^{- 1} {\mathcal R}_1\,. 
$$
By the hypotheses of the proposition, we choose the order of regularization in Proposition \ref{proposizione regolarizzazione ordini bassi} as $N = 2 \tau + 2 \geq 2 \tau + 1$. By \eqref{prima def a delta}, since $0 < \delta < \frac{1}{1 + M}$, $\lambda^{\delta M - 1} \gamma^{- 1} \leq \lambda^{- \delta } \gamma^{- 1}$, hence, by the estimates \eqref{stima cal Z cal R (3)}, \eqref{stima cal D inv lip gamma}, one has that for for any $s_0 \leq s \leq S - \sigma$, 
$$
\begin{aligned}
\| {\mathcal F} h \|_s^{\Lip(\gamma)} & \leq C(s)  \lambda^{\delta M - 1} \gamma^{- 1} \Big( \| h \|_s^{\Lip(\gamma)} + \| {\mathcal I} \|_{s + \sigma}^{\Lip(\gamma)} \| h \|_{s_0}^{\Lip(\gamma)} \Big)\,, \\ 
& \leq C(s)  \lambda^{- \delta} \gamma^{- 1} \Big( \| h \|_s^{\Lip(\gamma)} + \| {\mathcal I} \|_{s + \sigma}^{\Lip(\gamma)} \| h \|_{s_0}^{\Lip(\gamma)} \Big)\,, \\
\| {\mathcal F} h \|_{s_0}^{\Lip(\gamma)} & \leq C(s_0) \lambda^{\delta M - 1} \gamma^{- 1} \| h \|_{s_0}^{\Lip(\gamma)} \leq C(s_0) \lambda^{- \delta} \gamma^{- 1} \| h \|_{s_0}^{\Lip(\gamma)}
\end{aligned}
$$
for $\sigma \gg 0$ and $C(s) \gg C(s_0) \gg 1$. By iterating the latter estimate, using \eqref{ansatz}, one gets that for any $n \in \N$, 
$$
\begin{aligned}
\| {\mathcal F}^n h \|_s^{\Lip(\gamma)} & \leq \Big( C(s)  \lambda^{- \delta} \gamma^{- 1} \Big)^n \Big( \| h \|_s^{\Lip(\gamma)} + \| {\mathcal I} \|_{s + \sigma}^{\Lip(\gamma)} \| h \|_{s_0}^{\Lip(\gamma)} \Big)\,, \quad \forall s_0 \leq s \leq S - \sigma\,, \\
\| {\mathcal F}^n h \|_{s_0}^{\Lip(\gamma)} & \leq \Big( C(s_0) \lambda^{- \delta} \gamma^{- 1}  \Big)^n \| h \|_{s_0}^{\Lip(\gamma)}
\end{aligned}
$$
Hence for any $s_0 \leq s \leq S - \sigma$, by using the smallness condition $  \lambda^{- \delta} \gamma^{- 1} \leq \varepsilon(S) \ll 1$, the operator ${\rm Id} + {\mathcal F}$ is invertible by Neumann series and 
\begin{equation}\label{stima neumann per cal P1}
\|  ({\rm Id} + {\mathcal F})^{- 1} h \|_s^{\Lip(\gamma)} \lesssim_s \| h \|_s^{\Lip(\gamma)} + \| {\mathcal I} \|_{s + \sigma}^{\Lip(\gamma)} \| h \|_{s_0}^{\Lip(\gamma)}\,, \quad \forall s_0 \leq s \leq S - \sigma\,. 
\end{equation}
Then by the estimates \eqref{stima cal D inv lip gamma}, \eqref{stima neumann per cal P1}, the operator ${\mathcal P}_1^{- 1} = ({\rm Id} + {\mathcal F})^{- 1} {\mathcal D}^{- 1}$ satisfies the bound 
\begin{equation}\label{stima cal P inverso}
\| {\mathcal P}_1^{- 1} h \|_s^{\Lip(\gamma)} \lesssim_s \gamma^{- 1} \lambda^{- 1} \Big(\| h \|_{s + 2 \tau + 1}^{\Lip(\gamma)} + \| {\mathcal I} \|_{s + \sigma}^{\Lip(\gamma)} \| h \|_{s_0 + 2 \tau + 1}^{\Lip(\gamma)} \Big)\,, \quad \forall s_0 \leq s \leq S - \sigma\,. 
\end{equation}
{\sc Estimate of ${\mathcal P}^{- 1}$.}
By using that ${\mathcal P}^{- 1} = {\mathcal A} {\mathcal V} {\mathcal P}_1^{- 1} {\mathcal V}^{- 1} {\mathcal A}^{- 1}$, the  estimate \eqref{stima cal P inverso}, together with the estimates \eqref{stime alpha alpha cappuccio}, \eqref{stima Phi M reg ordini bassi} and Lemmata \ref{lem:changevar}, \ref{lemma azione tame pseudo-diff} (to estimate ${\mathcal A}, {\mathcal V}$) imply the claimed bound on ${\mathcal P}^{- 1}$. 
\end{proof}

\section{Inversion of the linearized operator ${\mathcal L}$}\label{inversione linearized totale}
We now invert the linear operator ${\mathcal L}_1$ in \eqref{forma cal L1 dopo decoupling}. 
The following proposition holds. 
\begin{prop}\label{invertibilita cal L1}
Let $\gamma \in (0, 1)$, $\tau > 0$ and $M, \delta$ as in \eqref{costanti finale pre NM}. Then there exists $\sigma \equiv \sigma(\tau) \gg 0$ large enough such that for any $S > s_0 + \sigma$, there exists $\varepsilon(S) \ll 1$ small enough such that if \eqref{ansatz} holds and $ \lambda^{- \delta} \gamma^{- 1} \leq \varepsilon(S) \ll 1$, then the following holds. For any $\omega \in {\mathcal G}_\lambda(\gamma, \tau)$, the operator ${\mathcal L}_1$ is invertible and its inverse ${\mathcal L}_1^{- 1}$ is a real operator and satisfies the estimate
\begin{equation}\label{stima cal L 1 inverse}
\| {\mathcal L}_1^{- 1} h \|_s^{\Lip(\gamma)} \lesssim_s \lambda^{-  \delta M}  \Big( \| h \|_{s + \sigma}^{\Lip(\gamma)} + \| {\mathcal I} \|_{s + \sigma}^{\Lip(\gamma)} \| h \|_{s_0 + \sigma}^{\Lip(\gamma)} \Big)\,, \quad \forall s_0 \leq s \leq S - \sigma\,.  
\end{equation}
\end{prop}
\begin{proof}
 To shorten notations we write $\| \cdot \|_s$ instead of $\| \cdot \|_s^{\Lip(\gamma)}$. By Proposition \ref{invertibilita cal P1}, one gets that for $\omega \in {\mathcal G}_\lambda(\gamma, \tau)$, the operator ${\mathcal P}$ is invertible and hence the equation \eqref{sistema cal L1 c} can be solved by setting
$$
h_2  = {\mathcal P}^{- 1}[g_2] - {\mathcal P}^{- 1} {\mathcal L}_{- N}^{(3)} ({\mathcal L}_1^{(1)})^{- 1} [g_1]\,. 
$$
By replacing the latter expression of $h_2$ in formula \eqref{sistema cal L1 b}, one obtains that 
$$
h_1 = ({\mathcal L}_1^{(1)})^{- 1} g_1 -  ({\mathcal L}_1^{(1)})^{- 1} {\mathcal L}_{-N}^{(2)} {\mathcal P}^{- 1}[g_2] + ({\mathcal L}_1^{(1)})^{- 1} {\mathcal L}_{- N}^{(2)} {\mathcal P}^{- 1} {\mathcal L}_{- N}^{(3)} ({\mathcal L}_1^{(1)})^{- 1} [g_1],
$$
and hence for any $\omega \in {\mathcal G}_\lambda(\gamma, \tau)$ 
$$
{\mathcal L}_1^{- 1} = \begin{pmatrix}
({\mathcal L}_1^{(1)})^{- 1}   + ({\mathcal L}_1^{(1)})^{- 1} {\mathcal L}_{- N}^{(2)} {\mathcal P}^{- 1} {\mathcal L}_{- N}^{(3)} ({\mathcal L}_1^{(1)})^{- 1} & -  ({\mathcal L}_1^{(1)})^{- 1} {\mathcal L}_{- N}^{(2)} {\mathcal P}^{- 1} \\
- {\mathcal P}^{- 1} {\mathcal L}_{- N}^{(3)} ({\mathcal L}_1^{(1)})^{- 1} &  {\mathcal P}^{- 1} 
\end{pmatrix}.
$$
By applying \eqref{stima L 1 1 loss} in Lemma \ref{lemma inversione cal L 1 1} to estimate $({\mathcal L}_1^{(1)})^{- 1}$, Proposition \ref{invertibilita cal P1} to estimate ${\mathcal P}^{- 1}$, the estimates \eqref{prop cal Q (1) 1234} on ${\mathcal L}_{- N}^{(2)}, {\mathcal L}_{- N}^{(3)}$, using also \eqref{ansatz}, one obtains that for any $s_0 \leq s \leq S - \sigma$
$$
\begin{aligned}
& \| ({\mathcal L}_1^{(1)})^{- 1} {\mathcal L}_{- N}^{(2)} {\mathcal P}^{- 1} {\mathcal L}_{- N}^{(3)} ({\mathcal L}_1^{(1)})^{- 1} h \|_s\,,\, \| ({\mathcal L}_1^{(1)})^{- 1} {\mathcal L}_{- N}^{(2)} {\mathcal P}^{- 1} h \|_s\,,\, \| {\mathcal P}^{- 1} {\mathcal L}_{- N}^{(3)} ({\mathcal L}_1^{(1)})^{- 1} h \|_s \\
 & \lesssim_s \lambda^{- 1} \gamma^{- 1} \Big( \| h \|_{s + \sigma} + \| {\mathcal I} \|_{s + \sigma}^{\Lip(\gamma)} \| h \|_{s_0 + \sigma}^{\Lip(\gamma)} \Big)\,. 
\end{aligned}
$$
for $\sigma \equiv \sigma(\tau) \gg 0$. By sing that $\lambda^{- 1} \gamma^{- 1} \leq \lambda^{- \delta M}$, since by \eqref{prima def a delta}, $0 < \delta < \frac{1}{M}$, one then obtains the claimed bound on ${\mathcal L}_1^{- 1}$. 
\end{proof}
 We can finally invert the linearized operator ${\mathcal L}$ in \eqref{block representation cal L}. The following Proposition holds. 
\begin{prop}\label{invertibilita cal L}
Let $\gamma \in (0, 1)$, $\tau > 0$ and $M, \delta$ as in \eqref{costanti finale pre NM}. Then there exists $\overline \sigma \equiv \overline \sigma (\tau) \gg 2 \tau + 1 \gg 0$ large enough such that for any $S > s_0 + \overline \sigma$, there exists $\varepsilon(S) \ll 1$ small enough such that if \eqref{ansatz} holds and $\lambda^{- \delta } \gamma^{- 1} \leq \varepsilon(S) \ll 1$, then the following holds. For any $\omega \in {\mathcal G}_\lambda(\gamma, \tau)$, the operator ${\mathcal L}$ is invertible and its inverse ${\mathcal L}^{- 1}$ satisfies for any $s_0 \leq s \leq S - \overline \sigma$, the estimate
\begin{equation}\label{inverso L}
\| {\mathcal L}^{- 1} h \|_s^{\Lip(\gamma)} \lesssim_s \lambda^{-  \delta M} \Big( \| h \|_{s + \overline \sigma}^{\Lip(\gamma)} + \| {\mathcal I} \|_{s + \overline \sigma}^{\Lip(\gamma)} \| h \|_{s_0 + \overline \sigma}^{\Lip(\gamma)} \Big) \,. 
\end{equation}
\end{prop}
\begin{proof}
First of all ${\mathcal L}^{-1}$ is real since the operator ${\mathcal L}$ is real.
Then, by Propositions \ref{proposizione L1}, \ref{invertibilita cal L1}, for any $\omega \in {\mathcal G}_\lambda(\gamma, \tau)$, one has that ${\mathcal L}^{- 1} = {\bf \Phi}_N {\mathcal L}_1^{- 1} {\bf \Phi}_N^{- 1}$. Then the claimed bound follows by the estimate \eqref{stima Phi}, Lemma \ref{lemma azione tame pseudo-diff} and by the estimate \eqref{stima cal L 1 inverse}, using also \eqref{ansatz}.
\end{proof}

\section{Construction of an approximate solution}\label{sezione soluzione approssimata}
In this section we construct an approximate solution of the nonlinear equation ${\mathcal F}(\Omega,J) = 0$ where the nonlinear operator ${\mathcal F}$ is defined in \eqref{mappa nonlineare iniziale}. This is the starting point to implement the nonlinear Nash-Moser scheme of Section \ref{sezione:NASH}. The following proposition holds. 
\begin{prop}[{\bf Approximate solutions}]\label{costruzione soluzioni approssimate}
Let $\lambda > 1$, $\gamma \in (0, 1)$, $\tau > 0$, $\delta \in (0, 1)$ and assume that $\lambda^{ - \frac{\delta}{3}} \gamma^{- 1} \leq 1$. Then, there exists ${\mathcal I}_{app}(\cdot; \omega) := (\Omega_{app}(\cdot; \omega), 0)$, $ \Omega_{app }\in {\mathcal C}^\infty(\T^2)$ with zero average defined for $\omega \in {\rm DC}(\gamma, \tau)$ such that $\Omega_{app} \neq 0$ and for any $s \geq s_0$
\begin{equation}\label{upper bound approximate solution}
\| {\mathcal I}_{app} \|_s^{\Lip(\gamma)} = \| \Omega_{app} \|_s^{\Lip(\gamma)} \lesssim_s \lambda^{ - \frac23\delta} \gamma^{- 1} \,, \quad \| {\mathcal F}({\mathcal I}_{app}) \|_s^{\Lip(\gamma)} \lesssim_s \lambda^{ - \frac{\delta}{3}} \gamma^{- 2} \,.
\end{equation}
Moreover, for any $s \geq 2$, one has the following lower bound on the approximate solution
\begin{equation}\label{upper bound approximate solution}
\inf_{\omega \in {\rm DC}(\gamma, \tau)}\| {\mathcal I}_{app}(\cdot; \omega) \|_s = \inf_{\omega \in {\rm DC}(\gamma, \tau)} \| \Omega_{app}(\cdot; \omega) \|_s \gtrsim_s \lambda^{ - \frac23 \delta}\,. 
\end{equation}
Moreover, there exists $\mathtt K(f, \mathbf b) > 0$ such that 
\begin{equation}\label{non degeneracy Omega app}
\| \mathbf b \cdot \nabla \Omega_{app} \|_{L^2} \geq \lambda^{- \frac23 \delta}\mathtt K(f, \mathbf b)\,. 
\end{equation}
\end{prop}
\begin{proof}
We determine $\Omega_{app}$ as the only solution with zero average of the following equation
\begin{equation}\label{sistema soluzione approssimata}
\begin{aligned}
& L_\lambda \Omega_{app} = \lambda^{1  -\frac23 \delta} F\,, \\
&  L_\lambda := \lambda \omega \cdot \nabla - \Delta = {\rm diag}_{k \in \Z^2 \setminus \{ 0 \}} \Lambda (k)\,, \\
& \Lambda(k) \equiv \Lambda(k; \omega) =    i \lambda \omega \cdot k + |k|^2 \,, \quad k \in \Z^2 \setminus \{ 0 \}\,. 
 \end{aligned}
\end{equation}
By using that $\omega \in {\rm DC}(\gamma, \tau)$, one has that 
$$
L_\lambda^{- 1} = {\rm diag}_{k \in \Z^2 \setminus \{ 0 \}} \frac{1}{i \lambda \omega \cdot k + |k|^2},
$$
and 
\begin{equation}\label{pippo 0 0}
|\Lambda(k)| = |i \lambda \omega \cdot k + |k|^2| \geq \lambda |\omega \cdot k| \geq \lambda \gamma |k|^{- \tau}, \quad \forall k \in \Z^2 \setminus \{ 0\}\,.
\end{equation}
Moreover if $\omega_1, \omega_2 \in {\rm DC}(\gamma, \tau)$, $k \in \Z^2 \setminus \{ 0 \}$, one estimates
\begin{equation}\label{pippo 0 1}
\begin{aligned}
|\Lambda(k; \omega_1)^{- 1} - \Lambda(k; \omega_2)^{- 1}| & \leq \dfrac{|\Lambda(k; \omega_1) - \Lambda(k; \omega_2)|}{|\Lambda(k; \omega_1)||\Lambda(k; \omega_2)|}  \\
&   \stackrel{\eqref{pippo 0 0}}{\lesssim} \lambda^{- 1} \gamma^{- 2} |k|^{2 \tau + 1} |\omega_1 - \omega_2|\,.
\end{aligned}
\end{equation}
The bounds \eqref{pippo 0 0}, \eqref{pippo 0 1} easily imply that 
$$
\| L_\lambda^{- 1} \|_{{\mathcal B}(H^{s + \tau}_0, H^s_0)}\,,\, \| L_\lambda^{- 1} \|_{{\mathcal B}(H^{s +2\tau + 1 }_0, H^s_0)}^{\Lip(\gamma)} \lesssim \lambda^{- 1} \gamma^{- 1}\,, \quad \forall s \geq 0.
$$
Therefore $\Omega_{app} := \lambda^{1   -\frac23 \delta} L_\lambda^{- 1 } F$ satisfies the bound 
\begin{equation}\label{bound Omega app lemma}
\| \Omega_{app} \|_s^{\Lip(\gamma)} \lesssim \lambda^{ -\frac23 \delta} \gamma^{- 1} \| F \|_{s + 2 \tau + 1}^{\Lip(\gamma)} \lesssim_s \lambda^{ -\frac23 \delta} \gamma^{- 1}, \quad \forall s \geq 0 \,. 
\end{equation}
Clearly, by the Assumption \eqref{assumption b f}, the forcing term $F \neq 0$ and hence one also has that $\Omega_{app} \neq 0$. 
By \eqref{sistema soluzione approssimata} and by recalling \eqref{mappa nonlineare iniziale}, one has that 
$$
\begin{aligned}
& {\mathcal F}(\Omega_{app}, 0) = \begin{pmatrix}
  \lambda^{\delta } U_{app}\cdot \nabla\Omega_{app}   \\
  - \mathbf b \cdot \nabla \Omega_{app}
\end{pmatrix} \,, \\
& U_{app} = {\mathcal U} \Omega_{app} = (- \Delta)^{- 1} \nabla^\bot \Omega_{app}\,.
\end{aligned}
$$
Clearly the estimates \eqref{bound Omega app lemma} also implies that 
\begin{equation}\label{bound U app B app}
\| U_{app} \|_s^{\Lip(\gamma)} \lesssim_s \lambda^{ -\frac23 \delta} \gamma^{- 1}\,, \quad \forall s \geq 0\,. 
\end{equation}
Hence, the estimates \eqref{bound Omega app lemma}, \eqref{bound U app B app}, together with the interpolation estimate \eqref{interpolazione bassa alta}, allow to deduce that 
\begin{equation}\label{bound pazzotici sol approx}
\begin{aligned}
&  \lambda^\delta \| U_{app}\cdot \nabla\Omega_{app} \|_s^{\Lip(\gamma)} \lesssim_s \lambda^{ -\frac{ \delta}{3}} \gamma^{- 2}\,, \qquad \| \mathbf b \cdot \nabla \Omega_{app} \|_s^{\Lip(\gamma)} \lesssim_s \lambda^{- \frac23 \delta} \gamma^{- 1}, \quad \forall s \geq 0 \,.   \\
\end{aligned}
\end{equation}
Since $\lambda^{- \frac{\delta}{3}} \gamma^{- 1} \leq 1$, $\lambda > 1$, $\gamma \in (0, 1)$, $\delta \in (0, 1)$ one has 
$$
\lambda^{- \frac23 \delta} \gamma^{- 1} \leq  \lambda^{ -\frac{ \delta}{3}} \gamma^{- 2}
$$
therefore the estimates \eqref{bound Omega app lemma}, \eqref{bound pazzotici sol approx} imply the claimed upper bound \eqref{upper bound approximate solution} on $\Omega_{app}, {\mathcal F}({\mathcal I}_{app})$. Now, we prove the lower bound on $\Omega_{app}$. One has that 
\begin{equation}\label{formula Fourier Omega app}
\Omega_{app}(x) = \lambda^{1   - \frac23 \delta} L_\lambda^{- 1} F = \lambda^{1  -\frac23 \delta} \sum_{k \in \Z^2 \setminus \{ 0 \}} \dfrac{\widehat F(k)}{i \lambda \omega \cdot k + |k|^2} e^{i x \cdot k}\,.
\end{equation}
Note that for $\lambda > 1$ large enough, $k \in \Z^2 \setminus \{ 0 \}$, one has the estimate 
\begin{equation}\label{incredibile ma vero}
|i \lambda \omega \cdot k + |k|^2| \lesssim \lambda |\omega \cdot k| + |k|^2 \lesssim \lambda |k| + |k|^2 \lesssim \lambda |k|^2\,, 
\end{equation}
therefore
$$
\begin{aligned}
\| \Omega_{app} \|_s^2 & = \lambda^{2(1  - \frac23 \delta)} \sum_{k \in \Z^2 \setminus \{ 0 \}} \dfrac{ | k |^{2 s}}{|i \lambda \omega \cdot k + |k|^2|^2} |\widehat F(k)|^2 \\
& \gtrsim \lambda^{2(1  - \frac23 \delta)} \lambda^{- 2} \sum_{k \in \Z^2 \setminus \{ 0 \}} |k|^{2(s - 2)} |\widehat F(k)|^2 \\
& \gtrsim \lambda^{- \frac43  \delta} \| F \|_{s - 2}^2 \gtrsim_s \lambda^{- \frac43 \delta} 
\end{aligned}
$$
which then implies the lower bound \eqref{upper bound approximate solution}. We now prove \eqref{non degeneracy Omega app}. By the assumption \eqref{assumption b f}, one has that there exists $\overline k \in \Z^2 \setminus \{ 0 \}$ such that 
\begin{equation}\label{assumption b f nel lemma}
\mathbf b \cdot \overline k \neq 0, \quad \widehat F(\overline k) \neq 0\,.
\end{equation}
Hence, by \eqref{formula Fourier Omega app}
$$
\begin{aligned}
\| \mathbf b \cdot \nabla \Omega_{app} \|_{L^2} & \gtrsim \lambda^{1 - \frac23 \delta} \Big( \sum_{k \in \Z^2 \setminus \{ 0 \}} \dfrac{|\mathbf b \cdot k|^2 |\widehat F(k)|^2}{|i \lambda \omega \cdot k + |k|^2 |^2} \Big)^{\frac12} \\
& \stackrel{\eqref{incredibile ma vero}}{\gtrsim} 
\lambda^{- \frac23 \delta} \Big( \sum_{k \in \Z^2 \setminus \{ 0 \}} \dfrac{|\mathbf b \cdot k|^2 |\widehat F(k)|^2}{ |k|^4} \Big)^{\frac12} \\
& \geq C \lambda^{- \frac23 \delta} \dfrac{|\mathbf b \cdot \overline k |^2|\widehat F(\overline k)|^2}{|\overline k|^2} \,.
\end{aligned}
$$
The claimed bound then follows by defining $\mathtt K(f, \mathbf b) := C \dfrac{|\mathbf b \cdot \overline k |^2|\widehat F(\overline k)|^2}{|\overline k|^2} > 0$, by \eqref{assumption b f nel lemma}. The proof of the Lemma is then concluded. 
\end{proof}

\section{The Nash Moser scheme}\label{sezione:NASH}
 In this section we construct solutions of the equation ${\mathcal F}({\mathcal I}) = 0$, ${\mathcal I} = (\Omega, J)$ in \eqref{mappa nonlineare iniziale} by means of a Nash Moser nonlinear iteration. 
We define the constants 
\begin{equation}\label{costanti nash moser}
\begin{aligned}
& N_0 > 0, \quad N_n := N_0^{\chi^n}, \quad n \geq 0, \quad N_{- 1} := 1\,, \quad \chi := 3/2, \\
& \tau > 0\,, \quad M := 6(2 \tau + 3)\,, \quad 0 < \delta < \dfrac{1}{(M + 1)(\tau + 2)}\,, \\
& \overline \mu := 3 \overline \sigma + 3\,, \quad    \mathtt a := {\rm max}\{3 \overline \mu + 1\,,\, 2(2 \overline \sigma + \tau + 1) \} \,, \quad \mathtt b := \overline \mu +  \mathtt a + 1
\end{aligned}
\end{equation}
where $\bar \sigma \gg 1$ is given in Proposition \ref{invertibilita cal L} and we recall \eqref{costanti finale pre NM}. We denote by $\Pi_n$ the orthogonal projector $\Pi_{N_n}$ (see \eqref{def:smoothing}) 
on the finite dimensional space
$$
{\mathcal H}_n := \big\{ {\mathcal I} \in L_0^2(\T^{2}, \R^2): \ \  
{\mathcal I} = { \Pi}_n {\mathcal I} = (\Pi_n \Omega, \Pi_n J)  \big\},
$$
and $\Pi_n^\bot := {\rm Id} - \Pi_n$. 
The projectors $ \Pi_n $, $ \Pi_n^\bot$ satisfy the 
usual smoothing properties (cf. Lemma \ref{lem:smoothing}), namely 
\begin{equation}\label{smoothing-u1}
\|\Pi_{n} {\mathcal I} \|_{s+a}^{\Lip(\gamma)} 
\leq N_{n}^{a} \| {\mathcal I} \|_{s}^{\Lip(\gamma)}, 
\quad \ 
\|\Pi_{n}^\bot {\mathcal I} \|_{s}^{\Lip(\gamma)} 
\leq N_n^{- a} \| {\mathcal I} \|_{s + a}^{\Lip(\gamma)}, 
\quad \ 
s,a \geq 0.
\end{equation}
Note that Proposition \ref{costruzione soluzioni approssimate} provides an approximate solution ${\mathcal I}_{app}$ of the functional equation ${\mathcal F}({\mathcal I}) = 0$ and such that  $\| {\mathcal F}({\mathcal I}_{app}) \|_s^{\Lip(\gamma)}$ is uniformly bounded with respect to the large parameter $\lambda \gg 1$ for any $s \geq 0$. The following lemma is a direct consequence of Proposition \ref{costruzione soluzioni approssimate}. 

\begin{lem}
	{\bf (Initialization of the Nash-Moser iteration).}
	Let $\lambda > 1$, $\gamma \in (0, 1)$, $\tau > 0$, $\delta \in (0, 1)$ and assume that $\lambda^{ - \frac{\delta}{3}} \gamma^{- 2} \leq 1$.For any $s \geq 0$, there exists a constant $C(s)>0$ such that
	\begin{equation}\label{stima soluzione approssimata nash moser}
	\begin{aligned}
	& 	\| {\mathcal F}({\mathcal I}_{app}) \|_s^{\Lip(\gamma)}\,,\, \| {\mathcal I}_{app} \|_s^{\Lip(\gamma)} \leq C(s)\,, \\
	& \inf_{\omega \in {\rm DC}(\gamma, \tau)} \| {\mathcal I}_{app}(\cdot; \omega) \|_s = \inf_{\omega \in {\rm DC}(\gamma, \tau)} \| \Omega_{app}(\cdot; \omega) \|_s  \geq C(s) \lambda^{ -\frac23 \delta}, \quad s \geq 2
		\end{aligned}
	\end{equation}
	uniformly w.r. to $\lambda$.
\end{lem}


\begin{prop}\label{iterazione-non-lineare} 
{\bf (Nash-Moser)} 
There exist $ \varepsilon \in (0, 1)$, $N_0  > 0$, $\overline \lambda> 0$, $C_* > 0$ such that if
\begin{equation}\label{nash moser smallness condition}  
\begin{aligned}
	& \lambda \geq \overline \lambda\,, \quad  \lambda^{- \frac{\delta}{3}} \gamma^{- 2}  \leq \varepsilon, 
	\end{aligned}
\end{equation}
then the following properties hold for all $n \geq 0$. 
\\[1mm]
\noindent $({\mathcal P}1)_{n}$
There exists a constant $C_0 > 0$ large enough and a sequence $({\mathcal I}_n)_{n \geq 0}$ 
with ${\mathcal I}_0 := {\mathcal I}_{app}$, ${\mathcal I}_n - {\mathcal I}_0 : {\mathcal G}_n \to {\mathcal H}_{n - 1} $, 
satisfying
\begin{equation} \label{stima.bassa.NM}
\| {\mathcal I}_{n} \|_{s_0 + \overline \sigma}^{\Lip(\gamma)} \leq C_0.
\end{equation}
If $n \geq 1$, 
the difference $ h_n := {\mathcal I}_n -  {\mathcal I}_{n - 1}$ satisfies 
\begin{equation} \label{hn}
\| h_n \|_{s_0 + \overline \sigma}^{\Lip(\gamma)} 
\lesssim N_{n-1}^{2 \overline \sigma } N_{n-2}^{-\mathtt a} \lambda^{-  \delta M} , \quad \text{for} \quad n \geq 1\,.  
\end{equation}
The sets $\{ {\mathcal G}_n\}_{n \geq 0}$ are defined as follows. If $n=0$ we define ${\mathcal G}_0 := {\rm DC}(\gamma, \tau)$ (recall \eqref{def:DC}).  
If $n \geq 1$, we define 
\begin{equation}\label{G-n+1}
{\mathcal G}_{n + 1} := {\mathcal G}_n \cap {\mathcal G}_\lambda(\gamma_n, \tau; {\mathcal I}_n)
\end{equation} 
where $ \gamma_{n}:=\gamma (1 + 2^{-n}) $ 
and the sets ${\mathcal G}_\lambda(\gamma_n, \tau; {\mathcal I}_n)$ 
are defined in \eqref{prime melnikov inversione}.
\\[1mm]
\noindent $({\mathcal P}2)_{n}$
On the set ${\mathcal G}_n$, one has the estimate
\begin{equation} \label{stima.F.to.0}
\| {\mathcal F}({\mathcal I}_n) \|_{s_{0}}^{\Lip(\gamma)} 
\leq C_* N_{n - 1 }^{- \mathtt a}\, ;
\end{equation}
\noindent $({\mathcal P}3)_{n}$
On the set ${\mathcal G}_n$, one has the estimate
\begin{equation}\label{stima.alta.NM}
	\| {\mathcal I}_n \|_{s_0 + \mathtt b}^{\Lip(\gamma)} + \| {\mathcal F}({\mathcal I}_n) \|_{s_{0} + \mathtt b}^{\Lip(\gamma)} 
	\leq C_*  N_{n - 1}^{\mathtt a}
\end{equation}
\end{prop}

\begin{proof}
To simplify notations, in this proof we write $\| \cdot \|_s$ instead of $\| \cdot \|_s^{\Lip(\gamma)}$. 


%
%

\smallskip

\noindent
{\sc Proof of $({\mathcal P}1, 2, 3)_0$}.
By \eqref{stima soluzione approssimata nash moser}, we have $\| {\mathcal I}_0\|_{s} = \| {\mathcal I}_{app} \|_{s} \leq C(s)$ and
$\| {\mathcal F} ({\mathcal I}_0 ) \|_s = \| {\mathcal F} ({\mathcal I}_{app} ) \|_s \leq C(s)$ for any $s \geq 0$. Then \eqref{stima.bassa.NM}, \eqref{stima.F.to.0} and \eqref{stima.alta.NM} hold taking $\tfrac12 C_0, C_*(s_0 + \overline \sigma) \geq C(s_0 + \overline \sigma)$. In particular, we have
\begin{equation}\label{stima.w0}
	\| {\mathcal I}_0 \|_{s_0+\bar\sigma} \leq \tfrac12 C_0 \leq C_0\,.
\end{equation}

\smallskip

\noindent
{\sc Assume that $({\mathcal P}1,2,3)_n$ hold for some $n \geq 0$, 
and prove $({\mathcal P}1,2,3)_{n+1}$.}
By $({\mathcal P}1)_n$, one has $\| {\mathcal I}_n\|_{s_0 + \bar \sigma} \leq C_0$, for some $C_0 \gg 0$ large enough.   
The assumption 
\eqref{nash moser smallness condition} implies 
the smallness condition 
$\lambda^{- \delta} \gamma^{- 1} \ll 1$ 
of Proposition \ref{invertibilita cal L}. Indeed $\gamma \in (0, 1)$, $\delta > \frac{\delta}{3}$, $\lambda \gg 1$ implies that $\lambda^{- \delta} \gamma^{- 1} \leq \lambda^{- \frac{\delta}{3}} \gamma^{- 2} \ll 1$. 
Then Proposition \ref{invertibilita cal L} applies 
to the linearized operator 
\begin{equation}\label{definizione cal Ln}
{\mathcal L}_n  : = D {\mathcal F}({\mathcal I}_n).
\end{equation}
This implies that, for any $\omega \in {\mathcal G}_{n + 1}$, the operator ${\mathcal L}_n(\omega) \equiv {\mathcal L}(\omega, {\mathcal I}_n(\omega))$ admits a right inverse ${\mathcal L}_n(\omega)^{- 1}$ satisfying the tame estimates, for any $s_0 \leq s \leq s_0 + {\mathtt b} + 1$ (choose $S := s_0 + \mathtt b + 1$),
\begin{equation}\label{stima Tn}
\| {\mathcal L}_n^{-1} h \|_s 
\lesssim_s \lambda^{-  \delta M}\big( \| h \|_{s + \overline \sigma}
+ \| {\mathcal I}_n \|_{s + \overline \sigma} \,  
\| h \|_{s_0 + \overline \sigma} \big)\,,
\end{equation}
using the bound $\gamma_n = \gamma(1 + 2^{- n}) \in [\gamma, 2\gamma]$. Note that since ${\mathcal I}_{app} = {\mathcal I}_0 \in {\mathcal C}^\infty(\T^2, \R^2)$ satisfies \eqref{stima soluzione approssimata nash moser}, since ${\mathcal I}_n - {\mathcal I}_0 \in {\mathcal H}_{n - 1}$, by \eqref{smoothing-u1} one obtains that for any $s , \mu \geq 0$, 
\begin{equation}\label{stima cal In smoothing nel nash moser}
\begin{aligned}
\| {\mathcal I}_n \|_{s + \mu} & \leq \| {\mathcal I}_{app} \|_{s + \mu} + \|  {\mathcal I}_n - {\mathcal I}_{app} \|_{s + \mu} \\
& \lesssim_{s, \mu} 1 + N_{n - 1}^\mu \|  {\mathcal I}_n - {\mathcal I}_{app}  \|_{s} \\
& \lesssim_{s, \mu} 1 + N_{n - 1}^\mu (\| {\mathcal I}_n \|_s + \| {\mathcal I}_{app} \|_s) \\
& \lesssim_{s, \mu} N_{n}^\mu (1 + \| {\mathcal I}_n \|_s )\,. 
\end{aligned}
\end{equation}
By \eqref{stima Tn}, \eqref{stima.bassa.NM} and \eqref{stima cal In smoothing nel nash moser}, for $s= s_0$ and $s = s_0 + 1$, one gets that  
one has 
\begin{equation}\label{stima Tn norma bassa} 
\begin{aligned}
\| {\mathcal L}_n^{- 1} h \|_{s_0}
& \lesssim \lambda^{-  \delta M} \| h \|_{s_0 + \overline \sigma}\,, \\
\| {\mathcal L}_n^{- 1} h \|_{s_0 + 1}
& \lesssim \lambda^{-  \delta M} \Big(  \| h \|_{s_0 + \overline \sigma + 1} + \| {\mathcal I}_n \|_{s_0 + \overline \sigma + 1} \| h \|_{s_0 + \overline \sigma}\Big)  \\
& \lesssim N_n \lambda^{-  \delta M} \| h \|_{s_0 + \overline \sigma + 1}\,. 
\end{aligned}
\end{equation}
We define the next approximation as 
\begin{equation}\label{soluzioni approssimate}
{\mathcal I}_{n + 1} := {\mathcal I}_n + h_{n + 1}, \quad 
{h}_{n + 1} :=  - \Pi_n {\mathcal L}_n^{-1}  \Pi_n {\mathcal F}({\mathcal I}_n) 
\in {\mathcal H}_{n},
\end{equation}
defined for any $\omega \in {\mathcal G}_{n + 1}$, and the remainder  
$$
Q_n := {\mathcal F}({\mathcal I}_{n+1}) - {\mathcal F}({\mathcal I}_n) - {\mathcal L}_n h_{n+1}\,.
$$
We now estimate $h_{n + 1}$. By the estimates \eqref{stima Tn}, \eqref{stima Tn norma bassa}, \eqref{stima.bassa.NM}, \eqref{smoothing-u1}, \eqref{costanti nash moser}, \eqref{nash moser smallness condition}, \eqref{stima cal In smoothing nel nash moser} and  (using also that $\gamma^{- 1} = N_0 \leq N_n$), we get that 
\begin{equation}\label{dawn1}
	\begin{aligned}
		\| h_{n + 1} \|_{s_0} & \lesssim \lambda^{-  \delta M} \| \Pi_n {\mathcal F}({\mathcal I}_n) \|_{s_0 + \overline \sigma} \lesssim N_n^{\bar \sigma } \lambda^{-  \delta M}  \| {\mathcal F}({\mathcal I}_n) \|_{s_0} \\
		\| h_{n + 1} \|_{s_0 + \bar \sigma} & \lesssim N_n^{\bar \sigma} \| {\mathcal L}_n^{-1}  \Pi_n {\mathcal F}({\mathcal I}_n) \|_{s_0} \lesssim N_n^{2 \bar \sigma } \lambda^{-  \delta M} \| {\mathcal F}({\mathcal I}_n) \|_{s_0}\,,			
		\end{aligned}		
		\end{equation}
and
			\begin{equation}\label{stima h n + 1 noma alta}
			\begin{aligned}
		\| h_{n + 1} \|_{s_0 + \mathtt b} & \lesssim \lambda^{- \delta M} \Big( \| \Pi_n {\mathcal F}({\mathcal I}_n) \|_{s_0 + \mathtt b + \overline \sigma}  + \| {\mathcal I}_n \|_{s_0 + \mathtt b + \overline \sigma} \| \Pi_n {\mathcal F}({\mathcal I}_n) \|_{s_0 +  \overline \sigma}  \Big) \\
		& \lesssim \lambda^{-  \delta M} N_n^{2 \bar \sigma } \Big( \|  {\mathcal F}({\mathcal I}_n) \|_{s_0 + \mathtt b - 2}  + \| {\mathcal I}_n \|_{s_0 + \mathtt b } \|  {\mathcal F}({\mathcal I}_n) \|_{s_0}  \Big) \\
		& \lesssim  \lambda^{-  \delta M}  N_n^{2 \bar \sigma } \Big( \|  {\mathcal F}({\mathcal I}_n) \|_{s_0 + \mathtt b } + C_* N_{n-1}^{- {\mathtt a}} \| {\mathcal I}_n \|_{s_0+ {\mathtt b}} \Big)  \\
		& \lesssim  N_n^{2 \bar \sigma } \Big( \| {\mathcal F}({\mathcal I}_n) \|_{s_0 + \mathtt b } + \| {\mathcal I}_n\|_{s_0 + \mathtt b} \Big)\,. 
	\end{aligned}
\end{equation}
Moreover, \eqref{soluzioni approssimate} and the estimate \eqref{stima h n + 1 noma alta} imply that 
\begin{equation}\label{w n + 1 w n inductive}
\| {\mathcal I}_{n + 1} \|_{s_0 + \mathtt b} \leq \| {\mathcal I}_n \|_{s_0 + \mathtt b} + \| h_{n + 1} \|_{s_0 + \mathtt b} \lesssim N_n^{2 \bar \sigma }  \big( \| {\mathcal F}({\mathcal I}_n) \|_{s_0 + \mathtt b } + \| {\mathcal I}_n\|_{s_0 + \mathtt b} \big)\,. 
\end{equation}
Next, we estimate $\|{\mathcal F}({\mathcal I}_{n + 1}) \|_{s_0}$ and $\| {\mathcal F}({\mathcal I}_{n + 1}) \|_{s_0 + \mathtt b}$. By the definition of $h_{n+1}$ in \eqref{soluzioni approssimate}, we obtain that, for any $\omega \in {\mathcal G}_{n + 1}$, 
\begin{equation}\label{forma cal F w n + 1}
\begin{aligned}
{\mathcal F}({\mathcal I}_{n + 1}) & = {\mathcal F}({\mathcal I}_n) + {\mathcal L}_n h_{n + 1} + Q_n \\
& = {\mathcal F}({\mathcal I}_n) - {\mathcal L}_{n} \Pi_n {\mathcal L}_{n}^{-1} \Pi_{n} {\mathcal F}({\mathcal I}_n) + Q_n \\
& = {\mathcal F}({\mathcal I}_n) - \Pi_n {\mathcal L}_{n}  {\mathcal L}_{n}^{-1} \Pi_{n} {\mathcal F}({\mathcal I}_n) + [{\mathcal L}_{n}\,,\, \Pi_n] {\mathcal L}_{n}^{-1} \Pi_{n} {\mathcal F}({\mathcal I}_n) + Q_n  \\
& = \Pi_n^\bot {\mathcal F}({\mathcal I}_n) + [{\mathcal L}_{n}\,,\, \Pi_n] {\mathcal L}_{n}^{-1} \Pi_{n} {\mathcal F}({\mathcal I}_n) + Q_n \\
& \stackrel{\Pi_n = {\rm Id} - \Pi_n^\bot}{=} \Pi_n^\bot {\mathcal F}({\mathcal I}_n) + [\Pi_n^\bot\,,\, {\mathcal L}_{n}] {\mathcal L}_{n}^{-1} \Pi_{n} {\mathcal F}({\mathcal I}_n) + Q_n
\end{aligned}
\end{equation}
We estimate separately the three terms in \eqref{forma cal F w n + 1}. 
By \eqref{smoothing-u1}, we have 
\begin{equation}\label{stima norma bassa 1}
	\begin{aligned}
		& \| \Pi_n^\bot {\mathcal F}({\mathcal I}_n) \|_{s_0}  \leq  N_n^{- \mathtt b} \| {\mathcal F}({\mathcal I}_n) \|_{s_0 + \mathtt b}\,, \quad  \| \Pi_n^\bot {\mathcal F}({\mathcal I}_n) \|_{s_0 + \mathtt b}  \leq \| {\mathcal F}({\mathcal I}_n) \|_{s_0 + \mathtt b}\,. 
	\end{aligned}
\end{equation}
By \eqref{stima.bassa.NM}, \eqref{stima.F.to.0}, Lemma \ref{prop elementari mappa cal F}-$(ii)$, \eqref{smoothing-u1}, \eqref{stima Tn}, \eqref{nash moser smallness condition}, \eqref{stima cal In smoothing nel nash moser}, we have
\begin{equation}\label{stima norma bassa 2}
	\begin{aligned}
		 \|  [\Pi_n^\bot\,,\, {\mathcal L}_n] {\mathcal L}_n^{- 1} \Pi_n {\mathcal F}({\mathcal I}_n) \|_{s_0}   &\lesssim \lambda^\delta N_n^{1 - \mathtt b} \Big(  \|  {\mathcal L}_n^{- 1} \Pi_n {\mathcal F}({\mathcal I}_n) \|_{s_0 + \mathtt b} + \| {\mathcal I}_n \|_{s_0 + \mathtt b }  \|  {\mathcal L}_n^{- 1} \Pi_n {\mathcal F}({\mathcal I}_n) \|_{s_0 + 1}  \Big)  \\
		& \lesssim \lambda^\delta N_n^{1 - \mathtt b} \lambda^{-  \delta M} \Big( \| \Pi_n {\mathcal F}({\mathcal I}_n) \|_{s_0 + \mathtt b + \bar \sigma}  \\
		& \quad + \| {\mathcal I}_n \|_{s_0 + \mathtt b + \bar \sigma} \| \Pi_n {\mathcal F}({\mathcal I}_n) \|_{s_0 + \bar \sigma} + N_n \| \Pi_n {\mathcal F}({\mathcal I}_n) \|_{s_0 + \bar \sigma + 1}  \Big) \\
		& \lesssim  N_n^{2 - \mathtt b} \lambda^{- (  M - 1) \delta}  \Big( \| \Pi_n {\mathcal F}({\mathcal I}_n) \|_{s_0 + \mathtt b + \bar \sigma}  \\
		& \quad + \| {\mathcal I}_n \|_{s_0 + \mathtt b + \bar \sigma} \| \Pi_n {\mathcal F}({\mathcal I}_n) \|_{s_0 + \bar \sigma} +  \| \Pi_n {\mathcal F}({\mathcal I}_n) \|_{s_0 + \bar \sigma + 1}  \Big) \\
		& \lesssim  N_n^{ 2 \bar \sigma + 3 - \mathtt b}\lambda^{- ( M - 1) \delta} \Big(  \|  {\mathcal F}({\mathcal I}_n) \|_{s_0 + \mathtt b} + C_* N_{n - 1}^{- \mathtt a} \| {\mathcal I}_n \|_{s_0 + \mathtt b}  \Big) \\
		& \lesssim N_n^{ 2 \bar \sigma + 3 - \mathtt b} \Big(  \|  {\mathcal F}({\mathcal I}_n) \|_{s_0 + \mathtt b} +  \| {\mathcal I}_n \|_{s_0 + \mathtt b}  \Big) \quad \text{and} \\
		\|  [\Pi_n^\bot\,,\, {\mathcal L}_n] {\mathcal L}_n^{- 1} \Pi_n {\mathcal F}({\mathcal I}_n) \|_{s_0 + \mathtt b}   &  = \|  [ {\mathcal L}_n\,,\, \Pi_n] {\mathcal L}_n^{- 1} \Pi_n {\mathcal F}({\mathcal I}_n) \|_{s_0 + \mathtt b} \\
		&  \lesssim \lambda^\delta N_n  \Big(  \|  {\mathcal L}_n^{- 1} \Pi_n {\mathcal F}({\mathcal I}_n) \|_{s_0 + \mathtt b }   + \| {\mathcal I}_n \|_{s_0 + \mathtt b } \|  {\mathcal L}_n^{- 1} \Pi_n {\mathcal F}({\mathcal I}_n) \|_{s_0 + 1} \Big) \\
		& \lesssim  N_n \lambda^{- ( M - 1)\delta}  \Big(  \|  \Pi_n {\mathcal F}({\mathcal I}_n) \|_{s_0 + \mathtt b + \bar \sigma}   + \| {\mathcal I}_n \|_{s_0 + \mathtt b  + \bar \sigma} \| \Pi_n {\mathcal F}({\mathcal I}_n) \|_{s_0 + \bar \sigma} \Big) \\
		& \quad +  N_n\lambda^{- (M - 1)\delta}  \| {\mathcal I}_n \|_{s_0 + \mathtt b}  \| \Pi_n {\mathcal F}({\mathcal I}_n) \|_{s_0 + \bar \sigma + 1}  \\
		& \stackrel{M > 1\,,\, \lambda \gg 1}{\lesssim}  N_n^{2 \bar \sigma + 2}  \Big(  \|  {\mathcal F}({\mathcal I}_n) \|_{s_0 + \mathtt b}   + \| {\mathcal I}_n \|_{s_0 + \mathtt b} \|  {\mathcal F}({\mathcal I}_n) \|_{s_0 } \Big) \\
		& \lesssim  N_n^{2 \bar \sigma + 2}   \Big(  \|  {\mathcal F}({\mathcal I}_n) \|_{s_0 + \mathtt b}   + C_* N_{n - 1}^{- \mathtt a} \| {\mathcal I}_n \|_{s_0 + \mathtt b}    \Big) \\
		& \lesssim N_n^{2 \bar \sigma + 2}   \Big(  \|  {\mathcal F}({\mathcal I}_n) \|_{s_0 + \mathtt b}   +  \| {\mathcal I}_n \|_{s_0 + \mathtt b}   \Big)\,. 
			\end{aligned}
\end{equation}
By Lemma \ref{prop elementari mappa cal F}-$(iii)$, \eqref{smoothing-u1}, \eqref{stima.bassa.NM}, \eqref{dawn1}, \eqref{stima h n + 1 noma alta} and \eqref{nash moser smallness condition}, we have
\begin{equation}\label{stima norma bassa 3}
	\begin{aligned}
		\| Q_n \|_{s_0} & \lesssim \lambda^{\delta} N_n^2 \| h_{n + 1} \|_{s_0}^2  \lesssim \lambda^{\delta}  N_n^{2 \bar \sigma + 2} \lambda^{- 2 \delta M}  \| {\mathcal F}({\mathcal I}_n) \|_{s_0}^2 \\
		& \lesssim   N_n^{2 \bar \sigma + 2} \lambda^{- (2 M - 1) \delta }   \| {\mathcal F}({\mathcal I}_n) \|_{s_0}^2 \quad \text{and} \\
	\| Q_n \|_{s_0 + \mathtt b} & \lesssim  \lambda^\delta \| h_{n + 1} \|_{s_0 + \mathtt b + 1} \| h_{n + 1} \|_{s_0 + 1} \lesssim N_n^2  \lambda^\delta  \| h_{n + 1} \|_{s_0 + \mathtt b} \| h_{n + 1} \|_{s_0} \\
	& \lesssim   N_n^{2 \bar \sigma  + 2} \lambda^\delta \Big( \| {\mathcal F}({\mathcal I}_n) \|_{s_0 + \mathtt b } + \| {\mathcal I}_n\|_{s_0 + \mathtt b} \Big) N_n^{\bar \sigma} \lambda^{-  \delta M}  \| {\mathcal F}({\mathcal I}_n) \|_{s_0} \\
	& \lesssim N_n^{3 \bar \sigma + 2} \lambda^{- ( M - 1) \delta} \Big( \| {\mathcal F}({\mathcal I}_n) \|_{s_0 + \mathtt b } + \| {\mathcal I}_n\|_{s_0 + \mathtt b} \Big) N_{n - 1}^{- \mathtt a} \\
	& \lesssim N_n^{3 \bar \sigma + 2} \Big( \| {\mathcal F}({\mathcal I}_n) \|_{s_0 + \mathtt b } + \| {\mathcal I}_n\|_{s_0 + \mathtt b} \Big)
	\end{aligned}
\end{equation}
Therefore, by collecting \eqref{w n + 1 w n inductive}-\eqref{stima norma bassa 3}, we have proved the following inductive inequalities for any $n \geq 0$ and on the set ${\mathcal G}_{n + 1}$: 
\begin{equation}\label{stime induttive in proof NM}
\begin{aligned}
\| {\mathcal F}({\mathcal I}_{n + 1}) \|_{s_0 + \mathtt b} + \| {\mathcal I}_{n + 1} \|_{s_0 + \mathtt b} & \lesssim N_n^{\bar \mu}  \big(\| {\mathcal F}({\mathcal I}_n) \|_{s_0 + \mathtt b} + \| {\mathcal I}_n\|_{s_0 + \mathtt b} \big)\,, \\
\| {\mathcal F}({\mathcal I}_{n + 1}) \|_{s_0} & \lesssim N_n^{\bar \mu - \mathtt b}  \big( \| {\mathcal F}({\mathcal I}_n) \|_{s_0 + \mathtt b} + \| {\mathcal I}_n \|_{s_0 + \mathtt b} \big)  \\
& \qquad +    N_n^{\bar \mu} \lambda^{- (2 M - 1) \delta}  \| {\mathcal F}({\mathcal I}_n) \|_{s_0}^2\,, \\
\bar \mu & := 3 \bar \sigma + 3\,. 
\end{aligned}
\end{equation}

\medskip

\noindent
{\bf Proof of $({\mathcal P}1)_{n + 1}$.} By \eqref{dawn1}, \eqref{stima.F.to.0}, \eqref{nash moser smallness condition} and \eqref{costanti nash moser}, we obtain
\begin{equation}\label{hn.n+1}
	\| h_{n+1} \|_{s_0+\bar\sigma} \lesssim N_{n}^{2\bar\sigma } N_{n-1}^{-\mathtt a}  \lambda^{-  \delta M} 
\end{equation}
which proves \eqref{hn} at the step $n+1$. By \eqref{stima.w0}, by the definition of the constants in \eqref{costanti nash moser} and by the smallness condition in \eqref{nash moser smallness condition}, the estimate \eqref{stima.bassa.NM} at the step $n + 1$ follows since 
\begin{equation}
	\begin{aligned}
		\| {\mathcal I}_{n + 1} \|_{s_0 + \bar \sigma}  & \leq  \| {\mathcal I}_0 \|_{s_0 + \bar \sigma} + \sum_{k = 1}^{n + 1} \| h_k \|_{s_0 + \bar \sigma}  \\
		&  \leq \tfrac12 C_0 + C_*\sum_{k = 1}^\infty N_{k-1}^{2 \overline \sigma} N_{k -2}^{-\mathtt a} \lambda^{-  \delta M}  \leq \frac12 C_0 + C_1 \lambda^{- \delta M} \leq C_0
	\end{aligned}
\end{equation}
for $\lambda \gg 1$ large enough. 

\noindent
{\bf Proof of $({\mathcal P}2)_{n + 1}, ({\mathcal P}3)_{n + 1}$.} By the induction estimate \eqref{stime induttive in proof NM} and using the induction hypothesis on $({\mathcal P}2)_{n }, ({\mathcal P}3)_{n }$, we obtain that 
$$
\begin{aligned}
\| {\mathcal F}({\mathcal I}_{n + 1}) \|_{s_0 + \mathtt b} + \| {\mathcal I}_{n + 1} \|_{s_0 + \mathtt b} & \lesssim  N_n^{\bar \mu}  \big( \| {\mathcal F}({\mathcal I}_{n }) \|_{s_0 + \mathtt b} + \| {\mathcal I}_{n} \|_{s_0 + \mathtt b} \big) \leq CC_* N_n^{\bar \mu} N_{n - 1}^{\mathtt a}   \leq C_* N_n^{\mathtt a}
\end{aligned}
$$
by the choice of the constants in \eqref{costanti nash moser} and taking $N_0 \gg 1 $ large enough. This latter chain of inequalities proves $({\mathcal P}3)_{n + 1}$. Moreover, by \eqref{stime induttive in proof NM} and by the induction estimates \eqref{stima.F.to.0}, \eqref{stima.alta.NM}, we have
$$
\begin{aligned}
\| {\mathcal F}({\mathcal I}_{n + 1}) \|_{s_0} & \lesssim N_n^{\bar \mu - \mathtt b} \big( \| {\mathcal F}({\mathcal I}_n) \|_{s_0 + \mathtt b} + \| {\mathcal I}_n \|_{s_0 + \mathtt b} \big) +   N_n^{\bar \mu} \lambda^{- ( 2 M - 1) \delta}  \| {\mathcal F}({\mathcal I}_n) \|_{s_0}^2 \\
& \leq C C_*  N_n^{\bar \mu - \mathtt b} N_{n - 1}^{\mathtt a} + C C_*^2  N_n^{\bar \mu}  N_{n - 1}^{- 2 \mathtt a} \lambda^{- (2 M - 1) \delta } \\
& \leq C_* N_n^{- \mathtt a}
\end{aligned}
$$
provided 
$$
N_n^{\mathtt b - \bar \mu - \mathtt a} \geq 2C\,, \quad \lambda^{- (2 M - 1) \delta} \leq \dfrac{N_{n - 1}^{2 \mathtt a} N_n^{- \mathtt a - \bar \mu} }{2 C C_*}, \quad \forall n \geq 0\,. 
$$
The latter condition is verified by the smallness condition \eqref{nash moser smallness condition}, for $N_0, \lambda \gg 1$ large enough and by the choice of the constants in \eqref{costanti nash moser}. The proof of the claimed statement is then concluded. 
\end{proof}
\section{Measure estimates}\label{sezione stime di misura}
In this section we estimate the measure of the resonant set ${\mathcal O} \setminus {\mathcal G}_\infty$ where the set ${\mathcal G}_\infty$ is defined as 
 \begin{equation}\label{def cal G infty}
{\mathcal G}_\infty := \cap_{n \geq 0} {\mathcal G}_n
\end{equation}
where the sets ${\mathcal G}_n$ are given in Proposition \ref{iterazione-non-lineare}. We recall that we denote by ${\mathcal M}^{(d)}$ the Lebesgue measure on $\R^d$. The main result of this section is the following
 \begin{prop} \label{prop measure estimate finale}
Let 
 \begin{equation}\label{choice tau}
 \tau := 2\,.
 \end{equation}
Under the same assumption of Proposition \ref{iterazione-non-lineare}, we have that ${\mathcal M}^{(2)}( \mathcal{O} \setminus {\mathcal G}_\infty) \lesssim \gamma$.   
\end{prop}

The rest of this section is devoted to the proof 
of Proposition \ref{prop measure estimate finale}. 
By the definition \eqref{def cal G infty}, one has 
\begin{equation}\label{prima inclusione stime misura}
\mathcal{O} \setminus {\mathcal G}_\infty 
\subseteq (\mathcal{O} \setminus {\mathcal G}_0) \cup \bigcup_{n \geq 0} ({\mathcal G}_n \setminus {\mathcal G}_{n + 1})\,.
\end{equation}
Hence, it is enough to estimate the measure of $\mathcal{O} \setminus {\mathcal G}_0$ and ${\mathcal G}_n \setminus {\mathcal G}_{n + 1}$ for any $n \geq 0$. 
By \eqref{G-n+1}, \eqref{prime melnikov inversione}, one gets that ${\mathcal G}_0 ={\rm DC}(\gamma, 2)$, therefore it is standard the fact that 
\begin{equation}\label{measure diophantine}
{\mathcal M}^{(2)}(\mathcal{O} \setminus {\mathcal G}_0) \lesssim \gamma\,. 
\end{equation}
Moreover for $n \geq 0$, one has that ${\mathcal G}_n \setminus {\mathcal G}_{n + 1}$ can be written as 
\begin{equation}\label{formula G n - G n + 1}
\begin{aligned}
{\mathcal G}_n \setminus {\mathcal G}_{n + 1}  & = \bigcup_{k \in \Z^2 \setminus \{ 0 \}} {\mathcal R}_k({\mathcal I}_n)\,, \\
{\mathcal R}_k({\mathcal I}_n) & := \Big\{ \omega \in {\mathcal G}_n : |i \lambda \omega \cdot k + z(k; \omega, {\mathcal I}_n(\omega))| < \frac{\lambda \gamma_n}{|k|^2} \Big\}\,, \quad k \in \Z^2 \setminus \{ 0 \}\,. 
\end{aligned}
\end{equation}
In the next lemma, we estimate the measure of the resonant sets ${\mathcal R}_{k}({\mathcal I}_n)$, $k \in \Z^2 \setminus \{ 0 \}$ defined in \eqref{formula G n - G n + 1}.

\begin{lem}\label{stima misura risonanti sec melnikov}
For any $n \geq 0$, one has that 
$${\mathcal M}^{(2)}({\mathcal R}_{k}({\mathcal I}_n)) 
\lesssim \frac{\gamma}{ |k|^{3}}\,. 
$$ 
\end{lem}

\begin{proof}
For $k \in \Z^2 \setminus \{ 0 \}$, we write $z(k;\omega) \equiv z(k; \omega, {\mathcal I}_n(\omega))$, and we set 
$$
\phi(\omega) := i \lambda \omega \cdot k + z(k; \omega)\,. 
$$ 
Since $k \neq 0$, we write 
$$
\omega = \frac{k }{|k|} s + v, \quad v \cdot k = 0,
$$
and we estimate the measure of the set 
$$
\begin{aligned}
{\mathcal Q} & := \Big\{ s : |\psi(s)| < \frac{ \gamma_n\, \lambda}{ |k|^2}, \quad \frac{k}{|k|} s + v \in {\mathcal G}_n  \Big\}, \\
\psi(s) & := \phi\Big( \frac{k }{|k|} s + v\Big) = i \lambda |k| s + z(k; s)\,, \\
z(k; s) & \equiv z\Big(k; \frac{k }{|k|} s + v \Big)\,. 
\end{aligned}
$$
By the estimate \eqref{stima autovalori cal P 1 dopo descent} of Proposition \ref{proposizione regolarizzazione ordini bassi}, one has that
$$
|z(k; s_1) - z(k; s_2)| \lesssim \gamma^{- 1}\lambda^{\delta M}  |s_1 - s_2|, \quad \forall k \in \Z^2 \setminus \{ 0 \} ,
$$
and therefore 
$$
\begin{aligned}
| \psi(s_1) - \psi(s_2)| & \geq \Big( \lambda |k |- C \lambda^{\delta M} \gamma^{- 1}  \Big) |s_1 - s_2|\geq \frac{\lambda |k|}{2} |s_1 - s_2|,
\end{aligned}
$$
since, by \eqref{costanti nash moser}, \eqref{nash moser smallness condition}, one has that $\lambda^{\delta M - 1} \gamma^{- 1} \leq \lambda^{- \frac{\delta}{3}} \gamma^{- 2} \ll 1$. This implies that the measure of the set ${\mathcal Q}$ satisfies 
$$
{\mathcal M}^{(1)}({\mathcal Q}) \lesssim \frac{ \lambda \gamma_n}{\lambda |k|^{3}} \lesssim \frac{\gamma}{ |k|^{3}}
$$
and by the Fubini theorem one also gets that the claimed bound on the measure ${\mathcal M}^{(2)}({\mathcal R}_{k}({\mathcal I}_n))$.
\end{proof}


\begin{lem}\label{lemma triviale modi alti}
Let $n \geq 1$, $k\in \Z^2 \setminus \{ 0 \}$, $|k| \leq N_{n - 1}$. Then ${\mathcal R}_{k}({\mathcal I}_n) = \emptyset$.
\end{lem}
\begin{proof}
Let $\omega \in {\mathcal G}_n$. We shall prove that 
\begin{equation}\label{lower bound R ljj vuoto}
\begin{aligned}
& |i \lambda \omega \cdot k + z(k; \omega, {\mathcal I}_n(\omega))| \geq \frac{ \lambda \gamma_n}{|k|^2}\,, \quad  k \in \Z^2 \setminus \{ 0 \}\,, \quad |k| \leq N_{n - 1},
\end{aligned}
\end{equation}
implying that ${\mathcal R}_{k}({\mathcal I}_n) = \emptyset$. We shall prove the bound \eqref{lower bound R ljj vuoto}. By applying the estimate \eqref{stime delta 12 cal L (3)} in Proposition \ref{proposizione regolarizzazione ordini bassi}, one has
$$
\begin{aligned}
|i \lambda \omega \cdot k + z(k; \omega, {\mathcal I}_n(\omega))| & \geq |i \lambda \omega \cdot k + z(k; \omega, {\mathcal I}_{n - 1}(\omega))| - |z(k; \omega, {\mathcal I}_n(\omega)) - z(k; \omega, {\mathcal I}_{n - 1}(\omega))|\\
& \geq \frac{ \lambda \gamma_{n - 1}}{|k|^2} - C \lambda^{\delta M} \| {\mathcal I}_n - {\mathcal I}_{n - 1} \|_{s_0 + \overline \sigma} \\
& \stackrel{\eqref{hn}}{\geq} \frac{ \lambda \gamma_{n - 1}}{|k|^2}  - C  N_{n-1}^{2 \overline \sigma} N_{n-2}^{-\mathtt a}  \geq \frac{\lambda \gamma_n}{|k|^2},
\end{aligned}
$$
provided that 
$$
\dfrac{C |k|^2 N_{n - 1}^{2 \overline \sigma} }{N_{n - 2}^{\mathtt a} \lambda (\gamma_{n - 1} - \gamma_n)} \leq 1, \quad \forall 0 <|k| \leq N_{n - 1}\,.
$$
Since $|k| \leq N_{n - 1}$ and $\gamma_{n - 1} - \gamma_n = \frac{\gamma}{2^n}$ and $2^n \lesssim N_{n - 1}$, the latter condition reads
$$
C' N_{n - 1}^{2 \overline \sigma + 3} N_{n - 2}^{- \mathtt a}  \lambda^{- 1 } \gamma^{- 1} \leq 1.
$$
for some constant $C' \gg 0$ large enough. This latter condition is satisfied by the smallness condition \eqref{nash moser smallness condition} and since $\mathtt a \geq 2(2 \overline \sigma + \tau + 1) > \chi(2 \overline \sigma + \tau + 1) = \chi(2 \overline \sigma + 3)$ (recall \eqref{costanti nash moser} and that $\tau = 2$). The claimed bound \eqref{lower bound R ljj vuoto} has then been proved. 
\end{proof}
\begin{prop}\label{Gn Omegan Lambdan}
 One has ${\mathcal M}^{(2)}({\mathcal G}_0 \setminus {\mathcal G}_1)\lesssim \gamma$ and for any $n \geq 1$, ${\mathcal M}^{(2)}( {\mathcal G}_n \setminus {\mathcal G}_{n + 1}) \lesssim \gamma N_{n - 1}^{- 3}$. 
\end{prop}
\begin{proof}
By \eqref{formula G n - G n + 1}, Lemma \ref{stima misura risonanti sec melnikov} and by recalling \eqref{choice tau}, one gets 
$$
{\mathcal M}^{(2)}( {\mathcal G}_0 \setminus {\mathcal G}_1 ) \leq \sum_{k\in \Z^2 \setminus \{ 0 \}} {\mathcal M}^{(2)}( {\mathcal R}_k({\mathcal I}_0)) \lesssim \gamma \sum_{k \in \Z^2 \setminus \{ 0 \}} \frac{1}{|k|^{3}} \lesssim \gamma\,.
$$
Moreover, \eqref{formula G n - G n + 1} and Lemma \ref{lemma triviale modi alti} imply that for any $n \geq 1$
$$
{\mathcal G}_n \setminus {\mathcal G}_{n + 1} \subseteq \bigcup_{|k| \geq N_{n - 1}} {\mathcal R}_k({\mathcal I}_n)
$$
and hence, using again Lemma \ref{stima misura risonanti sec melnikov} and \eqref{choice tau}, one gets
$$
{\mathcal M}^{(2)}( {\mathcal G}_n \setminus {\mathcal G}_{n +1}) \leq \sum_{|k| \geq N_{n - 1}} {\mathcal M}^{(2)}( {\mathcal R}_k({\mathcal I}_n)) \lesssim \gamma \sum_{|k| \geq N_{n - 1}} \frac{1}{|k|^{3}}  \lesssim \gamma N_{n - 1}^{- 3}\,.
$$
 The proof of the lemma is then concluded. 
\end{proof}

\medskip

\begin{proof}
[\textsc{Proof of Proposition \ref{prop measure estimate finale}}]
It follows by the inclusion \eqref{prima inclusione stime misura}, the estimate \eqref{measure diophantine}, 
Proposition \ref{Gn Omegan Lambdan}, and the fact that the series $\sum_{n \geq 1} N_{n - 1}^{-3} < + \infty$ 
converges. 
\end{proof}

\section{Proof of Theorems \ref{thm:main}, \ref{thm:main2}}\label{sezione dim dei teoremi principali}

\medskip

\noindent{\sc Proof of Theorem \ref{thm:main2}.}

\noindent
Fix $\gamma := \lambda^{- \mathtt c}$, see \eqref{costanti nash moser}, \eqref{nash moser smallness condition}. Then 
$$
 \lambda^{- \frac{\delta}{3}} \gamma^{- 2} = \lambda^{- \frac{\delta}{3} + 2 \mathtt c} \ll 1
$$
by taking $\lambda \gg 1$ large enough, and $0 < \mathtt c < \frac{\delta}{6}$ which implies that the smallness condition \eqref{nash moser smallness condition} is fullfilled. We then set 
\begin{equation}
{\mathcal O}_\lambda := {\mathcal G}_\infty 
\end{equation}
where the set ${\mathcal G}_\infty$ is defined in \eqref{def cal G infty}. By Proposition \ref{prop measure estimate finale}, using that $\gamma = \lambda^{- \mathtt c}$, one has that 
$$
{\mathcal M}^{(2)}( {\mathcal O} \setminus {\mathcal O}_\lambda) \lesssim \lambda^{- \mathtt c} \to 0 \quad \text{as} \quad \lambda \to + \infty\,. 
$$
Moreover, by Proposition \ref{iterazione-non-lineare}-$({\mathcal P}1)_n$, using a telescoping argument, for any $\omega \in {\mathcal G}_\infty = {\mathcal O}_\lambda$,
the sequence $({\mathcal I}_n(\cdot; \omega))_{n \geq 0} \subset H^{s_0 + \overline \sigma}_0(\T^2, \R^2)$ converges to ${\mathcal I}_\infty(\cdot; \omega) \in H^{s_0 + \overline \sigma}_0(\T^2, \R^2)$ 
with respect to the norm $\| \cdot  \|_{s_0 + \overline \sigma}^{\Lip(\gamma)}$, and 
\begin{equation}\label{conv vn v infty}
\| {\mathcal I}_\infty \|_{s_0 + \overline \sigma}^{\Lip(\gamma)} \leq C, \quad \| {\mathcal I}_\infty - {\mathcal I}_n \|_{s_0 + \overline \sigma}^{\Lip(\gamma)} \lesssim N_{n}^{2 \overline \sigma } N_{n- 1}^{-\mathtt a} \lambda^{-  \delta M}, \quad \forall n \geq 0\,. 
\end{equation}

\noindent
By Proposition \ref{iterazione-non-lineare}-$({\mathcal P}2)_n$, for any $\omega \in {\mathcal O}_\lambda$, ${\mathcal F}({\mathcal I}_n) \to 0$ as $n \to \infty$, therefore, the estimate \eqref{conv vn v infty} implies that 
$$
\begin{aligned}
& {\mathcal F}\Big(\Omega(\cdot; \omega), J(\cdot; \omega) \Big) = 0 \quad \forall \omega \in {\mathcal O}_\lambda \quad \text{where}  \\
& \Big(\Omega(\cdot; \omega), J(\cdot; \omega) \Big) := {\mathcal I}_\infty(\cdot; \omega), \quad \omega \in {\mathcal O}_\lambda\,. 
\end{aligned}
$$ 
It remains only to prove the lower bound on $(\Omega, J)$
\begin{equation}\label{lower bound cal I infty}
\inf_{\omega \in {\mathcal O}_\lambda}\| {\mathcal I}_\infty(\cdot; \omega) \|_{s_0 + \overline \sigma} \geq \inf_{\omega \in {\mathcal O}_\lambda}\| \Omega(\cdot; \omega) \|_{s_0 + \overline \sigma} \gtrsim \lambda^{ - \frac23 \delta}\,. 
\end{equation}
Indeed, by the estimate \eqref{stima soluzione approssimata nash moser} and by \eqref{conv vn v infty} for $n = 0$ (recall that ${\mathcal I}_0 = {\mathcal I}_{app} = (\Omega_{app}, 0)$), one has that there is a constant $C_1 > 0$ such that 
\begin{equation}\label{stima Omega Omega app lambda - delta}
\begin{aligned}
 \| {\mathcal I}_{app} \|_{s_0 + \overline \sigma}& = \| \Omega_{app} \|_{s_0 + \overline \sigma}  \geq C_1 \lambda^{ - \frac23\delta}\,,  \\
 \| \Omega - \Omega_{app} \|_{s_0 + \overline \sigma} + \| J \|_{s_0 + \overline \sigma} & \leq \| {\mathcal I}_\infty - {\mathcal I}_{app} \|_{s_0 + \overline \sigma}  \lesssim  \lambda^{-  \delta M} {\lesssim} \lambda^{- 2 \delta}  
 \end{aligned}
\end{equation}
for $\lambda \gg 1$ large enough where we recall that by \eqref{costanti nash moser}, \eqref{choice tau}, $M \geq 2$. 
 Therefore 
$$
\begin{aligned}
\| {\mathcal I}_\infty \|_{s_0 + \overline \sigma} & \geq \| \Omega \|_{s_0 + \overline \sigma} \geq \| \Omega_{app} \|_{s_0 + \overline \sigma} - \| \Omega - \Omega_{app} \|_{s_0 + \overline \sigma} \\
&  \gtrsim  \lambda^{ - \frac23 \delta}  -  \lambda^{- 2 \delta} \gtrsim \lambda^{- \frac23 \delta}
\end{aligned}
$$
for $\lambda \gg 1$ large enough, uniformly w.r. to $\omega \in {\mathcal O}_\lambda$. The latter estimate then implies the claimed bound \eqref{lower bound cal I infty} and also that clearly $\Omega \neq 0$. We also prove that $J \neq 0$. Indeed, assume that by contradiction $J = 0$. Then ${\mathcal F}(\Omega, 0) = 0$, implying that the second component in \eqref{mappa nonlineare iniziale} becomes
\begin{equation}\label{contraddizione bla}
\mathbf b \cdot \nabla \Omega  \equiv 0 \quad \text{is identically zero.}
\end{equation}
On the other hand, by \eqref{non degeneracy Omega app} and by the estimate \eqref{stima Omega Omega app lambda - delta}, using that $s_0 + \overline \sigma \geq 1$, one has that 
$$
\begin{aligned}
& \| \mathbf b \cdot \nabla \Omega_{app} \|_{L^2} \geq \lambda^{- \frac23 \delta} \mathtt K(f, \mathbf b)\,,  \\
& \| \mathbf b \cdot \nabla \Omega - \mathbf b \cdot \nabla \Omega_{app} \|_{L^2} \lesssim \| \Omega - \Omega_{app} \|_{1} \lesssim \| \Omega - \Omega_{app} \|_{s_0 + \overline \sigma} \leq {\mathtt C} \lambda^{- 2 \delta}
\end{aligned}
$$ 
therefore 
$$
\begin{aligned}
\| \mathbf b \cdot \nabla \Omega \|_{L^2} & \geq \| \mathbf b \cdot  \nabla \Omega_{app} \|_{L^2} - \| \mathbf b \cdot \nabla \Omega - \mathbf b \cdot \nabla \Omega_{app} \|_{L^2} \\
& \geq  \mathtt K(f, \mathbf b) \lambda^{- \frac23 \delta} - {\mathtt C} \lambda^{- 2 \delta} \geq \dfrac{\mathtt K(f, \mathbf b)}{2} \lambda^{- \frac23 \delta}
\end{aligned}
$$
by taking $\lambda^{\frac43\delta} \geq \dfrac{\mathtt K(f, \mathbf b)}{2 \mathtt C}$. The latter lower bound then implies that $\mathbf b \cdot \nabla \Omega$ is not identically zero and this is a contraddiction with \eqref{contraddizione bla}. The proof of Theorem \ref{thm:main2} is then concluded.

\bigskip

\noindent
{\sc Proof of Theorem \ref{thm:main}.}  We deduce Theorem \ref{thm:main} from Theorem \ref{thm:main2}. Let $(\Omega(\cdot; \omega), J(\cdot; \omega)) \in H^S_0 (\T^2, \R^2)$, $S \gg 0$, $\omega \in {\mathcal O}_\lambda$ be the solution of ${\mathcal F}\Big( \Omega(\cdot; \omega), J(\cdot; \omega) \Big) = 0$ constructed in Theorem \ref{thm:main2} and satisfying the estimate \eqref{stima Omega J main}. Then, by \eqref{eq:vmhd}, \eqref{equazione P}, \eqref{riscalamento}, \eqref{mappa nonlineare iniziale}
$$
\begin{aligned}
& U := \lambda^{\delta} (- \Delta)^{- 1} \nabla^\bot  \Omega\,, \quad B := \lambda^{\delta} (- \Delta)^{- 1} \nabla^\bot J\,, \\
&  P :=  \Delta^{- 1} \Big( \lambda^{1 + \eta} {\rm div} f+\dive [(B\cdot\nabla)B]-\dive[(U\cdot\nabla)U] \Big), \quad \delta = 3 \eta 
\end{aligned}
$$
solves the system \eqref{eq:mhd 2}. By the form of the Fourier multiplier $(- \Delta)^{- 1} \nabla^\bot$, one clearly has that
\begin{equation}\label{prop U B S + 1}
\begin{aligned}
& U, B \in H^{S + 1}_0(\T^2, \R^2) \quad \text{and} \\
& \| U \|_{s + 1} \simeq \lambda^\delta \| \Omega \|_s, \quad \| B \|_{s + 1} \simeq \lambda^\delta \| J \|_s, \quad \forall 0 \leq s \leq S\,. 
\end{aligned}
\end{equation}
Hence, the latter estimate, together with the second estimate in \eqref{stima Omega J main} and the interpolation estimate \eqref{interpolazione bassa alta} (and $\delta = 3 \eta$) imply that 
$$
\begin{aligned}
  \| U \|_{S + 1}&  \lesssim_S \lambda^{3 \eta}\,, \quad \| B \|_{S + 1}  \lesssim_S \lambda^{- 3 \eta} \quad \text{and} \\
  \| P \|_{S + 1}  & \lesssim \lambda^{1 + \eta} \| f \|_{S } + \| B \cdot \nabla B \|_{S} + \| U \cdot \nabla U \|_{S} \\
 & \lesssim_S \lambda^{1 + \eta} + \| B \|_{S + 1}^2 + \| U \|_{S + 1}^2 \\
 & \lesssim_S \lambda^{1 + \eta} + \lambda^{6 \eta} \stackrel{0 < \eta \ll 1}{\lesssim_S} \lambda^{1 + \eta}\,. 
\end{aligned}
$$
The upper bound for $U, B, P$ in \eqref{stima U B P main} then follows. We now prove the lower bounds. By the estimate \eqref{prop U B S + 1} and by the first estimate in \eqref{stima Omega J main}, one has that 
$$
\begin{aligned}
\| U \|_{S + 1} + \| B \|_{S + 1} & \geq \| U \|_{S + 1} \simeq \lambda^\delta  \| \Omega \|_S  \gtrsim_S \lambda^\delta \lambda^{- \frac23 \delta} \stackrel{\delta = 3 \eta}{\gtrsim_S} \lambda^\eta\,. 
\end{aligned}
$$
Now by assuming that $\dive f \neq 0$, using that $f \in H^S_0(\T^2, \R^2)$, $\dive f \in H^{S - 1}_0(\T^2, \R^2)$, by the estimate \eqref{interpolazione bassa alta}, using that $\| U \|_{S + 1} \lesssim_S \lambda^{3 \eta}\,,\, \| B \|_{S + 1} \lesssim_S \lambda^{- 3 \eta} \lesssim_S 1$, 
$$
\begin{aligned}
\| P \|_{S + 1} & \gtrsim \lambda^{1 + \eta} \| \dive f \|_{S- 2} - \| U \cdot \nabla U\|_S - \| B \cdot \nabla B \|_S  \\
& \gtrsim_S \lambda^{1 + \eta} - \| U \|_{S + 1}^2 - \| B \|_{S + 1}^2 \\
& \gtrsim_S \lambda^{1 + \eta} - \lambda^{6 \eta} \stackrel{\lambda \gg 1, 0 < \eta \ll 1 }{\gtrsim_S} \lambda^{1 + \eta}\,. 
\end{aligned}
$$
The proof of the claimed theorem is then concluded. 

\newpage
\appendix

\end{document}